\documentclass[12pt, english]{amsart}

\usepackage{imperfectcremona}

\title[Explicit Sarkisov program for regular surfaces]{Explicit Sarkisov program for regular surfaces over arbitrary fields and applications}

\author[F. Bernasconi]{Fabio Bernasconi}
\address{Dipartimento di Matematica “Guido Castelnuovo”, SAPIENZA Università di Roma, 
	Piazzale Aldo Moro 5, I-00185 Roma, Italia}
\curraddr{}
\email{fabio.bernasconi@uniroma1.it}

\author[A. Fanelli]{Andrea Fanelli}
\address{Univ. Bordeaux, CNRS, Bordeaux INP, IMB, UMR 5251, F-33400 Talence, France}
\curraddr{}
\email{andrea.fanelli@math.u-bordeaux.fr}

\author[J. Schneider]{Julia Schneider}
\address{Université Bourgogne Europe, CNRS, IMB UMR 5584, 21000 Dijon, France}
\curraddr{}
\email{julia.schneider@ube.fr}

\author[S. Zimmermann]{Susanna Zimmermann}
\address{Institut für Mathematik und Informatik, Spiegelgasse 1, Basel, Schweiz}
\curraddr{}
\email{susanna.zimmermann@unibas.ch}

\subjclass[2020]{}
\thanks{FB was supported by the grants $\#200021/169639$ and PZ00P2-21610 from the Swiss National Science Foundation.
AF was supported by the ANR project ``FRACASSO'' ANR- 22-CE40-0009-01.
JS was supported by the ERC starting grant $\#804334$, as well as by a Postdoc Grant of the University of Zurich, Verf\"ugung Nr. FK-22-125.
SZ was supported by Universit\'e d'Angers, the Centre Henri Lebesgue, the ERC StG Saphidir 101076412 and the Institut Universitaire de France.}

\begin{document}

%===========================================================

\begin{abstract}
	We prove the Sarkisov program for projective surfaces over excellent base rings, including the case of non-perfect base fields $\k$ of characteristic $p>0$. We classify the Sarkisov links between Mori fibre spaces and their relations for regular surfaces, generalising work of Iskovskikh.
	As an application, we discuss rationality problems for regular surfaces and the structure of the plane Cremona group.
\end{abstract}

\maketitle
\tableofcontents

%===========================================================

\section{Introduction} \label{Introduction}

The main goal of this work is to develop the explicit birational geometry of regular surfaces over {\it arbitrary} fields, with special focus on del Pezzo surfaces and conic bundles over genus-0 curves, generalising the classification results of Iskovskikh \cite{Isk79,Isko80,Isk96}.
Our core motivation, beyond an ancestral pleasure for generalisation and understanding the Cremona group $\Bir(\p^2_\k)$ of birational transformations of the 2-dimensional projective space over {\it imperfect} fields, comes from 3-dimensional birational geometry over algebraically closed fields of positive characteristic. Imperfect fields naturally appear when studying 3-dimensional del Pezzo fibrations and conic bundles as function fields of their bases.

Modern birational geometry methods have been fruitfully applied in recent years to deduce several structural results on the algebraic structure of the Cremona group over algebraically closed fields \cite{DI09,Bla09} (see \cite{Ser10} for a survey), even in higher rank, and/or over arbitrary perfect fields.
\\The key theoretical inputs required to apply these techniques are the following:
\begin{enumerate}
\item[($\text{in}_1$)] the Minimal Model Program (in short, MMP) for $\k$-surfaces (or $\k$-varieties);
\item[($\text{in}_2$)] classification results for {\it regular} rational del Pezzo surfaces and conic bundles over algebraically closed fields (in higher dimension, one requires classification, or some sort of boundedness, of Fano fibrations);
\item[($\text{in}_3$)] the study of the action of the Galois group on the Picard group of del Pezzo surfaces and conic bundles (in higher dimension, Fano fibrations) to discuss the general case of a perfect field.
\end{enumerate}
Thanks to the MMP ($\text{in}_1$), one can reduce many group-theoretic questions about the plane Cremona group to the study of automorphism groups of minimal rational del Pezzo surfaces and conic bundles. Combining ($\text{in}_2$) and ($\text{in}_3$) one obtains, over any perfect field, a classification of those minimal surfaces together with elementary birational maps (i.e.\ Sarkisov links)  and the relations between them.
\\Very recently, several striking applications on the structure of $\Bir(\p^n_\k)$, when $\k$ is perfect, have been obtained in this way:
\begin{enumerate}
\item[($\text{out}_1$)] generators and relations in $\Bir(\p^2_\k)$ \cite{Zim18,LZ20,SchneiderRelations,LS21,SZ21};
\item[($\text{out}_2$)] study of normal subgroups of $\Bir(\p^n_{\overline\k})$, $\text{char}(k)=0$, $n\ge 3$, and its non-simplicity \cite{BLZ21,BSY};
\item[($\text{out}_3$)] classification of maximal connected smooth algebraic subgroups of $\Bir(\p^n_{\overline\k})$, $n\ge 3$, up to conjugacy \cite{BFT22,BF22,BFT23,FFZ24,FZ24}. A large part of those classification results have been proved only in characteristic zero.
\end{enumerate}

The non-simplicity of the plane Cremona groups over algebraically closed fields have also been obtained using hyperbolic geometry and geometric group theory in the seminal paper \cite{CL13}, later generalised to arbitrary fields by Lonjou in \cite{Lon16}.
Further results on the generators of higher-rank Cremona groups and their quotients using motivic techniques and median geometry can be found in the recent works \cite{SL, GLU}.

\subsection*{Explicit Sarkisov program}
The minimal model program for surfaces over arbitrary fields (or even excellent rings) has been established in \cite{Sha66, Tan18}. Our first result is a classical application of the MMP to the birational geometry of Mori fibre spaces in dimension 2: the Sarkisov program \cite{Cor95, Isk96, HM13, Kal13}, which allows to decompose birational maps between Mori fibre spaces into elementary building blocks known as Sarkisov links.

\begin{theoremA*}[cf.\ \autoref{t-sarkisov-program}]
Any birational map between Mori fibre spaces of dimension 2 over an excellent integrally closed domain $T$ is a composition of Sarkisov links and isomorphisms of Mori fibre spaces.
\end{theoremA*}

Let $X$ be a minimal regular surface defined over a {\it perfect} field $\k$. After base-change to an algebraic closure $\overline{\k}$, the surface  $X_{\bar\k}$ stays regular since, over a perfect field, being regular coincides with being geometrically regular. Nonetheless, the Picard rank may increase: in general, $\rho(X_\k)=\rho(X_{\overline{\k}})^{\Gal(\overline{\k}/\k)}$. This is why in ($\text{in}_2$) one needs a classification even for {\it non-minimal} del Pezzo surfaces and conic bundles over $\overline{\k}$. When working over {\it imperfect} fields, one has to deal with purely inseparable field extensions, which do not modify the Picard rank but can produce {\it singularities}: there exist examples of regular del Pezzo surfaces which become non-regular, sometimes even non-normal or non-reduced after an inseparable base change of the base field (cf.\ \cite{Sch07, Mad16}). In the literature, these phenomena are often denoted as {\it pathologies} in positive characteristic. Despite this terminology, these are exactly the peculiar features of the geometry in positive characteristic which deserve a more comprehensive study.
\\So, when working over imperfect fields, ($\text{in}_2$) requires an upgrade:
\begin{enumerate}
\item[($\text{in}_2^{\text{im}}$)] Bounds on pathologies for del Pezzo surfaces and conic bundles over imperfect fields,
\end{enumerate}
where ``$\text{im}$'' stands for {\it imperfect}.

Several works appeared recently (cf.\ \cite{FS20a,PW22,BT22,Tan19,BM23}), where the authors investigated these wild phenomena and obtain a pretty good control in dimension 2. Focusing on the rationality problem, we obtain a birational rigidity result for pathological minimal del Pezzo surfaces, which implies the following non-rationality statement.

\begin{theoremB*}[cf.\ \autoref{prop: birational-rigidity-non-norm-dP}]
Let $X$ be a regular del Pezzo surface over a field $\k$ of Picard rank $1$. If $X$ is not geometrically normal, then $X$ is not $\k$-rational.
\end{theoremB*}

Since we can only apply Galois descent ($\text{in}_3$) from the separable closure, to isolate and classify the new phenomena over imperfect fields coming from purely inseparable extensions we develop the following.

\begin{enumerate}
\item[($\text{in}_3^{\text{im}}$)] Explicit Sarkisov links centred on closed points with purely inseparable residue field extension between geometrically rational surfaces.
\end{enumerate}

The term {\it explicit} aims to stress how our approach is very direct, often manipulating equations, and as a by-product we construct several interesting examples in \autoref{s: general points}.

We obtain a classification of del Pezzo surfaces of Picard rank $2$ in terms of the extremal rays of its Mori cone, giving rise to a Sarkisov link (links of type I respectively III in \autoref{prop: link_I}; links of type II in \autoref{prop: link II dP}; links of type IV in \autoref{prop: links IV}).
Our main result is then the classification of rational Mori fibre spaces in dimension 2 together with Sarkisov links between them over {\it arbitrary} fields, generalising the results over perfect fields in \cite{Isk96}.
As a first result, we show that if a minimal regular surface $X$ is rational, then $K_X^2 \geq 5$ (\autoref{cor: rat_deg_5}). In the rational case, we then obtain

\begin{theoremC*} [cf.\ \autoref{thm: classification_rational_surfaces}]
	Let $\k$ be an arbitrary field. Then the following hold, where $X_d$ denotes a del Pezzo surface of degree $d$.
	\begin{enumerate}
		\item\label{crs:1} Any rational Mori fibre space is isomorphic to one of the following:
		\begin{enumerate}
			\item $\p^2_{\k}$;
			\item a regular quadric surface $X_8\subset\p^3_{\k}$ with $\rho(X_8)=1$;
			\item a del Pezzo surface $X_6$ with $\rho(X_6)=1$;
			\item a del Pezzo surface $X_5$ with $\rho(X_5)=1$;
			\item the Hirzebruch surfaces $\mathbb{F}_n\to\p^1_\k$, $n\geq0$;
			\item a Mori conic bundle $X_5\to\p^1_\k$, where $X_5$ is a del Pezzo surface of Picard rank 2;
			\item a Mori conic bundle $X_6\to\p^1_{\k}$, where $X_6$ is a del Pezzo surface equipped with a birational morphism $X_6\to X_8$.
		\end{enumerate}
	\item\label{crs:2} Any Sarkisov link between rational Mori fibre spaces is one of the following:
		\begin{enumerate}[label=(\alph*)]
		\item links of type \I and \III:
		\begin{enumerate}[label=(\roman*), leftmargin=15mm]
			\item the blow-up $\F_1\to\p^2_{\k}$ in a rational point;
			\item the blow-up $X_5\to\p^2_{\k}$ of a point of degree $4$;
			\item the blow-up $X_6\to X_8$ of a point of degree $2$;
		\end{enumerate}
		\item links of type II:
		\begin{enumerate}[label=(\roman*)]
			\item elementary transformations $\F_n\rat \mathbb{F}_m$;
			\item elementary transformations $X_d\rat X_d'$ of conic fibrations $X_d/\p^1_{\k}$, $X_d'/\p^1_{\k}$, $d\in\{5,6\}$;
			\item $\chi_{\xx,\yy}\colon X_d \rat X_{d'}$, where $\chi_{\xx,\yy}$ is centred in the closed points $\xx$ and $\yy$ and the degrees $d \geq 5$ and $d'$ appear in the list of \autoref{prop: link II dP}.
		\end{enumerate}
		\item links of type $\IV$: exchanging the fibrations on $\F_0=\p^1_{\k}\times\p^1_{\k}$.
		\end{enumerate}
	\end{enumerate}
\end{theoremC*}

This classification specialises to  the case of separably closed fields in \autoref{thm:classification links p>2}. Moreover, for each of the cases we construct links with purely inseparable base points starting from $\p^2_{\k}$ if $\k$ is imperfect (cf. \autoref{thm: existence of Sarkisov links} for a precise statement).

\subsection*{Applications}
Once the explicit Sarkisov program is established, one can study {\it elementary relations} between Sarkisov links or, in more geometric terms, rank 3 fibrations (cf.\ \autoref{def: elementary relations}).
This provides, over any field, the classification of the contractions associated to the extremal faces of the nef cones of del Pezzo surfaces of Picard rank at most 3.

\begin{theoremD*}[cf.\ \autoref{thm:elementary relations}]
Let $\k$ be an arbitrary field. Then any elementary relation corresponding to a rank $3$ fibration $T/\Spec(\k)$ is described in the Figures \ref{fig:XB_ab}--\ref{X4_12}.
\end{theoremD*}

Further applications of our results generalise several structural results on the plane Cremona group to arbitrary fields.
The first result is about explicit quotients, generalising \cite{LZ20} and \cite{SchneiderRelations}.

\begin{theoremE*}[cf.\ \autoref{thm: quotients_p2}]
	Let $\k$ be a field such that $[\overline{\k}:\k]\geq3$.
	There exists a surjective group homomorphism
\[
\phi\colon \Bir(\p^2_{\k})\to (\bigoplus_{\mathcal{N}_{[\F_0/\p^1]}} \Z/2\Z) \ast\left(\Conv_{M\in I_5}(\bigoplus_{\mathcal{N}_{M}} \Z/2\Z)\right) \ast\left(\Conv_{M\in I_6}(\bigoplus_{\mathcal{N}_{M}} \Z/2\Z)\right) \ast\left(\Conv_{\mathcal{B}}\Z/2\right),
\]
	where the index set $\mathcal{N}_{[\F_0/\p^1]}$ is countably infinite and the remaining index sets are described in \autoref{thm: quotients_p2}.
\end{theoremE*}

\iffalse
The second application generalises the main result of \cite{LS21} to separably closed fields, which in turn generalised the celebrated Noether-Castelnuovo theorem, which holds over algebraically closed fields, to perfect fields.

\begin{theoremF*} [cf.\ \autoref{thm:involution}]
Let $\k$ be a separably closed field of characteristic $p\geq3$. Then $\Bir_\k(\p^2_{\k})$ is generated by involutions.
\end{theoremF*}

\fi

The second application is a classification result for smooth connected algebraic subgroups of $\Bir(\p^2_{\k})$.

\begin{theoremF*} [cf.\ \autoref{thm: max_alg_subgroups}]
		Let $\k$ be a field and let $G$ be a smooth connected algebraic group acting rationally on $\mathbb{P}^2_\k$.
		Then there exists a $G$-birational map $\varphi \colon \mathbb{P}^2_\k \dashrightarrow X$ such that $X$ is a regular projective $G$-surface (equivalently, $\varphi G \varphi^{-1} \subset \Autz_{X/\k}$) in the following list:
		\begin{enumerate}
			\item\label{mas:1} $X \simeq \p^2_\k$;
			\item\label{mas:2} $X \simeq \mathbb{F}_n$ for $n \in \mathbb{N} \setminus \left\{1\right\}$;
			\item\label{mas:3} $X \simeq Q \subset \mathbb{P}^3_\k$ is a quadric (in particular, a del Pezzo of degree 8) of Picard rank 1;
			\item\label{mas:4} $X$ is a del Pezzo surface of degree 6 and Picard rank 1.
			\item\label{mas:5} $X$ is a del Pezzo surface of degree 5 and Picard rank 1 such that $X_{\bar\k}$ is not smooth. Such $X$ exist only if $p \leq 5$ and $\k$ is not perfect.
		\end{enumerate}
		Moreover, the conjugacy classes of the group scheme $\Aut^\circ_{X/\k}$ of cases \eqref{mas:1}--\eqref{mas:5} are all pairwise disjoint, and \eqref{mas:1}, \eqref{mas:2} are maximal among smooth connected subgroups and \eqref{mas:3} (resp. \eqref{mas:4}) is maximal among smooth connected subgroups if and only if $Q_{\overline\k}\simeq\p^1_{\overline\k}\times\p^1_{\overline\k}$ (resp. $X$ is smooth).
\end{theoremF*}

The third application is the extension of \cite{LS21} to separably closed fields of characteristic $p\geq3$.

\begin{theoremG*}[cf.\ \autoref{thm:involution}]
	Let $\k$ be a separably closed field of characteristic $p\geq3$. Then $\Bir(\p^2_{\k})$ is generated by involutions.
\end{theoremG*}

This paper is organised as follows. In \autoref{preliminaries} we fix the notation and recall known facts on purely inseparable extensions and the birational geometry of regular surfaces over excellent rings we will need throughout the article.
In \autoref{sec:sarkisov}, we establish the Sarkisov program, proving \autoref{t-sarkisov-program}.
The main classification results \autoref{prop: birational-rigidity-non-norm-dP} and \autoref{thm: classification_rational_surfaces} are obtained in \autoref{s: rank 1 fibrations} and \autoref{s: Sarkisov_Links} via an explicit study of linear systems and numerical bounds on regular del Pezzo surfaces and conic bundles.
\autoref{sec:Elementary relations} is devoted to the classification of all possible elementary relations between Sarkisov links, while in \autoref{s: general points} we produce examples of Sarkisov links by blowing-up closed points with inseparable residue fields. 
The applications to the plane Cremona group are then discussed in \autoref{s-applications}.

\subsection*{Acknowledgements}

The authors would like to thank J. Blanc, S. Lamy, G. Martin, Yu. Prokhorov, S. Schröer, E. Yasinsky, H. Tanaka and A. Trepalin for many fruitful and interesting discussion on the content of the article.
We thank M. Brion for discussions on the Weil regularisation theorem for regular surfaces.
Finally, we would like to thank the referee for careful reading of our article and suggesting several improvements.
This project was completed during a stay of the four authors at the Mathematisches Forschungs Institut Oberwolfach as Research Fellows. We thank the institute for the hospitality. We also thank the Basaglia family for their warm hospitality in Valle Liona, where an important part of this work was discussed and developed.

\section{Preliminaries} \label{preliminaries}

\subsection{Notations}

We fix some notations we will use throughout this article.

\begin{enumerate}
\item In this article, $p$ is a prime number and $\k$ denotes a field.  We denote by $\k^{\text{sep}}$ (resp.~ $\overline{\k}$) a separable (resp.~ an algebraic) closure of $\k$.  We indicate by $\k^{\rm perf}$ the perfect closure of $\k$, which coincides with $\k^{p^{-\infty}}$ in the case the characteristic of $\k$ is $p>0$.
\item \label{item:p-degree} We say that a field $\k$ of characteristic $p>0$ is \emph{$F$-finite} if $[\k:\k^{p}] < + \infty$. The $p$\emph{-degree} (or degree of imperfection) of an $F$-finite field $\k$ is $\pdeg(\k) \coloneq \log_p[\k:\k^p]$.
\item We say that $X$ is a $\k$-\emph{variety} (or simply variety) if $X$ is an integral scheme that is separated and of finite type over $\k$. We denote by $\k(X)$ its fraction field.
\item Given a $\k$-{variety} $X$, we use $\xx$ or $\yy$ for closed points: we will say \emph{points} for closed points, and if we ever use generic points we will mention it.
\item We say $(X, \Delta)$ is a \emph{pair} if $X$ is a normal variety, $\Delta$ is a $\mathbb{Q}$-divisor and $K_X+\Delta$ is $\mathbb{Q}$-Cartier.
\item For the definition of the singularities of pairs appearing in the MMP (such as terminal, canonical, klt or lc) we follow \cite[Section 2.1]{kk-singbook}.
\item We say a proper morphism $\pi \colon X \to Y$ of normal Noetherian schemes is a \emph{contraction} if $\pi_*\mathcal{O}_X=\mathcal{O}_Y$. In particular, $\pi$ has geometrically connected fibres.
\item Given a Noetherian scheme $X$, we denote its Picard group by $\Pic(X)$.
\item Let $\pi \colon X \to Y$ be a projective morphism of normal Noetherian integral schemes.
Let $\text{Num}(X/Y)$ be the subgroup of $\Pic(X)$ composed of line bundles $L$ such that $L \cdot C=0$ for every integral curve $C$ such that $\pi(C)$ is a closed point.
The group $N^1(X/Y) \coloneqq \Pic(X)/\text{Num}(X/Y)$ is a finitely generated abelian group and it is called the \emph{relative numerical Néron-Severi group}.
We denote by $\rho(X/Y)$ the relative Picard number of $\pi$, i.e. the rank of the $\mathbb{Z}$-module $N^1(X/Y)$.
\item For the notions of positivity for line bundles and $\mathbb{Q}$-Cartier divisors (such as ample, nef, big) we refer to \cite{Laz04}.
\item We say a $\k$-variety of dimension $n$ is \emph{rational} if its function field is purely transcendental of degree $n$, i.e. $\k(X) \simeq \k(t_1, \dots, t_n)$. Note that rational varieties are geometrically integral.
\item We denote by $\mathbb{F}_n$ the $n$-th Hirzebruch surface $\mathbb{P}_{\mathbb{P}^1}(\mathcal{O}_{\mathbb{P}^1} \oplus \mathcal{O}_{\mathbb{P}^1}(n))$.
\end{enumerate}

\subsection{Purely inseparable extensions}

Recall that a field $\k$ is {\em perfect} if every irreducible polynomial is separable. In particular, $\overline{\k}/\k$ is a Galois extension. Examples include all fields in characteristic $0$, all algebraically closed fields, and finite fields.
Here, we are interested in imperfect fields.

Let $\k$ be a field of characteristic $p>0$. Then $\k$ is imperfect exactly if $\k\neq \k^p=\left\{a^p\mid a \in \k \right\}$. Examples are $\F_p(t), \overline{\F_p}(t)$ and, more generally, function fields of varieties of dimension at least 1 over a positive characteristic ground field.

We say an extension $\k \subset L$ is \emph{purely inseparable} if for every element $a \in L$ there exists an integer $e>0$ such that $a^{p^e} \in \k$. In particular, a purely inseparable extension is an algebraic extension.
For example, $\F_p(t)\subset \F_p(t^{1/p})$ is purely inseparable but $\F_p(t)\subset \F_{p^2}(t)$ is not.

To a finite purely inseparable extension $\k\subset L$, we have $[L:\k]=p^e$ for some $e \geq 0$.
Finally, by \cite[\href{https://stacks.math.columbia.edu/tag/030K}{Tag 030K}]{stacks-project} any algebraic field extension $\k \subset L$ can be factorised as $\k \subset F \subset L$, where the extension $\k \subset F$ is separable and the extension $ F \subset L$ is purely inseparable. We say that an algebraic field extension $\k \subset L$ is inseparable if, in the above factorisation, $F \subsetneq L$.

The $p$-rank of a field extension is an invariant that intuitively measures the degree to which the extension deviates from being separable.

\begin{definition}\label{def: p-basis}
  Let $\k \subset L$ be a field extension.
  A subset $\{t_i\}$ of $L$ is called a \emph{$p$-basis} of $L$ over $\k$ if the elements of the form	$t_E= \prod t_i^{d_i} \text{ for } 0 \leq d_i \leq p-1,$ are a basis of $L$ as a vector space over $\k L^p$, the composite field of $\k$ and $L^p$ in $L$.
  The $\prank(L:\k)$ of a finite extension $\k \subset L$ is the dimension of a $p$-basis of $L$ over $\k$.
\end{definition}

Equivalently, this invariant can be computed using K\"{a}hler differentials from the formula $\prank(L:\k)=\dim_{L} \Omega_{L/\k}$ (a proof can be found in \cite[\href{https://stacks.math.columbia.edu/tag/07P2}{Tag 07P2}]{stacks-project}).
Note that the $p$-degree of a field $\k$ defined in \autoref{item:p-degree} coincides with $\prank(\k : \k^p)$.
Let's see some elementary examples where we compute the degree and the $p$-rank of purely inseparable extensions.

\begin{example}
	Let $\k=\mathbb{F}_p(t_1, \dots, t_n)$.
	This is imperfect as soon as $n \geq 1$ and it has $p$-degree equal to $n$.
	Let $l_1, \dots, l_n \geq 0$ be integers and consider the purely inseparable extension $L=\mathbb{F}_p(t_1^{1/p^{l_1}}, \dots, t_n^{1/p^{l_n}})$ of $\k$.
	In this case,
	\[ [L:\k]=p^{\sum_j l_j},\, \text{ and } \prank(L:k) = |\left\{j \mid l_j >0 \right\}|.\]
\end{example}

We are used that over perfect fields the singularities properties remain stable after a base field extension.
This is no longer true over imperfect fields as their finite field extension are no longer necessarily \'{e}tale.

\begin{definition}
	Let $X$ be a $\k$-scheme of finite type.
	We say $X$ is \emph{geometrically integral} (resp.~\emph{connected}, \emph{irreducible}, \emph{reduced}, \emph{normal} or \emph{regular}) if $X_{\overline{\k}}$ is integral (resp.~{connected}, {irreducible}, {reduced}, normal or {regular}).
\end{definition}

\begin{remark}
Suppose that $\k$ is a separably closed field. Then any closed point $\xx$ of a $\k$-variety $X$ is reduced and geometrically irreducible by \cite[\href{https://stacks.math.columbia.edu/tag/038I}{Tag 038I}]{stacks-project}.
Note that  $\xx$ is geometrically reduced as a $\k$-scheme if and only if the extension $\k \subset \k(\xx)$ is trivial.
\end{remark}

Recall that being geometrically regular and of finite type over a field is equivalent to smoothness over the field by \cite[\href{https://stacks.math.columbia.edu/tag/07EL}{Tag 07EL}]{stacks-project} and \cite[\href{https://stacks.math.columbia.edu/tag/02H6}{Tag 02H6}]{stacks-project}.
As we work over imperfect fields, it is not sufficient to compute the sheaf of K\"ahler differentials $\Omega^1_{X/k}$ to check whether a variety is regular.
Nevertheless, a Jacobian criterion for regular rings still holds true using general derivations.

\begin{lemma}[{Jacobian criterion, \cite[\href{https://stacks.math.columbia.edu/tag/07PF}{Tag 07PF}]{stacks-project}}]\label{lem: jacobian}
	Let $R$ be a regular ring and let $f \in R$.
	Assume there exists a derivation $D \colon R \to R$ such that $D(f)$ is a unit of $R/(f)$. Then $R/(f)$ is regular.
\end{lemma}

Using \autoref{lem: jacobian}, it is not difficult to construct regular varieties which are not geometrically normal nor reduced.

\begin{example}\label{ex:Examples-to-start}
	Let $\k$ be an $F$-finite imperfect field and let $t_1, \dots, t_n$ be a $p$-basis of $\k^p \subset \k$ for $n \geq 1$.
	\begin{enumerate}
		\item $p=2$: the conic $C_1=(t_1x^2+y^2+z^2=0) \subset \mathbb{P}^2_\k$ is integral, not regular at $[0:1:1]$ and hence not normal, and $C_{1,\overline{\k}}$ is a double line;
		\item $p=2$: the conic $C_2=(t_1x^2+t_2y^2+z^2=0) \subset \mathbb{P}^2_\k$ is regular by applying \autoref{lem: jacobian} with the derivations $D_{t_1}$ and $D_{t_2}$, and $C_{2, \overline{\k}}$ is a double line;
		\item for any prime $p$, the Fermat--type hypersurface $H_p=(\sum_{i=1}^{n} t_i x_i^p=0) \subset \mathbb{P}^{n-1}_\k$ is regular by  \autoref{lem: jacobian} and geometrically $(H_p)_{\overline{\k}}$ is a $p$-fold hyperplane;
		\item $p=2,3$: consider the cubic curve given by $Q=(zy^2=x^3-t_1z^3) \subset \mathbb{P}^2_\k$. It is a regular geometrically integral curve such that $Q_{\overline{\k}}$  is the cuspidal cubic curve;
		\item $p=2$: the quadric surface $X:=(xy+z^2+t_1w^2=0) \subset \mathbb{P}^3_{\k,[w:x:y:z]}$ is regular. The base change $X_{\overline{\k}}$ is normal with a unique $A_1$-singularity at $[1:0:0:t_1^{\frac{1}{2}}]$.
	\end{enumerate}
	For further examples, we refer to \autoref{s: rank 1 fibrations} and \autoref{s: general points} where such phenomena appear systematically.
\end{example}

\begin{remark}\label{rmk:BlowUpOfRegular}
	Let $X$ be a regular surface, and let $\eta\colon Y\to X$ be the blow-up at a point $\xx\in X$.
	Then $Y$ is regular by \cite[Theorem VII.1.19]{Liu02}.
	However:
	\begin{enumerate}
		\item If $X$ is smooth, then $Y$ is not necessarily smooth, see for instance
    \autoref{rmk:weak-dp} and \autoref{ex:fibration not geometrically regular}, \autoref{ex:link deg2}, \autoref{ex: Mori-conic-double-lines} \autoref{ex:geometrically-non-normal-pdeg1},  for explicit examples.
		\item If $X$ is geometrically normal, then $Y$ is not necessarily geometrically normal
    see for instance \autoref{ex: Mori-conic-double-lines}, \autoref{ex:geometrically-non-normal-pdeg1}  for explicit examples.
	\end{enumerate}
\end{remark}

The following bounds the degree of imperfection of smooth closed points.

\begin{lemma}\label{lem:p-deg<3}
  Let $X$ be an integral $\k$-variety such that $\k$ is integrally closed in $k(X)$.
  Let $\xx$ be a smooth closed point of $X$.
  Then $\prank(\k(\xx) : \k) \leq \dim(X)$.
\end{lemma}

\begin{proof}
	We can suppose $X=\Spec(R)$ to be affine.
	We can further suppose  that $R$ is a local ring with maximal ideal corresponding to $\xx$.
	In particular, we have the following map of rings
	$\k \hookrightarrow R \twoheadrightarrow \k(\xx) $, which induces an exact sequence \[\k(\xx) \otimes_R \Omega_{R/\k} \rightarrow \Omega_{\k(\xx)/\k} \rightarrow \Omega_{\k(\xx)/R}=0.\]
	As $X$ is smooth at $\xx$, $\Omega_{R/\k}$ is a locally free $R$-module of rank equal to $\dim(X)$, thus $\dim_{\k(\xx)}\Omega_{\k(\xx)/\k} \leq \dim_{\k(\xx)} \k(\xx) \otimes_R \Omega_{R/\k} =n$.
\end{proof}

\begin{remark}
	The hypothesis of $\xx$ being a smooth point is necessary.
	Consider $\k=\mathbb{F}_2(t_1,t_2)$ and let $\xx=(x^2-t_1, y^2-t_2)$ in $\mathbb{A}^2_\k=\Spec \k[x,y]$. The integral conic $C=(x^2-t_1=0)$ contains $\xx \in C$ and $\prank(k(\xx):\k)=2$. Note that $C$ is not smooth at $\xx$ as it is not geometrically reduced.
\end{remark}

\subsection{Birational geometry of excellent surfaces}\label{sec:BirationalgeometryofExcellentSurfaces}
In this section, we review some notions of birational geometry and the Minimal Model Program (for short, MMP) for excellent surfaces.
For the definition of excellent ring, we refer to \cite[\href{https://stacks.math.columbia.edu/tag/07QS}{Tag 07QS}]{stacks-project}.
Examples of excellent rings important for us are fields (possibly imperfect), Noetherian complete local ring (such as the $p$-adic numbers $\mathbb{Z}_p$ and formal power series $k\llbracket     t \rrbracket$), the ring of integers $\mathcal{O}_K$ of number fields $K$ and finite type ring extension thereof (see \cite[\href{https://stacks.math.columbia.edu/tag/07QW}{Tag 07QW}]{stacks-project}).

Let $X$ be an integral scheme admitting a quasi-projective morphism $X \to T=\Spec (R)$, where $R$ is an integrally closed excellent ring of finite Krull dimension admitting a dualizing complex.
We will say that $X$ is a {\em variety defined over $T$}.
Notice that the local rings of $X$ are excellent rings by  \cite[\href{https://stacks.math.columbia.edu/tag/07QW}{Tag 07QW}]{stacks-project} and we will say, with a slight abuse of terminology, that $X$ is an {\em excellent $T$-variety}.
If $X$ is moreover a surface and $X \to T$ is surjective, then $T$ has dimension at most $2$.
As in this article we are primarily interested in studying Mori fibre spaces, we assume that $\dim(T) \leq 1$. We summarise our setting in the following

\begin{notation} \label{notation-excellent}
	Throughout this section, $T$ denotes the spectrum of an integrally closed excellent domain $R$ of Krull dimension at most 1.
	We say $X$ is a (quasi-)projective $T$-variety if $X$ is an integral scheme together with a (quasi-)projective morphism $X \to T$.
	A $T$-morphism (or morphism over $T$) is a morphism $X\to Y$ between $T$-varieties $X$ and $Y$ satisfying the commutative diagram:
	\[
	\begin{tikzcd}
	X\ar[rr]\ar[rd]&&Y\ar[ld]\\
	&T&
	\end{tikzcd}
	\]
	In the same way, one defines a birational map $X\rat Y$ over $T$.
\end{notation}

The case of $\dim(T)=0$, that is when $R=k$ is a field, is especially important for us, as this includes the case of imperfect fields.
The case where $T$ has dimension $1$ includes the study of arithmetic surfaces, such as $\p^1_\Z$ and regular models of curves over a Dedekind domain.
The MMP and the abundance conjecture for surfaces that are projective over $T$ have been established in the classical sense in \cite{Sha66}  and in the log setting in \cite{Tan18, DW22, 7authors}.

We quickly review how intersection numbers and the degree of closed points are defined.
\begin{definition}\label{def:d_x etc}Let $X$ be a projective integral surface over $T$.
\begin{enumerate}
\item 	To a closed point $\xx \in X$ mapping to $t \in T$, we define the \emph{degree} of $\xx$ to be $d_\xx:=[k(\xx):k(t)]$.
	When $T=\Spec(\k)$, we have $d_\xx=[k(\xx):\k]$.
\item
	We say an integral subscheme $C \subset X$ is a \emph{curve} if $\pi(C)=t$, where $t\in T$ is a closed point.
	Given a Cartier divisor $L$ we define $L \cdot C := \deg_{k(t)}(L|_C) \in \mathbb{Z}$.
\item If $C$ is an integral curve inside a projective regular surface $X$ over $T$, we define $d_C:=[H^0(C, \mathcal{O}_C):\k(t)]$ and we have $C^2:=C \cdot  C = d_C m_C$ for some $m_C \in \mathbb{Z}$.
\end{enumerate}
\end{definition}

Let $X$ be a regular surface over $T$ and let $f \colon Y \to X$ be the blow-up of $X$ at a closed point $\xx$.
The exceptional divisor $E \simeq \mathbb{P}^1_{\k(\xx)}$ satisfies $E^2=-d_\xx=-d_E$. Moreover, we have $K_Y = f^*K_X+E$ and therefore $K_Y \cdot E = -d_\xx$.

\begin{definition}\label{def: curve of first kind}
	Let $X$ be a projective regular surface over $T$.
	We say that an integral curve $E \subset X$ is an \emph{exceptional curve of the first kind} if $E^2<0$ and $K_X \cdot E <0$.
\end{definition}

Exceptional curves of the first kind are a natural generalisation of the usual $(-1)$-curves (that is, $E\simeq\p^1_\k$ and $E^2=-1$) over algebraically closed fields. Over a perfect field $\k$, an exceptional curve of the first kind consists of a set of disjoint $(-1)$-curves over $\bar\k$ that forms a $\Gal(\bar\k/\k)$-orbit.
As we allow the finite extension $\k \subset \k(x)$ to be inseparable or to work over an arithmetic base, we cannot apply Galois theory and argue in term of Galois orbits.

The following lemma shows that exceptional curves of the first kind have an easy description (cf. \cite[\href{https://stacks.math.columbia.edu/tag/0C2I}{Tag 0C2I}]{stacks-project}).

\begin{lemma}\label{lem: characterise_-1_curves}
	Let $X$ be a regular projective surface over $T$ and $E$ a curve on $X$.
	Then $E$ is an exceptional curve of the first kind if and only if $E \simeq \mathbb{P}^1_{H^0(E, \mathcal{O}_E)}$ and $E^2=-d_E$.
\end{lemma}

\begin{proof}
	Let $t=\pi(E)$.
	By adjunction, we have $(K_X+E) \cdot E =\deg_{k(t)}(\omega_E)$.
	If $E$ is an exceptional curve of the first kind, we have $\deg_{k(t)}(\omega_E)<0$ and therefore $\deg_{k(t)}(\omega_E)=-2d_E$ by \cite[Lemma 10.6.3]{kk-singbook} and by Serre duality $H^1(E, \mathcal{O}_E)=0$.
	As $\deg_{k(t)}\mathcal{O}_E(E)=-d_E$, we conclude by \cite[Lemma 10.6.4]{kk-singbook} that $E \simeq \mathbb{P}^1_{H^0(E, \mathcal{O}_E)}$.
	The other direction follows immediately using the adjunction formula.
\end{proof}

We recall the Castelnuovo contraction theorem for excellent regular surfaces.

\begin{proposition} \label{lem: Castelnuovo_contraction}
	Let $X$ be a regular projective surface over $T$ and let $E \subset X$ be an exceptional curve of the first kind.
	Then
	\begin{enumerate}
		\item there exists a birational contraction $f \colon X \to Y$ over $T$ such that the exceptional locus of $f$ coincides with $E$ and the closed point $\xx := f_*E$ has residue field $k(\xx) \simeq H^0(E, \mathcal{O}_E)$;
		\item $Y$ is regular.
	\end{enumerate}
	Moreover, any projective birational morphism $\pi \colon X \to Y$ of regular surfaces is a composition of contractions of exceptional curves of the first kind.
\end{proposition}

\begin{proof}
	This is proven in \cite[\href{https://stacks.math.columbia.edu/tag/0C2K}{Tag 0C2K}]{stacks-project} and \cite[\href{https://stacks.math.columbia.edu/tag/0C5S}{Tag 0C5S}]{stacks-project}.
\end{proof}

\begin{corollary}
	Any birational map  $\psi \colon X \rat Y$ between regular projective surfaces over $T$ can be factorised as a sequence of blow-ups and blow-downs of closed points.
\end{corollary}

\begin{proof}
	Let $W$ be a resolution of indeterminacies of $\psi$.
	By \cite{Lip78}, there exists a resolution of singularities $Z \to W$, which naturally admits projective birational morphisms to both $X$ and $Y$.
	Therefore we conclude by \autoref{lem: Castelnuovo_contraction}.
\end{proof}

\section{Sarkisov program for excellent surfaces}\label{sec:sarkisov}

Birational maps between Mori fibre spaces are in general hard to control.
The Sarkisov program aims to decompose birational maps between Mori fibre spaces into so-called Sarkisov links, which are more controllable type of birational maps.
It was proven in dimension two over arbitrary perfect fields in \cite[Appendix]{Cor95} by Corti and in \cite{Isk96} by Iskovskikh after an idea of Sarkisov. Both proofs are algorithmic, although the algorithm may be hard to execute explicity in general.
In higher dimension, Hacon-M$^{\text{c}}$Kernan proved in \cite{HM13} that any birational map between Mori fibre spaces over an algebraically closed field of characteristic zero is a composition of Sarkisov links. Their proof is non algorithmic and instead relies on the MMP as developed in \cite{BCHM10}.
Their proof was replicated in dimension two over perfect fields in \cite{LZ20}.

In this section, we give a proof of the Sarkisov program as stated in \cite{BLZ21} for excellent surfaces.

Throughout this section, $T$ is a scheme as in \autoref{notation-excellent}.
First, we recall the MMP for excellent surfaces and then recall the basic notions of the Sarkisov program as introduced in \cite{LZ20, BLZ21}.

\subsection{The MMP for excellent surfaces}
The MMP in the case where $X$ is a \emph{regular} surface is established in \cite{Sha66} (see \cite{Liu02} for a modern treatment).
For the proof of the Sarkisov program, we need the version of the MMP for klt surface pairs developed by Tanaka \cite{Tan18}.
From the point of view of higher dimensional MMP, recall that terminal surface singularities are regular \cite[Theorem 2.29]{kk-singbook}.

We will need the following basic result on push--forwards of ample divisors on surfaces multiple times.

\begin{lemma} \label{lem: push-forward-ample}
	Let $f \colon X \to Y$ be a proper birational morphism of projective $\mathbb{Q}$-factorial surfaces over $T$.
	If $H$ is an ample $\mathbb{Q}$-divisor on $X$, then $f_*H$ is ample.
\end{lemma}

\begin{proof}
Without loss of generality, we can suppose that $H$ is effective and thus we have $f^*f_*H =H+F$, where $F$ is effective.
First, we note by projection formula that $f_*H \cdot f_* H=f^*f_*H \cdot H =(H+F)\cdot H>0$, as $H$ is ample.
Let $C$ be an effective curve on $Y$. By projection formula again, we have $f_*H \cdot C=H \cdot f^*C>0$. Thus we conclude that $f_*H$ is ample by the Nakai--Moishezon criterion of ampleness.
\end{proof}

We now recall the general version of the cone theorem and the base point free theorem for klt surface pairs proven in \cite{Tan18}.
Recall that klt surface singularities are rational and $\mathbb{Q}$-factorial by \cite[Proposition 2.28 and 10.9]{kk-singbook}.

\begin{theorem}[{Cone theorem, \cite[Theorem 2.46]{7authors}}]
\label{t-cone-theorem}
Let $X$ be a quasi-projective surface over $T$ admitting a projective morphism $\pi \colon X \to B$ over $T$ to a quasi-projective variety $B$ over $T$ and suppose $(X,\Delta)$ is a klt pair.
Then there is a countable set of curves $(C_i)_{i \in I}$ such that
\begin{enumerate}
	\item $K_X \cdot C_i<0$ and $\pi(C_i)$ is a closed point;
	\item the cone theorem holds:  $$\overline{\NE}(X/T)=\overline{\NE}(X/T)_{K_X+\Delta \geq 0}+\sum_{i\in I} \mathbb{R}_{+}[C_i];$$
	\item if $A$ is a $\pi$-ample $\mathbb{R}$-divisor, then there is a finite index $I_A \subset I$ such that $$\overline{\NE}(X/T)=\overline{\NE}(X/T)_{K_X+\Delta +A \geq 0}+\sum_{i\in I_A} \mathbb{R}_{+}[C_i].$$
\end{enumerate}
\end{theorem}

We recall a special case of the cone theorem when the relative dimension is at most $1$:

\begin{lemma}[{Relative Cone Theorem, \cite[Lemma 2.13]{Tan18}}]
\label{l-cone-theorem}
Let $X$ be a quasi-projective surface over $T$ and $\pi\colon X\to B$ a projective $T$-morphism to a quasi-projective variety $B$ over $T$ with $\dim \pi(X)\geq1$.
Then there is a finite set of projective curves $(C_i)_{i \in I}$ such that
	\begin{enumerate}
		\item $\pi(C_i)$ is a closed point;
		\item $\overline{\NE}(X/B)=\sum_{i\in I}\R_{\geq0}[C_i]$;
	\end{enumerate}
\end{lemma}

\begin{theorem}[{Base point free theorem}] \label{t-bpf-theorem}
Let $X$ be a quasi-projective surface over $T$ admitting a projective morphism $\pi \colon X \to B$ over $T$ to a quasi-projective variety $B$ over $T$ and suppose $(X,\Delta)$ is a klt pair.
Let $L$ be a $\pi$-nef $\mathbb{Q}$-Cartier divisor such that $L-(K_X+\Delta)$ is $\pi$-big and $\pi$-nef. Then $L$ is semiample.
\end{theorem}

\begin{proof}
In the case where $Z$ is the aﬃne spectrum of a field of positive characteristic
for the Stein factorisation $X \to Z \to B$ and $L \equiv 0$, the statement follows from \cite[Theorem 1.3]{BT22}. The remaining cases are proven in \cite[Theorem 4.2]{Tan18}. 
\end{proof}

The cone and the contraction theorem then permit to run the MMP for excellent klt surfaces.
We recall the definition of Mori fibre space in dimension 2.

\begin{definition}\label{def-mfs}
	Let $X$ be a quasi-projective surface over $T$.
	We say a projective contraction $\pi\colon X\to B$ over $T$ is a \emph{Mori fibre space} if $X$ is a regular surface, $\rho(X/B)=1$ and $-K_X$ is $\pi$-ample.
\end{definition}

\begin{theorem}[{Minimal model program, \cite[Theorem 1.1]{Tan18}}] \label{t-mmp-excellent}
Let $X$ and $B$ be quasi-projective varieties over $T$ and
$\pi \colon X \to B$ be a projective morphism over $T$.
Suppose $(X,\Delta)$ is a klt surface pair.
Then we can run a $(K_X+\Delta)$-MMP over $B$ which ends either with a Mori fibre space or with a good minimal model \footnote{We refer to \cite[Theorem 1.1]{Tan18} for details on good minimal models.}.
\end{theorem}

\subsection{Rank $r$ fibrations, Sarkisov links and elementary relations}

We start with defining rank $r$ fibrations for projective regular surfaces $X$ over $T$. By taking the Stein factorisation, we can assume that the morphism $X \to T$ is a projective contraction.

We will show in this and the next section that in the context of regular surfaces over perfect fields, these notions are equivalent to the ones introduced in \cite{Isk96,Cor95}. See also \cite{Kal13} for a slightly different notion of rank $3$ fibrations leading to generating relations of Sarkisov links.

\begin{definition}\label{def: rank_r_fibration}
	Let $X$ be a projective surface over $T$.
	We say that a projective morphism $\pi \colon X \to B$ is a \emph{rank $r$ fibration} if $\pi$ is a contraction between normal varieties over $T$ and
	\begin{enumerate}
		\item \label{item:decrease dimension-Picard} $\dim B <\dim X$ and $\rho(X/B)=r$;
		\item \label{item: regularity} $X$ is regular;
		\item  \label{item: Fano} $-K_X$ is $\pi$-ample.
	\end{enumerate}
\end{definition}

Notice that in the case $r=1$, $\pi$ is a Mori fibre space (cf.\ \autoref{def-mfs}).

\begin{lemma} \label{lem: rank_r_MDS}
	Let $\pi \colon X \rightarrow B$ be a rank $r$ fibration over $T$. Then $X$ is a Mori dream space over $B$, i.e. for any Weil divisor $D$ we can run a $D$-MMP over $B$ which will terminate with $f \colon X \to Y$.
	Moreover, $Y$ is regular and $Y \rightarrow B$ is a rank $r'$ fibration, for some $r' \leq r$.
\end{lemma}

\begin{proof}
	%Let $n \gg 0$ such that there exists an ample $\mathbb{Q}$-divisor $H$ such that $D-nK_X \sim_{\mathbb{Q}}H$ and $(X, \frac{1}{n}H)$ is klt.
	Choose a suﬃciently large natural number $n>0$ such that $A=\frac{1}{n}D-K_X$ is ample. By the existence of log resolution for excellent surfaces, we can repeat the proof of \cite[Lemma 2.8]{GNT19} and show that there exists an eﬀective $\mathbb{Q}$-divisor $H$ such that $H \sim_\mathbb{Q} A$ and $(X, H)$ is a klt pair.
	By \autoref{t-mmp-excellent}, as $\frac{1}{n}D \sim_{\mathbb{Q}} (K_X+H)$, we can run a $D$-MMP which terminates with $f \colon X \to Y$.
	As $H$ is ample, by \autoref{lem: push-forward-ample} $f$ is a composition of steps of $K_X$-MMP and therefore $Y$ is regular as well.
	As $-K_Y=f_*(-K_X)$, we conclude by \autoref{lem: push-forward-ample} that $Y$ is a rank $r'$ fibration over $B$.
\end{proof}

\begin{remark} \label{rem: compare_BLZ21}
	As a consequence of \autoref{lem: rank_r_MDS}, we note that the definition of rank $r$ fibration introduced in \cite{BLZ21} specialises to \autoref{def: rank_r_fibration} in the case of surfaces.

	Note that as $\dim(B) \leq 1$, the condition (RF4) of \cite[Definition 3.1]{BLZ21} is automatic and \autoref{item:decrease dimension-Picard} coincides with (RF2).

	We are left to verify that \autoref{item: regularity} and \autoref{item: Fano} are equivalent to (RF1), (RF3) and (RF5).
	One direction is proven in \autoref{lem: rank_r_MDS}.
	For the other direction, we argue by contradiction.
	Suppose that (RF1), (RF3) and (RF5) hold and that $X$ is regular and $-K_X$ is big over $B$ but not ample. Then there exists an integral curve $C$ such that $-K_X \cdot C \leq 0$.
	As $-K_X$ is big, then $C^2<0$.
In particular, by \autoref{t-mmp-excellent} we can run an MMP for $(X, \varepsilon C)$ for $\varepsilon$ sufficiently small and in particular there exists a birational contraction $\pi \colon X \to X'$ such that $\Exc(\pi)=C$.
	As $-K_X \cdot C \leq 0$, the surface $X'$ is not terminal, thus contradicting (RF3).
\end{remark}

Before we define Sarkisov link in terms of rank $2$ fibrations in \autoref{def: Sarkisov_link}, we first generalize the classical notion of Sarkisov links over a perfect field as in \cite{Cor95, Isk96} to Sarkisov links over an excellent ring, and will observe that both definitions are equivalent.

Following \cite{Cor95, Isk96}, a birational map $\chi \colon X_1 \dashrightarrow X_2$ between two $T$-surfaces is a \emph{Sarkisov link} if $X_i$ admits a Mori fibre space structure to some base $B_i$ for $i=1,2$ fitting into a commutative diagram of type \I, \II, \III, or \IV, as shown in Figure~\ref{fig:SarkisovTypes}, where $fib$ denotes a Mori fibre space, $div$ denotes a birational contraction of regular surfaces, and each non-horizontal morphism has relative Picard rank $1$.

\begin{figure}[ht]\label{figure-4-links}
\[
{
\def\arraystretch{2.2}
\begin{array}{cc}
\begin{tikzcd}[ampersand replacement=\&,column sep=1.3cm,row sep=0.16cm]
 \&\& Z=X_2 \ar[dd,"fib"]\ar[ddll,"div"] \\ \\
X_1\ar[uurr,"\chi",dashed, bend left=20]  \ar[dr,"fib",swap] \&  \& B_2 \ar[dl] \\
\& B_1=T \&
\end{tikzcd}
&
\begin{tikzcd}[ampersand replacement=\&,column sep=1.3cm,row sep=0.18cm]
 \& Z\ar[ddl,"div",swap]\ar[ddr,"div"]\& \\ \\
X_1 \ar[rr,"\chi",dashed] \ar[d,"fib",swap] \&  \& X_2 \ar[d,"fib"] \\
B_1 \ar[rr,"\simeq"]\&  \& B_2
\end{tikzcd}
\\
\text{type }\I& \text{type }\II
\\
\begin{tikzcd}[ampersand replacement=\&,column sep=1.3cm,row sep=0.16cm]
X_1=Z \ar[ddrr,"\chi=div"] \ar[dd,"fib",swap] \&\& \\ \\
B_1 \ar[dr] \& \& X_2 \ar[dl,"fib"] \\
\&  T=B_2  \&
\end{tikzcd}
&
\begin{tikzcd}[ampersand replacement=\&,column sep=1.7cm,row sep=0.16cm]
X_1=Z \ar[rr,"\simeq"] \ar[dd,"fib",swap]  \&\& X_2 \ar[dd,"fib"] \\ \\
B_1 \ar[dr] \& \& B_2 \ar[dl] \\
\& T \&
\end{tikzcd}
\\
\text{type }\III & \text{type }\IV
\end{array}
}
\]
\caption{The four types of Sarkisov links in dimension $2$.}
\label{fig:SarkisovTypes}
\end{figure}

\begin{lemma}[{\cite[p.250]{Cor95}}]\label{lem:Sarkisov links=rank 2 fibration}
The surfaces over $T$ on the top row of each diagram in \autoref{fig:SarkisovTypes} are rank $2$ fibrations above the variety on the bottom row of the diagram.
\end{lemma}
\begin{proof}
Let $Z$ be a surface on the top row of the diagram and $B$ the variety at the bottom row of the diagram. As each non-horizontal arrow has relative Picard rank $1$, we have that $\rho(Z/B)=2$. So we only have to show that $-K_Z$ is relatively ample over $B$.
As there are at least two extremal contractions from $Z$ over $B$ (divisorial contractions and/or Mori fibre spaces) and $\rho(Z/B)=2$, there are exactly two $K_Z$-negative extremal contractions from $Z$ over $B$ and $\overline{\text{NE}}(Z/B)=\R_{\geq 0} [f_1] + \R_{\geq 0} [f_2]$, where $f_1,f_2$ are extremal curves.
Then $-K_{Z}\cdot f_i>0$ for $i=1,2$, and we conclude by Kleiman's criterion that $-K_{Z}$ is relatively ample over $B$.
\end{proof}

Conversely, given a rank $2$ fibration $\pi\colon Z\to B$, the cone theorem \autoref{l-cone-theorem} implies that the Mori cone $\overline{\NE}(Z/B)=\NE(Z/B)$ is spanned by two extremal rays generated by effective curves. Moreover, by the base point free theorem \autoref{t-bpf-theorem} all nef divisors are semiample.
Therefore we can play the $2$-ray game on $Z$ over $B$  (see~\cite[Section~2.F]{BLZ21}), which induces a commutative diagram over $T$
\[
	\begin{tikzcd}\label{diag:SarkisovLink}
		& Z\ar[ld,"f_1",swap]\ar[rd,"f_2"]\ar[ddd,"\pi", bend left=20] & \\
		X_1\ar[d,"\pi_1"]\ar[rr,dashed,"\chi"] & & X_2\ar[d,"\pi_2"]\\
		B_1\ar[dr] && B_2\ar[dl]\\
		&B&
	\end{tikzcd}
\]
such that $\pi_i\colon X_i\to B_i$ are Mori fibre spaces.
The case when $f_1$ is a divisorial contraction and $f_2$ an isomorphism (or vice-versa) is precisely when $f$ is a link of type \I (resp. \III).
If both $f_1$ and $f_2$ are divisorial contractions, then $f$ is a link of type \II.
Finally, if both $f_1$ and $f_2$ are isomorphisms, then $f$ is a link of type \IV.
As we suppose that $X_1 \to T$ is a contraction, the links of type \I, \III, \IV appear only when $\dim(T)=0$ by dimension reasons.

From the previous discussion, we can give the following concise definition of Sarkisov link.

\begin{definition}[rephrased]\label{def: Sarkisov_link}

	A birational map $\chi\colon X_1\dashrightarrow X_2$ is called a \emph{Sarkisov link} if it fits into a commutative diagram obtained from the $2$-ray game of a rank $2$ fibration $Z/B$, as described above.
\end{definition}

Note that if $\chi \colon X_1 \rat X_2$ is a Sarkisov link $\chi$ coming from a rank $2$ fibration $Z/B$, then also its inverse $\chi^{-1} \colon X_2 \rat X_1$ is a Sarkisov link coming from $Z/B$. Moreover, any composition of $\chi$ with isomorphisms at source and target are again Sarkisov links coming from $Z/B$.

\begin{definition} \label{def: birmori}
Let $X$ be a $T$-surface admitting a Mori fibre space structure $X/B$. Let
$$\BirMori_T(X) \coloneqq
\{f\colon X_1\rat X_2 \text{ birational} \mid \text{$X_i/B_i$ Mori fibre spaces birational to $X$}\}$$
be the groupoid of birational transformations of Mori fibre spaces birational to $X$ over $T$.
Note that $\BirMori_T(X)$ contains the groups $\Bir_T(X)$ and $\Aut_T(X)$.
\end{definition}
The Sarkisov program gives a description of the groupoid $\BirMori_T(X)$ in terms of generators and relations.

\begin{definition}
	Let $X \to B$ be a Mori fibre space over $T$.
	A \emph{relation} in the groupoid $\BirMori_T(X)$ is an equality of the form
	\[
		\varphi_n\circ\cdots\circ\varphi_1=\id_{X_1},
	\]
	where $\varphi_i\colon X_i/B_i\rat X'_{i}/B'_i$ are elements of $\BirMori_T(X)$ such that $X_{i+1}/B_{i+1}=X'_i/B'_i$, and $X'_n/B'_n=X_1/B_1$.

	We call a relation \emph{trivial} if it is of one of the following forms:
	\begin{enumerate}
		\item $n=2$ and $\varphi_2=\varphi_1^{-1}$, or
		\item $n=4$ and $\varphi_1,\varphi_3$ are Sarkisov links, $\varphi_2,\varphi_4$ are isomorphisms, or
		\item $\varphi_1,\ldots,\varphi_n$ are isomorphisms.
	\end{enumerate}
\end{definition}

We now recall elementary relations between Sarkisov links.

\begin{proposition}[{\cite[Proposition~4.3]{BLZ21}}] \label{prop: elementary_relations}
	Let $Z/B$ be a rank $3$ fibration over $T$. Then, there are only finitely many Sarkisov links $\chi_i$ dominated by $Z/B$, and they fit in a relation \[\chi_n\circ\cdots\circ\chi_1=\id.\]
\end{proposition}
\begin{proof}
	By \autoref{lem: rank_r_MDS}, $Z/B$ is a Mori dream space and there are only finitely many rank $1$ and rank $2$ fibrations over $B$ coming from $Z$.
	Then the same proof of \cite[Proposition~4.3]{BLZ21} adapts to our setting.
\end{proof}

\begin{definition}\label{def: elementary relations}
	In the situation of \autoref{prop: elementary_relations}, we say that $\chi_n\circ\cdots\circ\chi_1=\id$ is an  \emph{elementary relation} between Sarkisov links coming from the rank $3$ fibration $Z/B$.
\end{definition}

\begin{remark}
	Note that the elementary relation is uniquely defined by $Z/B$, up to taking the inverse, cyclic permutations and insertion of isomorphisms.
\end{remark}

\subsection{Statement and proof of the Sarkisov program}

We state and prove the Sarkisov program for excellent surfaces.

\begin{theorem}[Sarkisov program for excellent surfaces]\label{t-sarkisov-program}
	Let $T$ be an excellent integrally closed domain of dimension at most $1$.
	\begin{enumerate}
		\item Any birational map $f \colon X \rat X'$ between two Mori fibre spaces $X/B$, $X'/B'$ over $T$ of dimension 2 is a composition of Sarkisov links and isomorphisms of Mori fibre spaces.
		\item Any relation between Sarkisov links is generated by trivial and elementary relations.
	\end{enumerate}
\end{theorem}

Our proof follows the strategy of \cite{HM13, Kal13} using the geography of ample models for adjoint divisors of the form $K+\Theta$.
As we work with surfaces, no flops appear and the proof can be simplified as done for example in \cite{LZ20}.\smallskip

Let $X_1/B_1,\dots,X_n/B_n$ be a finite collection of rank $r$ fibrations over $T$ together with birational maps $\theta_i\colon X_i\rat X_{i+1}$, $i=1,\dots,n$, and $X_{n+1}:=X_1$, such that $\theta_n\circ\cdots\circ\theta_1=\rm{id}_{X_1}$.
Let $Z$ be a common resolution of the $\theta_i$ and denote by $f_i \colon Z \to X_i$ the natural birational contraction.
We can choose a sufficiently small ample  $\mathbb{Q}$-divisor $A \geq 0$ such that each $f_i$ is made of steps of $(K_Z+A)$-MMP.
By the cone theorem \autoref{t-cone-theorem}, there is only a finite number of curves $\{C_i\}_{i\in I}$ such that $(K_Z+A)\cdot C_i<0$. Therefore there are only finitely many outputs of $(K_Z+A)$-MMP and the $X_1,\dots,X_n$ are among them.

For each $i$, we construct an effective ample divisor $\Delta_i$ on $Z$ such that $(Z, A +\Delta_i)$ is klt and $f_i \colon Z \to X_i$ is the $(K_Z+A+\Delta_i)$-ample model and it is obtained as a composition of $(K_Z+A)$-MMP steps with scaling of $\Delta_i$.
The construction goes as follows: let $\Theta_i \geq 0$ such that $-(K_{X_i}+{f_i}_*A)+\Theta_i$ is ample and consider an effective divisor $D_i \sim_{\mathbb{Q}} -(K_{X_i}+{f_i}_*A)+\Theta_i$. Set $\Delta_i:=\varepsilon_i A+f_i^*D_i$ for some $\varepsilon_i>0$ sufficiently small. Then ${f_i}_*(K_Z+A+\Delta_i)=\varepsilon_i{f_i}_*(A)+\Theta_i$, which is ample by \autoref{lem:  push-forward-ample}, and so $f_i$ is the ample model of $(K_Z+A+\Delta_i)$.

By adding possibly further divisors, we can assume that the $\Delta_i$ generate $N^1(Z)$.
Consider the cone
$$\mathcal{C}^{\circ} \coloneqq \left\{ D = \lambda (K_Z+A+\sum t_i r_i \Delta_i) \mid \lambda \geq 0, t_i \geq 0 \text{ and } \sum t_i \leq 1 \right\} \cap \overline{\text{Eff}}(Z/T) \subset N^1(Z/T),$$
and denote by $\mathcal{C}$ the section:
$$\mathcal{C} \coloneqq \left\{ D \in \mathcal{C}^{\circ} \mid (K_Z+A) \cdot D =-1 \right\}.$$

Let $\{g_j\colon Z\to S_j\}_{j\in J}$ be the collection of associated outputs of the $(K+A+\Delta)$-MMP with $\Delta\in\mathcal C^{\circ}$.
Notice it is a finite collection by the cone theorem \autoref{t-cone-theorem}.
Let $D$ be a big divisor in $\mathcal{C}$. Then there exists a partition $I=I^{+} \cup I^{-}$ such that
\[
D\in \bigcap_{i \in I^{+}} C_i^{>0} \cap \bigcap_{i \in I^{-}} C_i^{\leq 0}.
\]
There exists $j\in J$ such that $\varphi_D=g_j$ and the curves $C_i$, $i\in I^-$ are precisely the curves contracted by $g_j$.
Notice that
\[
\mathcal A_j:=\bigcap_{i \in I^{+}} C_i^{>0} \cap \bigcap_{i \in I^{-}} C_i^{\leq 0}
\]
is the set of all $D\in\mathcal C$ such that $g_j=\varphi_D$.

\begin{proposition}
	The cones $\mathcal{C}^o$ and $\mathcal{C}$ are rational polyhedral and convex.
	Moreover, the interior of $\mathcal{C}$ admits a chamber decomposition:
	$\Int(\mathcal{C})= \bigcup_j \mathcal{A}_{j}.$
\end{proposition}

\begin{proof}
	The proof is identical to \cite[Lemma~3.2]{LZ20} and  \cite[Theorem~3.3]{LZ20}.
\end{proof}

Moving along a segment $[\Delta,K+A]\cap \overline{\text{Eff}}(Z/T)$ for $\Delta\in \mathcal C$ represents steps in the $(K+A)$-MMP with scaling of $\Delta$.
More precisely, we have the following statement, which is an adaption of \cite[Theorem 3.3]{HM13}.

\begin{lemma}\label{lem:border-crossing}
\begin{enumerate}
\item\label{border-crossing:1} If $D$ belongs to $\overline{A_{j}}\setminus A_{j}$ with associated morphism $\varphi_D\colon Z\to Y$, then there exists a unique morphism $f\colon S_j\to Y$ such that $f\circ g_j=\varphi_D$. Moreover, $\varphi_D$ only depends on the face $\mathcal{F}\subset\overline{A_j}$ such that $D$ is contained in the relative interior of $\mathcal F$.
\item\label{border-crossing:1} If $\overline{A_j}\cap A_k\neq\varnothing$ and $g_{j,k}\colon S_j\to S_k$ is the morphism from \autoref{border-crossing:1} such that $g_k=g_{j,k}\circ g_j$, then $\rho(S_j/S_k)$ is equal to the codimension of $\overline{A_j}\cap \overline{A_k}$ in $\overline{A_j}$.
\end{enumerate}
\end{lemma}
\begin{proof}
The proof is identical to the proof of \cite[Proposition 3.7, Proposition 3.9]{LZ20}.
\end{proof}

Define $\partial^{+} \mathcal{C}:=\mathcal{C} \cap (\overline{\text{Eff}}(Z/T) \setminus \overline{\text{Big}}(Z/T)).$
As explained in \cite[Section 3.2]{LZ20}, there is natural structure of polyhedral complex on $\partial^{+} \mathcal{C}$.
We say that a face of $\partial^+\mathcal C$ is an {\em inner} face if it intersects the relative interior of $\partial^+\mathcal C$.

\begin{proposition}[{cf. \cite[Proposition 3.10]{LZ20}}]\label{prop: sarkisov-crucial}
	Let $\mathcal{F}^r$ be an inner face in the polyhedral complex $\partial^{+} \mathcal{C}$ of codimension $r$. Then there exists a rank $r$ fibration $X^r/B$ such that
	\begin{enumerate}
		\item the induced morphism $Z \to X^r$ is equal to some $g_j$;
		\item the chamber $\mathcal{A}_{j}$ satisfies $\mathcal{F}^r \subset \overline{\mathcal{A}_{j}}$;
		\item If $\mathcal{F}^{r'} \subset \mathcal{F}^r$ is a smaller face, then the rank $r'$ fibrations factorises through $X^r/B$.
	\end{enumerate}
\end{proposition}

\begin{proof}
	We follow the proof of \cite[Proposition 3.10]{LZ20}, paying attention that we are not working over a perfect field, but over a general excellent ring $T$.
	The only part in \cite[Proposition 3.10]{LZ20} where the hypothesis that $T=\Spec(\k)$ with $\k$ perfect is used is (1), so we reprove it here.

	Let $D$ be a class in the relative interior of $\mathcal F^r$. By definition, there exists
	an ample class $\Delta \in \mathcal{C}$ such that $D$ lies in the segment $[\Delta, K_Z + A]$.
	Moreover for
	sufficiently small $\varepsilon>0$, $D - \varepsilon (K_Z+A)$ lies in a chamber $\mathcal{A}_{j}$, and $j$ does not depend on $D$ nor $\varepsilon$.
	Let $Z \to B$ be the morphism associated to $D$ and $Z \to S_j$ be the morphism associated to $D - \varepsilon (K_Z+A)$.
	The morphism $S_j \to B$ from \autoref{lem:border-crossing}\autoref{border-crossing:1} is a contraction and $K_Z+A +\lambda \Delta_i$ is trivial over $B$ for some $j$ and some small $\lambda\geq0$.
	In particular, $S_j \to B$ is a rank $r'$ fibration. We are left to show $r'=r$.
We run a $K_{S_j}$-MMP over $B$ (\autoref{t-mmp-excellent}) which is a composition of contractions and terminates with a Mori fibre space $S^l \to B'$, where $l=\rho(S^l/B) \in \left\{1,2 \right\}$. We divide into two cases:
\begin{enumerate}
	\item[(i)] $S_j=S^{r'}\to S^{r'-1} \to \dots \to S^1 \to B'=B$ , or;
	\item[(ii)] $S_j=S^{r'}\to S^{r'-1} \to \dots \to S^2 \to B'
	\to B$, where $\dim(B)<\dim(B')$,
\end{enumerate}
where the $S^i$ are rank $i$ fibrations over $B$ and the morphisms $S^{j} \to S^{j-1}$ are contractions of exceptional curves of the first kind.
Note that case (ii) appears only in the case where the base scheme $T$ is the spectrum of a field $\k$.

	In case (i), we can argue as in \cite[Proposition 3.10, proof of (\dag)]{LZ20}.
	In case (ii), $\mathcal A_j$ and the chamber $\mathcal A_{S^2}$ corresponding to $Z\to S^2$ share a codimension $2$ face $\mathcal F^{r'-2}$ by \autoref{lem:border-crossing}\autoref{border-crossing:1}.
	Moreover, $S^2$ is a del Pezzo surface with a conic bundle structure.
	By the cone theorem (\autoref{t-cone-theorem}), there exists a curve $\Gamma\subset S^2$ with $\Gamma^2 \leq 0$ and $K_{S^2}\cdot \Gamma<0$.
	If $\Gamma^2<0$, there exists a contraction $S^2 \to S^1$ over $B$ (\autoref{lem: Castelnuovo_contraction}) and we can reduce to case (1).
	Otherwise, the face $\mathcal F^B$ corresponding to $B$ is defined by
	$D \cdot \Gamma = 0$ and $D \cdot \Gamma' = 0$, where $\Gamma'$ is another curve in $S^2$ with $(\Gamma')^2=0$ and $K_{S^2}\cdot \Gamma'<0$.
	Hence $\mathcal F^B$ is a codimension 2 face in $\partial^{+}\mathcal{C}$. Intersecting with $\mathcal{F}^{r'-2}$ we obtain a codimension $r'$ face in $\partial^{+}\mathcal{C}$ and conclude, as in the previous case, that $r' = r$.
\end{proof}

\begin{proof}[Proof of \autoref{t-sarkisov-program}]
	The proofs of \cite[Proposition 3.14, Proposition 3.15]{LZ20} applies verbatim using \autoref{prop: sarkisov-crucial}.
\end{proof}

\section{Rank $1$ fibrations on surfaces}\label{s: rank 1 fibrations}

In this section, $T$ satisfies the assumption of \autoref{notation-excellent} and we suppose that the structure morphism $X \to T$ is a projective contraction.
The aim of this section is to prove several structure results on del Pezzo surfaces and surface conic bundles (resp. \autoref{ss-delPezzo} and \autoref{ss-conic}), which will be useful to classify Sarkisov links in the next section.

We now collect some results on rank $r$ fibrations depending on $\dim(B)$.

\subsection{del Pezzo surfaces} \label{ss-delPezzo}

In this section, we prove properties of del Pezzo surfaces over any field.

\begin{definition} \label{def: dP_surface}
	Let $\k$ be a field.
	Let $X$ be a normal projective $\k$-surface such that $H^0(X, \mathcal{O}_X)=\k$.
	We say $X$ is a \emph{del Pezzo surface} if $X$ is regular and $-K_X$ is ample.
\end{definition}

Notice that rank $r$ fibrations over the spectrum of a field are exactly del Pezzo surfaces of Picard rank $r$.

While del Pezzo surfaces over a perfect field are geometrically regular, there is more interesting behaviour over imperfect fields. In \autoref{thm: dP-base-change-classification} below, which collects the recent results \cite{BT22, PW22}, we will see, however, that over an imperfect field with $p\geq 11$ every del Pezzo surface is geometrically regular, and there are only geometrically non-normal del Pezzo surfaces for $p\in\{2,3\}$. See also \autoref{rmk:BlowUpOfRegular} for references to explicit examples.

The following result is useful for us as it allows us to build up on the theory of  \emph{weak del Pezzo surfaces} over algebraically closed fields (that is, a smooth projective surface $Y$ with $-K_Y$ big and nef; see \cite[Section~8.1.3]{Dol12} or also \cite{MSweakdP} for a nice overview).
It gives a tool to bound the type of singularities that appear on geometrically normal del Pezzo surfaces.

\begin{proposition}\label{rmk:weak-dp}
	Let $X$ be a geometrically normal del Pezzo surface over a field $\k$ and $f\colon Y \to X_{\overline{\k}}$ its minimal resolution over the base-change to the algebraic closure. Then the following hold:
	\begin{enumerate}
		\item\label{prop: geom_normal_implies_canonical-DP} $X$ is geometrically canonical;
		\item\label{weak-dp:1} $Y$ is a weak del Pezzo surface;
		\item\label{weak-dp:2} $\rho(Y/X_{\overline\k})$ is the number of $(-2)$-curves in $Y$;
		\item\label{lem: rho_volume_dP} $K_X^2=10-\rho(X_{\k^{\sep}})-\rho(Y/X_{\overline{\k}})$.
	\end{enumerate}
\end{proposition}
\begin{proof}
	\autoref{prop: geom_normal_implies_canonical-DP} is proved in \cite[Theorem 3.3]{BT22}.

	As $X_{\overline{\k}}$ has only canonical singularities, we deduce that $-K_Y=f^*(-K_{X_{\overline{\k}}})$ is big and nef, proving \autoref{weak-dp:1}. As $Y$ is a weak del Pezzo surface obtained as the minimal resolution of a canonical del Pezzo surface, it follows that $\rho(Y/X_{\overline{\k}})$ is the number of $(-2)$-curves on $Y$, giving \autoref{weak-dp:2}.

	As $Y$ is a smooth weak del Pezzo surface, it holds that $\rho(Y)=10-K_Y^2$ by \cite[Lemma 5.1]{BT22}, and so we compute that  $K_X^2=K_Y^2=10-\rho(Y)=10-\rho(X_{\overline{\k}})-\rho(Y/X_{\overline{k}})$.
	As $\rho(X_{\overline{\k}})=\rho(X_{\k^{\sep}}) $ by \cite[Proposition 2.4]{Tan18b}, we obtain \autoref{lem: rho_volume_dP}.
\end{proof}

In the following we will see what happens if a del Pezzo surface $X$ is not geometrically normal. This can happen only in positive characteristic. We recall normalised base change to the algebraic closure. (Note that \autoref{thm:normalised-base-change} is also trivially true for $p=0$, taking $f=\id_X$ and $C=0$.)

\begin{theorem}[{\cite[Theorems 1.1, 1.2]{PW22}}]\label{thm:normalised-base-change}
	Given a normal variety $X$ over a field $\k$ of characteristic $p$ such that $\k$ is integrally closed in $k(X)$, we denote by
	\[
	f \colon Y:=(X \times_\k \bar{\k})_{\reduced}^n \to X
	\]
	the normalised base change to the algebraic closure.
	Then there exists an effective Weil divisor $C$ such that
	\begin{equation}\label{eq: normalised-base-change}
		K_Y+(p-1)C \sim f^*K_X.
	\end{equation}
	If $X$ is geometrically integral, then $(p-1)C$ is the conductor of the normalisation.
\end{theorem}

For del Pezzo surfaces, we recall the classification of their normalised base change, as obtained in \cite{PW22, BT22, Tan19}.
We denote by $H$ an hyperplane class in $\mathbb{P}(1,1,m)$. For a Hirzebruch surface $\mathbb{F}_n$, we denote by $\Sigma_n$ a section with self-intersection $\Sigma_n^2=-n$.

\begin{theorem} \label{thm: dP-base-change-classification}
	Let $X$ be a del Pezzo surface over a field $\k$ of characteristic $p>0$.
	\begin{enumerate}
		\item If $p \geq 11$, then $X$ is smooth.
		\item If $p \geq 5$, then $X$ is geometrically canonical.
		\item If $p=3$ and $X$ is not geometrically normal, then $(Y,C)$ belongs to the following list, using the notation from \autoref{thm:normalised-base-change}:
		\begin{center}
			\begin{tabular}{ | m{1.5cm} | m{1.5cm}| m{1.5cm} |  m{1.5cm}|  }
				\hline
				Y & $C$ &  $(f^*K_X)^2$  & $ \rho(X_{\k^{\sep}})$ \\
				\hline
				\hline
				$\mathbb{P}^2$ & $H$ & $1$ & $1$ \\
				\hline
				$\mathbb{P}(1,1,3)$ & $H$  & $3$ & $1$ \\
				\hline
			\end{tabular}
		\end{center}
		\item If $p=2$ and $X$ is not geometrically normal, then $(Y,C)$ belongs to the following list, using the notation from \autoref{thm:normalised-base-change}:
		\begin{center}
			\begin{tabular}{ | m{1.5cm} | m{1.5cm}| m{1.5cm} | m{1.5cm}|}
				\hline
				$Y$ & $C$ & $(f^*K_X)^2$ & $ \rho(X_{\k^{\sep}})$ \\
				\hline
				\hline
				$\mathbb{P}^2$ & $H$ & $4$ & $1$ \\
				\hline
				$\mathbb{P}^2$ & $2H$  & $1$ & $1$ \\
				\hline
				$\mathbb{P}(1,1,2)$  & $2H$  & $2$ & $1$ \\
				\hline
				$\mathbb{P}(1,1,4)$  & $2H$  & $4$ & $1$ \\
				\hline
				$\mathbb{P}^1 \times \mathbb{P}^1$  & $F$  & $4$ & $2$ \\
				\hline
				$\mathbb{P}^1 \times \mathbb{P}^1$  & $\Sigma_0+F$  & $2$ & $2$ \\
				\hline
				$\mathbb{F}_1$  & $\Sigma_1$  & $5$ & $2$ \\
				\hline
				$\mathbb{F}_1$  & $\Sigma_1+F$  & $3$ & $2$ \\
				\hline
				$\mathbb{F}_2$  & $\Sigma_2$  & $6$ & $2$ \\
				\hline
				$\mathbb{F}_2$  & $\Sigma_2+F$  & $4$ & $2$ \\
				\hline
				$\mathbb{F}_4$  & $\Sigma_4$  & $8$ & $2$ \\
				\hline
				$\mathbb{F}_4$  & $\Sigma_4+F$  & $6$ & $2$ \\
				\hline
			\end{tabular}
		\end{center}
	\end{enumerate}
In the table, the divisors in the boxes of the column $C$ describe only the linear equivalence class of $C$.
\end{theorem}

\begin{proof}
	The case where $p \geq 11$ is shown in \cite[Proposition 5.2]{BT22} and when $p \geq 5$ in \cite[Theorem 3.7]{BT22}.
	The restriction for the possibilities appearing in the list for $p= 2,3$ are proven in \cite[Theorem 4.6]{Tan19}.
	Note that $\rho(X_{\k^{\sep}})=\rho(Y)$ by \cite[Proposition 2.4]{Tan18b}.
\end{proof}

If $R$ is the local ring of $X \times_\k \k^{p^{-\infty}}$ at its generic point, the thickening exponent of geometric non-reducedness defined in \cite[Theorem 7.3]{Tan21} is the non-negative integer such that $\length_{R}(R)=p^{\epsilon(X/\k)}$.
Recall that by \cite[Lemma 4.5]{Tan19} we have the equation:
\begin{equation}\label{eq: self-intersection}
 K_X^2=p^{\epsilon(X/\k)}(f^*K_X)^2.
\end{equation}
Finally, $X$ is geometrically reduced if and only if $\epsilon(X/\k)=0$ by definition, giving $K_X^2=(f^*K_X)^2$. In particular, in this case \autoref{thm: dP-base-change-classification} gives the value of $K_X^2$, as well as an upper bound for $\rho(X)$, for geometrically integral del Pezzo surfaces that are not geometrically normal.
Using this, we obtain:

\begin{corollary}\label{cor: geometrically integral del Pezzo}
	Let $X$ be a del Pezzo surface over any field $\k$ of characteristic $p\geq 0$.
	Then $(f^*K_X)^2 \leq 9$.
	If $X$ is not geometrically normal, then $\rho(X) \leq 2$ and
	\begin{enumerate}
		\item $p=3$, $\rho(X)=1$ and $(f^*K_X)^2 \in \left\{ 1, 3\right\}$;
		\item $p=2$, $\rho(X) \leq 2$ and $(f^*K_X)^2 \in \left\{ 1, 2, 3, 4, 5, 6, 8 \right\}$. 
		\item Moreover, if $p=2$ and $\rho(X)=1$, then $(f^*K_X)^2\in\{1,2,4\}$.
	\end{enumerate}
	If $X$ is geometrically integral, we can replace $(f^*K_X)^2$ with $K_X^2$ in the previous statements.
\end{corollary}
\begin{proof}
    If $X$ is geometrically normal, we have that $(f^*K_X)^2=K_X^2$ and \autoref{rmk:weak-dp}\autoref{lem: rho_volume_dP} implies that $K_X^2 \leq 9$.

	If $X$ is not geometrically normal, we consult \autoref{thm: dP-base-change-classification} and see that either $p=3$ or $p=2$, with certain $(f^*K_X)^2\leq 9$ and certain $\rho(X_{\k^{\rm sep}})$, where $f\colon Y\to X_{\bar\k}$ is the normalised base change.
	As $\rho(X)\leq\rho(X_{\k^{\rm sep}})\leq 2$, the statement follows.
	
	For the third statement, we show that if $(f^*K_X)^2\in\{3, 5, 6, 8\}$, then $\rho(X)=2$.
	In these cases, we see from the list in \autoref{thm: dP-base-change-classification} that $\rho(X_{\k^{\rm sep}})=2$ and $X_{\k^{\rm sep}}$ has a Mori conic bundle structure and a birational contraction $X \to Y$. In particular, the extremal rays of $\NE(X_{\k}^{\sep})$ cannot be interchanged by the Galois group $\Gal(\k^{\sep}/\k)$ and thus $\rho(X)=2$.
\end{proof}

\subsubsection{Classification of del Pezzo surfaces of high degree}

We classify del Pezzo surfaces of degree $\geq 7$. At the end of this section, in \autoref{cor: moredP}, we give a classification of all geometrically integral del Pezzo surfaces with Picard rank $1$.

\begin{lemma} \label{lem: SB}
	Let $X$ be a geometrically integral del Pezzo surface over a field $\k$ with $K_X^2=9$.
	Then $X$ is a Severi--Brauer surface (that is, $X_{\bar\k}\simeq\p^2_{\bar\k}$). In particular, $X(\k)\neq \emptyset$ if and only if $X \simeq \mathbb{P}^2_\k$.
\end{lemma}

\begin{proof}
	As $K_X^2=9$, by \autoref{cor: geometrically integral del Pezzo} we deduce that $X$ is geometrically normal.
	By \autoref{rmk:weak-dp}\autoref{lem: rho_volume_dP}, we have that $X_{\bar{\k}}$ is a smooth weak del Pezzo surface of degree $9$ and thus it is isomorphic to $\mathbb{P}^2_{\bar\k}$. The last assertion is Châtelet's theorem \cite[§IV.I, p.283]{Cha44}.
\end{proof}

\begin{lemma} \label{lem: dP8_normal}
	A geometrically integral del Pezzo surface over a field $\k$ with $K_X^2=8$ is geometrically normal.
\end{lemma}

\begin{proof}
	%We can assume that the base-field $\k$ is separably closed. 
	Suppose that $X$ is not geometrically normal.
	By \autoref{thm: dP-base-change-classification} the normalised base change $X_{\bar{\k}}^n$ is isomorphic to the Hirzebruch surface $\mathbb{F}_4$.
	In particular, $X_{\bar{\k}}^n$ has two extremal contractions, one of which is birational and descends to a birational contraction $f \colon X \to Y$ defined over $\k$ since $\rho(X)=2$ by \autoref{cor: geometrically integral del Pezzo}. %we have that 
	
	%In particular, $X_{\bar{\k}}^n$ has two extremal contractions, one of which $\varphi \colon X_{\bar{\k}}^n \to Z$ is birational.
	%By \autoref{cor: geometrically integral del Pezzo}, we have $\rho(X)=2$ and thus $\varphi$ descends to a birational contraction $f \colon X \to Y$ defined over $\k$.
	As $K_X^2=8$, $Y$ is a del Pezzo surface of degree at least $9$, and hence $K_Y^2=9$ by \autoref{cor: geometrically integral del Pezzo}. Therefore, by \autoref{lem: SB}, $X$ is obtained by blowing up a Severi--Brauer surface in a rational point, which gives a contradiction to geometric non-normality.
	%Using geometric integrality, \autoref{thm: dP-base-change-classification} tells us that $K_X^2=8$ can happen only if $p=2$, and if the normalised base change $X_{\bar{k}}^n$ is isomorphic to the Hirzebruch surface $\mathbb{F}_4$.
	%In particular, $X_{\bar{\k}}^n$ has two extremal contractions, one of which $\varphi \colon X_{\bar{\k}}^n \to Z$ is birational.
 	%By \cite[Proposition 2.4]{Tan18b}, the Picard rank does not change under a purely inseparable base change (that is, $\rho(X_{\bar\k})=\rho(X)$), and hence there is also a birational contraction $f \colon X \to Y$ defined over $\k$. As $K_X^2=8$, $Y$ is a del Pezzo surface of degree at least $9$, and hence $K_Y^2=9$ by \autoref{cor: geometrically integral del Pezzo}. Therefore, by \autoref{lem: SB}, $X$ is obtained by blowing up a Severi--Brauer surface in a rational point, which gives a contradiction to geometric non-normality.
\end{proof}

\begin{lemma} \label{lem: dP8}
	Let $X$ be a geometrically integral del Pezzo surface over any field $\k$ of degree $K_X^2=8$.
	Then either $X\simeq \F_1$, or $X$ is geometrically a quadric (that is, $X_{\bar \k}\subset \p^3_{\bar \k}$ is a quadric) and either
	\begin{enumerate}
		\item $X_{\bar{\k}} \simeq \mathbb{P}^1_{\bar\k} \times \mathbb{P}^1_{\bar \k}$, or
		\item $X$ is not smooth, $p=2$ and $X_{\bar\k} \simeq \mathbb{P}(1,1,2)$.
	\end{enumerate}
\end{lemma}

\begin{proof}
	By \autoref{lem: dP8_normal}, $X$ is geometrically normal and thus \autoref{rmk:weak-dp}\autoref{lem: rho_volume_dP} implies that either $X_{\bar{\k}}$ is smooth with $\rho(X_{\k^{\sep}})=2$, or $X_{\bar{\k}}$ has an $A_1$-singularity with $\rho(X)=1$.
	In the former case, by \cite[Section 8.4.1]{Dol12} we have that either $X_{\overline{\k}}$ is isomorphic to $\mathbb{P}^1_{\overline{\k}} \times \mathbb{P}^1_{\overline{\k}}$ or to $\mathbb{F}_1$. In the second case, $X\simeq\F_1$ because there is a unique curve $E\subset X$ such that $E_{\overline\k}$ is a $(-1)$-curve and thus there exists a birational contraction $X \to Y$ with exceptional locus $E$, where $Y$ is a del Pezzo of degree $K_Y^2>8$, and thus $K_Y^2=9$, and $Y$ contains a rational point. Hence, $Y\simeq \mathbb{P}^2_\k$ by \autoref{lem: SB}.

	In the latter case, we have that $X_{\bar{\k}}$ has one $A_1$-singular point
	$\yy$, which can appear only in characteristic $2$ by \cite[Theorem 6.1]{Sch08}.
	By blowing-up the $\yy$, the minimal resolution $Y \to X_{\bar\k}$	is a weak del Pezzo surface of degree $8$ by \autoref{rmk:weak-dp}\autoref{weak-dp:1}.
	By \cite[8.4.1]{Dol12}, $Y \simeq \mathbb{F}_2$ and thus $X_{\bar\k} \simeq \mathbb{P}(1,1,2)$.
\end{proof}

\begin{lemma} \label{lem: dp8-vs-quadric}
	Let $X$ be a geometrically integral del Pezzo surface of degree $8$ over any field $\k$.
	Then $X$ is a quadric in $\mathbb{P}^3_\k$ if and only if $-K_X=2H$ for some Cartier divisor $H$.
\end{lemma}

\begin{proof}
	If $X$ is a quadric in $\p^3_\k$, then $H$ is the hyperplane section and we conclude by the adjunction formula.
	Suppose now that $-K_X=2H$ for some Cartier divisor $H$. Then $H_{\overline{\k}}$ is a Cartier divisor on $X_{\overline{\k}}$. In particular, $-K_{X_{\overline{\k}}}$ is divisible by 2 in the Picard group, so $X_{\overline{k}}$ is not isomorphic to $\mathbb{F}_1$ and hence $X_{\overline{k}}$ is a quadric in $\p^3_{\bar\k}$ by \autoref{lem: dP8} and $H_{\overline{\k}}$ is the hyperplane section.
	Therefore $H$ is very ample and embeds $X$ in $\mathbb{P}^3_\k$ as a quadric.
\end{proof}

We will need the following result on the degree of closed points of del Pezzo surfaces of large anticanonical degree.
We first recall the definition of the index of a variety over a field.

\begin{definition}
Let $X$ be a variety over a field $\k$.
The \emph{index} ${\rm ind}(X)$ of $X$ is the greatest common divisor of the degrees of closed points on $X$.
\end{definition}

\begin{lemma} \label{lem: index_SB}
    Let $X$ be a geometrically integral del Pezzo surface of degree 9 over any field $\k$.
    Then $\rm {ind}(X)$ is either 1 or 3. Moreover, $\rm {ind}(X)=1$ if and only if $X \simeq \mathbb{P}^2_{\k}$.
\end{lemma}

\begin{proof}
      By \autoref{lem: SB}, $X$ is a Severi--Brauer surface. 
      Thus the result follows from \cite[Theorem 68 and Corollary 70]{Kol16}.
\end{proof}

\begin{proposition} \label{prop: divisibility-dp8}
Let $X$ be a geometrically integral del Pezzo surface of degree $8$ over a field $\k$.
 Then ${\rm ind}(X)$ divides $4$.
Moreover, if ${\rm ind}(X)=1$, either $X\simeq\F_1$, or $-K_X=2H$ for some Cartier divisor $H$ on $X$.
Moreover, if ${\rm ind}(X)=4$, then $X$ is smooth.
\end{proposition}

\begin{proof}
By \autoref{lem: dP8} we can suppose that $X_{\bar{\k}}$ is a quadric, as otherwise $X \simeq \mathbb{F}_1$ and it has a rational point.

If $X$ is not geometrically regular, then $X_{\overline{k}}$ has an $A_1$-singularity and the closed point of $X$ corresponding to the non-smooth locus has degree 2 by \cite[Lemma 2.25]{BT22} and thus the index is either 1 or 2.
If $X$ is geometrically regular, then it is an involution surface and the index divides 4 by \cite[Example 3.3]{AuelBernardara}, thus concluding the proof.
%Note that if $X$ is a quadric in $\mathbb{P}^3_{\k}$, it is easy to see that ${\rm ind}(X)$ divides 2 by intersecting $X$ with a general line .
%In general, by \cite[Theorem 1.1]{Lie17}, then there exists an embedding $X \hookrightarrow Y$, where $Y$ is a Severi--Brauer variety of dimension 3.
%As $-K_Y=2X$, and $|X|$ is a very ample linear system, we conclude by Bertini's theorem \cite[Corollary 3.4.14]{FOV99} that there exist $H_1, H_2 \in |X|$ such that the complete intersection $C =H_1 \cap H_2 \subset Y$ intersects $X$ transversally. As $C \cdot X=8$, we deduce that ${\rm ind}(X)$ divides $8$, unless $C$ intersects $X$ in two closed point $\xx$ and $\yy$ of degree $3$ and $5$ respectively. In this remaining case, ${\rm ind}(X)=1$, and $Y$ is a $3$-dimensional Severi--Brauer variety with ${\rm ind}(Y)=1$ and therefore $Y \simeq \mathbb{P}^3_{\k}$ by \cite[Theorem 53]{Kol16} and thus $X$ is a quadric in $\mathbb{P}^3_{\k}$ and the index of $X$ divides 2.
\end{proof}

\begin{remark}\label{rem: points on quadric}
	Note that a geometrically integral del Pezzo surface $X$ over $\k$ of degree $8$, which is not $\mathbb{F}_1$, is a quadric in $\p^3_\k$ if and only if $\Pic(X)$ coincides with the Galois fixed part of $\Pic(X_{\k^{\sep}})$.
	This condition is satisfied if $X$ has a rational point or the Brauer group of $\k$ is trivial. In particular, a rational del Pezzo surface of degree 8 is either $\mathbb{F}_1$ or a quadric with a rational point.

	Over a perfect field, if $X$ is smooth of Picard rank $1$, then either $X\subset\p^3_\k$ is a quadric and $\ind(X)\in\{1,2\}$, or $X$ is not a quadric and $\ind(X)\in\{2,4\}$ \cite[Table~3 and Remark~7.4]{AuelBernardara}.
\end{remark}

\begin{lemma} \label{lem: descent_semiamp}
Let $X$ be a normal projective variety over $\k$.
Let $\k'$ be a purely inseparable extension of $\k$ and let $Y$ be the normalisation of $X \times \k'$, with the induced morphism $\beta \colon Y \to X$. 
Then 
\begin{enumerate}
\item The induced group homomorphism $\tilde{\beta} \colon N^1(X)_{\mathbb{Q}} \to N^1(Y)_{\mathbb{Q}}$ is bijective.
\item Suppose that $h^1(X, \mathcal{O}_X)=0$. 
Let $L$ be a semi-ample $\mathbb{Q}$-Cartier $\mathbb{Q}$-divisor on $Y$. Then there exists a semi-ample $\mathbb{Q}$-Cartier $\mathbb{Q}$-divisor $M$ on $X$ such that $\beta^{*}M \sim_{\mathbb{Q}}L.$
\end{enumerate}
\end{lemma}

\begin{proof}
    The first item follows from \cite[Proposition 2.4]{Tan18b}. 
    For (2), if $h^1(X, \mathcal{O}_X)=0$ then we have $N^1(X)_{\mathbb{Q}} \simeq \Pic(X)_{\mathbb{Q}}$, and thus by (1) there exists a $\mathbb{Q}$-Cartier $\mathbb{Q}$-divisor $L$ such that $\beta^*M \sim_{\mathbb{Q}}L$.
    We are left to check $M$ is semi-ample. For this, we can reduce to the case where $\k'$ is a finite extension of $\k$ such that $L$ descends to $X \times \k'$.
    As semi-ampleness descends for finite morphism by the Nakai--Moishezon criterion, we conclude that $M$ is semi-ample as well.
\end{proof}

\begin{lemma}\label{lem:no-minimal-dp7}
	Let $X$ be a del Pezzo surface with $K_X^2=7$ over a field $\k$.
	Then $\rho(X) \in \left\{2, 3 \right\}$, and there is a birational morphism $X\to\p^2_\k$. In particular, $X$ is rational.
\end{lemma}

\begin{proof}
	By the values assumed by $(f^*K_X)^2$ in \autoref{thm: dP-base-change-classification}, $X$ is geometrically normal, and thus geometrically canonical by \autoref{prop: geom_normal_implies_canonical-DP}.
	The minimal resolution $g\colon Y\to X_{\bar\k}$ is a weak del Pezzo surface of degree $7$ (see \autoref{rmk:weak-dp}\autoref{weak-dp:1}), and it admits a morphism $Y\to(\F_1)_{\bar\k}$ that is the blow-up at a single point $\xx\in\F_1(\bar\k)$ by \cite[Section~8.1.3]{Dol12}.
	Recall that by \autoref{rmk:weak-dp}
	we have $\rho(Y/X_{\bar\k})+\rho(X_{k^{\rm sep}})=3$, and that $\rho(Y/X_{\bar\k})$ is equal to the number of $(-2)$-curves on $Y$.

	Denote by $\Sigma_1$ the strict transform of the negative section of $(\F_1)_{\bar\k}$ in $Y$, by $f$ the strict transform of the fibre through $\xx$ on $Y$, and by $E_{\xx}$ the exceptional divisor.
	We distinguish the two cases $\xx\notin\Sigma_1$ and $\xx\in\Sigma_1$, illustrated in \autoref{fig:ConfigurationsdP7}.

	If $\xx\notin \Sigma_1$, then $Y$ has exactly three $(-1)$-curves, namely $\Sigma_1$, $f$ and $E_{\xx}$.
	Hence, $\rho(Y/X_{\bar\k})=0$, so $Y=X_{\bar\k}$ and  $\rho(X_{\k^{\rm sep}})=3$.
	%By \cite[Proposition 2.4]{Tan18b}
	By \autoref{lem: descent_semiamp}, the birational contraction of $\Sigma_1$ (respectively $f$, respectively $E_\xx$) on $X_{\bar\k}$ descends to a contraction on $X_{\k^{\rm sep}}$; we denote again by $\Sigma_1$ (respectively $f$, respectively $E_\xx$) the contracted curves on $X_{\k^{\rm sep}}$.
	Note that $\Sigma_1$ and $E_\xx$ both intersect $f$ but do not intersect each other. Hence, $\Gal(\k^{\rm sep}/\k)$ fixes $f$, and preserves the set of the two disjoint $(-1)$-curves $\{E_\xx, \Sigma_1\}$.
	Hence, the contraction of them descends to a birational contraction $X\to Z$ over $\k$, where $Z$ is a del Pezzo surface with $K_Z^2=9$ with either a rational point or a point of degree $2$.
	As $Z$ is a Severi--Brauer surface by \autoref{lem: SB}, it follows that $\rm {ind}(X)=1$ and thus $Z \simeq \mathbb{P}^2_{\k}$ by \autoref{lem: index_SB}.
	%Hence $Z\simeq\p^2_\k$ by \autoref{lem: SB} and \cite[Corollary 70]{Kol16}, and $\rho(X)\in\{2,3\}$.

	If $\xx\in\Sigma_1$, then $Y$ has exactly three negative curves, namely the two $(-1)$-curves $f$ and $E_\xx$, and the $(-2)$-curve $\Sigma_1$. Hence, $\rho(Y/X_{\bar\k})=1$ and $\rho(X_{\k^{\rm sep}})=2$. Note that $E_\xx$ intersects $\Sigma_1$, but $f$ does not.
	Note that the contraction of the curve $g_*E_{\xx}$ on $X_{\bar{\k}}$ gives a birational morphism $X_{\bar{\k}} \to \mathbb{P}^2_{\bar\k}$ which descends by \cite[Proposition 2.4]{Tan18b} to a birational morphism $p \colon X_{{\k}^{\sep}} \to \mathbb{P}^2_{\k^{\sep}}$.
	Note that the contraction of $g_*f$ on  $X_{\bar{\k}}$ gives a birational morphism $X_{\bar{\k}} \to Q$ (over $\bar\k$), where $K_Q^2=8$, which descends to a birational morphism $X_{\k^{\rm sep}} \to W$  (over $\k^{\rm sep}$), where $W$ is a del Pezzo surface of degree $8$.
	As $K_{\p^2}^2 \neq K_W^2$, we have that the two extremal rays on $\overline{\text{NE}}(X_{\k^{\rm sep}})$ defining the birational contractions are fixed by the Galois group  $\Gal(\k^{\rm sep}/\k)$, concluding that $\rho(X)=2$ and that both birational contractions on $X_{\k^{\rm sep}}$ descend to birational contractions from $X$.
	In particular, there exists a birational morphism $p \colon X \to Z$, where $Z$ is a Severi--Brauer surface.
	As $K_Z^2=9$, we have that $p$ is the blow-up of a closed point of degree $2$ on $Z$, which implies $Z \simeq \mathbb{P}^2_\k$ by \autoref{lem: index_SB} and thus $X$ is rational.
\end{proof}

\begin{figure}
\label{fig:ConfigurationsdP7}
    \centering
    \begin{tikzpicture}
    \draw[label=above] (-3.3,0) -- (-0.7,0) node[midway,below] {$f$} node[midway,above] {\tiny$-1$};
    \draw[] (-1,-0.3)--(-1,1.3) node[near end,right] {$E_{\xx}$} node[near end,left] {\tiny$-1$};
    \draw[] (-3,-0.3)--(-3,1.3) node[near end,left]
    {$\Sigma_1$} node[near end,right] {\tiny$-1$};
    
    \draw[label=above] (3.3,0) -- (0.7,0) node[midway,below] {$E_{\xx}$} node[midway,above] {\tiny$-1$};
    \draw[] (3,-0.3)--(3,1.3) node[near end,right] {$f$} node[near end,left] {\tiny$-1$};
    \draw[] (1,-0.3)--(1,1.3) node[near end,left] {$\Sigma_1$} node[near end,right] {\tiny$-2$};
    \end{tikzpicture}
    \caption{Cases: $\xx\notin \Sigma_1$ (left) and $\xx\in\Sigma_1$ (right)}
    \label{fig:ConfigurationsdP7}
\end{figure}

 \begin{corollary} \label{cor: moredP}
 	Let $X$ be a geometrically integral del Pezzo surface over a field $\k$ with $N^1(X)=\mathbb{Z}[H]$ for an ample Cartier divisor $H$.
 	Then $H \equiv -K_X$ if and only if we are in one of the following cases:
 	\begin{enumerate}
 		\item\label{moredP-1} $X$ is a non-trivial Severi-Brauer surface,
 		\item\label{moredP-2} $K_X^2=8$ and $X$ is not a quadric in $\p^3_\k$, or
 		\item\label{moredP-3} $K_X^2 \leq 6$.
 	\end{enumerate}
	Moreover, if $X$ is not in \autoref{moredP-1}, \autoref{moredP-2}, or \autoref{moredP-3}, then either $X=\p^2_\k$ and $3H\equiv -K_X$, or $X$ is a quadric in $\p^3_\k$ and $2H\equiv -K_X$.
 \end{corollary}

 \begin{proof}
	Since $X$ is geometrically integral, \autoref{cor: geometrically integral del Pezzo} implies that $K_X^2\leq 9$.
 	If $K_X^2=9$, then $X$ is a Severi-Brauer surface by \autoref{lem: SB}.
	In this case, $-K_X$ does not generate $\Pic(X)$ if and only if $dH=-K_X$ for some $d>1$. As $X_{\bar{\k}} \simeq \mathbb{P}^2_{\bar\k}$ and $\Pic(X_{\bar{\k}})=\mathbb{Z}[L]$ with $-K_{X_{\bar{\k}}}=3L$, the fact that $-K_X$ is divisible by $H$ is equivalent to $H_{\bar\k}=L$ on $\p^2_{\bar\k}$, and thus to $X$ being a trivial Severi--Brauer variety.

 	If $K_X^2=8$, then $-K_X=dH$ is only possible for $d\in\{1,2\}$. By \autoref{lem: dp8-vs-quadric}, $X$ is a quadric if and only if $d=2$.

	Suppose that $K_X^2 < 8$. Then by \autoref{lem:no-minimal-dp7} we have $K_X^2 \leq 6$.
	Let $d \geq 1$ such that $-K_X=dH$.
	If $d=1$, then $-K_X=H$ and we are done.
	If $d\geq 2$, then $d=2$, $K_X^2=4$, and $H^2=1$. We devote the rest of the proof to show that this case does not occur. From now on, let $X$ be a del Pezzo surface of degree $4$.

	Suppose first that $X$ is geometrically normal with $K_X^2=4$. Let $\pi \colon Y \to X_{\overline{\k}}$ be the minimal resolution, so that $Y$ is a weak del Pezzo surface of degree $4$ by \autoref{rmk:weak-dp}\autoref{weak-dp:1}.
	In this case, as $Y$ is obtained as the blow-up of $\mathbb{P}^2_{\bar{\k}}$, there exists a $(-1)$-curve $l$ on $Y$ and thus $-K_Y \cdot l=1$.
	As $-K_Y=\pi^*(-K_{X_{\overline{\k}}}) =\pi^*(dH_{\overline{\k}})$, we have $1=-K_Y\cdot l=dH_{\bar \k}\cdot l\geq d$ and thus we conclude that $d=1$.

	Suppose now that $X$ is not geometrically normal with $K_X^2=4$.
	Let $f\colon Y \to X_{\bar{\k}}$ be the normalisation. 
	Looking again at the list of \autoref{thm: dP-base-change-classification}, we have $p=2$. We distinguish the four possible cases to investigate.
		Suppose that $Y=\mathbb{P}(1,1,4)$ and $(p-1)C= 2L$, where $L$ is a hyperplane section (a line through the vertex of the cone over the quartic normal rational curve). Then $-(K_{\mathbb{P}(1,1,4)}+2L)=4L$ is not divisible in the Picard group (as the Cartier index of $L$ is 4).
		We conclude that $K_X$ is not divisible by formula \autoref{eq: normalised-base-change}.
	Suppose now that $Y=\p^1_{\bar{\k}} \times \p^1_{\bar{\k}}$ and $(p-1)C=F$, then $K_Y+(p-1)C=\mathcal{O}(-1, -2)$, which is not divisible.
		The case where $Y=\mathbb{F}_2$ is proven analogously as, in this case, the conductor $(p-1)C=\Sigma_2+F$, and thus $K_{\mathbb{F}_2}+ \Sigma_2+F \sim -\Sigma_2-3F$ is not divisible (note that its intersection with $F$ is 1), we conclude that $K_X$ is also not divisible.
Suppose finally that $Y=\p^2_{\bar{\k}}$ and $(p-1)C=L$.
	In this case, as explained in \cite{Rei94, FS20a} we consider the co-cartesian square of the normalisation,
	\begin{equation}\label{normalisation_square}
		\begin{tikzcd}
			L \ar[hook,r, "j"] \ar[d,"f|_L"]& \p^2_{\bar{\k}} \ar[d, "f"]  \\
			\p^1_{\bar{\k}} \ar[hook, r,"i"] & X_{\bar{\k}}.
		\end{tikzcd}
	\end{equation}
	which induces a short exact sequence of Picard groups \cite[cf. page 1791]{FS20a}: \[ 0 \to \Pic(X_{\bar{\k}}) \xrightarrow{(f^*, i^*)} \Pic(\mathbb{P}^2_{\bar{\k}}) \oplus  \Pic(\mathbb{P}^1_{\bar{\k}}) \xrightarrow{g} \Pic(L),\]
where $g=j^*-(f|_L)^*$.
	 As explained in \cite{Rei94}, the morphism $L \to \mathbb{P}^1_{\bar{\k}} $ is 2 to 1, and thus $g(n,m)=n-2m$.
	As $(f^*,i^*)(-K_{X_{\bar{\k}}})$ is $(2,1)$, we conclude that $-K_{X_{\bar{\k}} }$, and henceforth $-K_X$, is not divisible.
\end{proof}

\begin{corollary} \label{cor: moredP-sepclosed}
	Let $\k$ be a separably closed field.
 Let $X$ be a geometrically integral del Pezzo surface over $\k$ with $N^1(X)=\mathbb{Z}[H]$ for an ample Cartier divisor $H$. Assume that $K_X^2\geq 5$.
 Then, $X$ is one of the following:
 \begin{enumerate}
 	\item $X=\p^2_\k$;
	\item $p=2$, $K_X^2=8$, and $X_{\bar\k}$ has an $A_1$-singularity;
	\item $p=5$, $K_X^2=5$, and $X_{\bar\k}$ has an $A_4$-singularity.
 \end{enumerate}
\end{corollary}

\begin{proof}
	First of all, note that since we work over a separably closed field, we have that $\rho(X_{\bar\k})=\rho(X)$ by \cite[Proposition 2.4(2)]{Tan18b}, and hence \autoref{rmk:weak-dp}\autoref{lem: rho_volume_dP} implies that $X$ is geometrically regular exactly if $K_X^2=9$.

	Suppose $K_X^2=9$. As any Severi--Brauer variety over a separably closed field has a rational point by \cite[Corollary 3.5.71]{Poo17}, we find that $X\simeq\p^2_\k$ by \autoref{lem: SB}.

	Suppose $K_X^2=8$. As $X$ is not geometrically regular, \autoref{lem: dP8} implies that $X_{\bar\k}\simeq \p(1,1,2)$ and $p=2$; in particular, $X_{\bar\k}$ has an $A_1$-singularity.

	$K_X^2=7$ is excluded by \autoref{lem:no-minimal-dp7}.

	For the remaining cases, we write now $Y \to X_{\bar{\k}}$ for the minimal resolution, and as $\rho(X_{\bar\k})=\rho(X)=1$, we obtain $\rho(Y/X_{\bar{\k}})=9-K_X^2$ from \autoref{rmk:weak-dp}\autoref{lem: rho_volume_dP}.

	We exclude $K_X^2=6$: by \autoref{rmk:weak-dp} we have $\rho(Y/X_{\bar{\k}})=3$.
	By \cite[Table 8.4, page 431]{Dol12}, we deduce $X_{\bar{\k}}$ has two singular points, one of type $A_1$ and one of type $A_2$. This is impossible as they cannot both be twisted form of regular sufaces by \cite[Theorem 6.1]{Sch08}.

	If $K_X^2=5$, by \autoref{cor: geometrically integral del Pezzo} $X$ is geometrically normal. By \autoref{rmk:weak-dp} we find $\rho(Y/X_{\bar{\k}})=4$.
	By \cite[Table 8.5]{Dol12}, we deduce $X_{\bar{\k}}$ has one $A_4$-singularity, and hence, $p=5$ by \cite[Theorem 6.1]{Sch08}.
\end{proof}
Note that one can similarly deduce the possible type of singularities for geometrically normal del Pezzo surfaces with $K_X^2 \leq 4$, using \cite[Theorem 6.1]{Sch08} and \cite[Table~8.6, Section~9.2.2, Table~8.8, Table~8.10]{Dol12}, respectively.For the non-taut singularities, a more detailed classification has been proven by Stadlmayr in \cite{Sta21}.

\subsubsection{Geiser and Bertini involutions.}

We extend the existence of Geiser and Bertini involutions on smooth del Pezzo surfaces of degree $1$ and $2$ to certain regular del Pezzo surfaces over arbitrary fields. This will be needed later when we discuss Sarkisov links of type IV and involutions in the Cremona group.

\begin{proposition} \label{lem: Geiser_Bertini_involutions}
Let $\k$ be an arbitrary field and let $X$ be a geometrically integral del Pezzo surface with $h^1(X, \mathcal{O}_X)=0$. Assume that $K_X^2=2$ (resp. $K_X^2=1)$. Then the following hold:
\begin{enumerate}
    \item The linear system $|-K_X|$ (resp. $|-2K_X|$) is base point free and it induces a double cover $\varphi_{-K_X} \colon X \to \mathbb{P}^2_\k$ (resp. $\varphi_{-2K_X} \colon X \to \mathbb{P}_\k(1,1,2)$).
    \item If $\varphi_{-K_X}$ (resp. $\varphi_{-2K_X}$) is separable, then it is a Galois cover whose Deck transformation is a biregular involution on $X$ called \emph{Geiser involution} (resp.\ \emph{Bertini involution}).
    \item If $\rho(X)>1$, then $\varphi_{-K_X}$ (resp.\ $\varphi_{-2K_X}$) is separable.
\end{enumerate}
\end{proposition}

\begin{proof}
    We prove the case $K_X^2=2$, as the case $K_X^2=1$ is analogous.
    In \cite[Theorem III.3.5]{k-rat-curves} and \cite[Theorem 2.15]{BT22} it is shown that if $X$ is geometrically canonical then $X$ is isomorphic to a hypersurface of $\P_\k(1,1,1,2)$ of degree $4$, and $|-K_X|$ induces a morphism $X\to\P^2_\k$ of degree $2$.
    In \cite[Proposition 2.4]{BT24} (see also \cite[Proposition 2.5]{BT24}), it is shown that the same strategy can be applied more generally if $X$ is a geometrically integral del Pezzo surface with $h^1(X, \mathcal{O}_X)=0.$ 
    If the map is separable of degree 2, it is Galois and its Deck transformation is a biregular involution on $X$.   
   
    If $\rho(X)>\rho(\mathbb{P}^2_\k)$, the maps are separable by \cite[Proposition 2.4(1)]{Tan18b}, concluding.
\end{proof}

\begin{remark}
    The condition on the Picard rank to be at least $2$ is necessary for the anticanonical maps to be separable and to show the existence of the Geiser and Bertini involutions. 
    Consider the following del Pezzo surface of degree $2$ with canonical singularities and Picard rank 1 over an algebraically closed field of characteristic $2$:
    $$ X=\{w^2=x^3y+y^3z+z^3x\}\subset \P(1,1,1,2).$$
    It is easy to see that $X$ is normal, it has $7A_1$ singularities and its anticanonical morphism is purely inseparable onto $\mathbb{P}^2_{\k}$. 
    One can construct a regular twisted form over $k=\overline{\mathbb{F}_2}(t_0, t_1, t_2)$ as follows:
    $$ Y=\{w^2=t_0x^3y+t_1y^3z+t_2z^3x\}.$$
    
    A similar pair of examples for del Pezzo of degree $1$ is given by
    $$\{w^2=y^3+x_0x_1y^2+x_0^5x_1+x_0x_1^5 \}\subset \P(1,1,2,3). $$
\end{remark}

It is useful to know the action of such involutions on the Picard group of the del Pezzo surfaces.
For this, we recall the description of the equivariant Picard group in the case of Galois finite cover.

\begin{lemma} \label{lem: pic_iso}
Let $ f \colon X \to Y$ be a finite Galois morphism of projective varieties with Galois group $G$.
Then the induced morphism of abelian groups $f^* \colon \Pic(Y)_{\mathbb{Q}} \to \Pic(X)_{\mathbb{Q}}^{G}$ is an isomorphism.
\end{lemma}
\begin{proof}
    Consider the Hochschild--Serre spectral sequence in the \'etale topology (see \cite[Theorem 2.20]{Mil80}):
    $$ H^p(G,H^q_{\acute{e}t}(X, \mathbb{G}_m)) \Rightarrow H^{p+q}_{\acute{e}t}(Y, \mathbb{G}_m). $$
    Looking at the first terms of the spectral sequence, and using Hilbert's 90 (\cite[Proposition 4.9]{Mil80}), we have the exact sequence
    $$ 0 \to \Pic(Y) \to \Pic(X)^G \to H^2(G, H^0(X, \mathcal{O}_X^*)).$$
    Since $G$ is finite, $H^2(G, H^0(X, \mathcal{O}_X^*))$ is a torsion group (e.g. \cite[Chapitre 2, Corollaire 3]{Serre-coh-galoisienne}) and thus we conclude that the injective homomorphism $\Pic(Y) \to \Pic(X)^G$ is surjective after tensoring with $\mathbb{Q}$.
\end{proof}

\begin{proposition}\label{lem: action-Bertini-Geiser}
Let $Y$ be a del Pezzo surface of degree $d=1$ (resp.\ $d=2$), and assume that the Bertini (resp.\  Geiser) involution $\sigma\in\Aut(Y)$ exists.
%Let $Y$ be a del Pezzo surface of degree $d\in\{1,2\}$ \st{with $\rho(Y) \geq 2$}. If $d=1$ (resp. $d=2$), let $\sigma\in\Aut(Y)$ be the Bertini (respectively, Geiser) involution on $Y$.
Then $N^1(Y)^{\sigma}_{\mathbb{Q}}=\mathbb{Q}[-K_Y]$.
%\FB{I would erase from now on and just say that the invariant part has Picard rank 1, no?}
%Then $\sigma^*\colon N^1(Y)_\Q\to N^1(Y)_\Q$ does not fix any extremal ray. 
In particular, if $\rho(Y)=2$, then $\sigma^*$ interchanges the two extremal rays of the nef cone of $Y$.
\end{proposition}
\begin{proof}
%By \autoref{lem: Geiser_Bertini_involutions}, the Geiser and the Bertini involutions exist.
Denote by  $G=\langle\sigma\rangle \simeq \Z/2\Z$ the subgroup of automorphisms generated by the Geiser (resp.\ Bertini) involution. By \autoref{lem: pic_iso}, we have $N^1(X)_{\mathbb{Q}}^{G} \simeq N^1(\mathbb{P}^2_{\k})_{\mathbb{Q}} \simeq \mathbb{Q}$ (resp. $N^1(Y)_{\mathbb{Q}}^{G} \simeq N^1(\mathbb{P}(1,1,2))_{\mathbb{Q}}$). As the numerical class $-K_Y$ is preserved by the Geiser (resp. Bertini) involution, we have $N^1(Y)_{\mathbb{Q}}^{G}=\mathbb{Q}[-K_Y]$. %As multiples of $K_Y$ are never extremal, $\sigma^*$ does not fix any extremal ray.
\end{proof}

\subsection{Conic bundles} \label{ss-conic}
	In this section, we give properties of rank $r$ fibrations $X\to B$ over $T$ in the case when $\dim B=1$ and $T$ is the spectrum of an excellent regular ring of dimension at most $1$.

\begin{definition}\label{def: conic-bundle}
	Let $X$ be a regular quasi-projective surface over $T$ and let $\pi \colon X \to B$ be a projective contraction onto a regular curve $B$ over $T$.
	We say $\pi$ is a \emph{conic bundle} if $-K_X$ is ample over $B$ and it is a \emph{Mori conic bundle} if $\rho(X/B)=1$.
	We say $\pi$ is {\em generically a conic bundle} if $-K_X$ is big over $B$.
\end{definition}

Notice that $\pi$ is a conic bundle if and only if $\pi$ is a rank $\rho(X/B)$ fibration over a curve. Moreover, $\pi$ is a Mori conic bundle if and only if $\pi$ is a rank $1$ fibration over a curve.
Note that, since the generic fibre of $\pi$ is a curve, $\pi \colon X \to B$ is generically a conic bundle if and only if $-K_{X_{k(B)}}$ is ample. Therefore $\pi$ is generically a conic bundle if and only if the generic fibre $X_{k(B)}$ is a regular conic in $\mathbb{P}^2_{k(B)}$ by \cite[10.6]{kk-singbook}.

Let $b \in B$ be a closed point with residue field $k(b)$. We denote by $X_b:=X \times_B \Spec k(b)$ the fibre over $b$.
We will identify $X_b$ as a subscheme of $X$ via the closed immersion $X_b \to X$, and we will denote the associated divisor with the same notation.

\begin{lemma}\label{l-conic}
	Let $\pi \colon X \to B$ be a rank $r$ fibration over $T$, where $B$ is a curve.
	Let $b \in B$ be a closed point such that the fibre $X_b$ is irreducible.
	Then $X_b$ is an integral conic in $\mathbb{P}^2_{\k(b)}$.
	Moreover, if $\charact \k(b) \neq 2$, then $X_b$ is a geometrically reduced $\k(b)$-scheme.
	Finally, if  $\k(b)$  is separably closed and $\charact(\k(b)) \neq 2$, then $X_{b}$ is a geometrically regular $\k(b)$-scheme.
\end{lemma}

\begin{proof}
	Running the $K_X$-MMP over $B$ is an isomorphism around $X_b$ (since $X_b$ is irreducible by assumption), so we can assume that $X/B$ is a Mori fibre space.
	The proof of \cite[Proposition 2.18]{BT22} adapts to this setting.
\end{proof}

\begin{proposition}\label{prop: geom_normal_implies_canonical-over_base}
	Let $X$ be a projective regular surface over a field $\k$.
	Suppose that $f \colon X \to B$ is a rank $r$ fibration and $B$ a curve.
	If $X$ is geometrically normal, then $B$ is geometrically regular.
	Moreover, if $p\neq 2$, $r=1$, and $\k$ is separably closed, then $X$ is geometrically regular.
\end{proposition}
\begin{proof}
	As $X$ is geometrically normal and $\pi_*\mathcal{O}_X=\mathcal{O}_B$, we conclude that $B$ is geometrically normal and thus geometrically regular as $\dim(B) = 1$.

	For the last statement, as $\k$ is separably closed and $r=1$, all closed fibres are geometrically irreducible.
	Since $p>2$, by \autoref{l-conic} all fibers of $f$ are smooth and therefore, as $X \to B \to \Spec(\k)$ is a composition of smooth morphisms (cf. \cite[Proposition 2.18]{BT22}), we conclude that $X$ is geometrically regular.
\end{proof}

\begin{example}\label{ex:fibration not geometrically regular}
	Let $\k$ be an imperfect field of characteristic $p>0$.
	Let $\pi_1 \colon  \mathbb{P}^1_\k \times \mathbb{P}^1_{\k} \to \mathbb{P}^1_{\k}$ be the natural projection onto the first term.
	Let $\yy \in \mathbb{P}^1_{\k}$ be a closed point whose residue field extension $k(\yy)/\k$ is purely inseparable of degree $p$.
	Let $\xx \in \mathbb{P}^1_{k(\yy)} = \pi_1^{-1}(\yy)$ be a $k(\yy)$-rational point. For example, one can take $\mathfrak{m}_\xx=(x_0^p-t,y_0)$ for some $t\in\k\setminus\k^p$.
	As will follow from \autoref{lem:k(x)=k(b)} below, the blow-up $\eta \colon X \to  \mathbb{P}^1_\k \times \mathbb{P}^1_{\k} $ at $\xx$ is endowed with a natural rank $2$ fibration $\pi_1 \circ \eta \colon X \to \mathbb{P}^1_\k$ and $X$ is not geometrically regular.
\end{example}

\begin{example}\label{ex:fibration not geometrically normal}
Let $\k$ be a field of characteristic $p>2$ and take the regular curve $B=(zx^{p-1}+y^p+tz^p=0)$ in $\p^2_\k$ for $t\in\k\setminus \k^p$, and set $X=B\times\p^1_\k$.
Then $X$ is not geometrically normal, $X\to B$ is a rank $1$ fibration (but $X\to\p^1_\k$ is not as the generic fibre is not of genus zero).
In $\charact(\k)=2$, there exist even \emph{rational} rank $1$ fibrations $X\to \p^1_\k$ where $X$ is not geometrically normal, see \autoref{ex: Mori-conic-double-lines} and \autoref{ex:geometrically-non-normal-pdeg1}.
\end{example}

The following shows that surfaces with a Mori conic bundle structure admitting a section are projective bundles.

\begin{lemma}\label{lem: MFS_conic_with_section}
	Let $X$ be a regular surface and let $\pi \colon X \to B$ be a Mori conic bundle over $T$.
	If $X_{\k(B)}(\k(B)) \neq \emptyset$, then for all $b \in B$ we have $X_{k(b)} \simeq \mathbb{P}^1_{\k(b)}$.
	Moreover, there exists a locally free sheaf $\mathcal{E}$ of rank $2$ on $B$ such that $X \simeq \mathbb{P}_B(\mathcal{E})$.
\end{lemma}

\begin{proof}
	As $X$ is regular, then $X_{k(B)}$ is a regular Fano curve. Since  $X_{k(B)}(k(B)) \neq \emptyset$, $X_{k(B)}$ is geometrically reduced. As it is also geometrically irreducible, it is a smooth conic by \cite[10.6]{kk-singbook} and thus $X_{\overline{\k(B)}} \simeq \p^1_{\overline{\k(B)}} $. Thus  we conclude that $X_{\k(B)} \simeq \mathbb{P}^1_{\k(B)}$ by Châtelet's theorem \cite[Proposition 4.5.10]{Cha44} (see also \cite[10.6]{kk-singbook}).
	Consider a section $\Sigma$ of $\pi$.
	For every closed point $b \in B$, the fibre $X_b$ has intersection $X_b \cdot \Sigma=d_b$ (see Definition~\ref{def:d_x etc} for the definition of $d_b$), which implies that the schematic intersection of $\Sigma$ with $X_b$ is a regular $\k(b)$-point.
	Then $X_b$ is isomorphic to $\mathbb{P}^1_{\k(b)}$ by \cite[10.6]{kk-singbook}.

	We are left to verify that $\pi$ is a $\mathbb{P}^1$-bundle.
	As $\Sigma$ is a Cartier $\pi$-ample divisor, we have that $X \simeq \text{Proj}_{B} \bigoplus_{m \geq 0} \pi_*\mathcal{O}_X(m\Sigma)$ as $B$-scheme.
	By cohomology and base change \cite[Theorem 12.11]{Har77}, we have that $\mathcal{E}:=\pi_*\mathcal{O}_X(\Sigma)$ is a vector bundle of rank $2$ on $B$.
	Again by cohomology and base change, we obtain that $\pi_*\mathcal{O}_X(m \Sigma) \simeq \text{Sym}^m \mathcal{E}$ for every $m\geq0$, thus concluding $X \simeq \mathbb{P}_{B}(\mathcal{E})$.
\end{proof}

\begin{corollary} \label{lem: Mori-conic-p>2}
	Let $\k$ be a separably closed field of characteristic $p>2$.
	Let $f \colon X \to B$ be a Mori conic bundle onto a curve $B$ of genus $0$. Then $f$ is a smooth morphism and $X$ is a Hirzebruch surface.
\end{corollary}

\begin{proof}
	As $p>2$, $B$ is geometrically integral and since $\k$ is separably closed, $B(\k) \neq \emptyset$ by \cite[Corollary 3.5.71]{Poo17} and thus $B \simeq \mathbb{P}^1_\k$ by \cite[10.6]{kk-singbook}.
	As $\k$ is separably closed, we have that $X_{\bar{\k}}$ has Picard rank $2$ by \cite[Proposition 2.4(2)]{Tan18b}, and thus all the fibres of $X \to \mathbb{P}^1_\k$ are geometrically irreducible.
	As $X_{\k(b)}$ is reduced by \cite[Proposition 2.18]{BT22} and $p>2$, it is a smooth conic. Since $\k$ is separably closed, this concludes that $f$ is a smooth morphism and thus  $X_{\bar{\k}}$ is a Hirzebruch surface $\F_n$ (over $\bar\k$).
	We are left to verify that $X$ is a Hirzebruch surface over $\k$. If $n=0$, there is another Mori conic bundle structure $g \colon X \to C\simeq \mathbb{P}^1_\k$ by \cite[Proposition 2.4(2)]{Tan18b}, and the natural morphism $f \times g \colon X \to B \times C \simeq \mathbb{P}^1_\k \times \mathbb{P}^1_\k$ is an isomorphism.
	If $n>0$, then by  \cite[Proposition 2.4(2)]{Tan18b} (using the birational contraction $X_{\bar\k}\to\mathbb{P}(1,1,n)$ contracting $\Sigma_n$) the negative section $\Sigma_n$ on $X_{\bar{k}}$ descends over $\k$ showing that $f$ has a section. This means that the generic fibre of $f$ has a rational point, and we conclude by \autoref{lem: MFS_conic_with_section}.
\end{proof}

\begin{corollary} \label{cor: self-intersection}
Let $\k$ be a field and let $X \to B$ be a regular Mori conic bundle over $\k$ onto a geometrically integral curve of genus 0.
If $X_{k(B)}(k(B)) \neq \emptyset$, then $K_X^2=8$.
\end{corollary}

\begin{proof}
	As $B$ is geometrically integral, it is smooth. By \autoref{lem: MFS_conic_with_section}, we have that $\pi$ is a smooth morphism and thus $X$ is smooth.
	Geometrically, we deduce that $X$ is a Hirzebruch surface, and thus $K_X^2=8$.
\end{proof}

We now improve a structure result on surface Mori conic bundles proved by Koll\'ar and Mella in \cite[Lemma 17]{KM17}, which shows that a Mori conic bundle with no sections and positive $K_X^2$ is not too far from being a del Pezzo surface (see \cite[Proposition 5.2]{Pro18} and the references therein for the case of perfect fields).

\begin{proposition} \label{prop: conic-bundle-vs-dP}
Let $\pi \colon X \to B$ be a Mori conic bundle over a field $\k$, where $B$ is a geometrically integral curve of genus $0$.
Suppose that $0<K_X^2<8$ and define $F_{\xx}:=\pi^*\xx$ where $\mathcal{O}_B(\xx)$ generates $N^1(B)$ for some closed point $\xx\in B$.
Then one of the following holds:
\begin{enumerate}
	\item $K_X^2=1$, $B \simeq \mathbb{P}^1_{\k}$, $|-K_X | = C+|F_{\xx}|$, where $C$ is a geometrically integral smooth rational curve of self-intersection $-3$;
	\item $K_X^2=2$, $|-K_X | = C+|sF_{\xx}|$, where $s \in \left\{1,2\right\} , C_{\overline{\k}}$ is a Galois-conjugate pair of disjoint smooth rational curves with self-intersection $-3$ (in this case, $B \simeq \mathbb{P}^1_{\k}$ if and only if $s=2$);
	\item $X$ is a weak del Pezzo surface with $K_X^2<7$.
	Moreover, if $K_X^2>2$ and $K_X^2 \neq 4$, then $X$ is a del Pezzo surface.
\end{enumerate}
\end{proposition}

\begin{proof}
	By \cite[Theorem 3.3]{Tan18}, we deduce that $h^1(X, \mathcal{O}_X)=h^1(\mathbb{P}^1, \mathcal{O}_{\mathbb{P}^1})=0$, and therefore by the Riemann--Roch formula we have that $h^0(-K_X) \geq 1+K_X^2$. Moreover, $-K_X$ is big as $h^0(X, -mK_X) \geq 1+\frac{m(m+1)}{2}K_X^2$.

	As $\pi$ is a Mori conic bundle and $K_X^2<8$, by \autoref{cor: self-intersection} $\pi$ does not admit a section and thus we deduce that
	$$\NS(X)=\mathbb{Z}[-K_X] \oplus \pi^*\NS(B). $$
	Indeed, by the strong version of the base point free theorem for surfaces (see \cite[Proposition 2.1.1]{Ber21} or \cite[Theorem 4.2]{Tan18}), we have that $M \coloneqq \NS(X)/ \pi^*\NS(B)$ is a free module of rank $1$. As $-K_X \cdot F_{\xx} =2 d_\xx$, we have that either $-K_X$ generates $M$ or the generator $L$ of $M$ satisfies $2L=-K_X$.
	In the second case, by Riemann--Roch the divisor $L+ nF_{\xx}$ is linearly equivalent to an effective divisor $S$ for $n$ sufficiently large.
	As $S$ is a section of $\pi$, we deduce that $X$ is a projective bundle over $B$ by \autoref{lem: MFS_conic_with_section} and thus $K_X^2=8$, reaching a contradiction.

	Suppose that $|-K_X|=|M|+C$, where $|M|$ is the mobile part of the linear system and $C \neq 0$ the fixed part.
	As $B$ is a conic, note that either $\deg(\xx)=1$ (and $-K_B=2\xx$ and $B \simeq \mathbb{P}^1_{\k}$) or $\deg(\xx)=2$ (and $-K_B=\xx$).
	Note that we have
	\[ 2d_{\xx}=-K_S \cdot F_{\xx}=(C \cdot F_{\xx})+(M \cdot F_{\xx}) \geq d_{\xx} \deg(C \to B),  \]
	which implies $\deg(C \to B)=2$
	and $ M \sim sF_{\xx}$ for some some $s \in \mathbb{Z}_{\geq 0}$.
	Note that $$s \deg(\xx) = h^0(B, \mathcal{O}_B(s\xx))-1=h^0(X, sF_{\xx})-1=\dim |M| = \dim |-K_X| \geq K_X^2$$
	and
	\[ \deg(\omega_C)=-2d_{\xx}s , \text{ and } C^2=K_X^2-4sd_{\xx}. \]

	We now distinguish two cases.
	Suppose $h^0(C, \mathcal{O}_C)= \k$. In this case, $\deg(\omega_C)=-2$ and thus $d_\xx=s=1$. This implies that $B \simeq \mathbb{P}^1_\k$. As $1=sd_\xx \geq K_X^2$, we conclude that  $K_X^2=1$ and $C^2=-3$. Note that in this case $C$ is geometrically connected and irreducible.

	Suppose $h^0(C, \mathcal{O}_C)$ is a degree 2 extension of $\k$.
	In this case, $\deg(\omega_C)=-4$ and thus $d_\xx s =2$ which implies that $2 \geq K_X^2$. Note that $K_X^2=2C \cdot M +C^2$ must be divisible by 2 as $h^0(C, \mathcal{O}_C)$ divides $C^2$. Thus $K_X^2=2$ and $C^2=2-8=-6$.
	Note that either $C_{\overline{\k}}$ is a Galois-conjugate pair of smooth rational curves (if the degree 2 extension is separable) or it is a double line (if the extension is not separable).

	Suppose now that $K_X^2>2$. By the previous discussion, $-K_X$ has no fixed components and thus $-K_X$ is big and semi-ample and $X$ is a weak del Pezzo surface. To conclude, we are left to study when a weak del Pezzo surface with a Mori conic bundle can fail to be a del Pezzo surface.

	Suppose there exists an integral curve $D$ such that $-K_X \cdot D=0$.
	We can write $D \sim a[-K_X] -b[F_{\xx}]$ for some integers $a,b$.
	We have therefore
	$$ 0=-K_X \cdot (-aK_X-bF_{\xx})= a(K_X^2)-2bd_\xx.$$
	Moreover, we have by adjunction that
	$$-2d_D=D^2=a(aK_X^2-4bd_\xx).$$
	Write $D \cdot F_\xx = d_D \cdot m$ for some $m>0$ and we have also that $D \cdot F_{\xx}=2ad_\xx$.
	All these equations together show that $abmd_\xx=2ad_\xx$, which shows that
	$$ bm=2 \text{ and } a(K_X^2)=\frac{4}{m}d_{\xx}. $$
Then $m=1,b=2$ or $m=2,b=1$. Since $d_{\xx}\in\{1,2\}$, we conclude as follows:
	\begin{enumerate}
		\item if $m=1,b=2$, then $aK_X^2=4d_{\xx}\in\{4,8\}$. As $8>K_X^2>2$, we obtain $K_X^2=4$ and either $a=1,d_D=2$ or $a=2,d_D=8$.
		\item If $m=2,b=1$, then $aK_X^2=2d_{\xx}\in\{2,4\}$. As $K_X^2>2$, we obtain $K_X^2=4$, $a=1$ and $d_D=2$.
	\end{enumerate}
\end{proof}

We obtain the following restriction when two Mori conic bundle structures exist. 

\begin{lemma} \label{lem:linkIV-possible-degree}
	Let $\k$ be a field and let $f\colon X \to B$ be a surface Mori conic bundle over a projective curve $B$.

	Assume that there exists a Mori conic bundle $g \colon X \to C$ such that $g^*H_{B} \cdot f^*H_{C}>0$ for $H_B$ (resp. $H_C$) ample on $B$ (resp. $C$).
	Then $X$ is a del Pezzo surface and $K_X^2\in\{1,2,4,8\}$.
	Moreover, if $X$ is geometrically integral, $B$ and $C$ are curves of genus 0.
\end{lemma}
\begin{proof}
	The property that $X$ is del Pezzo follows immediately from Kleiman's criterion.
	Denote by $f_1$ the fibre of $X \to B$ over a closed point $b$ of degree $d_1$, and denote $f_2$ the fibre of $X \to C$ over a closed point $c$ of degree $d_2$.
	Writing $K_X=-(a_1f_1+a_2f_2)$ for some $a_1, a_2 \in\Q$, we obtain from adjunction that
	\[2d_i=-K_X\cdot f_i=a_j f_1\cdot f_2,\]
	for $\{i,j\}=\{1,2\}$,
	with $f_1\cdot f_2\in d_1d_2\Z_{>0}$.
	We obtain \[K_X^2=2a_1a_2f_1\cdot f_2=\frac{8d_1d_2}{f_1\cdot f_2}.\]
	As $f_1 \cdot f_2$ is a positive multiple of $d_1d_2$, we conclude $K_X^2\in\{1,2,4,8\}$.

	To show that $B$ and $C$ are curves of genus 0, it is sufficient to show that $h^1(X, \mathcal{O}_X)=0$.
	By \cite[Proposition 4.8]{BM23} and the assumption that $X$ admits two distinct fibrations onto a curve, we deduce that $h^1(X, \mathcal{O}_X)=0$.
\end{proof}

We extend \cite[Theorem 5(1)]{Isko80} (see also \cite[Theorem 5.3]{Pro18}) to the regular case over arbitrary fields.

\begin{proposition}\label{prop: Mori-conic-bdl-dP}
	Let $\pi \colon X \to B$ be a geometrically integral surface Mori conic bundle over a field $\k$ onto a curve $B$ of genus $0$ with $3\leq K_X^2\leq7$ and $K_X^2\neq 4$.
	Then $X$ is a del Pezzo surface and there is a birational contraction $f \colon X \to Y$ such that $Y$ is a del Pezzo surface and one of the following holds:
	\begin{enumerate}
		\item\label{it:Mori-conic-bdl-dP--1} $Y=\p^2_\k$ and $K_X^2=5$.
		\item\label{it:Mori-conic-bdl-dP--2} $K_Y^2=8$ and $K_X^2=6$.
		\item\label{it:Mori-conic-bdl-dP--3} $K_Y^2=4$ and $K_X^2=3$.
	\end{enumerate}
\end{proposition}

\begin{proof}
	By \autoref{prop: conic-bundle-vs-dP}, $X$ is a del Pezzo surface and thus it admits a second $K_X$-negative extremal ray $\xi$.
	By \autoref{lem:linkIV-possible-degree}, we have that the $K_X$-negative extremal contraction $f \colon X\to Y$ associated to $\xi$ is birational. As $\rho(X)=2$, we have that $\rho(Y)=1$ and $f$ is the blow-up of one closed point $\xx$ on $Y$, with $K_X^2 = K_Y^2-\deg(\xx)$. As $K_X^2\geq 3$ and $K_Y^2 \leq 9$ by geometric integrality, we have $\deg(\xx) \leq 6$.
	Let $H$ be a Cartier divisor on $Y$ such that $N^1(Y)=\mathbb{Z}[H]$ and write $-K_Y \equiv iH$.

	Consider the pencil of curves $\{C_t\}$ on $Y$ corresponding to the fibration $\pi\colon X\to\p^1_\k$.
	Writing $C_t\equiv dH$ for some $d\geq 1$, we find that the strict transform $\tilde C_t$ in $X$ satisfies $0=\tilde C_t^2=d^2H^2-m_{\xx}(C_t)^2\deg(\xx)$, where $m_{\xx}(C_t)$ denotes the multiplicity of $C_t$ at $\xx$. 
	We note that the parity of the multiplicities of the prime factorization of the integer $K_Y^2=i^2H^2$ equals the one of $H^2$, which in turn equals the one of $\deg(\xx)$ since $d^2H^2=m_{\xx}(C_t)^2\deg(\xx)$. 
	Going through each $4\leq K_Y^2\leq9$ and each $1\leq\deg(\xx)\leq6$ such that $K_Y^2-\deg(\xx)\geq 3$, one can check that the only possibilities for $(K_Y^2, \deg(\xx))$ are $(9,1)$ (which is excluded by $K_X^2\neq8$) as well as $(9, 4)$, $(8, 2)$ and $(4, 1)$.
	%\textcolor{red}{Thus the multiplicities of the prime factors of $H^2$ (and thus of $K_X^2=i^2H^2$ (DO YOU MEAN $K_Y^2$??)) and $\deg(\xx)$ have the same parity. As $\deg(\xx) \leq 6$, $9 \geq K_Y^2 \geq 4$ and $K_X^2 \neq 8$, it is elementary to see that the only numerical possibilities for $(K_Y^2, \deg(\xx))$ are (9, 4),(8, 2),(4, 1).}

 We are left to show that in the case where $K_Y^2=9$, then $Y$ is $\mathbb{P}^2_\k$. As $Y$ is a Severi-Brauer surface by \autoref{lem: SB} and there exists a point $\xx$ with $\deg(\xx)=4$, then $Y$ is a trivial Severi-Brauer surface by \autoref{lem: index_SB}, that is, $Y\simeq\p^2_\k$.
\end{proof}

We discuss now del Pezzo surfaces of degree 4 with a Mori conic bundle structure and extend part of \cite[Theorem 5(3)]{Isko80} to arbitrary fields.

\begin{lemma} \label{lem: dP4_mfs}
Let $\k$ be any field.
Let $X$ be a geometrically integral del Pezzo surface of degree $K_X^2\in \{1,2,4\}$ admitting a Mori conic bundle structure $X \to B$ onto a curve of genus $0$.
Then $X$ admits a second Mori conic bundle structure.
\end{lemma}

\begin{proof}
	As $X$ is a del Pezzo surface of Picard rank 2, it admits a second $K_X$-negative extremal ray associate to a contraction $f \colon X \to Y$.
	Suppose by contradiction $f$ is birational.
	In this case, $Y$ is a del Pezzo surface of Picard rank 1 and $f$ is the blow-up of a closed point $\yy$ on $Y$ of degree $K_Y^2$.

	By \autoref{lem:no-minimal-dp7} $K_Y^2\neq7$.
	Let $(C_t)$ be a pencil of curves on $Y$ such that their strict transform becomes the conic bundle.
	As $K_Y$ generates $\Pic(Y)$, we have that $C_t \equiv -dK_Y$ for some $d \geq 1$.
	In particular, writing $m=m_\yy(C_t)$, we have
	\begin{equation}\label{eq:cf}
	0=\tilde{C}_t^2=C_t^2-m^2\deg(\yy)=d^2K_Y^2-m^2\deg(\yy).
	\end{equation}

	If $K_X^2=1$, then $K_Y^2\in\{2,3,4,5,6,8,9\}$ and so \eqref{eq:cf} becomes of the following
	\begin{align*}
	&0=2d^2-m^2,\ 0=3d^2-2m^2, \ 0=4d^2-3m^2,\ 0=5d^2-4m^2,\\
	 &0=6d^2-5m^2, \ 0=8d^2-7m^2,\ 0=9d^2-8m^2.
	\end{align*}
	However, none of them have solutions $d>0$, $m\geq0$.

	If $K_X^2=2$, then $K_Y^2\in\{3,4,5,6,8,9\}$ and so \eqref{eq:cf} becomes of the following
	\begin{align*}
	&0=3d^2-m^2, \ 0=4d^2-2m^2,\ 0=5d^2-3m^2,\ 0=6d^2-4m^2, \\
	 &0=8d^2-6m^2,\ 0=9d^2-7m^2,
	\end{align*}
	However, none of them have solutions $d>0$, $m\geq0$.

	If $K_X^2=4$, then $K_Y^2\in\{5,6,8,9\}$ and so \eqref{eq:cf} becomes of the following
	\[
	0=5d^2-m^2, \ 0=6d^2-2m^2, \ 0=8d^2-2m^2,\ 0=9d^2-5m^2.
	\]
	None of them have solutions $d>0$, $m\geq0$, $d>m$.
\end{proof}

%===================================
\subsection{Points in Sarkisov general position}

We introduce the notion of points in general position for surfaces in the context of the Sarkisov program (cf. \cite[Definition A.3]{Cor95} in the case of del Pezzo surfaces).

\begin{definition}\label{def: general_position}
	Let $X \to B$ be a rank $r$ fibration over $T$.
	We say that $n$ distinct closed points $p_1,\ldots,p_n$ on $X$ are in \emph{Sarkisov general position} over $B$ if the blow-up $Y\to X$ of $p_1,\ldots,p_n$ is a rank $r+n$ fibration over $B$.
\end{definition}

If $X=\p^2_\k$ and $\k$ is a perfect field, we recover the classical notion of points in general position (cf.\ \cite{Dol12}).
Note that being in Sarkisov general position is a relative notion over $B$: For example, a point on $\p^1_\k\times\p^1_\k$ can be in Sarkisov general position or not depending on whether one considers it over $\Spec\k$, the first $\p^1_\k$, or the second $\p^1_\k$.

The following is an elementary restraint on points in general position on a del Pezzo surface. We work again over an arbitrary field $\k$ of characteristic $p\geq0$.

\begin{lemma} \label{lem: how-much-we-can-blow-up}
	Let $X$ be a geometrically integral del Pezzo surface over a field $\k$. If $p_1,\ldots,p_n$ on $X$ are in Sarkisov general position over $\k$, then $\sum [k(p_i): \k] < K_X^2$. In particular, $\sum [k(p_i): \k] < 9$.
\end{lemma}

\begin{proof}
	Let $f\colon Y\to X$ be the blow-up of $p_1,\dots,p_n$.
	As $K_Y=f^*K_X+\sum_{i} E_i$, we have $K_Y^2=K_X^2-\sum_i [\k(p_i): \k]$. As $Y$ is del Pezzo by hypothesis, we have $K_Y^2>0$, which yields the first claim.
	The second claim follows from \autoref{cor: geometrically integral del Pezzo}.
	\end{proof}

	The following lemma gives sufficient and necessary conditions under which a closed point $\xx$ is in Sarkisov general position on a conic bundle.

	\begin{lemma}\label{lem:k(x)=k(b)}
		Let $\pi \colon X\to B$ be rank $r$ fibration over $T$ and $B$ a curve.
		Let $b\in B$ be a closed point and let $\eta\colon Y\to X$ be the blow-up of a point $\xx \in X_b$.
		Then $\pi \circ \eta\colon Y\to B$ is a rank $r+1$ fibration (that is, $\xx$ is in Sarkisov general position on $X$ over $B$) if and only if
		\begin{enumerate}
			\item\label{it:k(x)-1} $X_b$ is a regular $\k(b)$-conic, and
			\item\label{it:k(x)-2} the extension $\k(b) \subset \k(\xx)$ induced by $\pi$ is trivial.
		\end{enumerate}
	\end{lemma}

	\begin{proof}
		As the statement on ampleness is relative to $B$, it is sufficient to check the intersection products over $B$.
		Notice that if $X_b$ is not irreducible, then $K_Y$ is not $(\pi\circ\eta)$-ample and hence $\pi\circ\eta$ not a rank $r+1$ fibration. So, we can assume that $X_b$ is irreducible.
		Then $X_b$ is an integral (reduced) conic in $\p^2_{\k(b)}$ by \autoref{l-conic}.
		Let $E\subset Y$ be the exceptional divisor over $\xx$; it has self-intersection $E^2=-[k(\xx):k(b)]$. Denote by $\widetilde{X}_b\subset Y$ the strict transform of $X_b$.
		As $\eta^*X_b=\widetilde{X}_b+\mult_{\xx}(X_b)E$,
		we have
		\begin{align*}
			K_Y \cdot \widetilde{X}_b=K_X\cdot X_b-\mult_{\xx}(X_b)E^2=-2+\mult_{\xx}(X_b)\cdot [k(\xx):k(b)].
		\end{align*}
		This intersection is negative if and only if conditions \ref{it:k(x)-1} and \ref{it:k(x)-2} are satisfied.
		This is equivalent to $-K_Y$ being relatively ample by Kleiman's criterion (because $K_Y\cdot f=K_Y \cdot \eta(f)<0$ for any other closed fibre $f$ of $Y$ by hypothesis).
	\end{proof}

See \autoref{ex: Mori-conic-double-lines} and \autoref{ex:geometrically-non-normal-pdeg1} for examples of rational Mori conic bundles without any point in Sarkisov general position; that is, they do not admit any Sarkisov link of type \II.

\subsection{Non-rationality of geometrically non-normal del Pezzo surfaces}

We use the Sarkisov program to show birational rigidity of geometrically non-normal del Pezzo surfaces of Picard rank $1$.
We recall the definition of birational rigidity in the context of surfaces.

\begin{definition} \label{def: birational-rigidity}
	Let $\k$ be a field, let $X$ be a regular projective $\k$-surface and let $X \to B$ be a Mori fibre space over $\k$.
	We say $X/B$ is \emph{birationally rigid} if given a birational map $\chi \colon X \rat X'$, where $X' \to B'$ is a Mori fibre space, there exists an isomorphism $\sigma \colon B \to B'$ and a birational map $\varphi \in \Bir(X)$ such that $\chi \circ \varphi$ is an isomorphism and the following diagram commutes:
	\begin{equation}\label{diag:birat_rigidity}
		\begin{tikzcd}
			X \ar[r, dashed,"\varphi"] \ar[d]& X \ar[r, dashed,"\chi"] &  X' \ar[d]  \\
			B \ar[rr,"\sigma"] & & B'.
		\end{tikzcd}
	\end{equation}
	We say $X \to B$ is \emph{birationally super-rigid} if every birational map $X\rat X'$ to a Mori fibre space $X'\to B'$ is an isomorphism of Mori fibre spaces.
	 If $B=\Spec(\k)$, then $X$ is birationally super-rigid if and only if it is birationally rigid and $\Bir_\k(X)=\Aut_\k(X)$.
\end{definition}

We show, as an application of the Sarkisov program \autoref{t-sarkisov-program}, that geometrically non-normal del Pezzo surfaces of Picard rank $1$ are not rational. 

\begin{theorem}[Rigor-mortis statement]\label{prop: birational-rigidity-non-norm-dP}
	Let $\k$ be a field of characteristic $p>0$.
	Let $X$ be a del Pezzo surface over $\k$ of Picard rank $1$ and degree $K_X^2=d$.
	Suppose $X$ is not geometrically normal. Then $X$ is not rational. 
	If we assume $X$ is geometrically integral, we have
	\begin{enumerate}
		\item\label{rigor-mortis:1} if $(p,d) \neq (2,4)$, then $X$ is birationally super-rigid,
		\item\label{rigor-mortis:2} if $(p, d) =(2,4)$, then the only Mori fibre spaces birational to $X$ are $X$ itself and, for each rational point $\xx$ of $X$, the blowup $Y$ at $\xx$ endowed with a (generically non-smooth) conic bundle structure $Y \to \mathbb{P}^1_k$.
	\end{enumerate}
\end{theorem}

\begin{proof}
If $X$ is not geometrically integral, then $X$ is not $\k$-rational. Indeed, if $X$ is birational to $\mathbb{P}^2_\k$, then it is geometrically $R_0$ (i.e. regular in codimension 0). As $X$ is $S_1$, then $X_{\overline{\k}}$ is $S_1$ and $R_0$, thus reduced by \cite[\href{https://stacks.math.columbia.edu/tag/031R}{Tag 031R}]{stacks-project}, reaching a contradiction. Thus we can suppose that $X$ is geometrically integral from now on.

As $\rho(X)=1$, there are no Sarkisov links of type III and IV starting at $X$.
	By \autoref{thm: dP-base-change-classification}, we have that $p \in \left\{2, 3 \right\}$. Moreover, if $p=3$ then $K_X^2\in\left\{1,3 \right\}$, and if $p=2$ then $K_X^2 \in \left\{ 1, 2, 4 \right\}$.

	\autoref{rigor-mortis:1} It suffices to show that there is no rank $2$ fibration dominating $X$, equivalently there are no links of type I and II starting at $X$.

	If $K_X^2=1$, then there are no rank $2$ fibrations dominating $X$ (as for the blow-up $Y \to X$, we have $K_Y^2 \leq 0$) and therefore $X$ is birationally super-rigid.

	%If $K_X^2=2$, then $p=2$, the normalised base change  $(X \times_\k \overline{\k})_{\text{red}}^n \simeq \mathbb{P}(1,1,2)$ and the conductor is $2L$.
	If $K_X^2=2$, a rank $2$ fibration $Y$ dominating $X$ is obtained by blowing up a $\k$-rational point, which lies in the smooth locus of $X$.
	Thus $Y$ is a geometrically non-normal del Pezzo surface of Picard rank $2$ and $K_Y^2=1$, which does not appear in the list of \autoref{thm: dP-base-change-classification}.

	If $K_X^2=3$, then $p=3$ and the normalised base change is $\mathbb{P}(1,1,3)$ with conductor $2L$.
	As $K_X^2=3$, a rank $2$ fibration $Y$ dominating $X$ is obtained by blowing up a closed point $\xx$ of degree at most $2$.
	As $p=3$, $\k(\xx)$ is a separable extension of $\k$, and thus $\xx$ is contained in the smooth locus of $Y$. This implies that $Y$ is a del Pezzo surface which is geometrically non-normal with $\rho(Y)=2$, contradicting the list of \autoref{thm: dP-base-change-classification}.

	\autoref{rigor-mortis:2}
	If $K_X^2=4$, then any rank $2$ fibration $Y$ dominating $X$ is geometrically non-normal with $K_Y^2<4$.
	By \autoref{thm: dP-base-change-classification}, we have $\rho(Y_{\rm sep})=2$ and $K_Y^2\leq 3$.
	If $K_Y^2=2$, then $(Y_{\bar\k})^n$ is isomorphic to $\p^1\times\p^1$ by the list in \autoref{thm: dP-base-change-classification} and the two extremal contractions from $Y$ are onto a curve. 
	This is a contradiction to having a birational contraction $Y\to X$.
	Therefore, $K_Y^2=3$, and $f \colon Y \to X$ is the blow up at a $\k$-rational point with exceptional divisor $E \simeq \mathbb{P}^1_\k$.
	The divisor $-K_Y-E=-f^*K_X-2E$ is base-point-free and the associated contraction is a conic bundle $\pi\colon Y \to \mathbb{P}^1_k$.
	Geometrically, $Y$ is a cubic surface, $E \subset Y$ is a line and $\pi \colon Y \to \mathbb{P}^1$ is the projection from the line $E$.
	We claim that the generic fibre of $\pi$ is geometrically non-reduced. If not, the non-normal locus of $Y_{\overline{\k}}$ would be vertical with respect to $\pi_{\overline{\k}}$, contradicting the classification of the conductors in \autoref{thm: dP-base-change-classification}.
	This shows we cannot perform links of type II by \autoref{lem:k(x)=k(b)}. 
	This implies that the only surfaces with a Mori fibre space structure birational to $X$ are $X$ itself and cubic surfaces with a generically non-smooth Mori conic bundle structure.
	This concludes the proof.
\end{proof}

%===================================

\section{Classification of Sarkisov links} \label{s: Sarkisov_Links}

In this section, we classify Sarkisov links over arbitrary fields 
with respect to the numerical invariants of the involved surfaces (such as the self-intersection of the canonical class and the degree of the closed points that are being blown up).
The results here should be compared with \cite[Theorem 2.6]{Isk96}.

\subsection{Links of type \I}

We classify links of type $\I$ (and $\III$) for surfaces over arbitrary fields.

\begin{proposition} \label{prop: link_I} 
	Let $\k$ be any field and let $X$ be a geometrically integral del Pezzo surface with $N^1(X)=\Z H_X$ for $H_X$ an ample Cartier divisor on $X$.
	Let $\pi \colon Y \to X$ be the blow-up of a closed point $\xx$ with exceptional divisor $E$, such that
	\begin{center}
		\begin{tikzcd}[ampersand replacement=\&,column sep=1.3cm,row sep=0.16cm]
		 \&\& Y \ar[dd,"fib"]\ar[ddll,"div"] \\ \\
		X\ar[uurr,"\chi",dashed, bend left=20]  \ar[dr,"fib",swap] \&  \& B \ar[dl] \\
		\& \Spec k \&
		\end{tikzcd}
	\end{center}
	is a link of type \I. Then $B$ is a curve of genus $0$.
	Writing $H$ for the pull-back of $H_X$ on $Y$ and $C$ for the divisor on $Y$ giving the fibration to $B$, we have the following possibilities:
	\begin{enumerate}
	\item $K_X^2=9$ and
		\subitem $\deg(\xx)=1$: in this case $X \simeq \mathbb{P}^2_\k$, $Y\simeq \mathbb{F}_1$, $C\equiv H-E$.
		\subitem $\deg(\xx)=4$: in this case $X \simeq \mathbb{P}^2_\k$, $C\equiv 2H-E$.
	\item $K_X^2=8$, $\deg(\xx)=2$, and either
	\subitem (a) $-K_X=2H_X$, $B\simeq \mathbb{P}^1_k$, and $C\equiv H-E$, or
	\subitem (b) $-K_X=H_X$, $B(k) =\emptyset$, and $C\equiv H-2E$.
	\item $K_X^2=4$ and $\deg(\xx)=1$: $Y$ is a cubic surface and $f$ is the projection from a line with $C\equiv H-2E$.
\end{enumerate}
\end{proposition}

\begin{proof}
	By \autoref{prop: birational-rigidity-non-norm-dP}, $X$ is geometrically normal, unless $p=2$ and $K_X^2=4$. 
	Therefore, $h^1(X, \mathcal{O}_X)=0$ by \cite[Theorem~1.1]{BM23}. As $Y$ and $X$ are both regular, we have $h^1(Y, \mathcal{O}_{Y})=h^1(X, \mathcal{O}_X)=0$. By the Leray spectral sequence, we see that $h^1(B, \mathcal{O}_B)$ injects into $h^1(Y, \mathcal{O}_Y)$, and thus vanishes. So $B$ is a curve of genus $0$.
	Since $Y$ is a del Pezzo surface by assumption, by \autoref{lem: dP4_mfs} we have that $K_Y^2 =3 $ or $K_Y^2 \geq 5$.
	Suppose $K_X^2=9$.
	By \autoref{prop: Mori-conic-bdl-dP}, we conclude that $K_Y^2$ is either 8 or 5. In both cases, $X$ has a closed point of degree not divisible by 3, and thus $X \simeq \mathbb{P}^2_\k$.
	
	Suppose $K_X^2=8$.
	By  \autoref{prop: Mori-conic-bdl-dP}, we conclude that $K_Y^2$ is 6. Let $\left\{C_t\right\}$ be a pencil of curves on $X$ such that $\tilde{C_t}$ is base point free and let $d>0$ such that $C_t \equiv dH_X$.
	We have $0=\tilde{C_t}^2=C_t^2-2m_{\xx}(C_t)=d^2H_X^2-2m_{\xx}(C_t)^2$.
	Recall by \autoref{lem: dp8-vs-quadric} that either (a) $X$ is a quadric over $\k$ (and in this case $-K_X=2H_X$) or (b) $X$ is not a quadric over $\k$ (and in this case $-K_X=H_X$).
	In case (a), we thus have $2d^2=2m_{\xx}(C_t)^2$ while in case (b) we have $8d^2=2m_{\xx}(C_t)^2$.
	In case (a), consider the pencil of planes $\{H_t\}$ in $\mathbb{P}^3_\k$ passing through $\xx$ and write $C_t=Q \cap H_t$. The strict transforms $\left\{\widetilde{C}_t\right\}$ form a free pencil of conics on $X$, concluding the base of the fibration is $\mathbb{P}^1_\k$.
	In case (b), suppose by contradiction there exists a closed fibre $f$ over a rational point.
	In this case, we have $f \equiv H-2E $ and $2=-K_Y \cdot f=-K_Y \cdot (H-2E)=(H-E)(H-2E)=H^2+2E^2$.
	As $H^2=8$ and $E^2=-2$ we get a contradiction.
	
	The case of $K_X^2=7$ does not appear by \autoref{lem:no-minimal-dp7}
	The case of $K_X^2=6$ also does not appear as in this case $K_Y^2 \in \left\{3, 5 \right\}$ and this is not possible by \autoref{prop: Mori-conic-bdl-dP}.
	Finally, if $K_X^2=4$, then $\pi$ is the blow-up of $X$ at a rational point. Thus $Y$ is a del Pezzo of degree 3 with a rational line and $f$ is the projection from the line.
\end{proof}

\begin{remark}
\begin{enumerate}
	\item Links of type \III are the inverse of links of type \I, and thus \autoref{prop: link_I} classify also those.
	\item Links of type I with a non-separable base point appear only in characteristic $2$, starting from $\mathbb{P}^2_\k$ or a del Pezzo surface of degree $8$.
	\item The surface $Y$ in a link of type \I can fail to be geometrically normal. See \autoref{ex:geometrically-non-normal-pdeg1} for an example over a field with $p$-degree $1$, and \autoref{ex: Mori-conic-double-lines} over a field with $p$-degree $2$.
\end{enumerate}
\end{remark}

\subsection{Links of type \II}\label{sss: type II}

We classify links of type \II for surface Mori fibre spaces over arbitrary fields.
We distinguish according to whether the base is a point (Section~\ref{sss:IIoverSpec}) %(\noindent\ref{sss: type II}(a)) 
or a curve (Section~\ref{sss:IIovercurve}).%(\noindent\ref{sss: type II}(b)).
\bigskip

\begin{definition}[Notation]
Let us fix a notation for links of type \II: consider a Sarkisov link $\chi\colon X\rat Y$ of type \II with base-point $\xx$ and let $\yy$ be the base-point of $\chi^{-1}$. 
In this situation, we write $\chi=:\chi_{\xx,\yy}$.  For $a=\deg(\xx)$ and $b=\deg(\yy)$, we will write $\chi\colon X\overset{a:b}\rat Y$.
\end{definition}

%\noindent\ref{sss: type II}(a) \emph{Links of type $\II$ over $\Spec(\k)$.}
\subsubsection{Links of type $\II$ over $\Spec(\k)$}\label{sss:IIoverSpec}

\begin{remark}
	\autoref{lem: how-much-we-can-blow-up} implies that links of type \II between del Pezzo surfaces which have base points with non-separable residue field appear only in characteristic $p \leq 7$.
	Moreover, in characteristic $5$ and $7$ the base points must be purely inseparable.
\end{remark}

\begin{proposition}\label{prop: link II dP}
Let $\k$ be any field and let $X$ be a geometrically integral del Pezzo surface with $N^1(X)=\Z H_X$ for $H_X$ an ample Cartier divisor on $X$.
	Let $\pi \colon Y \to X$ be the blow-up of a closed point $\xx$ with exceptional divisor $E$, such that
	\begin{center}
	\begin{tikzcd}[ampersand replacement=\&,column sep=.8cm,row sep=0.16cm]
	 \& Y\ar[ddl,"\xx",swap]\ar[ddr,"\yy"]\& \\ \\
	X \ar[rr,"\chi",dashed] \ar[dr,swap] \&  \& X' \ar[dl] \\
	\& \Spec(\k) \&
\end{tikzcd}
	\end{center}
	is a link of type \II.
	Writing $H$ for the pull-back of $H_X$ on $Y$, and $E'$ for the exceptional divisor over $\yy$, we have the following possibilities:
	\[
	\begin{array}{l|ccc|c}
		X&\deg(\xx)&\deg(\yy)&E'& K_{X'}^2\\ \hline
		X=\p^2_{\k}, 3H_X=-K_X
&2 & 1 & 1H - 1E& 8\\
&3 & 3 & 3H - 2E& 9\\
&5 & 1 & 2H - 1E& 5\\
&6 & 6 & 12H - 5E& 9\\
&7 & 7 & 21H - 8E& 9\\
&8 & 8 & 48H - 17E& 9\\
\hline
K_X^2=8, 2H_X=-K_X
&1 & 2 & 1H - 2E& 9\\
&3 & 1 & 1H - 1E& 6\\
&4 & 4 & 4H - 3E& 8\\
&5 & 2 & 3H - 2E& 5\\
&6 & 6 & 12H - 7E& 8\\
&7 & 7 & 28H - 15E& 8\\
\hline
K_X^2= 9 , H_X=-K_X
&3 & 3 & 1H - 2E& 9\\
&6 & 6 & 4H - 5E& 9\\
\hline
K_X^2= 8 , H_X=-K_X
&4 & 4 & 2H - 3E& 8\\
&6 & 6 & 6H - 7E& 8\\
\hline
K_X^2= 6
&1 & 3 & 1H - 3E& 8\\
&2 & 2 & 1H - 2E& 6\\
&3 & 3 & 2H - 3E& 6\\
&4 & 4 & 4H - 5E& 6\\
&5 & 5 & 10H - 11E& 6\\
\hline
K_X^2= 5
&1 & 5 & 2H - 5E& 9\\
&2 & 5 & 3H - 5E& 8\\
&3 & 3 & 3H - 4E& 5\\
&4 & 4 & 8H - 9E& 5\\
\hline
K_X^2= 4
&2 & 2 & 2H - 3E& 4\\
&3 & 3 & 6H - 7E& 4\\
\hline
K_X^2= 3
&1 & 1 & 1H - 2E& 3\\
&2 & 2 & 4H - 5E& 3\\
\hline
K_X^2= 2
&1 & 1 & 2H - 3E& 2\\
\hline
	\end{array}
	\]
\end{proposition}

\begin{proof}
	We have $\Pic(Y)=\Pic(X) \oplus \mathbb{Z}[E]$ by \cite[Proposition 2.18]{Liu02}, thus we can
	write $E'\equiv dH-mE$ for some integers $d,m\in\Z$. Let us show that $m,d>0$. Indeed, as $E'$ is effective, we have that the pushforward of $dH-mE$ on $X$ is ample, and thus $d>0$. As $E'$ is an effective prime divisor and $E' \neq E$, we have $E' \cdot E \geq 0$, which implies $-mE^2 \geq 0$, concluding $m \geq 0$. As $E'^2<0$, we conclude $m>0$.
	By \autoref{cor: moredP}, we are in one of the following cases: $X=\p^2_{\k}$ (i.e. $K_X^2=9$ and $H_X^2=1$), or $K_X^2=8$ and $H_X^2=2$, or $K_X^2\in\{2,3,4,5,6,8,9\}$ and $-K_X=H_X$.
	Note that there is no rank $2$ fibration $Y \to X$ dominating a del Pezzo surface of degree $1$, as one would then have $K_Y^2 \leq 0$.

	Writing $a=\deg(\xx)$ and $b=\deg(\yy)$, we have the following two equations by \autoref{sec:BirationalgeometryofExcellentSurfaces}, 
	where $f$ is the blow-up $Y \to X$:
	\begin{align*}
		-b &=\deg(\yy)= K_Y\cdot E' = (f^*K_X+E)(dH-mE)=dK_XH_X+am, \\
		-b &=\deg(\yy)= (E')^2= (dH-mE)^2 = d^2H_X^2-am^2.
	\end{align*}
	Next, for a fixed $X$, we go through all cases $a\in\{1,\ldots, K_X^2-1\}$, $b\in\{1,\ldots,9-(K_X^2-a)\}$, and check when the value of $m$ is a non-negative integer.

	Writing $K_X=-\lambda H_X$, note that $(K_XH_X)^2/H_X^2=\lambda^2H_X^2=K_X^2$.
	The first equation gives $d=-\frac{b+am}{K_XH_X}$, so $d^2H_X^2=(b+am)^2H_X^2/(K_XH_X)^2=\frac{(b+am)^2}{K_X^2}$.
	The second equation becomes a quadratic equation in $m$, namely
	\begin{align*}
		0&=d^2H_X^2-am^2+b \\ &= \frac{(b+am)^2}{K_X^2} - am^2+b \text{, which implies }\\ 0 &=m^2a(a-K_X^2) +2abm + b(b+K_X^2).
		\end{align*}
	One can check that the only cases $(a,b)$ that give a non-negative integer $m$ are the ones from the statement, and, in addition, the following cases:
	\begin{enumerate}
		\item $K_X^2=9$, $H_X=-K_X$, and $(a,b)\in\{(2,1),(5,1),(7,7),(8,8)\}$. These cases are not possible because $3\mid a$ on a non-trivial Severi-Brauer surface.
		\item $K_X^2=8$, $H_X=-K_X$, and $(a,b)\in\{(1,2), (3,1), (5,2), (7,7)\}$. These cases are excluded as there are no points of odd degree on $X$ by \autoref{prop: divisibility-dp8}.
		\item $K_X^2=3$ and $(a,b)=(1,6)$: In this case $K_{X'}^2=8$, but in the list of del Pezzo surfaces of degree $8$ the case $(6,1)$ does not appear.
		\item $K_X^2=2$ and $(a,b)=(1,8)$: In this case $K_{X'}^2=9$, and in the list of del Pezzo surfaces of degree $9$ the case $(8,1)$ does not appear.
	\end{enumerate}
\end{proof}

%\noindent\noindent\ref{sss: type II}(b) \emph{Links of type \II over a curve.}
\subsubsection{Links of type \II over a curve}\label{sss:IIovercurve}

The following statement implies that the base-point of a Sarkisov link between conic fibrations must have ``direction transversal to the fibre''.
Given a regular surface $X$ and a closed point $\xx$, we denote by $Y \overset \xx\rightarrow  X$ the blow-up at $\xx$.
\begin{proposition}\label{lem:links type II conic fibrations}
	Let $\eta\colon X\to B$ be Mori fibre space onto a curve $B$ over $T$.
For $b\in B$, let $X_b \coloneqq X \times_B \Spec(\k(b))$ be the fibre above $b$.
	\begin{enumerate}
	    \item\label{links type II conic fibrations:2}
		Let $\xx\in X$ be a closed point such that $\eta(\xx)=b$ and $\k(\xx)=\k(b)$, and $X_{b}$ is regular. 
		Then there is a Sarkisov link $X\rat X'$ of type $\II$ over $B$ whose base-point is $\xx$.
		\item\label{links type II conic fibrations:1}
		Let $X \overset \xx\leftarrow  Y \overset \yy\rightarrow  X'$ be a Sarkisov link of type $\II$ over $B$.
		Let $b$ be the closed point of $B$ such that $\eta(\xx)=b$.
		Then $X_b$ is regular and  $\k(\xx)=\k(b)=\k(\yy)$.
		\item\label{links type II conic fibrations:3} Moreover, in the same notation of \autoref{links type II conic fibrations:1}, $X_b$ and $X'_b$ are geometrically regular over $\k(b)$.
		\item\label{links type II conic fibrations:4} Suppose $T=\Spec \k$. Then $X$ is smooth over $\k$ if and only if $X'$ is smooth over $\k$.
	\end{enumerate}
\end{proposition}

\begin{proof}
	\autoref{links type II conic fibrations:2} and \autoref{links type II conic fibrations:1} are just rephrasing of \autoref{lem:k(x)=k(b)}. Now we prove \autoref{links type II conic fibrations:3}.
	For this, we first remark that $X_b$ is a geometrically reduced conic over $\k(b)$ as $\xx$ is a regular $\k(b)$-rational point by \autoref{links type II conic fibrations:1}. %(indeed, it is obtained as the image of the exceptional divisor of $Y \to X$). 
	As $X_{b}$ is a regular and geometrically reduced conic, it is geometrically regular over $\k(b)$ by \cite[Lemma 2.17]{BT22}, thus concluding. By symmetry, the same holds for $X'_{b}$.
	
	We prove \autoref{links type II conic fibrations:4}.  
	Let $X_{\overline{\k}}$ be the base change to the algebraic closure, which is regular by hypothesis. 
	In particular, $X_{\overline{\k}}$ is normal and this implies that $B_{\overline{\k}}$ is normal, showing that $B$ is a geometrically regular curve over $\k$. 
	As the fibre $X'_{b}$ is geometrically regular over $\k(b)$ by \autoref{links type II conic fibrations:3}, the morphism $\eta'$ is smooth over a neighbourhood of $b$. As $X'$ is smooth over $B$ in a neighbourhood of $b$ and $B$ is geometrically regular, this shows that $X'$ is geometrically regular over $\k$.
	%Let $\xx\in X$ be the base-point of a Sarkisov link of type II and write $b=\eta(\xx)$.  Suppose that $X$ is geometrically regular over $k$ at every point of $X_b$.
	%Let $\pi\colon Y\to X$ be the blow-up of $\xx$ and let $\widetilde{X}_b\subset Y$ be the strict transform of $X_b$.
	%The exceptional divisor $E \simeq \mathbb{P}^1_{\k(\xx)}=\mathbb{P}^1_{\k(b)}$ of $\xx$ intersects $\widetilde{X}_b$ in a $\k(b)$-rational point. %In fact, it is the only non-smooth point of $Y$ above $b$ by \autoref{links type II conic fibrations:1}.
	%The second extremal ray of $\NE(Y/B)$ is generated by $\widetilde{X}_b$.
	%Thus $X'_{\k(b)} \simeq \mathbb{P}^1_{\k(\xx)}$ is geometrically regular over $\k(b)$, concluding.
	%on the fibre above $b$ (which is the image of $E$ by the contraction of $\tilde{X}_b$).
\end{proof}

\begin{remark}
	For links of type II over a curve, base points with non-separable residue fields appear in every characteristic.
\end{remark}

\subsection{Links of type \IV}

	We now discuss links of type \IV.
	Recall from \autoref{lem:linkIV-possible-degree} the presence of a Sarkisov link of \IV implies that $K_X^2\in\{1,2,4,8\}$.
In the case where the Mori conic bundle is smooth, we can be more precise.

\begin{lemma} \label{l-sarkisov-4-class}
	Let $\eta \colon X \to B$ be a Mori conic bundle onto a curve $B$ over a field $\k$.
	Suppose $X$ is geometrically integral and that there is a Sarkisov link of type \IV starting from $X$.
	Then, $\eta$ is smooth if and only if $K_X^2=8$.
	If this holds, then $X_{\bar{\k}}\simeq \p^1_{\bar{\k}}\times\p^1_{\bar{\k}}$.
\end{lemma}

\begin{proof}
	By hypothesis, $X$ is a rank $2$ fibration over a point and it has two Mori conic bundle structures onto curves. In particular, $X$ is a del Pezzo surface of Picard rank $\rho(X)=2$.
	By \cite[Theorem 1.1]{BM23}, we have $h^1(X, \mathcal{O}_X)=0$ and therefore by Kawamata--Viehweg vanishing \cite[Theorem 3.4]{Tan18} we deduce $h^1(B, \mathcal{O}_B)=0$, and thus $B$ is a geometrically integral conic, and thus smooth.

	Suppose $\eta$ is smooth.
	As $X_{\bar{\k}} \to B_{\bar{\k}} \simeq \p^1_{\bar{\k}}$ is smooth, we deduce that $X_{\overline{\k}}$ is regular and $\rho(X_{\bar{\k}})=2$.
	As $X_{\overline{\k}}$ admits another Mori conic bundle structure, we deduce that it is isomorphic to $ \p^1_{\bar{\k}} \times \p^1_{\bar{\k}}$.

	Suppose $K_X^2=8$. Then $X$ is geometrically normal by \autoref{thm: dP-base-change-classification} and thus geometrically canonical by \autoref{rmk:weak-dp}\autoref{prop: geom_normal_implies_canonical-DP}.
	By \autoref{rmk:weak-dp}\autoref{lem: rho_volume_dP} we have $8=K_X^2=10-\rho(X_{\bar{\k}})-\rho(Y/X_{\bar{\k}})$.
	This concludes that $\rho(X_{\bar{\k}})=2$ and $\rho(Y/X_{\overline{\k}})=0$, thus showing $X_{\overline{\k}}$ is a smooth rational surface of degree 8, admitting two Mori conic bundle structures. From this we deduce $X_{\overline{\k}}$ is isomorphic to $\p^1_{\bar{\k}} \times \p^1_{\bar{\k}}$.
\end{proof}

We say a link of type IV between $X \to B_1$ and $X \to B_2$ is \emph{represented by} $\varphi$ if $\varphi \colon X \to X$ is an isomorphism such that  $\varphi^* \colon N^1(X)_\mathbb{Q} \to N^1(X)_{\mathbb{Q}}$ interchanges the two extremal rays of the Mori cone of $X$.

\begin{proposition}\label{prop: links IV}
Let

\begin{center}
\begin{tikzcd}[ampersand replacement=\&,column sep=1.7cm,row sep=0.16cm]
X \ar[rr,"\simeq"] \ar[dd,"fib",swap]  \&\& X \ar[dd,"fib"] \\ \\
B_1 \ar[dr] \& \& B_2 \ar[dl] \\
\& \Spec(\k) \&
\end{tikzcd}
\end{center}

be a link of type \IV, where $X$ is geometrically integral. Then, the following holds:
\begin{enumerate}
	\item $K_X^2=8$: $X \simeq B_1 \times B_2$ for some regular projectives curves of genus 0 and the link is represented by a biregular involution if and only if $B_1 \simeq B_2$.
	\item $K_X^2=4$;
	\item $K_X^2=2$: the link is represented by the Geiser involution. % interchanges the two rulings.
	\item $K_X^2=1$: the link is represented by the Bertini involution. %interchanges the two rulings.
\end{enumerate}
\end{proposition}

\begin{proof}
    Note that, as $\rho(X)>1$, we have that $h^1(X, \mathcal{O}_X)=0$ by \cite[Theorem 1.1]{BM23}. By \autoref{lem:linkIV-possible-degree} we have that $K_X^2\in\{1,2,4,8\}$.
	If $K_X^2=8$, then the Mori conic bundle structures are both smooth by \autoref{l-sarkisov-4-class}.
	Note that the natural product morphism $\pi \colon X \to B_1 \times B_2$ is a finite morphism of degree $1$ as shown in the proof of \autoref{lem:linkIV-possible-degree}, and thus it is an isomorphism since source and target are both normal. This implies that the link is represented by an automorphism if and only if  $B_1 \simeq B_2$.
    
    If $K_X^2=2$ or $1$, we conclude by \autoref{lem: action-Bertini-Geiser}.
	%If $K_X^2=2$ (resp. $1$), then the Geiser (resp. Bertini) involution interchanges the two rulings and denote by $G=\mathbb{Z}/2$ the subgroup of automorphism generated by the involution. \textcolor{red}{By \autoref{lem: pic_iso}, we have $N^1(X)_{\mathbb{Q}}^{G} \simeq N^1(\mathbb{P}^2)_{\mathbb{Q}} \simeq \mathbb{Q}$ (resp.$N^1(X)_{\mathbb{Q}}^{G} \simeq N^1(\mathbb{P}(1,1,2))_{\mathbb{Q}}$). As the numerical class $-K_X$ is preserved by the Geiser (resp. Bertini), we have $N^1(X)_{\mathbb{Q}}^{G}=\mathbb{Q}[-K_X]$.} %\textcolor{red}{FIX FROM HERE} Indeed, in this case we have that the minimal resolution $\pi \colon Y \to X_{\overline{k}}$ is a weak del Pezzo surfaces of degree 2 (resp. 1) and the Geiser (resp. Bertini involution) of $X$ lifts to $Y$. By \cite[8.7.2 and 8.8.2]{Dol12} we see that the only invariant divisor for the involution is $-K_Y$ \textcolor{red}{ and some of the exceptional curves. In particular, on $X_{\overline{\k}}$ the only invariant one is $-K_{X_{\overline{\k}}$. Some extra work on the geometrically nn-normal case, which is only needed for the degree 2, i.e. Geiser}. 	This in particular concludes.
\end{proof}

\begin{remark}
	The hypothesis on geometrical integrality in \autoref{l-sarkisov-4-class} is necessary.
	Let $\k$ be an imperfect field of characteristic $p=2$, and let $C$ be a regular geometrically non-reduced conic.
	Then $X=C \times C$ is a del Pezzo surface of degree $8$ admitting a link of type \IV~but $X_{\bar\k}$ is not reduced.
	Note that such phenomenon appears only if $\k$ has characteristic 2 and the $p$-degree of $\k$ is at least $2$ by \cite[Theorem, page 36]{Sch01}.
	An example over $\mathbb{F}_2(s,t)$ of a regular, geometrically non-reduced conic is given by $\left\{x^2+ty^2+sz^2=0 \right\}$.
\end{remark}

\subsection{Sarkisov links between rational surfaces}
We specialise now to the case of rational surfaces.
We begin by extending a well--known result on the rationality of del Pezzo surfaces and Mori conic bundles \cite[Theorem 2.6]{Isk96} to the case of arbitrary fields. It is a direct consequence of the classification of Sarkisov links obtained above.

\begin{theorem} \label{cor: rat_deg_5}
	Let $\k$ be a field.
	Let $X$ be a del Pezzo surface with $\rho(X)=1$ or a Mori conic bundle $X \to B$.
	If $X$ is rational, then $K_X^2 \geq 5$.
\end{theorem}

\begin{proof}
	Assume that $X$ is rational and $K_X^2\leq 4$.
    Note that, as $X$ is rational, $B \simeq \mathbb{P}^1_{\k}$.
	Note that a link of type \II does not change the self-intersection of the anticanonical divisor (\autoref{lem:links type II conic fibrations} if $B$ is a curve, and \autoref{prop: link II dP} if $B=\Spec(\k)$; for the latter the assumption $K_X^2\leq 4$ is needed).
	The same is true for links of type \IV.
	Moreover, note that the only links of type \I or \III involving a Mori fibre space whose anticanonical divisor has self-intersection $\leq 4$ is between $X_4/\Spec(\k)$ and $X_3/C$ for some curve $C$ by \autoref{prop: link_I}. In particular, there is no Sarkisov link between two surfaces $Y,Z$ with $K_Y^2\leq 4$ and $K_Z^2\geq5$. Hence, $X$ is not rational by the Sarkisov program (\autoref{t-sarkisov-program}).
\end{proof}

We summarise the list of rational Mori fibre spaces over arbitrary fields and the list of possible Sarkisov links between them. Notice that the list coincides with the one given in \cite{Isk96} for perfect fields (see \cite[Figure~2.1]{LS21} for an illustration).
In \autoref{s: general points} we discuss the existence of these links and geometric criteria for a point to be in Sarkisov general position.

\begin{theorem} \label{thm: classification_rational_surfaces}
	Let $\k$ be an arbitrary field. Then the following hold, where $X_d$ denotes a del Pezzo surface of degree $d$.
	\begin{enumerate}
		\item\label{crs:1} Any rational Mori fibre space is isomorphic to one of the following:
		\begin{enumerate}
			\item $\p^2_{\k}$;
			\item a quadric surface $X_8\subset\p^3_{  \k}$ with $\rho(X_8)=1$;
			\item a del Pezzo surface $X_6$ with $\rho(X_6)=1$;
			\item a del Pezzo surface $X_5$ with $\rho(X_5)=1$;
			\item the Hirzebruch surfaces $\mathbb{F}_n\to\p^1_\k$, $n\geq0$;
			\item a Mori conic bundle $X_5\to\p^1_\k$, where $X_5$ is a del Pezzo surface equipped with a birational morphism $X_5\to\p^2_\k$,
			\item a Mori conic bundle $X_6\to\p^1_{\k}$, where $X_6$ is a del Pezzo surface equipped with a birational morphism $X_6\to X_8$;
		\end{enumerate}
	\item\label{crs:2} Any Sarkisov link between rational Mori fibre spaces is one of the following:
		\begin{enumerate}[label=(\alph*)]
		\item links of type \I and \III:
		\begin{enumerate}[label=(\roman*), leftmargin=15mm]
			\item the blow-up $\F_1\to\p^2_{\k}$ of a rational point;
			\item the blow-up $X_5\to\p^2_{\k}$ of a point of degree $4$;
			\item the blow-up $X_6\to X_8$ of a point of degree $2$;
		\end{enumerate}
		\item links of type II:
		\begin{enumerate}[label=(\roman*)]
			\item elementary transformations $\F_n\rat \mathbb{F}_m$;
			\item elementary transformations $X_d\rat X_d'$ of conic fibrations $X_d/\p^1_{\k}$, $X_d'/\p^1_{\k}$, $d\in\{5,6\}$;
			\item $\chi_{\xx,\yy}\colon X_d \rat X_{d'}$, where $d \geq 5$ and $d'$ appear in the list of \autoref{prop: link II dP}.
		\end{enumerate}
		\item links of type $\IV$: exchanging the fibrations on $\F_0=\p^1_{\k}\times\p^1_{\k}$.
		\end{enumerate}
	\end{enumerate}
\end{theorem}

\begin{proof}
\autoref{crs:1}: 
Let $X$ be a projective regular rational surface endowed with a Mori fibre space structure $X\to B$.
%Let $X\to B$ be a Mori fibre space that is rational.
By \autoref{cor: rat_deg_5} it holds that $K_X^2\geq 5$, and so $K_X^2\in\{5,6,8,9\}$ by \autoref{lem:no-minimal-dp7} and \autoref{prop: conic-bundle-vs-dP}.
If $K_X^2=9$, the claim follows from \autoref{lem: SB}.
If $K_X^2\in\{5,6\}$ then either $\rho(X)=1$ and $X$ is a del Pezzo surface, or $X$ admits a Mori conic bundle structure to $\p^1_\k$. In the latter case, $X$ is a del Pezzo surface admitting a birational morphism as described in the statement by \autoref{prop: Mori-conic-bdl-dP}.
If $K_X^2=8$, then either $X$ is a Mori conic bundle structure to $\p^1_\k$, or $\rho(X)=1$.
In the former case, $X=\F_n$ is a Hirzebruch surface because the only possible Mori conic bundle obtained from $\mathbb{P}^2_\k$ of degree 8 is $\mathbb{F}_1$ and a link of type II takes Hirzebruch surfaces to Hirzebruch surfaces.
In the latter case, since $X$ has a rational point (for example by the Lang-Nishimura theorem; see \cite[Theorem 3.6.11]{Poo17}), we have that $X$ is a quadric in $\p^3_\k$ by combining \autoref{prop: divisibility-dp8} with \autoref{lem: dp8-vs-quadric}.

\autoref{crs:2}: The list of links follows from \autoref{crs:1} and \autoref{prop: link_I}, \autoref{prop: link II dP}, \autoref{lem:links type II conic fibrations} and \autoref{prop: links IV}.
\end{proof}

We conclude by specialising \autoref{thm: classification_rational_surfaces} to the case of rational surfaces over separably closed fields. 
In this way, we classify the numerical possibilities of blown-up points with purely inseparable residue fields according to the characteristic of the ground field.

Recall that we write $d_\xx=[k(\xx):\k]$ for a point $\xx$ on a surface over a field $\k$. If $\k$ is separably closed, then $d_\xx=p^e$ for some $e\geq0$. In this case, together with \autoref{cor: moredP-sepclosed}, the following is a direct corollary of \autoref{thm: classification_rational_surfaces}:

\begin{corollary}\label{thm:classification links p>2}
Let $\k$ be a separably closed field of characteristic $p>0$.
Then the following hold, where $X_d$ denotes a del Pezzo surface of degree $d$:
\begin{enumerate}
\item\label{class:Mfs} Any rational Mori fibre space is isomorphic to one of the following:
	\begin{enumerate}
	\item $\p^2_{\k}$;
	\item the Hirzebruch surfaces $\mathbb{F}_n\to\p^1_\k$, $n\geq0$;
	\item $p=2$, a del Pezzo surface $X_8$ with $\rho(X_8)=1$ and $X_8 \subset \mathbb{P}^3_{\k}$ is a quadric such that $(X_8)_{\bar{\k}}$ is geometrically normal with an $A_1$-singularity;
	\item $p=2$, a conic fibration $X_5\to\p^1_\k$ equipped with a birational morphism $X_5\to\p^2_\k$. 
	\item $p=2$, a conic fibration $X_6\to\p^1_{\k}$ equipped with a birational morphism $X_6\to X_8$;
	\item $p=5$, a del Pezzo surface $X_5$ with $\rho(X_5)=1$ and $(X_{5})_{\bar\k}$ is geometrically normal with an $A_4$-singularity.
	\end{enumerate}
In case (d) (resp.\ (e)), if $X_5$ (resp.\ $X_6)$ is geometrically normal, then $(X_5)_{\bar{\k}}$ has an $A_3$-singularity (resp.~ $(X_6)_{\bar{\k}}$ has two $A_1$-singularities).
\item\label{class:links} Any Sarkisov link between rational Mori fibre spaces is one of the following:
	\begin{enumerate}[label=(\alph*)]
			\item\label{class:I+III} links of type \I and \III:
			\begin{enumerate}[label=(\roman*), leftmargin=15mm]
				\item the blow-up $\F_1\to\p^2_{\k}$ in a rational point;
				\item $p=2$, the blow-up $X_5\to\p^2_{\k}$ of a point of degree $4$;
				\item $p=2$, the blow-up $X_6\to X_8$ of a point of degree $2$;
			\end{enumerate}
		\item\label{class:II} links of type II:
		\begin{enumerate}[label=(\roman*)]
			\item elementary transformations $\F_n\rat \mathbb{F}_m$ with a base-point of degree $p^e\geq1$;
			\item $p=2$, elementary transformations $X_d\rat X_d'$ of conic fibrations $X_d/\p^1_{\k}$, $X_d'/\p^1_{\k}$, $d\in\{5,6\}$ with a base-point of degree $p^e\geq1$;
			\item $p=2$, $\chi_{\xx,\yy}\colon \p^2_{\k}\rat X_8$ with $d_\xx=2$, $d_\yy=1$, or its inverse;
			\item $p=2$, $\chi_{\xx,\yy}\colon \p^2_{\k}\rat\p^2_{\k}$ with $d_\xx=d_\yy=8$;
    			\item $p=2$, $\chi_{\xx,\yy}\colon X_8\rat X_8'$ with $d_\xx=d_\yy=4$;
			\item $p=3$, $\chi_{\xx,\yy}\colon \p^2_{\k}\rat \p^2_{\k}$ with
			$d_\xx=d_\yy=3$;
			\item $p=5$, $\chi_{\xx,\yy}\colon\p^2_{\k}\rat X_5$ with $d_\xx=5, d_\yy=1$, or its inverse;
			\item $p=7$, $\chi_{\xx,\yy}\colon \p^2_{\k}\rat\p^2_{\k}$ with $d_\xx=d_\yy=7$.
		\end{enumerate}
	\item links of type $\IV$: exchanging the fibrations on $\F_0=\p^1_{\k}\times\p^1_{\k}$.
	\end{enumerate}
\end{enumerate}
\end{corollary}

\begin{proof}
	This is a specialization of \autoref{thm: classification_rational_surfaces} to the case of separably closed fields.
	The list of geometric realisation of rational del Pezzo surfaces with Picard rank 1 and their geometric singularities follows from \autoref{cor: moredP-sepclosed}.
	We now describe the possible singularities in the case where $X$ is a geometrically normal Mori fibre space over a curve, and $K_X^2$ is either $5$ or $6$.
Since $X$ is a del Pezzo surface by \autoref{prop: conic-bundle-vs-dP}, the base change $X_{\overline{\k}}$ has canonical singularities. Let $Y \to X_{\overline{\k}}$ denote the minimal resolution. By \autoref{rmk:weak-dp}, we have 
$
\rho(Y/X_{\overline{\k}}) = 8 - K_Y^2,
$
which equals $3$ when $K_X^2 = 5$, and $2$ when $K_X^2 = 6$. 
By \cite[Theorem 6.1]{Sch08}, the surface $X_{\overline{\k}}$ has singularities of type either $3A_1$ or one $A_3$ singular points in the degree $5$ case, and $2A_1$ singular points in the degree $6$ case. Finally, by \cite[8.5.1]{Dol12}, the configuration of $3A_1$ singular points does not occur. This completes the classification in these cases.
	
	The list of possible links of type \I and \III follows from the fact that points of degree $2$ and $4$ only exist for $p=2$.
	Similarly, links $\chi_{\xx,\yy}$ of type \II from \autoref{prop: link II dP} can only exist if $d_\xx=p^e$, and $d_\yy=p^{e'}$; these are exactly the ones stated.
The statement follows from \autoref{lem: dP8}, \autoref{thm: classification_rational_surfaces}, \autoref{prop: link_I}, \autoref{prop: link II dP}, \autoref{lem:links type II conic fibrations} and \autoref{prop: links IV} and the fact that $\rho(X_{\bar\k})=\rho(X^{\rm sep})$.
\end{proof}

%==========================================================

 For $p=2$ (resp. $p=5$), Figure~\ref{fig:links_char2} (resp.~ Figure~\ref{fig:links_char5}) illustrates all possible Sarkisov links and Mori fibre spaces for separably closed fields, similar to \cite[Figure~2.1]{LS21}. For $d\in\{5,6,8\}$, $\DD_d$ denotes the set of rational del Pezzo surfaces of degree $d$ of Picard rank $1$, and $\CC_{d}$ denotes the set of rational Mori conic bundles that are del Pezzo surfaces of degree $d$. We do not illustrate the case for $p=3$ and $p=7$ as there is only one link that does not involve a Hirzebruch surface, namely $\chi_{\xx,\yy}\colon\p^2{_\k}\dashrightarrow\p^2_{\k}$ with $d_\xx=d_\yy=p$.

\begin{figure}[ht]
\begin{minipage}[c]{0.48\linewidth}
\[
\resizebox{1.2\linewidth}{!}{
\includegraphics[scale=1]{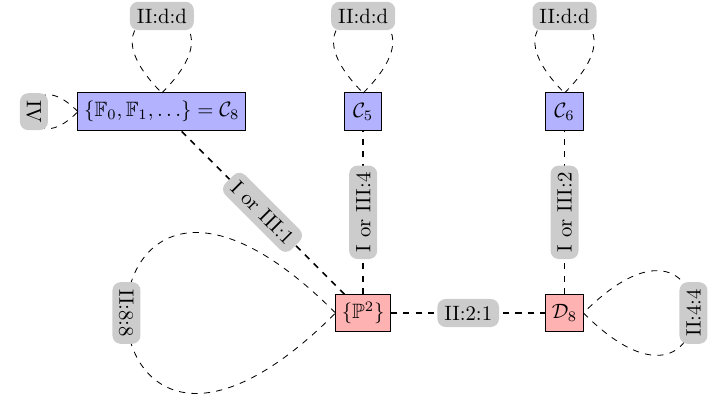}
}
\]
\caption{$p=2$}
\label{fig:links_char2}
\end{minipage}
\begin{minipage}[c]{0.48\linewidth}
\[
\resizebox{1.25\linewidth}{!}{
\includegraphics[scale=.035]{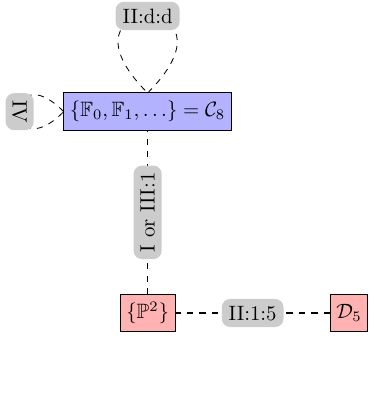}
}
\]
\caption{$p=5$}
\label{fig:links_char5}
\end{minipage}
\end{figure}

\section{Classification of elementary relations}\label{sec:Elementary relations}

In this section, we classify elementary relations appearing in the Sarkisov program.
Equivalently, we classify the Mori cones of rank 3 surface fibrations.

\subsection{How to compute elementary relations}\label{sec:ComputeRelations}
Elementary relations are all obtained by playing the so-called $2$-rays game in a certain way. In the following, we suggest that the reader keeps the example of the blow-up $Z\to \p^2_{\k}$ at two rational points in mind; $Z/\Spec(\k)$ is a rank $3$ fibration inducing a relation between five Sarkisov links, as depicted in Figure~\ref{P2_11}.

From now on, we fix $Z/B$ a rank $3$ fibration, giving rise to an elementary relation
\[
\chi_n\circ\cdots\circ\chi_1=\id,
\]
where $\chi_i\colon X_i\dashrightarrow X_{i+1}$ are Sarkisov links between Mori fibre spaces $X_i/B_i$, and $X_{n+1}=X_1$ (see \autoref{def: elementary relations}). Assume that $\chi_1\circ\chi_n$ as well as $\chi_{i+1}\circ\chi_i$ is not an isomorphism for all $i$. In particular, $n\geq 3$.
For each $i$, consider the Sarkisov diagram of $\chi_i$, with corresponding rank $2$ fibration $Y_i/\widehat B_i$:
\[
	\begin{tikzcd}
		&Y_i\ar[dr,"\beta_i"]\ar[dl,"\alpha_i",swap]& \\
		X_i\ar[d]\ar[rr,"\chi_i",dashed] &&X_{i+1}\ar[d]\\
		B_i\ar[dr]&&B_{i+1}\ar[dl]\\
		&\widehat B_i.&
	\end{tikzcd}\]

Given two adjacent Sarkisov links $\chi_{i-1},\chi_{i}$, we explain how the rank 3 fibration $Z/B$ uniquely determines $\chi_{i+1}$.

As $\chi_{i-1},\chi_i$ are adjacent and dominated by the rank $3$ fibration $Z/B$, there exist $\widehat\beta_{i-1}\colon Z\to Y_i$, $\widehat\alpha_i\colon Z\to Y_{i-1}$ that are isomorphisms exactly if $\beta_{i-1}, \alpha_i$ are, and they are blow-ups of closed points otherwise,
as well as morphisms $\widehat B_{i-1}\to B$ and $\widehat B_i\to B$ of relative Picard rank $\leq1$, such that the following cute diagram commutes:
\[
	\begin{tikzcd}
		&&Z\ar[dl,ultra thick,Green, swap,"\widehat\alpha_i"]\ar[dr,ultra thick,RoyalBlue,"\widehat\beta_{i-1}"]&&\\
		&Y_{i-1}\ar[dl,ultra thick,HotPink,"\alpha_{i-1}",swap]\ar[dr,ultra thick,RoyalBlue,"\beta_{i-1}"]&  &Y_i\ar[dl,ultra thick,Green,"\alpha_i",swap]\ar[dr,ultra thick,FireBrick,"\beta_i"]& \\
		X_{i-1}\ar[d]\ar[rr,"\chi_{i-1}",dashed] &&X_i\ar[d]\ar[rr,"\chi_i",dashed] &&X_{i+1}\ar[d]\\
		B_{i-1}\ar[dr]&&B_i\ar[dl]\ar[dr]&&B_{i+1}\ar[dl]\\
		&\widehat B_{i-1}\ar[dr]&&\widehat B_i\ar[dl]&\\
		&&B.&&
	\end{tikzcd}\]
	Here and in the rest of this section, arrows having the same color mean that either they are all isomorphisms, or that they are divisorial contractions contracting the (strict transform of the) same curve on $Z$. In particular, if $\beta_{i-1}$ (respectively $\alpha_i$) is a divisorial contraction, then $\hat\beta_{i-1}$ (respectively $\hat\alpha_i$) contract the same curve.
%In particular, $\widehat\beta_{i-1}$, $\widehat\alpha_i$ are isomorphisms exactly if $\beta_{i-1},\alpha_i$ are, and they are divisorial contractions otherwise, contracting the same curve respectively.

There are exactly two rank $2$ fibrations dominated by $Z/B$ that dominate $X_{i+1}/B_{i+1}$; one of them is $Y_i/\widehat B_i$ (corresponding to $\chi_i$). We explain now how to find the second one, $Y_{i+1}/\widehat B_{i+1}$, giving $\chi_{i+1}$.

Depending on whether $\beta_{i-1},\beta_i$ are isomorphisms or divisorial contractions, we will determine the rank $2$ fibration $Y_{i+1}/\widehat B_{i+1}$ giving $\chi_{i+1}$, with morphisms $\alpha_{i+1}\colon Y_{i+1}\to X_{i+1}$ and $\widehat\beta_i\colon Z\to Y_{i+1}$ such that the following diagram commutes:
\[\begin{tikzcd}
	&Z\ar[dl,ultra thick,RoyalBlue,swap,"\widehat\beta_{i-1}"]\ar[dr,ultra thick,FireBrick,"\widehat\beta_i"]&\\
	Y_i\ar[dr,ultra thick,FireBrick,"\beta_i"]&& Y_{i+1}\ar[dl,ultra thick,RoyalBlue,"\alpha_{i+1}"]\\
	&X_{i+1}.&
\end{tikzcd}\]
Note that it is not possible that $\beta_{i-1}$ and $\beta_i$ are both isomorphisms (because it would imply that $B_{i+1}/\widehat B_i$ as well as $\widehat B_i/B$ have relative Picard rank $1$, which is not possible since $\dim B_{i+1}\leq 1$). In all cases, \autoref{fig:computing relation} below illustrates the procedure.
\begin{enumerate}
	\item Suppose that $\beta_{i-1},\beta_i$ are divisorial contractions, contracting  curves $E_{i-1},E_i$, respectively.
	Then $\beta_i\circ\widehat\beta_{i-1}\colon Z\to Y_i\to X_{i+1}$ contracts first $E_{i-1}$ and then $E_i$.
Recall that $Z$ is del Pezzo and hence $E_{i-1}, E_i\subset Z$ are disjoint.
	Let $\widehat\beta_i\colon Z\to Y_{i+1}$ be the contraction of $E_i$ and $\alpha_{i+1}\colon Y_{i+1}\to X_{i+1}$ the contraction of $E_{i-1}$.
	Then $Y_{i+1}$ over $\widehat B_{i+1}=B_{i+1}$ is a rank $2$ fibration dominating $X_{i+1}/B_{i+1}$.
	\item Suppose that $\beta_{i-1}$ is the identity, and $\beta_i$ is a divisorial contraction.
	As $Z\simeq Y_i\overset{\beta_i}{\to}X_{i+1}\to B$ with $Z/B$ a rank $3$ fibration, we find that $B=\Spec(k)$ and $X_{i+1}/B$ is a rank $2$ fibration, dominating $X_{i+1}/B_{i+1}$. We take $Y_{i+1}=X_{i+1}$, $\widehat B_{i+1}=B$, $\alpha_{i+1}=\id_{X_{i+1}}$ and $\widehat\beta_i=\beta_i$.
	\item Suppose that $\beta_{i-1}$ is a divisorial contraction and $\beta_i$ is the identity.
	Then $Z/B_{i+1}$ is a rank $2$ fibration dominating $X_{i+1}/B_{i+1}$. We take  $\widehat B_{i+1}=B_{i+1}$, $Y_{i+1}=Z$, $\alpha_{i+1}=\widehat \beta_{i-1}$ and $\widehat\beta_i=\beta_i$.
\end{enumerate}
In particular, $\chi_{i-1}$ and $\chi_i$ determine $\chi_{i+1}$ uniquely as we claimed.
Hence, starting with two links $\chi_1$ and $\chi_2$, one can iteratively find $\chi_3$ up to $\chi_n$:
the relation is complete when $\chi_n\circ\cdots\circ\chi_1=\id$; this happens when $\alpha_{n+1}=\alpha_1$.

\begin{center}
\begin{figure}[ht]
	\begin{tikzpicture}[font=\footnotesize]
	\node[name=s, regular polygon, rotate=180, regular polygon sides=6,inner sep=1.8cm] at (0,0) {};
	\draw[-,ultra thick,RoyalBlue] (s.center) to["$\widehat\beta_{i-1}=\widehat\alpha_{i+1}$",yshift=-2em] (s.side 1);
	\draw[-,ultra thick,FireBrick] (s.center) to["$\widehat\beta_i$"] (s.side 2);
	\draw[-,ultra thick,Green,swap] (s.center) to["$\widehat \alpha_i$"] (s.side 6);
	\draw[-,ultra thick,FireBrick] (s.side 1) to[swap,"$\beta_i$",yshift=-1em] (s.corner 2);
	\draw[-,ultra thick,Green] (s.side 1) to["$\alpha_i$",yshift=-1em] (s.corner 1);
	\draw[-,ultra thick,Orange] (s.side 2) to[swap,"$\beta_{i+1}$"] (s.corner 3);
	\draw[-,ultra thick,RoyalBlue] (s.side 2) to["$\alpha_{i+1}$"] (s.corner 2);
	\draw[dotted] (s.side 3) to[swap] (s.corner 3);
	\draw[dotted] (s.side 5) to[swap] (s.corner 6);
	\draw[-,ultra thick,RoyalBlue] (s.side 6) to[swap,"$\beta_{i-1}$"] (s.corner 1);
	\draw[-,ultra thick,HotPink] (s.side 6) to["$\alpha_{i-1}$"] (s.corner 6);
	%%%
	\draw[dashed, bend right=99] (s.corner 6) to["$\chi_{i-1}$",swap] (s.corner 1);
	\draw[dashed, bend right=99] (s.corner 1) to["$\chi_{i}$",swap] (s.corner 2);
	\draw[dashed, bend right=99] (s.corner 2) to["$\chi_{i+1}$",swap] (s.corner 3);
	%%%
	\node[fill=white,circle,inner sep=2] at (s.center) {$Z_{/B}$};
	\node[fill=white,inner sep=2] at (s.corner 1) {${X_i}_{\tiny/B_i}$};
	\node[fill=white,inner sep=2] at (s.side 1) {${Y_i}_{\tiny/\widehat B_i}$};
	\node[fill=white,inner sep=2] at (s.corner 2) {${X_{i+1}}_{\tiny/B_{i+1}}$};
	\node[fill=white,inner sep=2] at (s.side 2) {${Y_{i+1}}_{/\widehat B_{i+1}}$};
	\node[fill=white,inner sep=2] at (s.corner 3) {${X_{i+2}}_{/B_{i+2}}$};
	\node[fill=white,inner sep=2] at (s.corner 6) {${X_{i-1}}_{/B_{i-1}}$};
	\node[fill=white,inner sep=2] at (s.side 6) {${Y_{i-1}}_{/\widehat B_{i-1}}$};
	\end{tikzpicture}
\caption{Computing an elementary relation}\label{fig:computing relation}
\end{figure}
\end{center}

For each $i$, write $a_i$ (respectively $b_i$) for the degree of the base point of $\chi_i$ (respectively $\chi_i^{-1}$).
Combining \autoref{prop: link II dP} and \autoref{prop: link_I}, we see that $K_{X_i}^2$ and $a_i$ determine $b_i$ uniquely (and hence also $K_{X_{i+1}}^2$).
In particular, the information $b_{i-1}$ of $\chi_{i-1}$, together with the information $K_{X_{i+1}}^2$ from $\chi_i$, determines the relevant information
%$a_{i+1}=b_{i-1}$ and $K_{X_{i+1}}^2$ 
of $\chi_{i+1}$.

\medskip

One can observe that the type of $\chi_{i-1}$ gives restrictions for the type of $\chi_i$ 
(these stem from the fact that the dimension of $B_i$ have to match, and that if $\beta_{i-1}$ is an isomorphism, then $\alpha_i$ cannot be an isomorphism, for otherwise the relation is redundant):
\begin{enumerate}
	\item If $\chi_{i-1}$ is of type \I, then $\chi_i$ is of type \II over a curve.
	\item If $\chi_{i-1}$ is of type \II over $\Spec(k)$, then $\chi_i$ is either of type \I, or of type \II over $\Spec(k)$.
	\item If $\chi_{i-1}$ is of type \II over a curve, then $\chi_i$ is either of type \II over the curve, or of type \III, or of type \IV.
	\item If $\chi_{i-1}$ is of type \III, then $\chi_i$ is of type \I, or of type \II over $\Spec(k)$.
	\item If $\chi_{i-1}$ is of type \IV, then $\chi_i$ is of type \II over a curve.
\end{enumerate}

%====================================
\subsection{The list of elementary relations}

Recall from \autoref{def: elementary relations} that an elementary relation corresponds to a rank $3$ fibration $Z\to B$. 
%In this section, we do not care about the existence of these relations, as they depend on the involved surfaces and points.
In this section, we list all elementary relations among Sarkisov links that correspond to a rank $3$ fibration $Z\to B$. They can all be obtained by following the $2$-rays game explained above. Computations for rational surfaces may be found in \cite[Appendix]{ZimHDR} and \cite[Appendix G]{lamybook} over a perfect field, and are here generalised for any field. 
Moreover, here we also list elementary relations for non-rational surfaces.

Figure~\ref{fig:XB_ab} represents elementary relations over a curve. Colored edges stand for divisorial contractions and they are labeled with the (strict transform on $Z$ of the) contracted curve.
Figures~\ref{P2_11}--\ref{X3_11} represent elementary relations over a point, denoted by $\piece{X}\xx\yy$,
whose edges are color coded according to the degree of the point whose blow-up is represented by the edge
\[
\textcolor{RoyalBlue}{\deg=1},\quad
\textcolor{FireBrick}{\deg=2}, \quad
\textcolor{Green}{\deg=3},\quad
\textcolor{Orange}{\deg=4},\quad
\textcolor{HotPink}{\deg=5},\quad
\textcolor{Turquoise}{\deg=6},\quad
\textcolor{Salmon}{\deg=7},
\]
and where $H$ denotes the pullback to $Z$ of the ample generator of $\Pic(X)$ and $E,F$ are exceptional divisors of the point $\xx,\yy\in X$. In the cases of diagrams $\piece{\F_0}0d$, $H_1,H_2$ denote the pull-backs to $Z$ of the two generators of $\F_0$.

\subsubsection{Elementary relations over a curve}

We list the elementary relations between Sarkisov links over curves.

\begin{remark}\label{rmk:strict transform fibre}
	Let $\k$ be a field and $\pi\colon X\to B$ a Mori fibre space above a curve $B$.
	Let $b\in B$ be such that $\textup{N}^1(B)=\Z[b]$.
	Let $\xx\in X$ be a point in Sarkisov general position over $B$. Consider the blow-up $\rho\colon Y\to X$ at $\xx$ with exceptional divisor $E_\xx$ and we write $f:=\rho^*\pi^*(b)$.
	The strict transform $\tilde X_{\pi(\xx)}\subset Y$ of the fibre over $\pi(\xx)$ is linearly equivalent to $\frac{d_\xx}{d_b} f-E_\xx$.
	In particular, if $B=\p^1_{\k}$ then $b$ is a $\k$-rational point and $\tilde X_{\pi(\xx)}\equiv d_\xx f - E_\xx$. 
\end{remark}

\begin{proposition}
Let $\k$ be a field.
Let $Z/B$ be a rank $3$ fibration where $B$ is a curve. Then the corresponding elementary relation is of the form $\piece{X/B}\xx\yy$ as shown in Figure~\ref{fig:XB_ab} with $X=X_1$, where $X\to B$ is a Mori fibre space, $E,F$ are the exceptional divisors of the points $\xx,\yy\in X$, $f$ denotes the pull-back to $Z$ of a closed point $b \in B$ generating $N^1(B)$, and $\delta_\xx=d_\xx/d_b$.

\begin{figure}[ht]
\[
\resizebox{.3\linewidth}{!}{
\begin{tikzpicture}
  \node[name=s, regular polygon, rotate=180, regular polygon sides=4,inner sep=1.242cm] at (0,0) {};
  \draw[-,ultra thick,Purple] (s.center) to["E"] (s.side 1);
  \draw[-,ultra thick,Green] (s.center) to["$\delta_{\yy}$f-F"] (s.side 2);
  \draw[-,ultra thick,FireBrick] (s.center) to["$\delta_{\xx}$f-E"] (s.side 3);
  \draw[-,ultra thick,RoyalBlue] (s.center) to["F"] (s.side 4);
%  \draw[-,ultra thick,RoyalBlue] (s.center) to["1"] (s.side 5);
  \draw[-,ultra thick,Green] (s.side 1) to[swap,"$\delta_{\yy}$f-F"] (s.corner 2);
  \draw[-,ultra thick,RoyalBlue] (s.side 1) to["F"] (s.corner 1);
  \draw[-,ultra thick,FireBrick] (s.side 2) to[swap,"$\delta_{\xx}$f-E"] (s.corner 3);
  \draw[-,ultra thick,Purple] (s.side 2) to["E"] (s.corner 2);
  \draw[-,ultra thick,RoyalBlue] (s.side 3) to[swap,"F"] (s.corner 4);
  \draw[-,ultra thick,Green] (s.side 3) to["$\delta_{\yy}$f-F"] (s.corner 3);
  \draw[-,ultra thick,FireBrick] (s.side 4) to["$\delta_{\xx}$f-E"] (s.corner 4);
    \draw[-,ultra thick,Purple] (s.side 4) to[swap,"E"] (s.corner 1);
    \node[fill=white,circle,inner sep=2] at (s.center) {${Z/B}$};
     \node[fill=white,inner sep=2] at (s.corner 1) {${X_1/B}$};
     \node[fill=white,inner sep=2] at (s.corner 2) {${X_2/B}$};
     \node[fill=white,inner sep=2] at (s.corner 3) {${X_3/B}$};
     \node[fill=white,inner sep=2] at (s.corner 4) {${X_4/B}$};
\end{tikzpicture}
}
\]
\caption{$\piece{X/\text{curve}}\xx\yy$}
\label{fig:XB_ab}
\end{figure}
\end{proposition}

\begin{proof}
Since $B$ is a curve and $\rho(Z/B)=3$, any birational morphism $Z\to X=X_1$ over $B$ is the composition of blow-ups of two points $\xx,\yy$ of degree $\deg(\xx)=a$ and $\deg(\yy)=b$. This is the lower left quadrant of the diagram.
Since $X_1\to B$ is a conic fibration, the Sarkisov links induced by blowing up $\xx$ and $\yy$ are links of type \II, see \autoref{lem:links type II conic fibrations} and \autoref{rmk:strict transform fibre}. This induces the adjacent quadrants of the diagram. It then closes canonically with the fourth quadrant.
\end{proof}

\subsubsection{Elementary relations over a point}

We now list all elementary relations over $\Spec(\k)$.

\begin{table}[t]
\begin{minipage}[t]{0.32\linewidth}
\strut\vspace*{-\baselineskip}\newline
\begin{tabular}{Cl}
\toprule
\text{Piece} & Figure \\
\midrule
\piece{\p^2}11 & Fig. \ref{P2_11}  \\
\piece{\p^2}12 & Fig. \ref{P2_12}  \\
\piece{\p^2}13 & Fig. \ref{P2_13}  \\
\piece{\p^2}14 & Fig. \ref{P2_14}  \\
\piece{\p^2}15 & Fig. \ref{P2_15}  \\
\piece{\p^2}16 & Fig. \ref{P2_16}  \\
\piece{\p^2}17 & Fig. \ref{P2_17}  \\
\piece{\p^2}22 & Fig. \ref{P2_22}  \\
\piece{\p^2}23 & Fig. \ref{P2_23}  \\
\piece{\p^2}24 & Fig. \ref{P2_24}  \\
\piece{\p^2}25 & Fig. \ref{P2_25}  \\
\piece{\p^2}26 & Fig. \ref{P2_26}  \\
\piece{\p^2}33 & Fig. \ref{P2_33}  \\
\piece{\p^2}34 & Fig. \ref{P2_34}  \\
\piece{\p^2}35 & Fig. \ref{P2_35}  \\
\piece{\p^2}44 & Fig. \ref{P2_44}  \\
\bottomrule \\
\end{tabular}
\end{minipage}
\begin{minipage}[t]{0.32\linewidth}
\strut\vspace*{-\baselineskip}\newline
\begin{tabular}{Cl}
\toprule
\text{Piece} & Figure \\
\midrule
\piece{X_8}11 & Fig. \ref{P2_12} \\
\piece{X_8}12 & Fig. \ref{P2_22} \\
\piece{X_8}13 & Fig. \ref{P2_23}  \\
\piece{X_8}14 & Fig. \ref{P2_24}  \\
\piece{X_8}15 & Fig. \ref{P2_25}  \\
\piece{X_8}16 & Fig. \ref{P2_26}  \\
\piece{X_8}22 & Fig. \ref{X8_22}  \\
\piece{X_8}23 & Fig. \ref{X8_23}  \\
\piece{X_8}24 & Fig. \ref{X8_24}  \\
\piece{X_8}25 & Fig. \ref{X8_25}  \\
\piece{X_8}33 & Fig. \ref{X8_33}  \\
\piece{X_8}34 & Fig. \ref{X8_34}  \\
\bottomrule \\
\end{tabular}
\end{minipage}
\begin{minipage}[t]{0.32\linewidth}
\strut\vspace*{-\baselineskip}\newline
\begin{tabular}{Cl}
\toprule
\text{Piece} & Figure \\
\midrule
\piece{X_6}11 & Fig. \ref{P2_23} \\
\piece{X_6}12 & Fig. \ref{X8_23}  \\
\piece{X_6}13 & Fig. \ref{X8_33}  \\
\piece{X_6}14 & Fig. \ref{X8_34} \\
\piece{X_6}22 & Fig. \ref{X6_22}  \\
\piece{X_6}23 & Fig. \ref{X6_23} \\
\midrule
\piece{X_5}11 & Fig. \ref{P2_15}  \\
\piece{X_5}12 & Fig. \ref{P2_25}  \\
\piece{X_5}13 & Fig. \ref{P2_35}  \\
\piece{X_5}22 & Fig. \ref{X8_25}  \\
\midrule
\piece{\F_0}0{1} & Fig. \ref{P2_12} \\
\piece{\F_0}0{2} & Fig. \ref{F0_02} \\
\piece{\F_0}0{3} & Fig. \ref{P2_13} \\
\piece{\F_0}0{4} & Fig. \ref{F0_04} \\
\piece{\F_0}0{5} & Fig. \ref{P2_15} \\
\piece{\F_0}0{6} & Fig. \ref{F0_06} \\
\piece{\F_0}0{7} & Fig. \ref{P2_17} \\
\bottomrule \\
\end{tabular}
\end{minipage}
\caption{The $27$ elementary relations over a point for which the surfaces appearing can be rational; from Fig. \ref{P2_11} to Fig. \ref{F0_06}}
\label{tab:pieces}
\end{table}

\begin{table}[t]
	\begin{minipage}[t]{0.5\linewidth}
	\strut\vspace*{-\baselineskip}\newline
	\begin{tabular}{Cl}
	\toprule
	\text{Piece} & Figure \\
	\midrule
	\piece{X_9}33 & compare with Fig. \ref{P2_33} \\
	\midrule
	\piece{X_8}22 & compare with Fig. \ref{X8_22} \\
	\piece{X_8}24 & compare with Fig. \ref{X8_24} \\
	\bottomrule \\
	\end{tabular}
	\end{minipage}
\begin{minipage}[t]{0.32\linewidth}
\strut\vspace*{-\baselineskip}\newline
\begin{tabular}{Cl}
\toprule
\text{Piece} & Figure \\
\midrule
\piece{X_3}11 & Fig. \ref{X3_11} \\
\midrule
\piece{X_4}11 & Fig. \ref{X4_11} \\
\piece{X_4}12 & Fig. \ref{X4_12} \\
\midrule
\piece{X_4}01 & Fig. \ref{X4_01} \\
\piece{X_4}02 & Fig. \ref{X4_02} \\
\piece{X_4}03 & Fig. \ref{X4_03} \\
\bottomrule \\
\end{tabular}
\end{minipage}
\caption{The $9$ elementary relations over a point for which the surfaces appearing are never rational (where $X_9$ denotes a non-trivial Severi-Brauer surface and $X_8$ a not rational del Pezzo surface of degree 8)}
\label{tab:pieces_nonrat}
\end{table}

We classify the possible rank 3 surface fibrations and all the possible contractions to rank 2 and rank 1 fibrations.

\begin{theorem}\label{thm:elementary relations}
Let $\k$ be a field. Let $Z/B$ be a rank $3$ fibration.
\begin{enumerate}
	\item If $B$ is a curve, then the corresponding elementary relation is as in Figure~\ref{fig:XB_ab}.
	\item If $B=\Spec(\k)$ and $Z$ is rational, then the corresponding elementary relation is one of the figures listed in \autoref{tab:pieces}.
	\item If $B=\Spec(\k)$ and $Z$ is not rational, then the corresponding elementary relation is one of the figures listed in \autoref{tab:pieces} or in \autoref{tab:pieces_nonrat}.
\end{enumerate}
\end{theorem}
\begin{proof}
If $\k$ is an arbitrary field, we obtain the elementary relations purely computationally when playing the $2$-rays game as described in \autoref{sec:Elementary relations}.

The distinction between rational and non-rational surfaces comes from the fact that rational Mori fibre spaces satisfy $K_X^2\geq5$ by \autoref{cor:  rat_deg_5}.

Moreover, only the following elementary relations over $\Spec(\k)$ involve a del Pezzo surface $X$ with $K_X^2\in\{8,9\}$ and $\Pic(X)=\Z[-K_X]$:
\begin{enumerate}
	\item If $X=X_9$ is a non-trivial Severi--Brauer surface, then $\ind(X_9)=3$ and so only $\piece{X_9}33$ exists, and it is analogous to \autoref{P2_33}, replacing $\p^2_{\k}$ by $X_9$ and its opposite Severi--Brauer surface $X_9^{\rm op}$, and replacing $3H$ by $H$ (see \cite[Figure~3]{BSY}).
	\item If $X=X_8$ is not a quadric in $\p^3_\k$, then $2\mid\ind(X_8)$ and so only $\piece{X_8}22$ and $\piece{X_8}24$ can exist. They are analogous to \autoref{X8_22} respectively \autoref{X8_24}, replacing $2H$ by $H$.
\end{enumerate}
\end{proof}

We specialize above theorem for separably closed fields.

\begin{corollary}\label{cor:elementary relations sep closed}
	Let $\k$ be a separably closed field, and let $Z/B$ be a rank $3$ fibration.
	\begin{enumerate}
		\item If $B$ is a curve then the corresponding elementary relation is $\piece{X}ab$ (Fig.~\ref{fig:XB_ab}).
		\item If $B=\Spec(\k)$ and $Z$ is rational, then the corresponding elementary relation is $\piece{\p^2}11$ or one of the following:
		\begin{enumerate}
		\item in $\mathrm{char}(\k)=2$, \\
		\begin{tabular}{l l l}
		$\piece{\P^2}12$ (Fig.~\ref{P2_12}) & $\piece{\P^2}14$ (Fig.~\ref{P2_14}) & $\piece{\P^2}22$ (Fig.~\ref{P2_22}) \\
		  $\piece{\P^2}24$ (Fig.~\ref{P2_24}) & $\piece{\P^2}44$ (Fig.~\ref{P2_44}) & $\piece{X_8}22$ (Fig.~\ref{X8_22})\\  $\piece{X_8}24$ (Fig.~\ref{X8_24})&
		$\piece{\F_0}02$ (Fig.~\ref{F0_02}) & $\piece{\F_0}04$ (Fig.~\ref{F0_04})\\
		\end{tabular}
		\item in $\mathrm{char}(\k)=3$, $\piece{\P^2}13$ (Fig.~\ref{P2_13}), $\piece{\P^2}33$ (Fig.~\ref{P2_33});
		\item in $\mathrm{char}(\k)=5$, $\piece{\P^2}15$ (Fig.~\ref{P2_15});
		\item in $\mathrm{char}(\k)=7$, $\piece{\P^2}17$ (Fig.~\ref{P2_17});
		\end{enumerate}
		\item If $B=\Spec(\k)$ and $Z$ is not rational, then the corresponding elementary relation is among the ones listed in \autoref{tab:pieces_nonrat} (depending on $\charact{\k}$), $\piece{X_9}33$ excluded.
	\end{enumerate}
\end{corollary}

\begin{proof}
	We check which of the relations of \autoref{thm:elementary relations} exist over separably closed fields.
	Recall that in \autoref{cor: moredP-sepclosed} we have seen that over a separably closed field the only del Pezzo surfaces $X$ with $\rho(X)=1$ and $K_X^2\geq5$ are $\p^2_\k$, or $p=2$ and $K_X^2=8$, or $p=5$ and $K_X^2=5$.
	Combining this with the fact that the degree of any point is a power of $p$, we obtain the list for rational surfaces using \autoref{cor: rat_deg_5}.

	If $Z$ is not rational, then either $K_X^2\leq 4$ and the possible pieces are listed in \autoref{tab:pieces_nonrat}, or $K_X^2\geq 5$ and as above either $p=2$ and $K_X^2=8$, or $p=5$ and $K_X^2=5$.
	Note that $\piece{X_8}11=\piece{\p^2}12$ and $\piece{X_5}11=\piece{\p^2}15$, so these appear only for rational surfaces. For $X_5$, the latter is the only relation involving points of degree a power of $p=5$; for $X_8$, there are also the relations $\piece{X_8}22$ and $\piece{X_8}24$, which are also listed in \autoref{tab:pieces_nonrat}.
\end{proof}

%%% P2, 11, 12, 13 %%%
\begin{figure}[ht]
\begin{minipage}[c]{0.32\linewidth}
\[
\resizebox{1.05\linewidth}{!}{
\begin{tikzpicture}[font=\footnotesize]
\node[name=s, regular polygon, rotate=180, regular polygon sides=5,inner sep=1.242cm] at (0,0) {};
\draw[-,ultra thick,RoyalBlue] (s.center) to["F"] (s.side 1);
\draw[-,ultra thick,black] (s.center) to[""] (s.side 2);
\draw[-,ultra thick,RoyalBlue] (s.center) to["H-E-F"] (s.side 3);
\draw[-,ultra thick,black] (s.center) to[""] (s.side 4);
\draw[-,ultra thick,RoyalBlue] (s.center) to["E",swap] (s.side 5);
\draw[-,ultra thick,black] (s.side 1) to[swap,""] (s.corner 2);
\draw[-,ultra thick,RoyalBlue] (s.side 1) to["E"] (s.corner 1);
\draw[-,ultra thick,RoyalBlue] (s.side 2) to[swap,"H-E-F"{yshift=-3pt}] (s.corner 3);
\draw[-,ultra thick,RoyalBlue] (s.side 2) to["F"{yshift=-3pt}] (s.corner 2);
\draw[-,ultra thick,black] (s.side 3) to[swap,""] (s.corner 4);
\draw[-,ultra thick,black] (s.side 3) to[""] (s.corner 3);
\draw[-,ultra thick,RoyalBlue] (s.side 4) to[swap,"E"{yshift=-3pt}] (s.corner 5);
\draw[-,ultra thick,RoyalBlue] (s.side 4) to["H-E-F"{yshift=-3pt}] (s.corner 4);
\draw[-,ultra thick,RoyalBlue] (s.side 5) to[swap,"F"] (s.corner 1);
\draw[-,ultra thick,black] (s.side 5) to[""] (s.corner 5);
\node[fill=white,circle,inner sep=2] at (s.center) {${X_7}$};
\node[fill=white,inner sep=2] at (s.corner 1) {${\p^2}$};
\node[fill=white,inner sep=2] at (s.side 1) {${\F_1}$};
\node[fill=white,inner sep=2] at (s.corner 2) {${\F_1}_{\!/\p^1}$};
\node[fill=white,inner sep=2] at (s.side 2) {${X_7}_{/\p^1}$};
\node[fill=white,inner sep=2] at (s.corner 3) {${\F_0}_{\!/\p^1}$};
\node[fill=white,inner sep=2] at (s.side 3) {${\F_0}$};
\node[fill=white,inner sep=2] at (s.corner 4) {${\F_0}_{\!/\p^1}$};
\node[fill=white,inner sep=2] at (s.side 4) {${X_7}_{/\p^1}$};
\node[fill=white,inner sep=2] at (s.corner 5) {${\F_1}_{\!/\p^1}$};
\node[fill=white,inner sep=2] at (s.side 5) {${\F_1}$};
\end{tikzpicture}
}
\]
\caption{$\piece{\P^2}11$}
\label{P2_11}
\end{minipage}
\begin{minipage}[c]{0.32\linewidth}
\[
\resizebox{1.15\linewidth}{!}{
\begin{tikzpicture}[font=\footnotesize]
\node[name=s, regular polygon, rotate=180, regular polygon sides=5,inner sep=1.242cm] at (0,0) {};
\draw[-,ultra thick,FireBrick] (s.center) to["F"] (s.side 1);
\draw[-,ultra thick,black] (s.center) to[""] (s.side 2);
\draw[-,ultra thick,FireBrick] (s.center) to["2(H-E)-F"{xshift=2pt}] (s.side 3);
\draw[-,ultra thick,RoyalBlue] (s.center) to["H-F"{xshift=4pt}] (s.side 4);
\draw[-,ultra thick,RoyalBlue] (s.center) to["E"] (s.side 5);
\draw[-,ultra thick,black] (s.side 1) to[swap,""] (s.corner 2);
\draw[-,ultra thick,RoyalBlue] (s.side 1) to["E"] (s.corner 1);
\draw[-,ultra thick,FireBrick] (s.side 2) to[swap,"2(H-E)-F"{yshift=-4pt}] (s.corner 3);
\draw[-,ultra thick,FireBrick] (s.side 2) to["F"{yshift=-3pt}] (s.corner 2);
\draw[-,ultra thick,RoyalBlue] (s.side 3) to[swap,"H-F"] (s.corner 4);
\draw[-,ultra thick,black] (s.side 3) to[""] (s.corner 3);
\draw[-,ultra thick,RoyalBlue] (s.side 4) to[swap,"E"] (s.corner 5);
\draw[-,ultra thick,FireBrick] (s.side 4) to["2(H-E)-F"{yshift=-4pt}] (s.corner 4);
\draw[-,ultra thick,FireBrick] (s.side 5) to[swap,"F"{yshift=-3pt}] (s.corner 1);
\draw[-,ultra thick,RoyalBlue] (s.side 5) to["H-F"] (s.corner 5);
\node[fill=white,circle,inner sep=2] at (s.center) {${X_6}$};
\node[fill=white,inner sep=2] at (s.corner 1) {${\p^2}$};
\node[fill=white,inner sep=2] at (s.side 1) {${\F_1}$};
\node[fill=white,inner sep=2] at (s.corner 2) {${\F_1}_{\!/\p^1}$};
\node[fill=white,inner sep=2] at (s.side 2) {${X_6}_{/\p^1}$};
\node[fill=white,inner sep=2] at (s.corner 3) {${\F_1}_{\!/\p^1}$};
\node[fill=white,inner sep=2] at (s.side 3) {${\F_1}$};
\node[fill=white,inner sep=2] at (s.corner 4) {${\p^2}$};
\node[fill=white,inner sep=2] at (s.side 4) {${X_7}$};
\node[fill=white,inner sep=2] at (s.corner 5) {${X_8}$};
\node[fill=white,inner sep=2] at (s.side 5) {${X_7'}$};
\end{tikzpicture}
}
\]
\caption{$\piece{\P^2}12$}
\label{P2_12}
\end{minipage}
\begin{minipage}[c]{0.32\linewidth}
\[
\resizebox{1.15\linewidth}{!}{
\begin{tikzpicture}[font=\footnotesize]
\node[name=s, regular polygon, rotate=180, regular polygon sides=6,inner sep=1.557cm] at (0,0) {};
\draw[-,ultra thick,ForestGreen] (s.center) to["F"] (s.side 1);
\draw[-,ultra thick,black] (s.center) to[""] (s.side 2);
\draw[-,ultra thick,ForestGreen] (s.center) to[swap,"3(H-E)-F"{xshift=-13pt, yshift=-3pt}] (s.side 3);
\draw[-,ultra thick,black] (s.center) to["2H-E-F"] (s.side 4);
\draw[-,ultra thick,ForestGreen] (s.center) to["3H-2F"{xshift=3pt}] (s.side 5);
\draw[-,ultra thick,RoyalBlue] (s.center) to["E"] (s.side 6);
\draw[-,ultra thick,black] (s.side 1) to[swap,""] (s.corner 2);
\draw[-,ultra thick,RoyalBlue] (s.side 1) to["E"] (s.corner 1);
\draw[-,ultra thick,ForestGreen] (s.side 2) to[swap,"3(H-E)-F"] (s.corner 3);
\draw[-,ultra thick,ForestGreen] (s.side 2) to["F"] (s.corner 2);
\draw[-,ultra thick,black] (s.side 3) to[swap,"2H-E-F"] (s.corner 4);
\draw[-,ultra thick,black] (s.side 3) to[""] (s.corner 3);
\draw[-,ultra thick,ForestGreen] (s.side 4) to[swap,"3H-2F"{yshift=4pt}] (s.corner 5);
\draw[-,ultra thick,ForestGreen] (s.side 4) to["3(H-E)-F"{yshift=2pt}] (s.corner 4);
\draw[-,ultra thick,RoyalBlue] (s.side 5) to[swap,"E"] (s.corner 6);
\draw[-,ultra thick,black] (s.side 5) to["2H-E-F"] (s.corner 5);
\draw[-,ultra thick,ForestGreen] (s.side 6) to[swap,"F"] (s.corner 1);
\draw[-,ultra thick,ForestGreen] (s.side 6) to["3H-2F"] (s.corner 6);
\node[fill=white,circle,inner sep=2] at (s.center) {${X_5}$};
\node[fill=white,inner sep=2] at (s.corner 1) {${\p^2}$};
\node[fill=white,inner sep=2] at (s.side 1) {${\F_1}$};
\node[fill=white,inner sep=2] at (s.corner 2) {${\F_1}_{\!/\p^1}$};
\node[fill=white,inner sep=2] at (s.side 2) {${X_5}_{/\p^1}$};
\node[fill=white,inner sep=2] at (s.corner 3) {${\F_0}_{\!/\p^1}$};
\node[fill=white,inner sep=2] at (s.side 3) {${\F_0}$};
\node[fill=white,inner sep=2] at (s.corner 4) {${\F_0}_{\!/\p^1}$};
\node[fill=white,inner sep=2] at (s.side 4) {${X_5}_{/\p^1}$};
\node[fill=white,inner sep=2] at (s.corner 5) {${\F_1}_{\!/\p^1}$};
\node[fill=white,inner sep=2] at (s.side 5) {${\F_1}$};
\node[fill=white,inner sep=2] at (s.corner 6) {${\p^2}$};
\node[fill=white,inner sep=2] at (s.side 6) {${X_6}$};
\end{tikzpicture}
}
\]
\caption{$\piece{\P^2}13$}
\label{P2_13}
\end{minipage}
\end{figure}

%%% P2, 14, 15, 16 %%%
\begin{figure}[ht]
\begin{minipage}[c]{0.32\linewidth}
\[
\resizebox{1.2\linewidth}{!}{
\begin{tikzpicture}[font=\footnotesize]
\node[name=s, regular polygon, rotate=180, regular polygon sides=6,inner sep=1.557cm] at (0,0) {};
\draw[-,ultra thick,Orange] (s.center) to["F"] (s.side 1);
\draw[-,ultra thick,black] (s.center) to[""] (s.side 2);
\draw[-,ultra thick,Orange] (s.center) to["4(H-E)-F"{yshift=-3pt,xshift=-13pt},swap] (s.side 3);
\draw[-,ultra thick,RoyalBlue] (s.center) to["2H-F-E"] (s.side 4);
\draw[-,ultra thick,black] (s.center) to[""] (s.side 5);
\draw[-,ultra thick,RoyalBlue] (s.center) to["E"] (s.side 6);
\draw[-,ultra thick,black] (s.side 1) to[swap,""] (s.corner 2);
\draw[-,ultra thick,RoyalBlue] (s.side 1) to["E"] (s.corner 1);
\draw[-,ultra thick,Orange] (s.side 2) to[swap,"4(H-E)-F"{yshift=3pt}] (s.corner 3);
\draw[-,ultra thick,Orange] (s.side 2) to["F"] (s.corner 2);
\draw[-,ultra thick,RoyalBlue] (s.side 3) to[swap,"2H-F-E"] (s.corner 4);
\draw[-,ultra thick,black] (s.side 3) to[""] (s.corner 3);
\draw[-,ultra thick,black] (s.side 4) to[swap,""] (s.corner 5);
\draw[-,ultra thick,Orange] (s.side 4) to["4(H-E)-F"] (s.corner 4);
\draw[-,ultra thick,RoyalBlue] (s.side 5) to[swap,"E"] (s.corner 6);
\draw[-,ultra thick,RoyalBlue] (s.side 5) to["2H-F-E"] (s.corner 5);
\draw[-,ultra thick,Orange] (s.side 6) to[swap,"F"] (s.corner 1);
\draw[-,ultra thick,black] (s.side 6) to[""] (s.corner 6);
\node[fill=white,circle,inner sep=2] at (s.center) {${X_4}$};
\node[fill=white,inner sep=2] at (s.corner 1) {${\p^2}$};
\node[fill=white,inner sep=2] at (s.side 1) {${\F_1}$};
\node[fill=white,inner sep=2] at (s.corner 2) {${\F_1}_{\!/\p^1}$};
\node[fill=white,inner sep=2] at (s.side 2) {${X_4}_{/\p^1}$};
\node[fill=white,inner sep=2] at (s.corner 3) {${\F_1}_{\!/\p^1}$};
\node[fill=white,inner sep=2] at (s.side 3) {${\F_1}$};
\node[fill=white,inner sep=2] at (s.corner 4) {${\p^2}$};
\node[fill=white,inner sep=2] at (s.side 4) {${X_5}$};
\node[fill=white,inner sep=2] at (s.corner 5) {${X_5}_{\!/\p^1}$};
\node[fill=white,inner sep=2] at (s.side 5) {${X_4}_{/\p^1}$};
\node[fill=white,inner sep=2] at (s.corner 6) {${X_5}_{\!/\p^1}$};
\node[fill=white,inner sep=2] at (s.side 6) {${X_5}$};
\end{tikzpicture}
}
\]
\caption{$\piece{\P^2}14$}
\label{P2_14}
\end{minipage}
\begin{minipage}[c]{0.31\linewidth}
\[
\resizebox{1.2\linewidth}{!}{
\begin{tikzpicture}[font=\footnotesize]
\node[name=s, regular polygon, rotate=180, regular polygon sides=7,inner sep=1.872cm] at (0,0) {};
\draw[-,ultra thick,HotPink] (s.center) to["F"] (s.side 1);
\draw[-,ultra thick,black] (s.center) to[""] (s.side 2);
\draw[-,ultra thick,HotPink] (s.center) to["\!\!\!5(H-E)-F",swap] (s.side 3);
\draw[-,ultra thick,black] (s.center) to["3H-2E-F"{yshift=5pt}] (s.side 4);
\draw[-,ultra thick,HotPink] (s.center) to["10H-5E-4F"{xshift=5pt}] (s.side 5);
\draw[-,ultra thick,RoyalBlue] (s.center) to["2H-F"{xshift=-10pt,yshift=-2pt}] (s.side 6);
\draw[-,ultra thick,RoyalBlue] (s.center) to["E"] (s.side 7);
\draw[-,ultra thick,black] (s.side 1) to[swap,""] (s.corner 2);
\draw[-,ultra thick,RoyalBlue] (s.side 1) to["E"] (s.corner 1);
\draw[-,ultra thick,HotPink] (s.side 2) to[swap,"5(H-E)-F"{yshift=3pt}] (s.corner 3);
\draw[-,ultra thick,HotPink] (s.side 2) to["F"] (s.corner 2);
\draw[-,ultra thick,black] (s.side 3) to[swap,"3H-2E-F"] (s.corner 4);
\draw[-,ultra thick,black] (s.side 3) to[""] (s.corner 3);
\draw[-,ultra thick,HotPink] (s.side 4) to[swap,"10H-5E-4F"{yshift=5pt,xshift=-10pt}] (s.corner 5);
\draw[-,ultra thick,HotPink] (s.side 4) to["5(H-E)-F"{yshift=3pt, xshift=2pt}] (s.corner 4);
\draw[-,ultra thick,RoyalBlue] (s.side 5) to[swap,"2H-F"] (s.corner 6);
\draw[-,ultra thick,black] (s.side 5) to["3H-2E-F"] (s.corner 5);
\draw[-,ultra thick,RoyalBlue] (s.side 6) to[swap,"E"] (s.corner 7);
\draw[-,ultra thick,HotPink] (s.side 6) to["10H-5E-4F"{yshift=-5pt},swap] (s.corner 6);
\draw[-,ultra thick,HotPink] (s.side 7) to[swap,"F"] (s.corner 1);
\draw[-,ultra thick,RoyalBlue] (s.side 7) to["2H-F"] (s.corner 7);
\node[fill=white,circle,inner sep=2] at (s.center) {${X_3}$};
\node[fill=white,inner sep=2] at (s.corner 1) {${\p^2}$};
\node[fill=white,inner sep=2] at (s.side 1) {${\F_1}$};
\node[fill=white,inner sep=2] at (s.corner 2) {${\F_1}_{\!/\p^1}$};
\node[fill=white,inner sep=2] at (s.side 2) {${X_3}_{/\p^1}$};
\node[fill=white,inner sep=2] at (s.corner 3) {${\F_0}_{\!/\p^1}$};
\node[fill=white,inner sep=2] at (s.side 3) {${\F_0}$};
\node[fill=white,inner sep=2] at (s.corner 4) {${\F_0}_{\!/\p^1}$};
\node[fill=white,inner sep=2] at (s.side 4) {${X_3}_{/\p^1}$};
\node[fill=white,inner sep=2] at (s.corner 5) {${\F_1}_{\!/\p^1}$};
\node[fill=white,inner sep=2] at (s.side 5) {${\F_1}$};
\node[fill=white,inner sep=2] at (s.corner 6) {${\p^2}$};
\node[fill=white,inner sep=2] at (s.side 6) {${X_4}$};
\node[fill=white,inner sep=2] at (s.corner 7) {${X_5}$};
\node[fill=white,inner sep=2] at (s.side 7) {${X_4'}$};
\end{tikzpicture}
}
\]
\caption{$\piece{\P^2}15$}
\label{P2_15}
\end{minipage}
\begin{minipage}[c]{0.32\linewidth}
\[
\resizebox{1.25\linewidth}{!}{
\begin{tikzpicture}[font=\footnotesize]
\node[name=s, regular polygon, rotate=180, regular polygon sides=8,inner sep=2.169cm] at (0,0) {};
\draw[-,ultra thick,Turquoise] (s.center) to["F"] (s.side 1);
\draw[-,ultra thick,black] (s.center) to[""] (s.side 2);
\draw[-,ultra thick,Turquoise] (s.center) to["6(H-E)-F"] (s.side 3);
\draw[-,ultra thick,RoyalBlue] (s.center) to["\!\!\!3H-2E-F",swap] (s.side 4);
\draw[-,ultra thick,Turquoise] (s.center) to[swap,"18H-7F-6E"{yshift=11pt,xshift=-3pt}] (s.side 5);
\draw[-,ultra thick,black] (s.center) to["5H-2F-E"{yshift=7pt}] (s.side 6);
\draw[-,ultra thick,Turquoise] (s.center) to["12H-5F"] (s.side 7);
\draw[-,ultra thick,RoyalBlue] (s.center) to["E"] (s.side 8);
\draw[-,ultra thick,black] (s.side 1) to[swap,""] (s.corner 2);
\draw[-,ultra thick,RoyalBlue] (s.side 1) to["E"] (s.corner 1);
\draw[-,ultra thick,Turquoise] (s.side 2) to[swap,"6(H-E)-F"] (s.corner 3);
\draw[-,ultra thick,Turquoise] (s.side 2) to["F"] (s.corner 2);
\draw[-,ultra thick,RoyalBlue] (s.side 3) to[swap,"3H-2E-F"] (s.corner 4);
\draw[-,ultra thick,black] (s.side 3) to[""] (s.corner 3);
\draw[-,ultra thick,Turquoise] (s.side 4) to[swap,"18H-7F-6E"] (s.corner 5);
\draw[-,ultra thick,Turquoise] (s.side 4) to["6(H-E)-F"] (s.corner 4);
\draw[-,ultra thick,black] (s.side 5) to[swap,"5H-2F-E"{yshift=4pt,xshift=-2pt}] (s.corner 6);
\draw[-,ultra thick,RoyalBlue] (s.side 5) to["3H-2E-F"{yshift=4pt,xshift=2pt}] (s.corner 5);
\draw[-,ultra thick,Turquoise] (s.side 6) to[swap,"12H-5F"] (s.corner 7);
\draw[-,ultra thick,Turquoise] (s.side 6) to["18H-7F-6E"] (s.corner 6);
\draw[-,ultra thick,RoyalBlue] (s.side 7) to[swap,"E"] (s.corner 8);
\draw[-,ultra thick,black] (s.side 7) to["5H-2F-E"] (s.corner 7);
\draw[-,ultra thick,Turquoise] (s.side 8) to[swap,"F"] (s.corner 1);
\draw[-,ultra thick,Turquoise] (s.side 8) to["12H-5F"] (s.corner 8);
\node[fill=white,circle,inner sep=2] at (s.center) {${X_2}$};
\node[fill=white,inner sep=2] at (s.corner 1) {${\p^2}$};
\node[fill=white,inner sep=2] at (s.side 1) {${\F_1}$};
\node[fill=white,inner sep=2] at (s.corner 2) {${\F_1}_{\!/\p^1}$};
\node[fill=white,inner sep=2] at (s.side 2) {${X_2}_{/\p^1}$};
\node[fill=white,inner sep=2] at (s.corner 3) {${\F_1}_{\!/\p^1}$};
\node[fill=white,inner sep=2] at (s.side 3) {${\F_1}$};
\node[fill=white,inner sep=2] at (s.corner 4) {${\p^2}$};
\node[fill=white,inner sep=2] at (s.side 4) {${X_3}$};
\node[fill=white,inner sep=2] at (s.corner 5) {${\p^2}$};
\node[fill=white,inner sep=2] at (s.side 5) {${\F_1}$};
\node[fill=white,inner sep=2] at (s.corner 6) {${\F_1}_{\!/\p^1}$};
\node[fill=white,inner sep=2] at (s.side 6) {${X_2}_{/\p^1}$};
\node[fill=white,inner sep=2] at (s.corner 7) {${\F_1}_{\!/\p^1}$};
\node[fill=white,inner sep=2] at (s.side 7) {${\F_1}$};
\node[fill=white,inner sep=2] at (s.corner 8) {${\p^2}$};
\node[fill=white,inner sep=2] at (s.side 8) {${X_3}$};
\end{tikzpicture}
}
\]
\caption{$\piece{\p^2}16$}
\label{P2_16}
\end{minipage}
\end{figure}

%%% P2, 17, 44 %%%
\begin{figure}[ht]
\begin{minipage}[c]{0.49\linewidth}
\[
\resizebox{1.15\linewidth}{!}{
\begin{tikzpicture}[font=\footnotesize]
\node[name=s, regular polygon, rotate=180, regular polygon sides=12,inner sep=3.357cm] at (0,0) {};
\draw[-,ultra thick,Salmon] (s.center) to["F"] (s.side 1);
\draw[-,ultra thick,black] (s.center) to[""] (s.side 2);
\draw[-,ultra thick,Salmon] (s.center) to["7(H-E)-F"{yshift=-7pt}] (s.side 3);
\draw[-,ultra thick,black] (s.center) to[""] (s.side 4);
\draw[-,ultra thick,Salmon] (s.center) to["21H-14E-6F"{yshift=4pt},swap] (s.side 5);
\draw[-,ultra thick,RoyalBlue] (s.center) to["\!\!\!6H-3E-2F",swap] (s.side 6);
\draw[-,ultra thick,Salmon] (s.center) to["42H-14E-15F"{yshift=15pt}] (s.side 7);
\draw[-,ultra thick,black] (s.center) to[""] (s.side 8);
\draw[-,ultra thick,Salmon] (s.center) to["35H-7E-13F"] (s.side 9);
\draw[-,ultra thick,black] (s.center) to[""] (s.side 10);
\draw[-,ultra thick,Salmon] (s.center) to["21H-8F"] (s.side 11);
\draw[-,ultra thick,RoyalBlue] (s.center) to["E"] (s.side 12);
\draw[-,ultra thick,black] (s.side 1) to[swap,""] (s.corner 2);
\draw[-,ultra thick,RoyalBlue] (s.side 1) to["E"] (s.corner 1);
\draw[-,ultra thick,Salmon] (s.side 2) to[swap,"7(H-E)-F"] (s.corner 3);
\draw[-,ultra thick,Salmon] (s.side 2) to["F"] (s.corner 2);
\draw[-,ultra thick,black] (s.side 3) to[swap,""] (s.corner 4);
\draw[-,ultra thick,black] (s.side 3) to[""] (s.corner 3);
\draw[-,ultra thick,Salmon] (s.side 4) to[swap,"21H-14E-6F",swap] (s.corner 5);
\draw[-,ultra thick,Salmon] (s.side 4) to["7(H-E)-F",swap] (s.corner 4);
\draw[-,ultra thick,RoyalBlue] (s.side 5) to[swap,"6H-3E-2F"] (s.corner 6);
\draw[-,ultra thick,black] (s.side 5) to[""] (s.corner 5);
\draw[-,ultra thick,Salmon] (s.side 6) to[swap,"42H-14E-15F"] (s.corner 7);
\draw[-,ultra thick,Salmon] (s.side 6) to["21H-14E-6F"] (s.corner 6);
\draw[-,ultra thick,black] (s.side 7) to[swap,""] (s.corner 8);
\draw[-,ultra thick,RoyalBlue] (s.side 7) to["6H-3E-2F"{yshift=4pt}] (s.corner 7);
\draw[-,ultra thick,Salmon] (s.side 8) to[swap,"35H-7E-13F"] (s.corner 9);
\draw[-,ultra thick,Salmon] (s.side 8) to["42H-14E-15F"] (s.corner 8);
\draw[-,ultra thick,black] (s.side 9) to[swap,""] (s.corner 10);
\draw[-,ultra thick,black] (s.side 9) to[""] (s.corner 9);
\draw[-,ultra thick,Salmon] (s.side 10) to[swap,"21H-8F"] (s.corner 11);
\draw[-,ultra thick,Salmon] (s.side 10) to["35H-7E-13F"] (s.corner 10);
\draw[-,ultra thick,RoyalBlue] (s.side 11) to[swap,"E"] (s.corner 12);
\draw[-,ultra thick,black] (s.side 11) to[""] (s.corner 11);
\draw[-,ultra thick,Salmon] (s.side 12) to[swap,"F"] (s.corner 1);
\draw[-,ultra thick,Salmon] (s.side 12) to["21H-8F"] (s.corner 12);
\node[fill=white,circle,inner sep=2] at (s.center) {${X_1}$};
\node[fill=white,inner sep=2] at (s.corner 1) {${\p^2}$};
\node[fill=white,inner sep=2] at (s.side 1) {${\F_1}$};
\node[fill=white,inner sep=2] at (s.corner 2) {${\F_1}_{\!/\p^1}$};
\node[fill=white,inner sep=2] at (s.side 2) {${X_1}_{/\p^1}$};
\node[fill=white,inner sep=2] at (s.corner 3) {${\F_0}_{\!/\p^1}$};
\node[fill=white,inner sep=2] at (s.side 3) {${\F_0}$};
\node[fill=white,inner sep=2] at (s.corner 4) {${\F_0}_{\!/\p^1}$};
\node[fill=white,inner sep=2] at (s.side 4) {${X_1}_{/\p^1}$};
\node[fill=white,inner sep=2] at (s.corner 5) {${\F_1}_{\!/\p^1}$};
\node[fill=white,inner sep=2] at (s.side 5) {${\F_1}$};
\node[fill=white,inner sep=2] at (s.corner 6) {${\p^2}$};
\node[fill=white,inner sep=2] at (s.side 6) {${X_2}$};
\node[fill=white,inner sep=2] at (s.corner 7) {${\p^2}$};
\node[fill=white,inner sep=2] at (s.side 7) {${\F_1}$};
\node[fill=white,inner sep=2] at (s.corner 8) {${\F_1}_{\!/\p^1}$};
\node[fill=white,inner sep=2] at (s.side 8) {${X_1}_{/\p^1}$};
\node[fill=white,inner sep=2] at (s.corner 9) {${\F_0}_{\!/\p^1}$};
\node[fill=white,inner sep=2] at (s.side 9) {${\F_0}$};
\node[fill=white,inner sep=2] at (s.corner 10) {${\F_0}_{\!/\p^1}$};
\node[fill=white,inner sep=2] at (s.side 10) {${X_1}_{/\p^1}$};
\node[fill=white,inner sep=2] at (s.corner 11) {${\F_1}_{\!/\p^1}$};
\node[fill=white,inner sep=2] at (s.side 11) {${\F_1}$};
\node[fill=white,inner sep=2] at (s.corner 12) {${\p^2}$};
\node[fill=white,inner sep=2] at (s.side 12) {${X_2}$};
\end{tikzpicture}
}
\]
\caption{$\piece{\P^2}17$}
\label{P2_17}
\end{minipage}\ \
\begin{minipage}[c]{0.49\linewidth}
\[
\resizebox{1.2\linewidth}{!}{
\begin{tikzpicture}[font=\footnotesize]
\node[name=s, regular polygon, rotate=180, regular polygon sides=12,inner sep=3.357cm] at (0,0) {};
\draw[-,ultra thick,Orange] (s.center) to["F"] (s.side 1);
\draw[-,ultra thick,black] (s.center) to[""] (s.side 2);
\draw[-,ultra thick,Orange] (s.center) to["4(2H-F)-E"] (s.side 3);
\draw[-,ultra thick,Orange] (s.center) to["16H-4E-7F"] (s.side 4);
\draw[-,ultra thick,black] (s.center) to[""] (s.side 5);
\draw[-,ultra thick,Orange] (s.center) to["24H-8E-9F"{yshift=5pt},swap] (s.side 6);
\draw[-,ultra thick,Orange] (s.center) to["24H-9E-8F"{yshift=29pt}] (s.side 7);
\draw[-,ultra thick,black] (s.center) to[""] (s.side 8);
\draw[-,ultra thick,Orange] (s.center) to["16H-7E-4F"{yshift=5pt,xshift=-3pt}] (s.side 9);
\draw[-,ultra thick,Orange] (s.center) to["4(2H-E)-F"] (s.side 10);
\draw[-,ultra thick,black] (s.center) to[""] (s.side 11);
\draw[-,ultra thick,Orange] (s.center) to["E"] (s.side 12);
\draw[-,ultra thick,black] (s.side 1) to[swap,""] (s.corner 2);
\draw[-,ultra thick,Orange] (s.side 1) to["E"] (s.corner 1);
\draw[-,ultra thick,Orange] (s.side 2) to[swap,"4(2H-F)-E"] (s.corner 3);
\draw[-,ultra thick,Orange] (s.side 2) to["F"] (s.corner 2);
\draw[-,ultra thick,Orange] (s.side 3) to[swap,"16H-4E-7F"] (s.corner 4);
\draw[-,ultra thick,black] (s.side 3) to[""] (s.corner 3);
\draw[-,ultra thick,black] (s.side 4) to[swap,""] (s.corner 5);
\draw[-,ultra thick,Orange] (s.side 4) to["4(2H-F)-E"] (s.corner 4);
\draw[-,ultra thick,Orange] (s.side 5) to[swap,"24H-8E-9F"] (s.corner 6);
\draw[-,ultra thick,Orange] (s.side 5) to["16H-4E-7F"] (s.corner 5);
\draw[-,ultra thick,Orange] (s.side 6) to[swap,"24H-9E-8F"] (s.corner 7);
\draw[-,ultra thick,black] (s.side 6) to[""] (s.corner 6);
\draw[-,ultra thick,black] (s.side 7) to[swap,""] (s.corner 8);
\draw[-,ultra thick,Orange] (s.side 7) to["24H-8E-9F"{yshift=4pt}] (s.corner 7);
\draw[-,ultra thick,Orange] (s.side 8) to[swap,"16H-7E-4F"] (s.corner 9);
\draw[-,ultra thick,Orange] (s.side 8) to["24H-9E-8F"] (s.corner 8);
\draw[-,ultra thick,Orange] (s.side 9) to[swap,"4(2H-E)-F"] (s.corner 10);
\draw[-,ultra thick,black] (s.side 9) to[""] (s.corner 9);
\draw[-,ultra thick,black] (s.side 10) to[swap,""] (s.corner 11);
\draw[-,ultra thick,Orange] (s.side 10) to["16H-7E-4F",swap] (s.corner 10);
\draw[-,ultra thick,Orange] (s.side 11) to[swap,"E"] (s.corner 12);
\draw[-,ultra thick,Orange] (s.side 11) to["4(2H-E)-F"] (s.corner 11);
\draw[-,ultra thick,Orange] (s.side 12) to[swap,"F"] (s.corner 1);
\draw[-,ultra thick,black] (s.side 12) to[""] (s.corner 12);
\node[fill=white,circle,inner sep=2] at (s.center) {${X_1}$};
\node[fill=white,inner sep=2] at (s.corner 1) {${\p^2}$};
\node[fill=white,inner sep=2] at (s.side 1) {${X_5}$};
\node[fill=white,inner sep=2] at (s.corner 2) {${X_5}_{\!/\p^1}$};
\node[fill=white,inner sep=2] at (s.side 2) {${X_1}_{/\p^1}$};
\node[fill=white,inner sep=2] at (s.corner 3) {${X_5}_{\!/\p^1}$};
\node[fill=white,inner sep=2] at (s.side 3) {${X_5}$};
\node[fill=white,inner sep=2] at (s.corner 4) {${\p^2}$};
\node[fill=white,inner sep=2] at (s.side 4) {${X_5}$};
\node[fill=white,inner sep=2] at (s.corner 5) {${X_5}_{\!/\p^1}$};
\node[fill=white,inner sep=2] at (s.side 5) {${X_1}_{/\p^1}$};
\node[fill=white,inner sep=2] at (s.corner 6) {${X_5}_{\!/\p^1}$};
\node[fill=white,inner sep=2] at (s.side 6) {${X_5}$};
\node[fill=white,inner sep=2] at (s.corner 7) {${\p^2}$};
\node[fill=white,inner sep=2] at (s.side 7) {${X_5}$};
\node[fill=white,inner sep=2] at (s.corner 8) {${X_5}_{\!/\p^1}$};
\node[fill=white,inner sep=2] at (s.side 8) {${X_1}_{/\p^1}$};
\node[fill=white,inner sep=2] at (s.corner 9) {${X_5}_{\!/\p^1}$};
\node[fill=white,inner sep=2] at (s.side 9) {${X_5}$};
\node[fill=white,inner sep=2] at (s.corner 10) {${\p^2}$};
\node[fill=white,inner sep=2] at (s.side 10) {${X_5}$};
\node[fill=white,inner sep=2] at (s.corner 11) {${X_5}_{\!/\p^1}$};
\node[fill=white,inner sep=2] at (s.side 11) {${X_1}_{/\p^1}$};
\node[fill=white,inner sep=2] at (s.corner 12) {${X_5}_{\!/\p^1}$};
\node[fill=white,inner sep=2] at (s.side 12) {${X_5}$};
\end{tikzpicture}
}
\]
\caption{$\piece{\P^2}44$}
\label{P2_44}
\end{minipage}
\end{figure}

%%% P2, 22, 23, 24 %%%
\begin{figure}[ht]
\begin{minipage}[c]{0.32\linewidth}
\[
\resizebox{.9\linewidth}{!}{
\begin{tikzpicture}[font=\footnotesize]
\node[name=s, regular polygon, rotate=180, regular polygon sides=5,inner sep=1.242cm] at (0,0) {};
\draw[-,ultra thick,FireBrick] (s.center) to["F"] (s.side 1);
\draw[-,ultra thick,RoyalBlue] (s.center) to["H-E",swap] (s.side 2);
\draw[-,ultra thick,black] (s.center) to[""] (s.side 3);
\draw[-,ultra thick,RoyalBlue] (s.center) to["H-F"] (s.side 4);
\draw[-,ultra thick,FireBrick] (s.center) to["E"] (s.side 5);
\draw[-,ultra thick,RoyalBlue] (s.side 1) to[swap,"H-E"] (s.corner 2);
\draw[-,ultra thick,FireBrick] (s.side 1) to["E"] (s.corner 1);
\draw[-,ultra thick,black] (s.side 2) to[swap,""] (s.corner 3);
\draw[-,ultra thick,FireBrick] (s.side 2) to["F"] (s.corner 2);
\draw[-,ultra thick,RoyalBlue] (s.side 3) to[swap,"H-F"{yshift=4pt}] (s.corner 4);
\draw[-,ultra thick,RoyalBlue] (s.side 3) to["H-E"{yshift=4pt}] (s.corner 3);
\draw[-,ultra thick,FireBrick] (s.side 4) to[swap,"E"] (s.corner 5);
\draw[-,ultra thick,black] (s.side 4) to[""] (s.corner 4);
\draw[-,ultra thick,FireBrick] (s.side 5) to[swap,"F"] (s.corner 1);
\draw[-,ultra thick,RoyalBlue] (s.side 5) to["H-F"] (s.corner 5);
\node[fill=white,circle,inner sep=2] at (s.center) {${X_5}$};
\node[fill=white,inner sep=2] at (s.corner 1) {${\p^2}$};
\node[fill=white,inner sep=2] at (s.side 1) {${X_7}$};
\node[fill=white,inner sep=2] at (s.corner 2) {${X_8}$};
\node[fill=white,inner sep=2] at (s.side 2) {${X_6}$};
\node[fill=white,inner sep=2] at (s.corner 3) {${X_6}_{\!/\p^1}$};
\node[fill=white,inner sep=2] at (s.side 3) {${X_5}_{/\p^1}$};
\node[fill=white,inner sep=2] at (s.corner 4) {${X_6'}_{\!/\p^1}$};
\node[fill=white,inner sep=2] at (s.side 4) {${X_6'}$};
\node[fill=white,inner sep=2] at (s.corner 5) {${X_8'}$};
\node[fill=white,inner sep=2] at (s.side 5) {${X_7'}$};
\end{tikzpicture}
}
\]
\caption{$\piece{\P^2}22$}
\label{P2_22}
\end{minipage}
\begin{minipage}[c]{0.32\linewidth}
\[
\resizebox{1.05\linewidth}{!}{
\begin{tikzpicture}[font=\footnotesize]
\node[name=s, regular polygon, rotate=180, regular polygon sides=5,inner sep=1.242cm] at (0,0) {};
\draw[-,ultra thick,ForestGreen] (s.center) to["F"] (s.side 1);
\draw[-,ultra thick,RoyalBlue] (s.center) to["H-E",swap] (s.side 2);
\draw[-,ultra thick,RoyalBlue] (s.center) to["2H-E-F"] (s.side 3);
\draw[-,ultra thick,ForestGreen] (s.center) to["3H-2F"{xshift=6pt}] (s.side 4);
\draw[-,ultra thick,FireBrick] (s.center) to["E"] (s.side 5);
\draw[-,ultra thick,RoyalBlue] (s.side 1) to[swap,"H-E"] (s.corner 2);
\draw[-,ultra thick,FireBrick] (s.side 1) to["E"] (s.corner 1);
\draw[-,ultra thick,RoyalBlue] (s.side 2) to[swap,"2H-E-F"] (s.corner 3);
\draw[-,ultra thick,ForestGreen] (s.side 2) to["F"] (s.corner 2);
\draw[-,ultra thick,ForestGreen] (s.side 3) to[swap,"3H-2F"{yshift=3pt}] (s.corner 4);
\draw[-,ultra thick,RoyalBlue] (s.side 3) to["H-E"] (s.corner 3);
\draw[-,ultra thick,FireBrick] (s.side 4) to[swap,"E"] (s.corner 5);
\draw[-,ultra thick,RoyalBlue] (s.side 4) to["2H-E-F"] (s.corner 4);
\draw[-,ultra thick,ForestGreen] (s.side 5) to[swap,"F"] (s.corner 1);
\draw[-,ultra thick,ForestGreen] (s.side 5) to["3H-2F"] (s.corner 5);
\node[fill=white,circle,inner sep=2] at (s.center) {${X_4}$};
\node[fill=white,inner sep=2] at (s.corner 1) {${\p^2}$};
\node[fill=white,inner sep=2] at (s.side 1) {${X_7}$};
\node[fill=white,inner sep=2] at (s.corner 2) {${X_8}$};
\node[fill=white,inner sep=2] at (s.side 2) {${X_5}$};
\node[fill=white,inner sep=2] at (s.corner 3) {${X_6}$};
\node[fill=white,inner sep=2] at (s.side 3) {${X_5}$};
\node[fill=white,inner sep=2] at (s.corner 4) {${X_8}$};
\node[fill=white,inner sep=2] at (s.side 4) {${X_7}$};
\node[fill=white,inner sep=2] at (s.corner 5) {${\p^2}$};
\node[fill=white,inner sep=2] at (s.side 5) {${X_6}$};
\end{tikzpicture}
}
\]
\caption{$\piece{\P^2}23$}
\label{P2_23}
\end{minipage}
\begin{minipage}[c]{0.32\linewidth}
\[
\resizebox{1.05\linewidth}{!}{
\begin{tikzpicture}[font=\footnotesize]
\node[name=s, regular polygon, rotate=180, regular polygon sides=6,inner sep=1.557cm] at (0,0) {};
\draw[-,ultra thick,Orange] (s.center) to["F"] (s.side 1);
\draw[-,ultra thick,RoyalBlue] (s.center) to["H-E"] (s.side 2);
\draw[-,ultra thick,Orange] (s.center) to["8H-4E-3F"{xshift=-10pt},swap] (s.side 3);
\draw[-,ultra thick,FireBrick] (s.center) to["4H-E-2F",swap] (s.side 4);
\draw[-,ultra thick,black] (s.center) to[""] (s.side 5);
\draw[-,ultra thick,FireBrick] (s.center) to["E"] (s.side 6);
\draw[-,ultra thick,RoyalBlue] (s.side 1) to[swap,"H-E"] (s.corner 2);
\draw[-,ultra thick,FireBrick] (s.side 1) to["E"] (s.corner 1);
\draw[-,ultra thick,Orange] (s.side 2) to[swap,"8H-4E-3F"] (s.corner 3);
\draw[-,ultra thick,Orange] (s.side 2) to["F"] (s.corner 2);
\draw[-,ultra thick,FireBrick] (s.side 3) to[swap,"4H-E-2F"] (s.corner 4);
\draw[-,ultra thick,RoyalBlue] (s.side 3) to["H-E"] (s.corner 3);
\draw[-,ultra thick,black] (s.side 4) to[swap,""] (s.corner 5);
\draw[-,ultra thick,Orange] (s.side 4) to["8H-4E-3F"{yshift=4pt}] (s.corner 4);
\draw[-,ultra thick,FireBrick] (s.side 5) to[swap,"E"] (s.corner 6);
\draw[-,ultra thick,FireBrick] (s.side 5) to["4H-E-2F"] (s.corner 5);
\draw[-,ultra thick,Orange] (s.side 6) to[swap,"F"] (s.corner 1);
\draw[-,ultra thick,black] (s.side 6) to[""] (s.corner 6);
\node[fill=white,circle,inner sep=2] at (s.center) {${X_3}$};
\node[fill=white,inner sep=2] at (s.corner 1) {${\p^2}$};
\node[fill=white,inner sep=2] at (s.side 1) {${X_7}$};
\node[fill=white,inner sep=2] at (s.corner 2) {${X_8}$};
\node[fill=white,inner sep=2] at (s.side 2) {${X_4}$};
\node[fill=white,inner sep=2] at (s.corner 3) {${X_8'}$};
\node[fill=white,inner sep=2] at (s.side 3) {${X_7'}$};
\node[fill=white,inner sep=2] at (s.corner 4) {${\p^2}$};
\node[fill=white,inner sep=2] at (s.side 4) {${X_5}$};
\node[fill=white,inner sep=2] at (s.corner 5) {${X_5}_{\!/\p^1}$};
\node[fill=white,inner sep=2] at (s.side 5) {${X_3}_{/\p^1}$};
\node[fill=white,inner sep=2] at (s.corner 6) {${X_5}_{\!/\p^1}$};
\node[fill=white,inner sep=2] at (s.side 6) {${X_5}$};
\end{tikzpicture}
}
\]
\caption{$\piece{\P^2}24$}
\label{P2_24}
\end{minipage}
\end{figure}

%%% P2 25, 26 33 %%%
\begin{figure}[ht]
\begin{minipage}[c]{0.32\linewidth}
\[
\resizebox{1.15\linewidth}{!}{
\begin{tikzpicture}[font=\footnotesize]
\node[name=s, regular polygon, rotate=180, regular polygon sides=6,inner sep=1.557cm] at (0,0) {};
\draw[-,ultra thick,HotPink] (s.center) to["F"{yshift=-3pt}] (s.side 1);
\draw[-,ultra thick,RoyalBlue] (s.center) to["H-E",swap] (s.side 2);
\draw[-,ultra thick,FireBrick] (s.center) to[sloped,"6H-3E-2F",swap] (s.side 3);
\draw[-,ultra thick,HotPink] (s.center) to[sloped,"15H-5E-6F",swap] (s.side 4);
\draw[-,ultra thick,RoyalBlue] (s.center) to["2H-F"] (s.side 5);
\draw[-,ultra thick,FireBrick] (s.center) to["E"] (s.side 6);
\draw[-,ultra thick,RoyalBlue] (s.side 1) to[swap,"H-E"] (s.corner 2);
\draw[-,ultra thick,FireBrick] (s.side 1) to["E"] (s.corner 1);
\draw[-,ultra thick,FireBrick] (s.side 2) to[swap,"6H-3E-2F"] (s.corner 3);
\draw[-,ultra thick,HotPink] (s.side 2) to["F"] (s.corner 2);
\draw[-,ultra thick,HotPink] (s.side 3) to[swap,"15H-5E-6F"] (s.corner 4);
\draw[-,ultra thick,RoyalBlue] (s.side 3) to["H-E"] (s.corner 3);
\draw[-,ultra thick,RoyalBlue] (s.side 4) to[swap,"2H-F"] (s.corner 5);
\draw[-,ultra thick,FireBrick] (s.side 4) to["6H-3E-2F"{yshift=4pt}] (s.corner 4);
\draw[-,ultra thick,FireBrick] (s.side 5) to[swap,"E"] (s.corner 6);
\draw[-,ultra thick,HotPink] (s.side 5) to["15H-5E-6F"] (s.corner 5);
\draw[-,ultra thick,HotPink] (s.side 6) to[swap,"F"] (s.corner 1);
\draw[-,ultra thick,RoyalBlue] (s.side 6) to["2H-F"] (s.corner 6);
\node[fill=white,circle,inner sep=2] at (s.center) {${X_2}$};
\node[fill=white,inner sep=2] at (s.corner 1) {${\p^2}$};
\node[fill=white,inner sep=2] at (s.side 1) {${X_7}$};
\node[fill=white,inner sep=2] at (s.corner 2) {${X_8}$};
\node[fill=white,inner sep=2] at (s.side 2) {${X_3}$};
\node[fill=white,inner sep=2] at (s.corner 3) {${X_5}$};
\node[fill=white,inner sep=2] at (s.side 3) {${X_4}$};
\node[fill=white,inner sep=2] at (s.corner 4) {${\p^2}$};
\node[fill=white,inner sep=2] at (s.side 4) {${X_7}$};
\node[fill=white,inner sep=2] at (s.corner 5) {${X_8}$};
\node[fill=white,inner sep=2] at (s.side 5) {${X_3}$};
\node[fill=white,inner sep=2] at (s.corner 6) {${X_5}$};
\node[fill=white,inner sep=2] at (s.side 6) {${X_4}$};
\end{tikzpicture}
}
\]
\caption{$\piece{\P^2}25$}
\label{P2_25}
\end{minipage}
\begin{minipage}[c]{0.32\linewidth}
\[
\resizebox{1.15\linewidth}{!}{
\begin{tikzpicture}[font=\footnotesize]
\node[name=s, regular polygon, rotate=180, regular polygon sides=8,inner sep=2.169cm] at (0,0) {};
\draw[-,ultra thick,Turquoise] (s.center) to["F"] (s.side 1);
\draw[-,ultra thick,RoyalBlue] (s.center) to["H-E"] (s.side 2);
\draw[-,ultra thick,Turquoise] (s.center) to["24H-12E-7F",swap] (s.side 3);
\draw[-,ultra thick,FireBrick] (s.center) to[sloped,"12H-5E-4F",swap] (s.side 4);
\draw[-,ultra thick,Turquoise] (s.center) to[sloped,"36H-12E-13F"{xshift=3pt}] (s.side 5);
\draw[-,ultra thick,RoyalBlue] (s.center) to["5H-E-2F"] (s.side 6);
\draw[-,ultra thick,Turquoise] (s.center) to["12H-5F"] (s.side 7);
\draw[-,ultra thick,FireBrick] (s.center) to["E"] (s.side 8);
\draw[-,ultra thick,RoyalBlue] (s.side 1) to[swap,"H-E"] (s.corner 2);
\draw[-,ultra thick,FireBrick] (s.side 1) to["E"] (s.corner 1);
\draw[-,ultra thick,Turquoise] (s.side 2) to[swap,"24H-12E-7F"] (s.corner 3);
\draw[-,ultra thick,Turquoise] (s.side 2) to["F"] (s.corner 2);
\draw[-,ultra thick,FireBrick] (s.side 3) to[swap,"12H-5E-4F"] (s.corner 4);
\draw[-,ultra thick,RoyalBlue] (s.side 3) to["H-E"] (s.corner 3);
\draw[-,ultra thick,Turquoise] (s.side 4) to[swap,"36H-12E-13F"] (s.corner 5);
\draw[-,ultra thick,Turquoise] (s.side 4) to["24H-12E-7F"] (s.corner 4);
\draw[-,ultra thick,RoyalBlue] (s.side 5) to[swap,"5H-E-2F"{yshift=4pt,xshift=-4pt}] (s.corner 6);
\draw[-,ultra thick,FireBrick] (s.side 5) to["12H-5E-4F"{yshift=4pt,xshift=4pt}] (s.corner 5);
\draw[-,ultra thick,Turquoise] (s.side 6) to[swap,"12H-5F"] (s.corner 7);
\draw[-,ultra thick,Turquoise] (s.side 6) to["36H-12E-13F"] (s.corner 6);
\draw[-,ultra thick,FireBrick] (s.side 7) to[swap,"E"] (s.corner 8);
\draw[-,ultra thick,RoyalBlue] (s.side 7) to["5H-E-2F"] (s.corner 7);
\draw[-,ultra thick,Turquoise] (s.side 8) to[swap,"F"] (s.corner 1);
\draw[-,ultra thick,Turquoise] (s.side 8) to["12H-5F"] (s.corner 8);
\node[fill=white,circle,inner sep=2] at (s.center) {${X_1}$};
\node[fill=white,inner sep=2] at (s.corner 1) {${\p^2}$};
\node[fill=white,inner sep=2] at (s.side 1) {${X_7}$};
\node[fill=white,inner sep=2] at (s.corner 2) {${X_8}$};
\node[fill=white,inner sep=2] at (s.side 2) {${X_2}$};
\node[fill=white,inner sep=2] at (s.corner 3) {${X_8}$};
\node[fill=white,inner sep=2] at (s.side 3) {${X_7}$};
\node[fill=white,inner sep=2] at (s.corner 4) {${\p^2}$};
\node[fill=white,inner sep=2] at (s.side 4) {${X_3}$};
\node[fill=white,inner sep=2] at (s.corner 5) {${\p^2}$};
\node[fill=white,inner sep=2] at (s.side 5) {${X_7}$};
\node[fill=white,inner sep=2] at (s.corner 6) {${X_8}$};
\node[fill=white,inner sep=2] at (s.side 6) {${X_2}$};
\node[fill=white,inner sep=2] at (s.corner 7) {${X_8}$};
\node[fill=white,inner sep=2] at (s.side 7) {${X_7}$};
\node[fill=white,inner sep=2] at (s.corner 8) {${\p^2}$};
\node[fill=white,inner sep=2] at (s.side 8) {${X_3}$};
\end{tikzpicture}
}
\]
\caption{$\piece{\p^2}26$}
\label{P2_26}
\end{minipage}
\begin{minipage}[c]{0.32\linewidth}
\[
\resizebox{1.2\linewidth}{!}{
\begin{tikzpicture}[font=\footnotesize]
\node[name=s, regular polygon, rotate=180, regular polygon sides=6,inner sep=1.557cm] at (0,0) {};
\draw[-,ultra thick,ForestGreen] (s.center) to["F"] (s.side 1);
\draw[-,ultra thick,ForestGreen] (s.center) to["3H-2E"] (s.side 2);
\draw[-,ultra thick,ForestGreen] (s.center) to[sloped,"6H-3E-2F",swap] (s.side 3);
\draw[-,ultra thick,ForestGreen] (s.center) to[sloped,"6H-3F-2E"] (s.side 4);
\draw[-,ultra thick,ForestGreen] (s.center) to["3H-2F"] (s.side 5);
\draw[-,ultra thick,ForestGreen] (s.center) to["E"] (s.side 6);
\draw[-,ultra thick,ForestGreen] (s.side 1) to[swap,"3H-2E"{yshift=-3pt}] (s.corner 2);
\draw[-,ultra thick,ForestGreen] (s.side 1) to["E"] (s.corner 1);
\draw[-,ultra thick,ForestGreen] (s.side 2) to[swap,"6H-3E-2F"] (s.corner 3);
\draw[-,ultra thick,ForestGreen] (s.side 2) to["F"] (s.corner 2);
\draw[-,ultra thick,ForestGreen] (s.side 3) to[swap,"6H-3F-2E"] (s.corner 4);
\draw[-,ultra thick,ForestGreen] (s.side 3) to["3H-2E"] (s.corner 3);
\draw[-,ultra thick,ForestGreen] (s.side 4) to[swap,"3H-2F"{yshift=4pt,xshift=-3pt}] (s.corner 5);
\draw[-,ultra thick,ForestGreen] (s.side 4) to["6H-3E-2F"{yshift=4pt}] (s.corner 4);
\draw[-,ultra thick,ForestGreen] (s.side 5) to[swap,"E"] (s.corner 6);
\draw[-,ultra thick,ForestGreen] (s.side 5) to["6H-3F-2E"] (s.corner 5);
\draw[-,ultra thick,ForestGreen] (s.side 6) to[swap,"F"] (s.corner 1);
\draw[-,ultra thick,ForestGreen] (s.side 6) to["3H-2F"] (s.corner 6);
\node[fill=white,circle,inner sep=2] at (s.center) {${X_3}$};
\node[fill=white,inner sep=2] at (s.corner 1) {${\p^2}$};
\node[fill=white,inner sep=2] at (s.side 1) {${X_6}$};
\node[fill=white,inner sep=2] at (s.corner 2) {${\p^2}$};
\node[fill=white,inner sep=2] at (s.side 2) {${X_6'}$};
\node[fill=white,inner sep=2] at (s.corner 3) {${\p^2}$};
\node[fill=white,inner sep=2] at (s.side 3) {${X_6}$};
\node[fill=white,inner sep=2] at (s.corner 4) {${\p^2}$};
\node[fill=white,inner sep=2] at (s.side 4) {${X_6'}$};
\node[fill=white,inner sep=2] at (s.corner 5) {${\p^2}$};
\node[fill=white,inner sep=2] at (s.side 5) {${X_6}$};
\node[fill=white,inner sep=2] at (s.corner 6) {${\p^2}$};
\node[fill=white,inner sep=2] at (s.side 6) {${X_6'}$};
\end{tikzpicture}
}
\]
\caption{$\piece{\P^2}33$}
\label{P2_33}
\end{minipage}
\end{figure}

%%% P2, 34, 35 %%%
\begin{figure}[ht]
\begin{minipage}[c]{0.49\linewidth}
\[
\resizebox{1.15\linewidth}{!}{
\begin{tikzpicture}[font=\footnotesize]
\node[name=s, regular polygon, rotate=180, regular polygon sides=8,inner sep=2.169cm] at (0,0) {};
\draw[-,ultra thick,Orange] (s.center) to["F"] (s.side 1);
\draw[-,ultra thick,ForestGreen] (s.center) to["3H-2E"] (s.side 2);
\draw[-,ultra thick,black] (s.center) to[""] (s.side 3);
\draw[-,ultra thick,ForestGreen] (s.center) to["\!\!\!9H-4E-3F",swap] (s.side 4);
\draw[-,ultra thick,Orange] (s.center) to[sloped,"12H-4E-5F"] (s.side 5);
\draw[-,ultra thick,ForestGreen] (s.center) to["3(2H-F)-E"] (s.side 6);
\draw[-,ultra thick,black] (s.center) to[""] (s.side 7);
\draw[-,ultra thick,ForestGreen] (s.center) to["E"] (s.side 8);
\draw[-,ultra thick,ForestGreen] (s.side 1) to[swap,"3H-2E"{yshift=-3pt}] (s.corner 2);
\draw[-,ultra thick,ForestGreen] (s.side 1) to["E"] (s.corner 1);
\draw[-,ultra thick,black] (s.side 2) to[swap,""] (s.corner 3);
\draw[-,ultra thick,Orange] (s.side 2) to["F"] (s.corner 2);
\draw[-,ultra thick,ForestGreen] (s.side 3) to[swap,"9H-4E-3F"{yshift=-5pt}] (s.corner 4);
\draw[-,ultra thick,ForestGreen] (s.side 3) to["3H-2E"] (s.corner 3);
\draw[-,ultra thick,Orange] (s.side 4) to[swap,"12H-4E-5F"] (s.corner 5);
\draw[-,ultra thick,black] (s.side 4) to[""] (s.corner 4);
\draw[-,ultra thick,ForestGreen] (s.side 5) to[swap,"3(2H-F)-E"{xshift=-6pt}] (s.corner 6);
\draw[-,ultra thick,ForestGreen] (s.side 5) to["9H-4E-3F"{yshift=3pt,xshift=5pt}] (s.corner 5);
\draw[-,ultra thick,black] (s.side 6) to[swap,""] (s.corner 7);
\draw[-,ultra thick,Orange] (s.side 6) to["12H-4E-5F"] (s.corner 6);
\draw[-,ultra thick,ForestGreen] (s.side 7) to[swap,"E"] (s.corner 8);
\draw[-,ultra thick,ForestGreen] (s.side 7) to["3(2H-F)-E"] (s.corner 7);
\draw[-,ultra thick,Orange] (s.side 8) to[swap,"F"] (s.corner 1);
\draw[-,ultra thick,black] (s.side 8) to[""] (s.corner 8);
\node[fill=white,circle,inner sep=2] at (s.center) {${X_2}$};
\node[fill=white,inner sep=2] at (s.corner 1) {${\p^2}$};
\node[fill=white,inner sep=2] at (s.side 1) {${X_6}$};
\node[fill=white,inner sep=2] at (s.corner 2) {${\p^2}$};
\node[fill=white,inner sep=2] at (s.side 2) {${X_5}$};
\node[fill=white,inner sep=2] at (s.corner 3) {${X_5}_{\!/\p^1}$};
\node[fill=white,inner sep=2] at (s.side 3) {${X_2}_{/\p^1}$};
\node[fill=white,inner sep=2] at (s.corner 4) {${X_5}_{\!/\p^1}$};
\node[fill=white,inner sep=2] at (s.side 4) {${X_5}$};
\node[fill=white,inner sep=2] at (s.corner 5) {${\p^2}$};
\node[fill=white,inner sep=2] at (s.side 5) {${X_6}$};
\node[fill=white,inner sep=2] at (s.corner 6) {${\p^2}$};
\node[fill=white,inner sep=2] at (s.side 6) {${X_5}$};
\node[fill=white,inner sep=2] at (s.corner 7) {${X_5}_{\!/\p^1}$};
\node[fill=white,inner sep=2] at (s.side 7) {${X_2}_{/\p^1}$};
\node[fill=white,inner sep=2] at (s.corner 8) {${X_5}_{\!/\p^1}$};
\node[fill=white,inner sep=2] at (s.side 8) {${X_5}$};
\end{tikzpicture}
}
\]
\caption{$\piece{\p^2}34$}
\label{P2_34}
\end{minipage}\ \ 
\begin{minipage}[c]{0.49\linewidth}
\[
\resizebox{1.15\linewidth}{!}{
\begin{tikzpicture}[font=\footnotesize]
\node[name=s, regular polygon, rotate=180, regular polygon sides=8,inner sep=2.169cm] at (0,0) {};
\draw[-,ultra thick,HotPink] (s.center) to["F"] (s.side 1);
\draw[-,ultra thick,ForestGreen] (s.center) to["3H-2E"] (s.side 2);
\draw[-,ultra thick,RoyalBlue] (s.center) to["4H-2E-F"] (s.side 3);
\draw[-,ultra thick,ForestGreen] (s.center) to[swap,"\!\!\!18H-7E-6F"] (s.side 4);
\draw[-,ultra thick,HotPink] (s.center) to[sloped,"30H-10E-11F"] (s.side 5);
\draw[-,ultra thick,ForestGreen] (s.center) to["15H-4E-6F"] (s.side 6);
\draw[-,ultra thick,RoyalBlue] (s.center) to["2H-F"] (s.side 7);
\draw[-,ultra thick,ForestGreen] (s.center) to["E"] (s.side 8);
\draw[-,ultra thick,ForestGreen] (s.side 1) to[swap,"3H-2E"{yshift=-3pt}] (s.corner 2);
\draw[-,ultra thick,ForestGreen] (s.side 1) to["E"] (s.corner 1);
\draw[-,ultra thick,RoyalBlue] (s.side 2) to[swap,"4H-2E-F"] (s.corner 3);
\draw[-,ultra thick,HotPink] (s.side 2) to["F"] (s.corner 2);
\draw[-,ultra thick,ForestGreen] (s.side 3) to[swap,"18H-7E-6F"] (s.corner 4);
\draw[-,ultra thick,ForestGreen] (s.side 3) to["3H-2E"] (s.corner 3);
\draw[-,ultra thick,HotPink] (s.side 4) to[swap,"30H-10E-11F"] (s.corner 5);
\draw[-,ultra thick,RoyalBlue] (s.side 4) to["4H-2E-F"] (s.corner 4);
\draw[-,ultra thick,ForestGreen] (s.side 5) to[swap,"15H-4E-6F"{yshift=3pt,xshift=-9pt}] (s.corner 6);
\draw[-,ultra thick,ForestGreen] (s.side 5) to["18H-7E-6F"{yshift=3pt,xshift=6pt}] (s.corner 5);
\draw[-,ultra thick,RoyalBlue] (s.side 6) to[swap,"2H-F"] (s.corner 7);
\draw[-,ultra thick,HotPink] (s.side 6) to["30H-10E-11F"] (s.corner 6);
\draw[-,ultra thick,ForestGreen] (s.side 7) to[swap,"E"] (s.corner 8);
\draw[-,ultra thick,ForestGreen] (s.side 7) to["15H-4E-6F"{yshift=7pt}] (s.corner 7);
\draw[-,ultra thick,HotPink] (s.side 8) to[swap,"F"] (s.corner 1);
\draw[-,ultra thick,RoyalBlue] (s.side 8) to["2H-F"] (s.corner 8);
\node[fill=white,circle,inner sep=2] at (s.center) {${X_1}$};
\node[fill=white,inner sep=2] at (s.corner 1) {${\p^2}$};
\node[fill=white,inner sep=2] at (s.side 1) {${X_6}$};
\node[fill=white,inner sep=2] at (s.corner 2) {${\p^2}$};
\node[fill=white,inner sep=2] at (s.side 2) {${X_4}$};
\node[fill=white,inner sep=2] at (s.corner 3) {${X_5}$};
\node[fill=white,inner sep=2] at (s.side 3) {${X_2}$};
\node[fill=white,inner sep=2] at (s.corner 4) {${X_5}$};
\node[fill=white,inner sep=2] at (s.side 4) {${X_4}$};
\node[fill=white,inner sep=2] at (s.corner 5) {${\p^2}$};
\node[fill=white,inner sep=2] at (s.side 5) {${X_6}$};
\node[fill=white,inner sep=2] at (s.corner 6) {${\p^2}$};
\node[fill=white,inner sep=2] at (s.side 6) {${X_4}$};
\node[fill=white,inner sep=2] at (s.corner 7) {${X_5}$};
\node[fill=white,inner sep=2] at (s.side 7) {${X_2}$};
\node[fill=white,inner sep=2] at (s.corner 8) {${X_5}$};
\node[fill=white,inner sep=2] at (s.side 8) {${X_4}$};
\end{tikzpicture}
}
\]
\caption{$\piece{\P^2}35$}
\label{P2_35}
\end{minipage}
\end{figure}

%%% X_8 22, 23, 24%%%
\begin{figure}[ht]
\begin{minipage}[c]{0.32\linewidth}
\[
\resizebox{1.15\linewidth}{!}{
\begin{tikzpicture}[font=\footnotesize]
\node[name=s, regular polygon, rotate=180, regular polygon sides=6,inner sep=1.557cm] at (0,0) {};
\draw[-,ultra thick,FireBrick] (s.center) to["F"] (s.side 1);
\draw[-,ultra thick,black] (s.center) to[""] (s.side 2);
\draw[-,ultra thick,FireBrick] (s.center) to[sloped,"2(H-E)-F",swap] (s.side 3);
\draw[-,ultra thick,FireBrick] (s.center) to[sloped,"2(H-F)-E"] (s.side 4);
\draw[-,ultra thick,black] (s.center) to[""] (s.side 5);
\draw[-,ultra thick,FireBrick] (s.center) to["E"] (s.side 6);
\draw[-,ultra thick,black] (s.side 1) to[swap,""] (s.corner 2);
\draw[-,ultra thick,FireBrick] (s.side 1) to["E"] (s.corner 1);
\draw[-,ultra thick,FireBrick] (s.side 2) to[swap,"2(H-E)-F"] (s.corner 3);
\draw[-,ultra thick,FireBrick] (s.side 2) to["F"] (s.corner 2);
\draw[-,ultra thick,FireBrick] (s.side 3) to[swap,"2(H-F)-E"] (s.corner 4);
\draw[-,ultra thick,black] (s.side 3) to[""] (s.corner 3);
\draw[-,ultra thick,black] (s.side 4) to[swap,""] (s.corner 5);
\draw[-,ultra thick,FireBrick] (s.side 4) to["2(H-E)-F"] (s.corner 4);
\draw[-,ultra thick,FireBrick] (s.side 5) to[swap,"E"] (s.corner 6);
\draw[-,ultra thick,FireBrick] (s.side 5) to["2(H-F)-E"] (s.corner 5);
\draw[-,ultra thick,FireBrick] (s.side 6) to[swap,"F"] (s.corner 1);
\draw[-,ultra thick,black] (s.side 6) to[""] (s.corner 6);
\node[fill=white,circle,inner sep=2] at (s.center) {${X_4}$};
\node[fill=white,inner sep=2] at (s.corner 1) {${X_8}$};
\node[fill=white,inner sep=2] at (s.side 1) {${X_6}$};
\node[fill=white,inner sep=2] at (s.corner 2) {${X_6}_{\!/\p^1}$};
\node[fill=white,inner sep=2] at (s.side 2) {${X_4}_{/\p^1}$};
\node[fill=white,inner sep=2] at (s.corner 3) {${X_6'}_{\!/\p^1}$};
\node[fill=white,inner sep=2] at (s.side 3) {${X_6'}$};
\node[fill=white,inner sep=2] at (s.corner 4) {${X_8'}$};
\node[fill=white,inner sep=2] at (s.side 4) {${X_6''}$};
\node[fill=white,inner sep=2] at (s.corner 5) {${X_6''}_{\!/\p^1}$};
\node[fill=white,inner sep=2] at (s.side 5) {${X_4'}$};
\node[fill=white,inner sep=2] at (s.corner 6) {${X_6'''}_{\!/\p^1}$};
\node[fill=white,inner sep=2] at (s.side 6) {${X_6'''}$};
\end{tikzpicture}
}
\]
\caption{$\piece{X_8}22$}
\label{X8_22}
\end{minipage}
\begin{minipage}[c]{0.33\linewidth}
\[
\resizebox{1.15\linewidth}{!}{
\begin{tikzpicture}[font=\footnotesize]
\node[name=s, regular polygon, rotate=180, regular polygon sides=6,inner sep=1.557cm] at (0,0) {};
\draw[-,ultra thick,ForestGreen] (s.center) to["F"] (s.side 1);
\draw[-,ultra thick,black] (s.center) to[""] (s.side 2);
\draw[-,ultra thick,ForestGreen] (s.center) to[swap,sloped,"3(H-E)-F"] (s.side 3);
\draw[-,ultra thick,FireBrick] (s.center) to[sloped,"3H-2E-2F"] (s.side 4);
\draw[-,ultra thick,RoyalBlue] (s.center) to["H-F"] (s.side 5);
\draw[-,ultra thick,FireBrick] (s.center) to["E"] (s.side 6);
\draw[-,ultra thick,black] (s.side 1) to[swap,""] (s.corner 2);
\draw[-,ultra thick,FireBrick] (s.side 1) to["E"] (s.corner 1);
\draw[-,ultra thick,ForestGreen] (s.side 2) to[swap,"3(H-E)-F"] (s.corner 3);
\draw[-,ultra thick,ForestGreen] (s.side 2) to["F"] (s.corner 2);
\draw[-,ultra thick,FireBrick] (s.side 3) to[swap,"3H-2E-2F"] (s.corner 4);
\draw[-,ultra thick,black] (s.side 3) to[""] (s.corner 3);
\draw[-,ultra thick,RoyalBlue] (s.side 4) to[swap,"H-F"] (s.corner 5);
\draw[-,ultra thick,ForestGreen] (s.side 4) to["3(H-E)-F"] (s.corner 4);
\draw[-,ultra thick,FireBrick] (s.side 5) to[swap,"E"] (s.corner 6);
\draw[-,ultra thick,FireBrick] (s.side 5) to["3H-2E-2F"] (s.corner 5);
\draw[-,ultra thick,ForestGreen] (s.side 6) to[swap,"F"] (s.corner 1);
\draw[-,ultra thick,RoyalBlue] (s.side 6) to["H-F"{yshift=3pt}] (s.corner 6);
\node[fill=white,circle,inner sep=2] at (s.center) {${X_3}$};
\node[fill=white,inner sep=2] at (s.corner 1) {${X_8}$};
\node[fill=white,inner sep=2] at (s.side 1) {${X_6}$};
\node[fill=white,inner sep=2] at (s.corner 2) {${X_6}_{\!/\p^1}$};
\node[fill=white,inner sep=2] at (s.side 2) {${X_3}_{/\p^1}$};
\node[fill=white,inner sep=2] at (s.corner 3) {${X_6'}_{\!/\p^1}$};
\node[fill=white,inner sep=2] at (s.side 3) {${X_6'}$};
\node[fill=white,inner sep=2] at (s.corner 4) {${X_8'}$};
\node[fill=white,inner sep=2] at (s.side 4) {${X_5'}$};
\node[fill=white,inner sep=2] at (s.corner 5) {${X_6''}$};
\node[fill=white,inner sep=2] at (s.side 5) {${X_4}$};
\node[fill=white,inner sep=2] at (s.corner 6) {${X_6'''}$};
\node[fill=white,inner sep=2] at (s.side 6) {${X_5}$};
\end{tikzpicture}
}
\]
\caption{$\piece{X_8}23$}
\label{X8_23}
\end{minipage}\ \
\begin{minipage}[c]{0.32\linewidth}
\[
\resizebox{1.05\linewidth}{!}{
\begin{tikzpicture}[font=\footnotesize]
\node[name=s, regular polygon, rotate=180, regular polygon sides=8,inner sep=2.169cm] at (0,0) {};
\draw[-,ultra thick,Orange] (s.center) to["F"] (s.side 1);
\draw[-,ultra thick,black] (s.center) to[""] (s.side 2);
\draw[-,ultra thick,Orange] (s.center) to["4(H-E)-F"] (s.side 3);
\draw[-,ultra thick,FireBrick] (s.center) to["\!\!\!4H-3E-2F",swap] (s.side 4);
\draw[-,ultra thick,Orange] (s.center) to[sloped,"8H-4E-5F"] (s.side 5);
\draw[-,ultra thick,black] (s.center) to[""] (s.side 6);
\draw[-,ultra thick,Orange] (s.center) to["4H-3F"] (s.side 7);
\draw[-,ultra thick,FireBrick] (s.center) to["E"] (s.side 8);
\draw[-,ultra thick,black] (s.side 1) to[swap,""] (s.corner 2);
\draw[-,ultra thick,FireBrick] (s.side 1) to["E"] (s.corner 1);
\draw[-,ultra thick,Orange] (s.side 2) to[swap,"4(H-E)-F"] (s.corner 3);
\draw[-,ultra thick,Orange] (s.side 2) to["F"] (s.corner 2);
\draw[-,ultra thick,FireBrick] (s.side 3) to[swap,"4H-3E-2F"] (s.corner 4);
\draw[-,ultra thick,black] (s.side 3) to[""] (s.corner 3);
\draw[-,ultra thick,Orange] (s.side 4) to[swap,"8H-4E-5F"] (s.corner 5);
\draw[-,ultra thick,Orange] (s.side 4) to["4(H-E)-F"] (s.corner 4);
\draw[-,ultra thick,black] (s.side 5) to[swap,""] (s.corner 6);
\draw[-,ultra thick,FireBrick] (s.side 5) to["4H-3E-2F"{yshift=3pt}] (s.corner 5);
\draw[-,ultra thick,Orange] (s.side 6) to[swap,"4H-3F"] (s.corner 7);
\draw[-,ultra thick,Orange] (s.side 6) to["8H-4E-5F"] (s.corner 6);
\draw[-,ultra thick,FireBrick] (s.side 7) to[swap,"E"] (s.corner 8);
\draw[-,ultra thick,black] (s.side 7) to[""] (s.corner 7);
\draw[-,ultra thick,Orange] (s.side 8) to[swap,"F"] (s.corner 1);
\draw[-,ultra thick,Orange] (s.side 8) to["4H-3F"] (s.corner 8);
\node[fill=white,circle,inner sep=2] at (s.center) {${X_2}$};
\node[fill=white,inner sep=2] at (s.corner 1) {${X_8}$};
\node[fill=white,inner sep=2] at (s.side 1) {${X_6}$};
\node[fill=white,inner sep=2] at (s.corner 2) {${X_6}_{\!/\p^1}$};
\node[fill=white,inner sep=2] at (s.side 2) {${X_2}_{/\p^1}$};
\node[fill=white,inner sep=2] at (s.corner 3) {${X_6}_{\!/\p^1}$};
\node[fill=white,inner sep=2] at (s.side 3) {${X_6}$};
\node[fill=white,inner sep=2] at (s.corner 4) {${X_8}$};
\node[fill=white,inner sep=2] at (s.side 4) {${X_4}$};
\node[fill=white,inner sep=2] at (s.corner 5) {${X_8}$};
\node[fill=white,inner sep=2] at (s.side 5) {${X_6}$};
\node[fill=white,inner sep=2] at (s.corner 6) {${X_6}_{\!/\p^1}$};
\node[fill=white,inner sep=2] at (s.side 6) {${X_2}_{/\p^1}$};
\node[fill=white,inner sep=2] at (s.corner 7) {${X_6}_{\!/\p^1}$};
\node[fill=white,inner sep=2] at (s.side 7) {${X_6}$};
\node[fill=white,inner sep=2] at (s.corner 8) {${X_8}$};
\node[fill=white,inner sep=2] at (s.side 8) {${X_4}$};
\end{tikzpicture}
}
\]
\caption{$\piece{X_8}24$}
\label{X8_24}
\end{minipage}
\end{figure}

%%% X_8 25, 33, 34%%%
\begin{figure}[ht]
\begin{minipage}[c]{0.32\linewidth}
\[
\resizebox{1.15\linewidth}{!}{
\begin{tikzpicture}[font=\footnotesize]
\node[name=s, regular polygon, rotate=180, regular polygon sides=10,inner sep=2.7720000000000002cm] at (0,0) {};
\draw[-,ultra thick,HotPink] (s.center) to["F"] (s.side 1);
\draw[-,ultra thick,black] (s.center) to[""] (s.side 2);
\draw[-,ultra thick,HotPink] (s.center) to["5(H-E)-F"] (s.side 3);
\draw[-,ultra thick,FireBrick] (s.center) to["\!\!\!5H-4E-2F",swap] (s.side 4);
\draw[-,ultra thick,FireBrick] (s.center) to["\!\!\!8H-5E-4F",swap] (s.side 5);
\draw[-,ultra thick,HotPink] (s.center) to[sloped,"20H-10E-11F"] (s.side 6);
\draw[-,ultra thick,black] (s.center) to[""] (s.side 7);
\draw[-,ultra thick,HotPink] (s.center) to["15H-5E-9F"] (s.side 8);
\draw[-,ultra thick,FireBrick] (s.center) to["3H-2F"] (s.side 9);
\draw[-,ultra thick,FireBrick] (s.center) to["E"] (s.side 10);
\draw[-,ultra thick,black] (s.side 1) to[swap,""] (s.corner 2);
\draw[-,ultra thick,FireBrick] (s.side 1) to["E"] (s.corner 1);
\draw[-,ultra thick,HotPink] (s.side 2) to[swap,"5(H-E)-F"] (s.corner 3);
\draw[-,ultra thick,HotPink] (s.side 2) to["F"] (s.corner 2);
\draw[-,ultra thick,FireBrick] (s.side 3) to[swap,"5H-4E-2F"] (s.corner 4);
\draw[-,ultra thick,black] (s.side 3) to[""] (s.corner 3);
\draw[-,ultra thick,FireBrick] (s.side 4) to[swap,"8H-5E-4F"] (s.corner 5);
\draw[-,ultra thick,HotPink] (s.side 4) to["5(H-E)-F"] (s.corner 4);
\draw[-,ultra thick,HotPink] (s.side 5) to[swap,"20H-10E-11F"] (s.corner 6);
\draw[-,ultra thick,FireBrick] (s.side 5) to["5H-4E-2F"] (s.corner 5);
\draw[-,ultra thick,black] (s.side 6) to[swap,""] (s.corner 7);
\draw[-,ultra thick,FireBrick] (s.side 6) to["8H-5E-4F"{yshift=3pt}] (s.corner 6);
\draw[-,ultra thick,HotPink] (s.side 7) to[swap,"15H-5E-9F"] (s.corner 8);
\draw[-,ultra thick,HotPink] (s.side 7) to["20H-10E-11F"] (s.corner 7);
\draw[-,ultra thick,FireBrick] (s.side 8) to[swap,"3H-2F"] (s.corner 9);
\draw[-,ultra thick,black] (s.side 8) to[""] (s.corner 8);
\draw[-,ultra thick,FireBrick] (s.side 9) to[swap,"E"] (s.corner 10);
\draw[-,ultra thick,HotPink] (s.side 9) to["15H-5E-9F"] (s.corner 9);
\draw[-,ultra thick,HotPink] (s.side 10) to[swap,"F"] (s.corner 1);
\draw[-,ultra thick,FireBrick] (s.side 10) to["3H-2F"] (s.corner 10);
\node[fill=white,circle,inner sep=2] at (s.center) {${X_1}$};
\node[fill=white,inner sep=2] at (s.corner 1) {${X_8}$};
\node[fill=white,inner sep=2] at (s.side 1) {${X_6}$};
\node[fill=white,inner sep=2] at (s.corner 2) {${X_6}_{\!/\p^1}$};
\node[fill=white,inner sep=2] at (s.side 2) {${X_1}_{/\p^1}$};
\node[fill=white,inner sep=2] at (s.corner 3) {${X_6}_{\!/\p^1}$};
\node[fill=white,inner sep=2] at (s.side 3) {${X_6}$};
\node[fill=white,inner sep=2] at (s.corner 4) {${X_8}$};
\node[fill=white,inner sep=2] at (s.side 4) {${X_3}$};
\node[fill=white,inner sep=2] at (s.corner 5) {${X_5}$};
\node[fill=white,inner sep=2] at (s.side 5) {${X_3}$};
\node[fill=white,inner sep=2] at (s.corner 6) {${X_8}$};
\node[fill=white,inner sep=2] at (s.side 6) {${X_6}$};
\node[fill=white,inner sep=2] at (s.corner 7) {${X_6}_{\!/\p^1}$};
\node[fill=white,inner sep=2] at (s.side 7) {${X_1}_{/\p^1}$};
\node[fill=white,inner sep=2] at (s.corner 8) {${X_6}_{\!/\p^1}$};
\node[fill=white,inner sep=2] at (s.side 8) {${X_6}$};
\node[fill=white,inner sep=2] at (s.corner 9) {${X_8}$};
\node[fill=white,inner sep=2] at (s.side 9) {${X_3}$};
\node[fill=white,inner sep=2] at (s.corner 10) {${X_5}$};
\node[fill=white,inner sep=2] at (s.side 10) {${X_3}$};
\end{tikzpicture}
}
\]
\caption{$\piece{X_8}25$}
\label{X8_25}
\end{minipage}
\begin{minipage}[c]{0.32\linewidth}
\[
\resizebox{1.1\linewidth}{!}{
\begin{tikzpicture}[font=\footnotesize]
\node[name=s, regular polygon, rotate=180, regular polygon sides=6,inner sep=1.557cm] at (0,0) {};
\draw[-,ultra thick,ForestGreen] (s.center) to["F"] (s.side 1);
\draw[-,ultra thick,RoyalBlue] (s.center) to["H-E"] (s.side 2);
\draw[-,ultra thick,ForestGreen] (s.center) to[sloped,"6H-4E-3F",swap] (s.side 3);
\draw[-,ultra thick,ForestGreen] (s.center) to[sloped,"6H-3E-4F"] (s.side 4);
\draw[-,ultra thick,RoyalBlue] (s.center) to["H-F"] (s.side 5);
\draw[-,ultra thick,ForestGreen] (s.center) to["E"] (s.side 6);
\draw[-,ultra thick,RoyalBlue] (s.side 1) to[swap,"H-E"] (s.corner 2);
\draw[-,ultra thick,ForestGreen] (s.side 1) to["E"] (s.corner 1);
\draw[-,ultra thick,ForestGreen] (s.side 2) to[swap,"6H-4E-3F"] (s.corner 3);
\draw[-,ultra thick,ForestGreen] (s.side 2) to["F"] (s.corner 2);
\draw[-,ultra thick,ForestGreen] (s.side 3) to[swap,"6H-3E-4F"] (s.corner 4);
\draw[-,ultra thick,RoyalBlue] (s.side 3) to["H-E"] (s.corner 3);
\draw[-,ultra thick,RoyalBlue] (s.side 4) to[swap,"H-F"] (s.corner 5);
\draw[-,ultra thick,ForestGreen] (s.side 4) to["6H-4E-3F"{yshift=3pt}] (s.corner 4);
\draw[-,ultra thick,ForestGreen] (s.side 5) to[swap,"E"] (s.corner 6);
\draw[-,ultra thick,ForestGreen] (s.side 5) to["6H-3E-4F"] (s.corner 5);
\draw[-,ultra thick,ForestGreen] (s.side 6) to[swap,"F"] (s.corner 1);
\draw[-,ultra thick,RoyalBlue] (s.side 6) to["H-F"] (s.corner 6);
\node[fill=white,circle,inner sep=2] at (s.center) {${X_2}$};
\node[fill=white,inner sep=2] at (s.corner 1) {${X_8}$};
\node[fill=white,inner sep=2] at (s.side 1) {${X_5}$};
\node[fill=white,inner sep=2] at (s.corner 2) {${X_6}$};
\node[fill=white,inner sep=2] at (s.side 2) {${X_3}$};
\node[fill=white,inner sep=2] at (s.corner 3) {${X_6'}$};
\node[fill=white,inner sep=2] at (s.side 3) {${X_5'}$};
\node[fill=white,inner sep=2] at (s.corner 4) {${X_8}$};
\node[fill=white,inner sep=2] at (s.side 4) {${X_5}$};
\node[fill=white,inner sep=2] at (s.corner 5) {${X_6}$};
\node[fill=white,inner sep=2] at (s.side 5) {${X_3}$};
\node[fill=white,inner sep=2] at (s.corner 6) {${X_6'}$};
\node[fill=white,inner sep=2] at (s.side 6) {${X_5'}$};
\end{tikzpicture}
}
\]
\caption{$\piece{X_8}33$}
\label{X8_33}
\end{minipage}
\begin{minipage}[c]{0.32\linewidth}
\[
\resizebox{1.15\linewidth}{!}{
\begin{tikzpicture}[font=\footnotesize]
\node[name=s, regular polygon, rotate=180, regular polygon sides=8,inner sep=2.169cm] at (0,0) {};
\draw[-,ultra thick,Orange] (s.center) to["F"] (s.side 1);
\draw[-,ultra thick,RoyalBlue] (s.center) to["H-E"] (s.side 2);
\draw[-,ultra thick,Orange] (s.center) to["12H-8E-5F"] (s.side 3);
\draw[-,ultra thick,ForestGreen] (s.center) to[sloped,"12H-7E-6F"] (s.side 4);
\draw[-,ultra thick,Orange] (s.center) to[sloped,"16H-8E-9F"] (s.side 5);
\draw[-,ultra thick,RoyalBlue] (s.center) to["3H-E-2F"] (s.side 6);
\draw[-,ultra thick,Orange] (s.center) to["4H-3F"] (s.side 7);
\draw[-,ultra thick,ForestGreen] (s.center) to["E"] (s.side 8);
\draw[-,ultra thick,RoyalBlue] (s.side 1) to[swap,"H-E"] (s.corner 2);
\draw[-,ultra thick,ForestGreen] (s.side 1) to["E"] (s.corner 1);
\draw[-,ultra thick,Orange] (s.side 2) to[swap,"12H-8E-5F"] (s.corner 3);
\draw[-,ultra thick,Orange] (s.side 2) to["F"] (s.corner 2);
\draw[-,ultra thick,ForestGreen] (s.side 3) to[swap,"12H-7E-6F"] (s.corner 4);
\draw[-,ultra thick,RoyalBlue] (s.side 3) to["H-E"] (s.corner 3);
\draw[-,ultra thick,Orange] (s.side 4) to[swap,"16H-8E-9F"] (s.corner 5);
\draw[-,ultra thick,Orange] (s.side 4) to["12H-8E-5F"] (s.corner 4);
\draw[-,ultra thick,RoyalBlue] (s.side 5) to[swap,"3H-E-2F"] (s.corner 6);
\draw[-,ultra thick,ForestGreen] (s.side 5) to["12H-7E-6F"] (s.corner 5);
\draw[-,ultra thick,Orange] (s.side 6) to[swap,"4H-3F"] (s.corner 7);
\draw[-,ultra thick,Orange] (s.side 6) to["16H-8E-9F"] (s.corner 6);
\draw[-,ultra thick,ForestGreen] (s.side 7) to[swap,"E"] (s.corner 8);
\draw[-,ultra thick,RoyalBlue] (s.side 7) to["3H-E-2F"] (s.corner 7);
\draw[-,ultra thick,Orange] (s.side 8) to[swap,"F"] (s.corner 1);
\draw[-,ultra thick,Orange] (s.side 8) to["4H-3F"] (s.corner 8);
\node[fill=white,circle,inner sep=2] at (s.center) {${X_1}$};
\node[fill=white,inner sep=2] at (s.corner 1) {${X_8}$};
\node[fill=white,inner sep=2] at (s.side 1) {${X_5}$};
\node[fill=white,inner sep=2] at (s.corner 2) {${X_6}$};
\node[fill=white,inner sep=2] at (s.side 2) {${X_2}$};
\node[fill=white,inner sep=2] at (s.corner 3) {${X_6}$};
\node[fill=white,inner sep=2] at (s.side 3) {${X_5}$};
\node[fill=white,inner sep=2] at (s.corner 4) {${X_8}$};
\node[fill=white,inner sep=2] at (s.side 4) {${X_4}$};
\node[fill=white,inner sep=2] at (s.corner 5) {${X_8}$};
\node[fill=white,inner sep=2] at (s.side 5) {${X_5}$};
\node[fill=white,inner sep=2] at (s.corner 6) {${X_6}$};
\node[fill=white,inner sep=2] at (s.side 6) {${X_2}$};
\node[fill=white,inner sep=2] at (s.corner 7) {${X_6}$};
\node[fill=white,inner sep=2] at (s.side 7) {${X_5}$};
\node[fill=white,inner sep=2] at (s.corner 8) {${X_8}$};
\node[fill=white,inner sep=2] at (s.side 8) {${X_4}$};
\end{tikzpicture}
}
\]
\caption{$\piece{X_8}34$}
\label{X8_34}
\end{minipage}
\end{figure}

%%% X6_22, X6_23
\begin{figure}[ht]
\begin{minipage}[c]{0.32\linewidth}
\[
\resizebox{1.05\linewidth}{!}{
\begin{tikzpicture}[font=\footnotesize]
\node[name=s, regular polygon, rotate=180, regular polygon sides=6,inner sep=1.557cm] at (0,0) {};
\draw[-,ultra thick,FireBrick] (s.center) to["F"] (s.side 1);
\draw[-,ultra thick,FireBrick] (s.center) to["H-2E"] (s.side 2);
\draw[-,ultra thick,FireBrick] (s.center) to[sloped,"2H-3E-2F",swap] (s.side 3);
\draw[-,ultra thick,FireBrick] (s.center) to[sloped,"2H-2E-3F"] (s.side 4);
\draw[-,ultra thick,FireBrick] (s.center) to["H-2F"] (s.side 5);
\draw[-,ultra thick,FireBrick] (s.center) to["E"] (s.side 6);
\draw[-,ultra thick,FireBrick] (s.side 1) to[swap,"H-2E"{yshift=-3pt}] (s.corner 2);
\draw[-,ultra thick,FireBrick] (s.side 1) to["E"] (s.corner 1);
\draw[-,ultra thick,FireBrick] (s.side 2) to[swap,"2H-3E-2F"] (s.corner 3);
\draw[-,ultra thick,FireBrick] (s.side 2) to["F"] (s.corner 2);
\draw[-,ultra thick,FireBrick] (s.side 3) to[swap,"2H-2E-3F"] (s.corner 4);
\draw[-,ultra thick,FireBrick] (s.side 3) to["H-2E"] (s.corner 3);
\draw[-,ultra thick,FireBrick] (s.side 4) to[swap,"H-2F"] (s.corner 5);
\draw[-,ultra thick,FireBrick] (s.side 4) to["2H-3E-2F"{yshift=3pt}] (s.corner 4);
\draw[-,ultra thick,FireBrick] (s.side 5) to[swap,"E"] (s.corner 6);
\draw[-,ultra thick,FireBrick] (s.side 5) to["2H-2E-3F"] (s.corner 5);
\draw[-,ultra thick,FireBrick] (s.side 6) to[swap,"F"] (s.corner 1);
\draw[-,ultra thick,FireBrick] (s.side 6) to["H-2F"] (s.corner 6);
\node[fill=white,circle,inner sep=2] at (s.center) {${X_2}$};
\node[fill=white,inner sep=2] at (s.corner 1) {${X_6}$};
\node[fill=white,inner sep=2] at (s.side 1) {${X_4}$};
\node[fill=white,inner sep=2] at (s.corner 2) {${X_6'}$};
\node[fill=white,inner sep=2] at (s.side 2) {${X_4'}$};
\node[fill=white,inner sep=2] at (s.corner 3) {${X_6''}$};
\node[fill=white,inner sep=2] at (s.side 3) {${X_4''}$};
\node[fill=white,inner sep=2] at (s.corner 4) {${X_6}$};
\node[fill=white,inner sep=2] at (s.side 4) {${X_4}$};
\node[fill=white,inner sep=2] at (s.corner 5) {${X_6'}$};
\node[fill=white,inner sep=2] at (s.side 5) {${X_4'}$};
\node[fill=white,inner sep=2] at (s.corner 6) {${X_6''}$};
\node[fill=white,inner sep=2] at (s.side 6) {${X_4''}$};
\end{tikzpicture}
}
\]
\caption{$\piece{X_6}22$}
\label{X6_22}
\end{minipage}
\begin{minipage}[c]{0.32\linewidth}
\[
\resizebox{1\linewidth}{!}{
\begin{tikzpicture}[font=\footnotesize]
\node[name=s, regular polygon, rotate=180, regular polygon sides=8,inner sep=2.169cm] at (0,0) {};
\draw[-,ultra thick,Green] (s.center) to["F"] (s.side 1);
\draw[-,ultra thick,FireBrick] (s.center) to["H-2E"] (s.side 2);
\draw[-,ultra thick,Green] (s.center) to["4H-6E-3F"] (s.side 3);
\draw[-,ultra thick,FireBrick] (s.center) to[sloped,"4H-5E-4F"] (s.side 4);
\draw[-,ultra thick,Green] (s.center) to[sloped,"6H-6E-7F"] (s.side 5);
\draw[-,ultra thick,FireBrick] (s.center) to["3H-2E-4F"] (s.side 6);
\draw[-,ultra thick,Green] (s.center) to["2H-3F"] (s.side 7);
\draw[-,ultra thick,FireBrick] (s.center) to["E"] (s.side 8);
\draw[-,ultra thick,FireBrick] (s.side 1) to[swap,"H-2E"{yshift=-3pt}] (s.corner 2);
\draw[-,ultra thick,FireBrick] (s.side 1) to["E"] (s.corner 1);
\draw[-,ultra thick,Green] (s.side 2) to[swap,"4H-6E-3F"] (s.corner 3);
\draw[-,ultra thick,Green] (s.side 2) to["F"] (s.corner 2);
\draw[-,ultra thick,FireBrick] (s.side 3) to[swap,"4H-5E-4F"] (s.corner 4);
\draw[-,ultra thick,FireBrick] (s.side 3) to["H-2E"] (s.corner 3);
\draw[-,ultra thick,Green] (s.side 4) to[swap,"6H-6E-7F"] (s.corner 5);
\draw[-,ultra thick,Green] (s.side 4) to["4H-6E-3F"] (s.corner 4);
\draw[-,ultra thick,FireBrick] (s.side 5) to[swap,"3H-2E-4F"{yshift=3pt,xshift=-5pt}] (s.corner 6);
\draw[-,ultra thick,FireBrick] (s.side 5) to["4H-5E-4F"{yshift=3pt,xshift=5pt}] (s.corner 5);
\draw[-,ultra thick,Green] (s.side 6) to[swap,"2H-3F"] (s.corner 7);
\draw[-,ultra thick,Green] (s.side 6) to["6H-6E-7F"] (s.corner 6);
\draw[-,ultra thick,FireBrick] (s.side 7) to[swap,"E"] (s.corner 8);
\draw[-,ultra thick,FireBrick] (s.side 7) to["3H-2E-4F"] (s.corner 7);
\draw[-,ultra thick,Green] (s.side 8) to[swap,"F"] (s.corner 1);
\draw[-,ultra thick,Green] (s.side 8) to["2H-3F"] (s.corner 8);
\node[fill=white,circle,inner sep=2] at (s.center) {${X_1}$};
\node[fill=white,inner sep=2] at (s.corner 1) {$X_6$};
\node[fill=white,inner sep=2] at (s.side 1) {${X_4}$};
\node[fill=white,inner sep=2] at (s.corner 2) {$X_6'$};
\node[fill=white,inner sep=2] at (s.side 2) {${X_3}$};
\node[fill=white,inner sep=2] at (s.corner 3) {${X_6''}$};
\node[fill=white,inner sep=2] at (s.side 3) {${X_4'}$};
\node[fill=white,inner sep=2] at (s.corner 4) {$X_6'''$};
\node[fill=white,inner sep=2] at (s.side 4) {${X_3'}$};
\node[fill=white,inner sep=2] at (s.corner 5) {$X_6$};
\node[fill=white,inner sep=2] at (s.side 5) {${X_4}$};
\node[fill=white,inner sep=2] at (s.corner 6) {$X_6'$};
\node[fill=white,inner sep=2] at (s.side 6) {${X_3}$};
\node[fill=white,inner sep=2] at (s.corner 7) {${X_6''}$};
\node[fill=white,inner sep=2] at (s.side 7) {${X_4'}$};
\node[fill=white,inner sep=2] at (s.corner 8) {$X_6'''$};
\node[fill=white,inner sep=2] at (s.side 8) {${X_3'}$};
\end{tikzpicture}
}
\]
\caption{$\piece{X_6}23$}
\label{X6_23}
\end{minipage}
\end{figure}

%%% F0, 02, 04, 06
\begin{figure}[ht]
\begin{minipage}[c]{0.32\linewidth}
\[
\resizebox{1.05\linewidth}{!}{
\begin{tikzpicture}[font=\footnotesize]
\node[name=s, regular polygon, rotate=180, regular polygon sides=6,inner sep=1.557cm] at (0,0) {};
\draw[-,ultra thick,FireBrick] (s.center) to["F"] (s.side 1);
\draw[-,ultra thick,black] (s.center) to["$H_2$",swap] (s.side 2);
\draw[-,ultra thick,FireBrick] (s.center) to[sloped,"$2H_2-F$",swap] (s.side 3);
\draw[-,ultra thick,black] (s.center) to[sloped,"\tiny$H_1+H_2-F$"] (s.side 4);
\draw[-,ultra thick,FireBrick] (s.center) to[sloped,"$2H_1-F$",swap] (s.side 5);
\draw[-,ultra thick,black] (s.center) to["$H_1$"] (s.side 6);
\draw[-,ultra thick,black] (s.side 1) to[swap,"$H_2$"] (s.corner 2);
\draw[-,ultra thick,black] (s.side 1) to["$H_1$"] (s.corner 1);
\draw[-,ultra thick,FireBrick] (s.side 2) to[swap,"$2H_2-F$"] (s.corner 3);
\draw[-,ultra thick,FireBrick] (s.side 2) to["F"] (s.corner 2);
\draw[-,ultra thick,black] (s.side 3) to[swap,"$H_1+H_2-F$"] (s.corner 4);
\draw[-,ultra thick,black] (s.side 3) to["$H_2$"] (s.corner 3);
\draw[-,ultra thick,FireBrick] (s.side 4) to[swap,"$2H_1-F$"{yshift=3pt,xshift=-4pt}] (s.corner 5);
\draw[-,ultra thick,FireBrick] (s.side 4) to["$2H_2-F$"{yshift=3pt,xshift=4pt}] (s.corner 4);
\draw[-,ultra thick,black] (s.side 5) to[swap,"$H_1$"] (s.corner 6);
\draw[-,ultra thick,black] (s.side 5) to["$H_1+H_2-F$"] (s.corner 5);
\draw[-,ultra thick,FireBrick] (s.side 6) to[swap,"F"] (s.corner 1);
\draw[-,ultra thick,FireBrick] (s.side 6) to["$2H_1-F$"] (s.corner 6);
\node[fill=white,circle,inner sep=2] at (s.center) {${X_6}$};
\node[fill=white,inner sep=2] at (s.corner 1) {${\F_0}_{\!/\p^1}$};
\node[fill=white,inner sep=2] at (s.side 1) {${\F_0}$};
\node[fill=white,inner sep=2] at (s.corner 2) {${\F_0}_{\!/\p^1}$};
\node[fill=white,inner sep=2] at (s.side 2) {${X_6}_{/\p^1}$};
\node[fill=white,inner sep=2] at (s.corner 3) {${\F_0}_{\!/\p^1}$};
\node[fill=white,inner sep=2] at (s.side 3) {${\F_0}$};
\node[fill=white,inner sep=2] at (s.corner 4) {${\F_0}_{\!/\p^1}$};
\node[fill=white,inner sep=2] at (s.side 4) {${X_6}_{/\p^1}$};
\node[fill=white,inner sep=2] at (s.corner 5) {${\F_0}_{\!/\p^1}$};
\node[fill=white,inner sep=2] at (s.side 5) {${\F_0}$};
\node[fill=white,inner sep=2] at (s.corner 6) {${\F_0}_{\!/\p^1}$};
\node[fill=white,inner sep=2] at (s.side 6) {${X_6}_{/\p^1}$};
\end{tikzpicture}
}
\]
\caption{$\piece{\F_0}02$}
\label{F0_02}
\end{minipage}
\begin{minipage}[c]{0.32\linewidth}
\[
\resizebox{1\linewidth}{!}{
\begin{tikzpicture}[font=\footnotesize]
\node[name=s, regular polygon, rotate=180, regular polygon sides=8,inner sep=2.169cm] at (0,0) {};
\draw[-,ultra thick,Orange] (s.center) to["F"] (s.side 1);
\draw[-,ultra thick,black] (s.center) to[""] (s.side 2);
\draw[-,ultra thick,Orange] (s.center) to["4$H_2$-F",swap] (s.side 3);
\draw[-,ultra thick,black] (s.center) to[sloped,"$H_1$+2$H_2$-F",swap] (s.side 4);
\draw[-,ultra thick,Orange] (s.center) to[sloped,"4($H_1$+$H_2$)-3F"{xshift=3pt}] (s.side 5);
\draw[-,ultra thick,black] (s.center) to[sloped,"2$H_1$+$H_2$-F",swap] (s.side 6);
\draw[-,ultra thick,Orange] (s.center) to["4$H_1$-F"] (s.side 7);
\draw[-,ultra thick,black] (s.center) to[""] (s.side 8);
\draw[-,ultra thick,black] (s.side 1) to[swap,""] (s.corner 2);
\draw[-,ultra thick,black] (s.side 1) to[""] (s.corner 1);
\draw[-,ultra thick,Orange] (s.side 2) to[swap,"4$H_2$-F"] (s.corner 3);
\draw[-,ultra thick,Orange] (s.side 2) to["F"] (s.corner 2);
\draw[-,ultra thick,black] (s.side 3) to[swap,"$H_1+2H_2-F$"] (s.corner 4);
\draw[-,ultra thick,black] (s.side 3) to[""] (s.corner 3);
\draw[-,ultra thick,Orange] (s.side 4) to[swap,"4($H_1$+$H_2$)-3F"] (s.corner 5);
\draw[-,ultra thick,Orange] (s.side 4) to["4$H_2$-F"] (s.corner 4);
\draw[-,ultra thick,black] (s.side 5) to[swap,"2$H_1$+$H_2$-F"{yshift=3pt,xshift=-12pt}] (s.corner 6);
\draw[-,ultra thick,black] (s.side 5) to["$H_1$+2$H_2$-F"{yshift=3pt,xshift=12pt}] (s.corner 5);
\draw[-,ultra thick,Orange] (s.side 6) to[swap,"4$H_1$-F"] (s.corner 7);
\draw[-,ultra thick,Orange] (s.side 6) to["4($H_1$+$H_2$)-3F"] (s.corner 6);
\draw[-,ultra thick,black] (s.side 7) to[swap,""] (s.corner 8);
\draw[-,ultra thick,black] (s.side 7) to["2$H_1$+$H_2$-F"] (s.corner 7);
\draw[-,ultra thick,Orange] (s.side 8) to[swap,"F"] (s.corner 1);
\draw[-,ultra thick,Orange] (s.side 8) to["4$H_1$-F"] (s.corner 8);
\node[fill=white,circle,inner sep=2] at (s.center) {${X_4}$};
\node[fill=white,inner sep=2] at (s.corner 1) {${\F_0}_{\!/\p^1}$};
\node[fill=white,inner sep=2] at (s.side 1) {${\F_0}$};
\node[fill=white,inner sep=2] at (s.corner 2) {${\F_0}_{\!/\p^1}$};
\node[fill=white,inner sep=2] at (s.side 2) {${X_4}_{/\p^1}$};
\node[fill=white,inner sep=2] at (s.corner 3) {${\F_0}_{\!/\p^1}$};
\node[fill=white,inner sep=2] at (s.side 3) {${\F_0}$};
\node[fill=white,inner sep=2] at (s.corner 4) {${\F_0}_{\!/\p^1}$};
\node[fill=white,inner sep=2] at (s.side 4) {${X_4}_{/\p^1}$};
\node[fill=white,inner sep=2] at (s.corner 5) {${\F_0}_{\!/\p^1}$};
\node[fill=white,inner sep=2] at (s.side 5) {${\F_0}$};
\node[fill=white,inner sep=2] at (s.corner 6) {${\F_0}_{\!/\p^1}$};
\node[fill=white,inner sep=2] at (s.side 6) {${X_4}_{/\p^1}$};
\node[fill=white,inner sep=2] at (s.corner 7) {${\F_0}_{\!/\p^1}$};
\node[fill=white,inner sep=2] at (s.side 7) {${\F_0}$};
\node[fill=white,inner sep=2] at (s.corner 8) {${\F_0}_{\!/\p^1}$};
\node[fill=white,inner sep=2] at (s.side 8) {${X_4}_{/\p^1}$};
\end{tikzpicture}
}
\]
\caption{$\piece{\F_0}04$}
\label{F0_04}
\end{minipage}
\begin{minipage}[c]{0.32\linewidth}
\[
\resizebox{1\linewidth}{!}{
\begin{tikzpicture}[font=\footnotesize]
\node[name=s, regular polygon, rotate=180, regular polygon sides=12,inner sep=3.357cm] at (0,0) {};
\draw[-,ultra thick,Turquoise] (s.center) to["F"] (s.side 1);
\draw[-,ultra thick,black] (s.center) to["$H_2$"] (s.side 2);
\draw[-,ultra thick,Turquoise] (s.center) to[sloped,"$6H_2-F$",swap] (s.side 3);
\draw[-,ultra thick,black] (s.center) to[swap,"$H_1+3H_2-F$"] (s.side 4);
\draw[-,ultra thick,Turquoise] (s.center) to[sloped,"$6H_1+12H_2-5F$",swap] (s.side 5);
\draw[-,ultra thick,black] (s.center) to[sloped,"$3H_1+4H_2-2F$",swap] (s.side 6);
\draw[-,ultra thick,Turquoise] (s.center) to[sloped,"$12(H_1+H_2)-7F$"] (s.side 7);
\draw[-,ultra thick,black] (s.center) to[sloped,"$4H_1+3H_2-2F$",swap] (s.side 8);
\draw[-,ultra thick,Turquoise] (s.center) to[sloped,swap,"$12H_1+6H_2-5F$"] (s.side 9);
\draw[-,ultra thick,black] (s.center) to["$3H_1+H_2-F$"] (s.side 10);
\draw[-,ultra thick,Turquoise] (s.center) to[sloped,"$6H_1-F$",swap] (s.side 11);
\draw[-,ultra thick,black] (s.center) to["$H_1$"] (s.side 12);
\draw[-,ultra thick,black] (s.side 1) to[swap,"$H_2$"] (s.corner 2);
\draw[-,ultra thick,black] (s.side 1) to["$H_1$"] (s.corner 1);
\draw[-,ultra thick,Turquoise] (s.side 2) to[swap,"$6H_2-F$"] (s.corner 3);
\draw[-,ultra thick,Turquoise] (s.side 2) to["F"] (s.corner 2);
\draw[-,ultra thick,black] (s.side 3) to[swap,"$H_1+3H_2-F$"] (s.corner 4);
\draw[-,ultra thick,black] (s.side 3) to["$H_2$"] (s.corner 3);
\draw[-,ultra thick,Turquoise] (s.side 4) to[swap,"$6H_1+12H_2-5F$"] (s.corner 5);
\draw[-,ultra thick,Turquoise] (s.side 4) to["$6H_2-F$"] (s.corner 4);
\draw[-,ultra thick,black] (s.side 5) to[swap,"$3H_1+4H_2-2F$"] (s.corner 6);
\draw[-,ultra thick,black] (s.side 5) to["$H_1+3H_2-F$"] (s.corner 5);
\draw[-,ultra thick,Turquoise] (s.side 6) to[swap,"$12(H_1+H_2)-7F$"] (s.corner 7);
\draw[-,ultra thick,Turquoise] (s.side 6) to["$6H_1+12H_2-5F$"] (s.corner 6);
\draw[-,ultra thick,black] (s.side 7) to[swap,"$4H_1+3H_2-2F$"{yshift=3pt,xshift=-15pt}] (s.corner 8);
\draw[-,ultra thick,black] (s.side 7) to["$3H_1+4H_2-2F$"{yshift=3pt,xshift=15pt}] (s.corner 7);
\draw[-,ultra thick,Turquoise] (s.side 8) to[swap,"$12H_1+6H_2-5F$"] (s.corner 9);
\draw[-,ultra thick,Turquoise] (s.side 8) to["$12(H_1+H_2)-7F$"] (s.corner 8);
\draw[-,ultra thick,black] (s.side 9) to[swap,"$3H_1+H_2-F$"] (s.corner 10);
\draw[-,ultra thick,black] (s.side 9) to["$4H_1+3H_2-2F$"] (s.corner 9);
\draw[-,ultra thick,Turquoise] (s.side 10) to[swap,"$6H_1-F$"] (s.corner 11);
\draw[-,ultra thick,Turquoise] (s.side 10) to["$12H_1+6H_2-5F$"] (s.corner 10);
\draw[-,ultra thick,black] (s.side 11) to[swap,"$H_1$"] (s.corner 12);
\draw[-,ultra thick,black] (s.side 11) to["$3H_1+H_2-F$"] (s.corner 11);
\draw[-,ultra thick,Turquoise] (s.side 12) to[swap,"F"] (s.corner 1);
\draw[-,ultra thick,Turquoise] (s.side 12) to["$6H_1-F$"] (s.corner 12);
\node[fill=white,circle,inner sep=2] at (s.center) {${X_2}$};
\node[fill=white,inner sep=2] at (s.corner 1) {${\F_0}_{\!/\p^1}$};
\node[fill=white,inner sep=2] at (s.side 1) {${\F_0}$};
\node[fill=white,inner sep=2] at (s.corner 2) {${\F_0}_{\!/\p^1}$};
\node[fill=white,inner sep=2] at (s.side 2) {${X_2}_{/\p^1}$};
\node[fill=white,inner sep=2] at (s.corner 3) {${\F_0}_{\!/\p^1}$};
\node[fill=white,inner sep=2] at (s.side 3) {${\F_0}$};
\node[fill=white,inner sep=2] at (s.corner 4) {${\F_0}_{\!/\p^1}$};
\node[fill=white,inner sep=2] at (s.side 4) {${X_2}_{/\p^1}$};
\node[fill=white,inner sep=2] at (s.corner 5) {${\F_0}_{\!/\p^1}$};
\node[fill=white,inner sep=2] at (s.side 5) {${\F_0}$};
\node[fill=white,inner sep=2] at (s.corner 6) {${\F_0}_{\!/\p^1}$};
\node[fill=white,inner sep=2] at (s.side 6) {${X_2}_{\!/\p^1}$};
\node[fill=white,inner sep=2] at (s.corner 7) {${\F_0}_{\!/\p^1}$};
\node[fill=white,inner sep=2] at (s.side 7) {${\F_0}$};
\node[fill=white,inner sep=2] at (s.corner 8) {${\F_0}_{\!/\p^1}$};
\node[fill=white,inner sep=2] at (s.side 8) {${X_2}_{/\p^1}$};
\node[fill=white,inner sep=2] at (s.corner 9) {${\F_0}_{\!/\p^1}$};
\node[fill=white,inner sep=2] at (s.side 9) {${\F_0}$};
\node[fill=white,inner sep=2] at (s.corner 10) {${\F_0}_{\!/\p^1}$};
\node[fill=white,inner sep=2] at (s.side 10) {${X_2}_{/\p^1}$};
\node[fill=white,inner sep=2] at (s.corner 11) {${\F_0}_{\!/\p^1}$};
\node[fill=white,inner sep=2] at (s.side 11) {${\F_0}$};
\node[fill=white,inner sep=2] at (s.corner 12) {${\F_0}_{\!/\p^1}$};
\node[fill=white,inner sep=2] at (s.side 12) {${X_2}_{\!/\p^1}$};
\end{tikzpicture}
}
\]
\caption{$\piece{\F_0}06$}
\label{F0_06}
\end{minipage}
\end{figure}

%% X_4/P^1 01, 02, 03
\begin{figure}[ht]
	\begin{minipage}[c]{0.32\linewidth}
	\[
	\resizebox{1.05\linewidth}{!}{
	\begin{tikzpicture}[font=\footnotesize]
\node[name=s, regular polygon, rotate=180, regular polygon sides=6,inner sep=1.557cm] at (0,0) {};
\draw[-,ultra thick,RoyalBlue] (s.center) to["F"{yshift=-3pt}] (s.side 1);
\draw[-,ultra thick,black] (s.center) to["$H_2$",swap] (s.side 2);
\draw[-,ultra thick,RoyalBlue] (s.center) to[sloped,"$H_2-F$",swap] (s.side 3);
\draw[-,ultra thick,black] (s.center) to[sloped,"$H_1$+$H_2$-2F"] (s.side 4);
\draw[-,ultra thick,RoyalBlue] (s.center) to[sloped,swap,"$H_1-F$"] (s.side 5);
\draw[-,ultra thick,black] (s.center) to["$H_1$"] (s.side 6);
\draw[-,ultra thick,black] (s.side 1) to[swap,"$H_2$"] (s.corner 2);
\draw[-,ultra thick,black] (s.side 1) to["$H_1$"] (s.corner 1);
\draw[-,ultra thick,RoyalBlue] (s.side 2) to[swap,"$H_2-F$"] (s.corner 3);
\draw[-,ultra thick,RoyalBlue] (s.side 2) to["F"] (s.corner 2);
\draw[-,ultra thick,black] (s.side 3) to[swap,"$H_1+H_2-2F$"{yshift=-2pt}] (s.corner 4);
\draw[-,ultra thick,black] (s.side 3) to["$H_2$"] (s.corner 3);
\draw[-,ultra thick,RoyalBlue] (s.side 4) to[swap,"$H_1-F$"{yshift=3pt}] (s.corner 5);
\draw[-,ultra thick,RoyalBlue] (s.side 4) to["$H_2-F$"{yshift=3pt}] (s.corner 4);
\draw[-,ultra thick,black] (s.side 5) to[swap,"$H_1$"] (s.corner 6);
\draw[-,ultra thick,black] (s.side 5) to["$H_1+H_2-2F$"{yshift=-2pt}] (s.corner 5);
\draw[-,ultra thick,RoyalBlue] (s.side 6) to[swap,"F"] (s.corner 1);
\draw[-,ultra thick,RoyalBlue] (s.side 6) to["$H_1-F$"] (s.corner 6);
\node[fill=white,circle,inner sep=2] at (s.center) {${X_3}$};
\node[fill=white,inner sep=2] at (s.corner 1) {${X_4}_{\!/\p^1}$};
\node[fill=white,inner sep=2] at (s.side 1) {${X_4}$};
\node[fill=white,inner sep=2] at (s.corner 2) {${X_4}_{\!/\p^1}$};
\node[fill=white,inner sep=2] at (s.side 2) {${X_3}_{/\p^1}$};
\node[fill=white,inner sep=2] at (s.corner 3) {${X_4'}_{\!/\p^1}$};
\node[fill=white,inner sep=2] at (s.side 3) {${X_4'}$};
\node[fill=white,inner sep=2] at (s.corner 4) {${X_4'}_{\!/\p^1}$};
\node[fill=white,inner sep=2] at (s.side 4) {${X_3}_{/\p^1}$};
\node[fill=white,inner sep=2] at (s.corner 5) {${X_4''}_{\!/\p^1}$};
\node[fill=white,inner sep=2] at (s.side 5) {${X_4''}$};
\node[fill=white,inner sep=2] at (s.corner 6) {${X_4''}_{\!/\p^1}$};
\node[fill=white,inner sep=2] at (s.side 6) {${X_3}_{/\p^1}$};
\end{tikzpicture}
	}
	\]
	\caption{$\piece{X_4/\p^1}01$}
	\label{X4_01}
	\end{minipage}
	\begin{minipage}[c]{0.32\linewidth}
	\[
	\resizebox{1.05\linewidth}{!}{
	\begin{tikzpicture}[font=\footnotesize]
\node[name=s, regular polygon, rotate=180, regular polygon sides=8,inner sep=2.169cm] at (0,0) {};
\draw[-,ultra thick,FireBrick] (s.center) to["F"{yshift=-3pt}] (s.side 1);
\draw[-,ultra thick,black] (s.center) to["$H_2$",swap] (s.side 2);
\draw[-,ultra thick,FireBrick] (s.center) to[sloped,"$2H_2-F$",swap] (s.side 3);
\draw[-,ultra thick,black] (s.center) to[sloped,"$H_1$+2$H_2$-2F",swap] (s.side 4);
\draw[-,ultra thick,FireBrick] (s.center) to[sloped,"2($H_1$+$H_2$)-3F"{xshift=3pt}] (s.side 5);
\draw[-,ultra thick,black] (s.center) to[sloped,"2$H_1$+$H_2$-2E",swap] (s.side 6);
\draw[-,ultra thick,FireBrick] (s.center) to["$2H_1-F$"] (s.side 7);
\draw[-,ultra thick,black] (s.center) to["$H_1$"] (s.side 8);
\draw[-,ultra thick,black] (s.side 1) to[swap,"$H_2$"] (s.corner 2);
\draw[-,ultra thick,black] (s.side 1) to["$H_1$"] (s.corner 1);
\draw[-,ultra thick,FireBrick] (s.side 2) to[swap,"$2H_2-F$"] (s.corner 3);
\draw[-,ultra thick,FireBrick] (s.side 2) to["F"] (s.corner 2);
\draw[-,ultra thick,black] (s.side 3) to[swap,"$H_1+2H_2-2F$"] (s.corner 4);
\draw[-,ultra thick,black] (s.side 3) to["$H_2$"] (s.corner 3);
\draw[-,ultra thick,FireBrick] (s.side 4) to[swap,"$2(H_1+H_2)-3F$"] (s.corner 5);
\draw[-,ultra thick,FireBrick] (s.side 4) to["$2H_2-F$"] (s.corner 4);
\draw[-,ultra thick,black] (s.side 5) to[swap,"$2H_1+H_2-2F$"{yshift=3pt,xshift=-15pt}] (s.corner 6);
\draw[-,ultra thick,black] (s.side 5) to["$H_1+2H_2-2F$"{yshift=3pt,xshift=15pt}] (s.corner 5);
\draw[-,ultra thick,FireBrick] (s.side 6) to[swap,"$2H_1-F$"] (s.corner 7);
\draw[-,ultra thick,FireBrick] (s.side 6) to["$2(H_1+H_2)-3F$"] (s.corner 6);
\draw[-,ultra thick,black] (s.side 7) to[swap,"$H_1$"] (s.corner 8);
\draw[-,ultra thick,black] (s.side 7) to["$2H_1+H_2-2F$"] (s.corner 7);
\draw[-,ultra thick,FireBrick] (s.side 8) to[swap,"F"] (s.corner 1);
\draw[-,ultra thick,FireBrick] (s.side 8) to["$2H_1-F$"] (s.corner 8);
\node[fill=white,circle,inner sep=2] at (s.center) {${X_2}$};
\node[fill=white,inner sep=2] at (s.corner 1) {${X_4}_{\!/\p^1}$};
\node[fill=white,inner sep=2] at (s.side 1) {${X_4}$};
\node[fill=white,inner sep=2] at (s.corner 2) {${X_4}_{\!/\p^1}$};
\node[fill=white,inner sep=2] at (s.side 2) {${X_2}_{/\p^1}$};
\node[fill=white,inner sep=2] at (s.corner 3) {${X_4'}_{\!/\p^1}$};
\node[fill=white,inner sep=2] at (s.side 3) {${X_4'}$};
\node[fill=white,inner sep=2] at (s.corner 4) {${X_4'}_{\!/\p^1}$};
\node[fill=white,inner sep=2] at (s.side 4) {${X_2'}_{/\p^1}$};
\node[fill=white,inner sep=2] at (s.corner 5) {${X_4}_{\!/\p^1}$};
\node[fill=white,inner sep=2] at (s.side 5) {${X_4}$};
\node[fill=white,inner sep=2] at (s.corner 6) {${X_4}_{\!/\p^1}$};
\node[fill=white,inner sep=2] at (s.side 6) {${X_2}_{/\p^1}$};
\node[fill=white,inner sep=2] at (s.corner 7) {${X_4'}_{\!/\p^1}$};
\node[fill=white,inner sep=2] at (s.side 7) {${X_4'}$};
\node[fill=white,inner sep=2] at (s.corner 8) {${X_4'}_{\!/\p^1}$};
\node[fill=white,inner sep=2] at (s.side 8) {${X_2'}_{/\p^1}$};
\end{tikzpicture}
	}
	\]
	\caption{$\piece{X_4/\p^1}02$}
	\label{X4_02}
	\end{minipage}
\begin{minipage}[c]{0.32\linewidth}
\[
\resizebox{1.05\linewidth}{!}{
\begin{tikzpicture}[font=\footnotesize]
\node[name=s, regular polygon, rotate=180, regular polygon sides=12,inner sep=3.357cm] at (0,0) {};
\draw[-,ultra thick,Green] (s.center) to["F"{yshift=-3pt}] (s.side 1);
\draw[-,ultra thick,black] (s.center) to["$H_2$",swap] (s.side 2);
\draw[-,ultra thick,Green] (s.center) to[sloped,"$3H_2-F$",swap] (s.side 3);
\draw[-,ultra thick,black] (s.center) to["$H_1+3H_2-2F$",swap] (s.side 4);
\draw[-,ultra thick,Green] (s.center) to[sloped,"$3H_1+6H_2-5F$",swap] (s.side 5);
\draw[-,ultra thick,black] (s.center) to[sloped,"$3H_1+4H_2-4F$",swap] (s.side 6);
\draw[-,ultra thick,Green] (s.center) to[sloped,"$6H_1+6H_2-7E$"] (s.side 7);
\draw[-,ultra thick,black] (s.center) to[sloped,"$4H_1+3H_2-4F$",swap] (s.side 8);
\draw[-,ultra thick,Green] (s.center) to[sloped,"$6H_1+3H_2-5E$",swap] (s.side 9);
\draw[-,ultra thick,black] (s.center) to["$3H_1+H_2-2F$"] (s.side 10);
\draw[-,ultra thick,Green] (s.center) to[sloped,"$3H_1-F$",swap] (s.side 11);
\draw[-,ultra thick,black] (s.center) to["$H_1$"] (s.side 12);
\draw[-,ultra thick,black] (s.side 1) to[swap,"$H_2$"] (s.corner 2);
\draw[-,ultra thick,black] (s.side 1) to["$H_1$"] (s.corner 1);
\draw[-,ultra thick,Green] (s.side 2) to[swap,"$3H_2-F$"] (s.corner 3);
\draw[-,ultra thick,Green] (s.side 2) to["F"] (s.corner 2);
\draw[-,ultra thick,black] (s.side 3) to[swap,"$H_1+3H_2-2F$"] (s.corner 4);
\draw[-,ultra thick,black] (s.side 3) to["$H_2$"] (s.corner 3);
\draw[-,ultra thick,Green] (s.side 4) to[swap,"$3H_1+6H_2-5F$"] (s.corner 5);
\draw[-,ultra thick,Green] (s.side 4) to["$3H_2-F$"] (s.corner 4);
\draw[-,ultra thick,black] (s.side 5) to[swap,"$3H_1+4H_2-4F$"] (s.corner 6);
\draw[-,ultra thick,black] (s.side 5) to["$H_1+3H_2-2F$"] (s.corner 5);
\draw[-,ultra thick,Green] (s.side 6) to[swap,"$6H_1+6H_2-7E$"] (s.corner 7);
\draw[-,ultra thick,Green] (s.side 6) to["$3H_1+6H_2-5F$"] (s.corner 6);
\draw[-,ultra thick,black] (s.side 7) to[swap,"$4H_1+3H_2-4F$"{yshift=3pt,xshift=-18pt}] (s.corner 8);
\draw[-,ultra thick,black] (s.side 7) to["$3H_1+4H_2-4F$"{yshift=3pt,xshift=18pt}] (s.corner 7);
\draw[-,ultra thick,Green] (s.side 8) to[swap,"$6H_1+3H_2-5F$"] (s.corner 9);
\draw[-,ultra thick,Green] (s.side 8) to["$6H_1+6H_2-7E$"] (s.corner 8);
\draw[-,ultra thick,black] (s.side 9) to[swap,"$3H_1+H_2-2F$"] (s.corner 10);
\draw[-,ultra thick,black] (s.side 9) to["$4H_1+3H_2-4F$"] (s.corner 9);
\draw[-,ultra thick,Green] (s.side 10) to[swap,"$3H_1-F$"] (s.corner 11);
\draw[-,ultra thick,Green] (s.side 10) to["$6H_1+3H_2-5E$"] (s.corner 10);
\draw[-,ultra thick,black] (s.side 11) to[swap,"$H_1$"] (s.corner 12);
\draw[-,ultra thick,black] (s.side 11) to["$3H_1+H_2-2F$"] (s.corner 11);
\draw[-,ultra thick,Green] (s.side 12) to[swap,"F"] (s.corner 1);
\draw[-,ultra thick,Green] (s.side 12) to["$3H_1-F$"] (s.corner 12);
\node[fill=white,circle,inner sep=2] at (s.center) {${X_1}$};
\node[fill=white,inner sep=2] at (s.corner 1) {${X_4}_{\!/\p^1}$};
\node[fill=white,inner sep=2] at (s.side 1) {${X_4}$};
\node[fill=white,inner sep=2] at (s.corner 2) {${X_4}_{\!/\p^1}$};
\node[fill=white,inner sep=2] at (s.side 2) {${X_1}_{/\p^1}$};
\node[fill=white,inner sep=2] at (s.corner 3) {${X_4'}_{\!/\p^1}$};
\node[fill=white,inner sep=2] at (s.side 3) {${X_4'}$};
\node[fill=white,inner sep=2] at (s.corner 4) {${X_4'}_{\!/\p^1}$};
\node[fill=white,inner sep=2] at (s.side 4) {${X_1}_{/\p^1}$};
\node[fill=white,inner sep=2] at (s.corner 5) {${X_4''}_{\!/\p^1}$};
\node[fill=white,inner sep=2] at (s.side 5) {${X_4''}$};
\node[fill=white,inner sep=2] at (s.corner 6) {${X_4''}_{\!/\p^1}$};
\node[fill=white,inner sep=2] at (s.side 6) {${X_1}_{\!/\p^1}$};
\node[fill=white,inner sep=2] at (s.corner 7) {${X_4}_{\!/\p^1}$};
\node[fill=white,inner sep=2] at (s.side 7) {${X_4}$};
\node[fill=white,inner sep=2] at (s.corner 8) {${X_4}_{\!/\p^1}$};
\node[fill=white,inner sep=2] at (s.side 8) {${X_1}_{/\p^1}$};
\node[fill=white,inner sep=2] at (s.corner 9) {${X_4'}_{\!/\p^1}$};
\node[fill=white,inner sep=2] at (s.side 9) {${X_4'}$};
\node[fill=white,inner sep=2] at (s.corner 10) {${X_4'}_{\!/\p^1}$};
\node[fill=white,inner sep=2] at (s.side 10) {${X_1}_{/\p^1}$};
\node[fill=white,inner sep=2] at (s.corner 11) {${X_4''}_{\!/\p^1}$};
\node[fill=white,inner sep=2] at (s.side 11) {${X_4''}$};
\node[fill=white,inner sep=2] at (s.corner 12) {${X_4''}_{\!/\p^1}$};
\node[fill=white,inner sep=2] at (s.side 12) {${X_1}_{\!/\p^1}$};
\end{tikzpicture}
}
\]
\caption{$\piece{X_4/\p^1}03$}
\label{X4_03}
\end{minipage}
\end{figure}

%%% X4 11, 12, X3 11 %%%
\begin{figure}[ht]
	\begin{minipage}[c]{0.32\linewidth}
	\[
	\resizebox{1\linewidth}{!}{
	\begin{tikzpicture}[font=\footnotesize]
\node[name=s, regular polygon, rotate=180, regular polygon sides=6,inner sep=1.557cm] at (0,0) {};
\draw[-,ultra thick,RoyalBlue] (s.center) to["F"] (s.side 1);
\draw[-,ultra thick,RoyalBlue] (s.center) to["H-2E"] (s.side 2);
\draw[-,ultra thick,RoyalBlue] (s.center) to[sloped,"2H-3E-2F",swap] (s.side 3);
\draw[-,ultra thick,RoyalBlue] (s.center) to[sloped,"2H-2E-3F"] (s.side 4);
\draw[-,ultra thick,RoyalBlue] (s.center) to["H-2F"] (s.side 5);
\draw[-,ultra thick,RoyalBlue] (s.center) to["E"] (s.side 6);
\draw[-,ultra thick,RoyalBlue] (s.side 1) to[swap,"H-2E"] (s.corner 2);
\draw[-,ultra thick,RoyalBlue] (s.side 1) to["E"] (s.corner 1);
\draw[-,ultra thick,RoyalBlue] (s.side 2) to[swap,"2H-3E-2F"] (s.corner 3);
\draw[-,ultra thick,RoyalBlue] (s.side 2) to["F"] (s.corner 2);
\draw[-,ultra thick,RoyalBlue] (s.side 3) to[swap,"2H-2E-3F"] (s.corner 4);
\draw[-,ultra thick,RoyalBlue] (s.side 3) to["H-2E"] (s.corner 3);
\draw[-,ultra thick,RoyalBlue] (s.side 4) to[swap,"H-2F"] (s.corner 5);
\draw[-,ultra thick,RoyalBlue] (s.side 4) to["2H-3E-2F"{yshift=3pt}] (s.corner 4);
\draw[-,ultra thick,RoyalBlue] (s.side 5) to[swap,"E"] (s.corner 6);
\draw[-,ultra thick,RoyalBlue] (s.side 5) to["2H-2E-3F"] (s.corner 5);
\draw[-,ultra thick,RoyalBlue] (s.side 6) to[swap,"F"] (s.corner 1);
\draw[-,ultra thick,RoyalBlue] (s.side 6) to["H-2F"] (s.corner 6);
\node[fill=white,circle,inner sep=2] at (s.center) {${X_1}$};
\node[fill=white,inner sep=2] at (s.corner 1) {${X_3}$};
\node[fill=white,inner sep=2] at (s.side 1) {${X_2}$};
\node[fill=white,inner sep=2] at (s.corner 2) {${X_3}$};
\node[fill=white,inner sep=2] at (s.side 2) {${X_2'}$};
\node[fill=white,inner sep=2] at (s.corner 3) {${X_3}$};
\node[fill=white,inner sep=2] at (s.side 3) {${X_2''}$};
\node[fill=white,inner sep=2] at (s.corner 4) {${X_3}$};
\node[fill=white,inner sep=2] at (s.side 4) {${X_2}$};
\node[fill=white,inner sep=2] at (s.corner 5) {${X_3}$};
\node[fill=white,inner sep=2] at (s.side 5) {${X_2'}$};
\node[fill=white,inner sep=2] at (s.corner 6) {${X_3}$};
\node[fill=white,inner sep=2] at (s.side 6) {${X_2''}$};
\end{tikzpicture}
	}
	\]
	\caption{$\piece{X_3}11$}
	\label{X3_11}
	\end{minipage}
\begin{minipage}[c]{0.32\linewidth}
\[
\resizebox{1\linewidth}{!}{
\begin{tikzpicture}[font=\footnotesize]
\node[name=s, regular polygon, rotate=180, regular polygon sides=6,inner sep=1.557cm] at (0,0) {};
\draw[-,ultra thick,RoyalBlue] (s.center) to["F"] (s.side 1);
\draw[-,ultra thick,black] (s.center) to[""] (s.side 2);
\draw[-,ultra thick,RoyalBlue] (s.center) to["\!\!\!H-2E-F",swap] (s.side 3);
\draw[-,ultra thick,RoyalBlue] (s.center) to["H-E-2F"] (s.side 4);
\draw[-,ultra thick,black] (s.center) to[""] (s.side 5);
\draw[-,ultra thick,RoyalBlue] (s.center) to["E"] (s.side 6);
\draw[-,ultra thick,black] (s.side 1) to[swap,""] (s.corner 2);
\draw[-,ultra thick,RoyalBlue] (s.side 1) to["E"] (s.corner 1);
\draw[-,ultra thick,RoyalBlue] (s.side 2) to[swap,"H-2E-F"] (s.corner 3);
\draw[-,ultra thick,RoyalBlue] (s.side 2) to["F"] (s.corner 2);
\draw[-,ultra thick,RoyalBlue] (s.side 3) to[swap,"H-E-2F"] (s.corner 4);
\draw[-,ultra thick,black] (s.side 3) to[""] (s.corner 3);
\draw[-,ultra thick,black] (s.side 4) to[swap,""] (s.corner 5);
\draw[-,ultra thick,RoyalBlue] (s.side 4) to["H-2E-F"{yshift=3pt}] (s.corner 4);
\draw[-,ultra thick,RoyalBlue] (s.side 5) to[swap,"E"] (s.corner 6);
\draw[-,ultra thick,RoyalBlue] (s.side 5) to["H-E-2F"] (s.corner 5);
\draw[-,ultra thick,RoyalBlue] (s.side 6) to[swap,"F"] (s.corner 1);
\draw[-,ultra thick,black] (s.side 6) to[""] (s.corner 6);
\node[fill=white,circle,inner sep=2] at (s.center) {${X_2}$};
\node[fill=white,inner sep=2] at (s.corner 1) {${X_4}$};
\node[fill=white,inner sep=2] at (s.side 1) {${X_3}$};
\node[fill=white,inner sep=2] at (s.corner 2) {${X_3}_{\!/B}$};
\node[fill=white,inner sep=2] at (s.side 2) {${X_2}_{\!/B}$};
\node[fill=white,inner sep=2] at (s.corner 3) {${X_3'}_{\!/B}$};
\node[fill=white,inner sep=2] at (s.side 3) {${X_3'}$};
\node[fill=white,inner sep=2] at (s.corner 4) {${X_4}$};
\node[fill=white,inner sep=2] at (s.side 4) {${X_3}$};
\node[fill=white,inner sep=2] at (s.corner 5) {${X_3}_{\!/B}$};
\node[fill=white,inner sep=2] at (s.side 5) {${X_2}_{/B}$};
\node[fill=white,inner sep=2] at (s.corner 6) {${X_3'}_{/B}$};
\node[fill=white,inner sep=2] at (s.side 6) {${X_3'}$};
\end{tikzpicture}
}
\]
\caption{$\piece{X_4}11$}
\label{X4_11}
\end{minipage}
\begin{minipage}[c]{0.32\linewidth}
\[
\resizebox{1.05\linewidth}{!}{
\begin{tikzpicture}[font=\footnotesize]
\node[name=s, regular polygon, rotate=180, regular polygon sides=8,inner sep=2.169cm] at (0,0) {};
\draw[-,ultra thick,FireBrick] (s.center) to["F"] (s.side 1);
\draw[-,ultra thick,black] (s.center) to["H-2E"] (s.side 2);
\draw[-,ultra thick,FireBrick] (s.center) to["2H-4E-F"] (s.side 3);
\draw[-,ultra thick,RoyalBlue] (s.center) to["\!\!\!2H-3E-2F",swap] (s.side 4);
\draw[-,ultra thick,FireBrick] (s.center) to["4H-4E-5F"{yshift=8pt}] (s.side 5);
\draw[-,ultra thick,black] (s.center) to["3H-2E-4F"] (s.side 6);
\draw[-,ultra thick,FireBrick] (s.center) to["2H-3F"] (s.side 7);
\draw[-,ultra thick,RoyalBlue] (s.center) to["E"] (s.side 8);
\draw[-,ultra thick,black] (s.side 1) to[swap,"H-2E"{yshift=-3pt}] (s.corner 2);
\draw[-,ultra thick,RoyalBlue] (s.side 1) to["E"] (s.corner 1);
\draw[-,ultra thick,FireBrick] (s.side 2) to[swap,"2H-4E-F"] (s.corner 3);
\draw[-,ultra thick,FireBrick] (s.side 2) to["F"] (s.corner 2);
\draw[-,ultra thick,RoyalBlue] (s.side 3) to[swap,"2H-3E-2F"] (s.corner 4);
\draw[-,ultra thick,black] (s.side 3) to["H-2E"] (s.corner 3);
\draw[-,ultra thick,FireBrick] (s.side 4) to[swap,"4H-4E-5F"] (s.corner 5);
\draw[-,ultra thick,FireBrick] (s.side 4) to["2H-4E-F"] (s.corner 4);
\draw[-,ultra thick,black] (s.side 5) to[swap,"3H-2E-4F"{yshift=3pt,xshift=-8pt}] (s.corner 6);
\draw[-,ultra thick,RoyalBlue] (s.side 5) to["2H-3E-2F"{yshift=3pt,xshift=8pt}] (s.corner 5);
\draw[-,ultra thick,FireBrick] (s.side 6) to[swap,"2H-3F"] (s.corner 7);
\draw[-,ultra thick,FireBrick] (s.side 6) to["4H-4E-5F"] (s.corner 6);
\draw[-,ultra thick,RoyalBlue] (s.side 7) to[swap,"E"] (s.corner 8);
\draw[-,ultra thick,black] (s.side 7) to["3H-2E-4F"] (s.corner 7);
\draw[-,ultra thick,FireBrick] (s.side 8) to[swap,"F"] (s.corner 1);
\draw[-,ultra thick,FireBrick] (s.side 8) to["2H-3F"] (s.corner 8);
\node[fill=white,circle,inner sep=2] at (s.center) {${X_1}$};
\node[fill=white,inner sep=2] at (s.corner 1) {${X_4}$};
\node[fill=white,inner sep=2] at (s.side 1) {${X_3}$};
\node[fill=white,inner sep=2] at (s.corner 2) {${X_3}_{\!/\p^1}$};
\node[fill=white,inner sep=2] at (s.side 2) {${X_1}_{/\p^1}$};
\node[fill=white,inner sep=2] at (s.corner 3) {${X_3}_{\!/\p^1}$};
\node[fill=white,inner sep=2] at (s.side 3) {${X_3}$};
\node[fill=white,inner sep=2] at (s.corner 4) {${X_4}$};
\node[fill=white,inner sep=2] at (s.side 4) {${X_2}$};
\node[fill=white,inner sep=2] at (s.corner 5) {${X_4}$};
\node[fill=white,inner sep=2] at (s.side 5) {${X_3}$};
\node[fill=white,inner sep=2] at (s.corner 6) {${X_3}_{\!/\p^1}$};
\node[fill=white,inner sep=2] at (s.side 6) {${X_1}_{/\p^1}$};
\node[fill=white,inner sep=2] at (s.corner 7) {${X_3}_{\!/\p^1}$};
\node[fill=white,inner sep=2] at (s.side 7) {${X_3}$};
\node[fill=white,inner sep=2] at (s.corner 8) {${X_4}$};
\node[fill=white,inner sep=2] at (s.side 8) {${X_2}$};
\end{tikzpicture}
}
\]
\caption{$\piece{X_4}12$}
\label{X4_12}
\end{minipage}
\end{figure}
\FloatBarrier

We conclude this section with showing an explicit factorisation of the birational map with purely inseparable base points given in \cite[Remark 1.3]{LZ20} via Sarkisov links.

\begin{example}
  For a field $k$ of characteristic $2$, the birational involution $f\colon\p^2_{\k} \dashedrightarrow\p^2_{\k}$ given by
  \[f\colon [x:y:z]\mapsto [xz:yz:x^2+ty^2]\]
  has as base locus the closed points given by the ideal $(x,y)$ (corresponding to $[0:0:1]$) and $(x^2+t,z)$ (whose reduced base change to the algebraic closure is $[t^{1/2}:1:0]$). Its resolution $X_6$ is a del Pezzo surface of degree $6$ (regular but not smooth) and of Picard rank $3$. One can check that $X_6/\Spec k$ is a rank $3$ fibration that dominates five rank $2$ fibrations, depicted in Figure~\ref{P2_12}.
  Note that this gives two decomposition of $f$ into Sarkisov links: $\p^2_{\k} \link{2}{1} X_8\link{1}{2} \p^2_{\k}$, and $\p^2_{\k} \linkI{1}\F_1\link{2}{2}\F_1\linkIII{1}\p^2_{\k}$.
\end{example}

\section{Existence of Sarkisov links}\label{s: general points}

So far we have shown that the Sarkisov program exists over any field (\autoref{t-sarkisov-program}), we have studied properties of Mori fibre spaces (\autoref{s: rank 1 fibrations}) and listed all numerically possible Sarkisov links over any field (\autoref{s: Sarkisov_Links}) and all their elementary relations (\autoref{sec:Elementary relations}).
However, it is not clear that all these Sarkisov links actually exist for some Mori fibre space over some field.
In this section, we explicitly construct examples of Sarkisov links over imperfect fields starting from $\p^2_{\k}$, rational quadric surfaces $Q \subset \p^3_{\k}$ and several Mori conic bundles.

In the case the residue field $\k(\xx)$ of a closed point blown-up is a separable extension of $\k$, the picture is analogue to the case where $\k$ is perfect.
In particular, to construct points in Sarkisov general position with separable residue field we can consider the separable closure $\k^{\sep}$ over $\k$ and then construct $\mathrm{Gal}(\k^{\sep}/\k)$-orbits of $\k^{\rm sep}$-points in general position in $\p^2_{\k^{\rm sep}}$ (see for instance \cite[Section 4.2]{LZ20} for the construction of a closed point of degree $8$ in $\p^2_\k$ in Sarkisov general position).
Instead, in the case where the residue field $k(\xx)$ is a purely inseparable extension of $\k$, the base change $\xx_{\bar{\k}}$ to the algebraic closure is a non-reduced irreducible subscheme of $\p^2_{\bar{\k}}$ and we cannot use Galois theory to prove it is in Sarkisov general position.
For this reason, we prove when the blow-up of $\mathbb{P}^2_{\k}$ is a del Pezzo surface using a more schematic language in \autoref{ss: sarkisov_position}. 
We use this criterion to construct closed points which are twisted forms of irreducible non-reduced sub-schemes whose length is a power of $p$ to construct Sarkisov links with purely inseparable base points in \autoref{ss: existece_Sarkisov_links}.
%on the intuitive level, we can think of blowing-up closed points with purely inseparable extension as blowing-up irreducible non-reduced sub-schemes whose length is a power of $p$ on the algebraic closure.

\subsection{Points in $\P^2_{\k}$ of small degree}

We describe closed points of small degree in $\mathbb{P}^2_{\k}$ with explicit coordinates, up to automorphisms.
As any closed point of $\p^2_{\k}$ is contained in some open subset isomorphic to $\A^2_{\k}$, we then classify points of small degree in $\A^2_{\k}$.

Let $q=p^e$ be a power of $p$ for some $e > 0$.
The maximal ideals with purely inseparable residue field on the affine line are straightforward to describe.
Any closed point $\xx\in \A^1_\k$ of degree $q$ is given by a maximal ideal generated by an irreducible element  $f\in\k[x]$ of degree $q$. If $k(\xx)$ is a purely inseparable extension of $\k$, then $f=x^q-t$ for some $t\in\k\setminus \k^p$.

\begin{remark}
	Let $\k$ be a field and let $\xx\in\A^2$ be a point of degree $p$ prime. Then either $\k(\xx)/\k$ is purely inseparable and $\charact\k=p$, or $\xx_{\bar\k}$ consists of $p$ distinct 	points in $\A^2_{\bar\k}$ forming a $\Gal(\k^{\rm sep}/\k)$-orbit of size $p$.
\end{remark}

For the case of the plane, we need the following characterisation of maximal ideals.
\begin{lemma}[{\cite[\href{https://stacks.math.columbia.edu/tag/00F0}{Tag 00F0}]{stacks-project}}]\label{lem: 2-dim-max-ideal}
	Let $\k$ be any field. Then any maximal ideal $\mathfrak{m}$ in $\k[x,y]$ is of the form $\mathfrak{m}=(g,h)$ with $g \in 	\k[x]$ is irreducible and the class of $h \in\k[x,y]$ is irreducible in $\k[x, y]/(g)$.
\end{lemma}

With the same notation of \autoref{lem: 2-dim-max-ideal}, we can write
$h=\sum_{i=0}^d a_i(x)y^i$, where $\deg(a_i) < \deg(g)$. Then, $\k(\xx) \simeq (\k[x]/(g))[y]/(h)$ and thus $[\k(\xx): \k]=\deg(g) \cdot d$. 
We collect the special case where $[\k(\xx): \k]=p$ as it will be useful for later examples.

\begin{lemma}\label{lem:general_form_maximal_ideal}
	Let $\k$ be any field and let $\xx\in\A^2_\k$ a point of prime degree $p$.
	Then, up to a change of coordinates, we have $\mathfrak{m}_{\xx}=(g(x), b(x) y-f(x))$, where $g\in\k[x]$ is irreducible of degree $p$ and $b,f\in\k[x]$ is of degree $\deg(b),\deg(f)\leq p-1$.
	In particular, if $\k(\xx)$ is purely inseparable over $\k$, then $\mathfrak{m}_{\xx}=(x^p-t, b(x)y-f(x))$ for some $t\in\k\setminus\k^p$.
\end{lemma}
%\begin{proof}
%	By \autoref{lem: 2-dim-max-ideal}, we have $\mathfrak{m}_{\xx}=(g(x),h(x,y))$ where $g \in \k[x]$ is prime and $h\in 	\k[x,y]$ is a polynomial such that its residue class in $\k[x,y]/(g)$ is irreducible.
	
%	Write $h =\sum_{i=0}^{d} a_i(x)y^i$ with $\deg(a_i) <\deg(g)$. Then, there is the the isomorphism $k[x, y]/m \simeq ((k[x]/(g))[y])/(h)$ which implies that $\deg(k(x)) = \deg(g) \cdot d$.
	
%	Consider the first projection $\pi\colon \A^2_\k\to\A^1_\k=\Spec \k[x]$, then $\mathfrak{m}_{\pi(\xx)}=(g(x))$.
%	As the degree of $\xx$ is a prime, we have $\deg(g) \in \left\{ 1,p \right\}$.
%	If $\deg(g)=1$, then $g(x)=x-a$ for some $a \in \k$ and thus $\mathfrak{m}_{\xx}=(x-a, h(a,y))$ and $h(a,y)\in\k[y]$ is irreducible of degree $p$. We conclude by switching the role of $x$ and $y$.

%	If $\deg(g)=p$, let $s\in\bar\k$ be a root of $g$. Since $k(\xx)$ has degree $p$ over $\k$, the class of $h$ in $\k[x,y]/(g)\simeq \k(s)[y]$ is linear.
%	Then	$\mathfrak{m}_{\xx}=(g(x), b(x)y-f(x))$ for some $b,f\in\k[x]$ of degree $\leq p-1$ and thus we conclude as well.
%Finally, if $\k(\xx)$ is purely inseparable of degree $p$, $g(x)=x^p-t$ for some $t\in\k\setminus \k^p$.
%\end{proof}

%====================================

We now collect some facts on closed reduced subschemes of small degree in $\mathbb{P}^2_{\k}$ and their relation to lines, conics and cubics.

\begin{lemma}\label{lem:deg-2-conic-or-line}
	Let $\k$ be a field, and $\xx\in\P_\k^2$ be a closed point of degree $2$. %such that $\k(\xx)/\k$ is purely inseparable.	By \autoref{lem:general_form_maximal_ideal}, we have $\mathfrak{m}_{\xx}=(x^2-t, h(x)y-f(x))$ for some $f,h\in\k[x]$ of degree $\leq1$.
	Then there exists a $\k$-line containing $\xx$.
\end{lemma}
\begin{proof}
Consider the short exact sequence
$$ 0 \to \mathcal{I}_{\xx} \otimes \mathcal{O}_{\mathbb{P}^2_{\k}}(1) \to \mathcal{O}_{\mathbb{P}^2_{\k}}(1) \to \k(\xx) \to 0.$$
As $h^0(\mathbb{P}^2_{\k}, \mathcal{O}_{\mathbb{P}^2_{\k}}(1))=3$ and $h^0(\k(\xx))=2$, we deduce that $H^0(\mathbb{P}^2_{\k}, \mathcal{I}_{\xx} \otimes \mathcal{O}_{\mathbb{P}^2_{\k}}(1)) \neq \emptyset$, thus concluding the proof.
%	Consider a linear polynomial $L(x,y)=Ax+By+C\in\k[x,y]$, and write $h(x)=ax+b$ and $f(x)=cx+d$ with $a,b,c,d\in\k$.
%	We plug in the geometric coordinates $(t^{1/2}, \frac{at^{1/2}+b}{ct^{1/2}+d})$ of the point $\xx$, obtaining $L(t^{1/2}, \frac{at^{1/2}+b}{ct^{1/2}+d})\in \k(t^{1/2})=\k\oplus t^{1/2}\k$.
%	We have $L(t^{1/2}, \frac{at^{1/2}+b}{ct^{1/2}+d})=0$ if and only if
%	\[
%	Aat+Bd+Cb+t^{1/2}(Ab+Bc+Ca)=0
%	\]
%Since $1,t^{1/2}$ are $\k$-linearly independent, this equation holds if and only if $(Aat+Bd+Cb=0)$ and $(Ab+Bc+Ca=0)$. It corrsponds to the intersection $[A:B:C]$ of the two corresponding lines in $\p^2_{\k}$ and it is a $\k$-rational point whose coefficients are rational in $a,b,c,d$.
\end{proof}

\begin{lemma} \label{lem: rat_pt_Fano_curve}
    Let $C$ be an integral Gorenstein projective curve with $H^0(C, \mathcal{O}_C)=\k$ and $-K_C$ ample. 
    Assume there exists an effective Cartier divisor $D$ of odd degree, then $C$ is geometrically integral and $C \simeq \mathbb{P}^1_\k$.
    In particular, if $C$ has a closed point $P$ such that $[\k(P):\k] \geq 3$ is odd, we conclude that $C \simeq \mathbb{P}^1_\k$.
\end{lemma}

\begin{proof}
%If $C$ is regular, then it follows from \autoref{lem: rat_pt_Fano_curve_1}.
%Suppose $C$ is not regular. As it is integral, then the locus of non regularity is a $\k$-rational point $Q$.
As $D$ has degree $d=2n+1$ with $n \geq 0$, by the Riemann--Roch theorem for Gorenstein curves \cite[Theorem 2.7]{Tan18}, there exists an effective Cartier divisor $Q \sim D+nK_C$. Thus $Q$ is a regular $k$-rational point on $C$.
As $C$ is integral and $Q$ is regular, it is easy to see that the projection $\pi_{Q} \colon C \to \mathbb{P}^1_\k$ is a surjective morphism of degree 1, and thus we conclude $C \simeq \mathbb{P}^1_{\k}$.

For the last statement, $C$ is a conic by \cite[10.6]{kk-singbook}. 
It is sufficient to note that for an integral conic the locus of non-regularity is either empty or a $\k$-rational point. Thus $P$ is Cartier and we conclude.
\end{proof}

\begin{lemma}\label{lem:deg-3-conic-or-line}
	Let $\k$ be a field, and $\xx\in\mathbb{P}_\k^2$ be a point of degree $3$.
Then either $\xx$ is contained in a line, or there exists a web of geometrically integral conics containing $\xx$.
\end{lemma}
\begin{proof}
Let us suppose that $\xx$ does not belong to a line. 
Consider the short exact sequence
\[0 \to \mathcal{I}_{\xx} \otimes \mathcal{O}_{\mathbb{P}^2_{\k}}(2) \to \mathcal{O}_{\mathbb{P}^2_{\k}}(2) \to \k(\xx) \to 0.\]
As $h^0(\mathbb{P}^2_{\k}, \mathcal{O}(2))=6$ and $h^0(\mathbb{P}^2_{\k}, \k(\xx))=3$, we deduce that there is a web of conics containing $\xx$. As $\xx$ is not contained in a line, all the conics of the web are integral and thus by \autoref{lem: rat_pt_Fano_curve} we conclude all members are geometrically irreducible.
%	Consider a polynomial $F=ax^2+by^2+c+dxy+ex+gy\in\k[x,y]$ of degree $2$, and write $h(x)=\sum_{i=0}^2h_ix^i$ and $f(x)=\sum_{i=0}^2 f_ix^i$ with $h_i,f_i\in\k$.
%	We plug in the geometric coordinates $(t^{1/3}, \frac{f(t^{1/3})}{h(t^{1/3})})$ of the point $\xx$, obtaining 	$F(t^{1/3}, \frac{f(t^{1/3})}{h(t^{1/3})})\in \k(t^{1/3})$. We have $\frac{f(t^{1/3})}{h(t^{1/3})})=0$  if and only if
%	\[
%	F_1+F_2t^{1/3}+F_3t^{2/3}=0
%	\]
%	for some linear polynomials $F_1,F_2,F_3\in \k[a,b,c,d,e,g]$.
%Since $1,t^{1/3},t^{2/3}$ are $\k$-linearly independent, the polynomial $F\in\k[x,y]$ vanishes at $(t^{1/3}, \frac{f(t^{1/3})}{h(t^{1/3})})$ exactly if the point $[a:b:c:d:e:g]\in\p^5_\k(\k)$ corresponding to $F$ lies in the intersection of the hyperplanes $(F_1=0)$, $(F_2=0)$ and $(F_3=0)$.
%Their intersection is a linear subspace $\mathcal L$ of $\p^5_\k$ of dimension $\geq2$. With \autoref{rmk:Conic geom integral or linear}, either $\mathfrak m_\xx$ contains a linear form, or $\mathcal L(\k)$ contains a web of conics, where all members are geometrically irreducible. As the characteristic of $\k$ is 3, we deduce that the generic member is smooth.
\end{proof}

\begin{lemma}\label{lem:deg-5-conic-or-line}
	Let $\k$ be a field, and $\xx\in\mathbb{P}^2_{\k}$ be a point of degree $5$.
	Then $\xx$ is contained in a line, or there exists a  geometrically integral conic containing $\xx$.
\end{lemma}

\begin{proof}
Let us suppose that $\xx$ does not belong to a line.
Consider
\[0 \to \mathcal{I}_{\xx} \otimes \mathcal{O}(2) \to \mathcal{O}_{\mathbb{P}^2_{\k}}(2) \to \k(\xx) \to 0.\]
As $h^0(\mathbb{P}^2_{\k}, \mathcal{O}_{\mathbb{P}^2_{\k}}(2))=6$ and $h^0(\k(\xx))=5$, we deduce that there is a conic $C$ containing $\xx$, which is integral as $\xx$ is not contained in a line.
The conic $C$ must be geometrically integral by \autoref{lem: rat_pt_Fano_curve}.
\end{proof}

\begin{lemma} \label{lem:cubic_curve}
    Let $\k$ be a field, and $Z \subset \in\mathbb{P}^2_{\k}$ be a reduced 0-dimensional subscheme of degree at most $9$.
    Then there exists a cubic curve $Q$ containing $Z$.
\end{lemma}

\begin{proof}
Consider the short exact sequence 
$$ 0 \to \mathcal{I}_{Z} \otimes \mathcal{O}_{\mathbb{P}^2_{\k}}(3) \to \mathcal{O}_{\mathbb{P}^2_{\k}}(3) \to \mathcal{O}_{Z} \to 0.$$
As $h^0(\mathbb{P}^2_{\k}, \mathcal{O}(3))=10$ and $h^0(\mathbb{P}^2_{\k}, \mathcal{O}_Z) \leq 9$, there exists a cubic curve containing $Z$.
\end{proof}

\subsection{Points in Sarkisov general position in $\mathbb{P}^2_{\k}$} \label{ss: sarkisov_position}
Let $\xx_{1}, ..., \xx_{r}$ be distinct closed points in $\mathbb{P}^2_{\k}$ and let $d_i \coloneqq [\k(\xx_i) :k]$.
Let $X$ be the blow-up of $\mathbb{P}^2_{\k}$ at $\xx_{1}, \dots, \xx_{r}$.
In this subsection, we aim to classify when $X$ is a del Pezzo surface.

We briefly recall what happens in the case of closed points with residue fields that are separable over $k$. Blowing-up a collection of points $T \coloneqq \left\{\xx_1, \dots, \xx_r \right\}$ on $X$ whose residue fields are separable extension of $\k$ is the same as blowing up a collection of $\Gal(\k^{\rm sep}/\k)$-orbits of rational points on $X_{\k^{\rm sep}}$. In that case, the blown-up surface $X$ is del Pezzo if and only if the base change $T_{\k^{\rm sep}}$ is in general position in $\p^2_{\k^{\rm sep}}$, i.e. the irreducible components of $T_{\k^{\rm sep}}$ satisfy that no three are collinear, no six lie on an integral conic and no eight on a singular cubic with one being the singular point. 
As in the more general setting of inseparable residue field we cannot argue in terms of Galois orbits, we prove a statement which involves only the ground field. Fortunately, the same strategy as in \cite[Théorème 1]{demazure80} works, with some extra care.

We start with a couple of preliminaries results.
\begin{lemma} \label{lem: easy_criterion}
The following hold.
    \begin{enumerate}
        \item If $X$ is a del Pezzo surface, then $\sum_{i} d_i \leq 8.$
        \item If $\sum_{i \geq 1} d_i \leq 8,$ then $X$ is a del Pezzo surface if and only if there does not exist an irreducible curve $C$ such that $K_X \cdot C \geq 0.$
    \end{enumerate}
\end{lemma}

\begin{proof}
If $X$ is a del Pezzo surface, then $K_X^2>0$.
As the equality $K_X^2=9 - \sum_{i}d_i$ holds, (1) follows immediately.
As for (2), it is a consequence of the Nakai--Moishezon criterion for ampleness.
\end{proof}

\begin{lemma} \label{lem: sing_curves}
    Let $\k$ be a field. The following hold:
    \begin{enumerate}
        \item  A singular integral conic $C \subset \mathbb{P}^2_{\k}$ has a unique singular point $\xx$, which is a rational point. Moreover, if $C$ is not geometrically reduced, then $p=2$.
        \item A singular integral cubic $C\subset \mathbb{P}^2_{\k}$ has a unique singular point $\xx$, and $(\deg(\xx),\mult_\xx(C))\in\{(1,2), (1,3),(3,2)\}$.
       Moreover, if $C$ is not geometrically reduced, then $p=3$ and the singular point is rational.
    \end{enumerate}
\end{lemma}

\begin{proof}
We discuss the case of cubics (2), as the case of conics follows the same strategy and it is easier.

Let $\xx$ be a singular point of an integral cubic $C$. Passing to an algebraic closure, $\xx_{\bar\k}$ is supported on the singular locus of the cubic $C_{\bar\k}$. We deduce the result going through the well-known classification of singular cubic curves over $\overline{\k}$. 

If $C_{\overline{\k}}$ is integral, it contains a unique singular closed point with multiplicity 2 (which is either a node or a cusp).
As also $C_{\k^{\sep}}$ is integral and it is singular, there is a unique closed point $\yy \in C_{\k^{\sep}}$ contained in its singular locus.  
We now prove that $\yy$ is a rational point of $C_{\k^{\sep}}$.
By the classification on the algebraic closure, we know that $C_{\k^{\sep}}^{\nu} \simeq \mathbb{P}^1_{\k^{\sep}}$ and that the ramification locus $D \subset \mathbb{P}^1_{\k^{\sep}}$ of the normalisation $\nu \colon C_{\k^{\sep}}^{\nu} \to C_{\k^{\sep}}$ satisfies $\dim_{\k^{\sep}}H^0(\mathcal{O}_D)=2$.
By \cite[Proposition A.2]{FS20a} we have $\length \mathcal{O}_{D, \yy}=2\length \mathcal{O}_{P, \yy}$, where $P$ is the conductor locus of $\nu$ and the length are computed as $\mathcal{O}_{C_{\k^{\sep},\yy}}$-module.
Therefore,
$2= \dim_{\k^{\sep}}H^0(\mathcal{O}_D) = 2 \dim_{\k^{\sep}}H^0(\mathcal{O}_P) \geq 2[\k(\yy):\k^{\sep}]$ and we conclude that $\yy$ is a $\k^{\sep}$-rational point.
By Galois descent, as $\yy$ is the unique singular point of $C_{\k^{\sep}}$ and it is rational, it descends to a rational point $\xx$ of $C$ and thus $(\deg(\xx),\mult_\xx(C))=(1,2)$.

If $C_{\bar\k}$ not integral, then we distinguish the three following cases:
\begin{itemize}
    \item Suppose $C_{\overline{\k}}=L \cup D$, where $L$ is a line and $D$ is an irreducible conic such that $L$ and $D$ have different support. Then the same holds true for $C_{\k^{\sep}}$ as $C_{\overline{\k}} \to C_{\k^{\sep}}$ is a universal homeomorphism. As $L$ and $D$ cannot be interchanged by the action of $\Gal(\k^{\sep}/\k)$ (as their degree is different), this contradicts that $C$ is integral. 
    \item Suppose $C_{\overline{\k}}=L_1 \cup L_2 \cup L_3$, where $L_i$ are distinct lines. Then the same holds true for $C_{\k^{\sep}}$ as $C_{\overline{\k}} \to C_{\k^{\sep}}$ is a universal homeomorphism. 
    We distinguish two cases now. 
    If the three lines do not have a common intersection point, as $C$ is integral, the Galois group acts transitively on the irreducible components and we deduce that $\deg(\xx)=3$ and $\mult_\xx(C)=2$.
    If the three lines have a common intersection point, we instead deduce that $\deg(\xx)=1$ and $\mult_\xx(C)=3$.
    \item Suppose $C_{\bar\k}$ is a triple line. Then $C_{\bar\k}$ has equation of the form $(\alpha x+\beta y+\gamma z)^3=0$ with $\alpha,\beta,\gamma\in \overline{\k}$.
    As $C$ is integral, we deduce that $p=3$ by \cite[Theorem 7.3]{Tan21} and thus the equation of $C$ is $
    (a x^3+b y^3+c z^3=0)$ for $a, b, c \in \k$ and we can assume $a=1$ without any loss of generality.
    As $C$ is singular and integral,  by the Jacobian criterion \autoref{lem: jacobian} we deduce that its equation is of the form $x^3+ty^3+sz^3=0$, with $t\in\k \setminus \k^3$ and $s$ belongs to the field generated by $\k^3$ and $t$.
    This implies that $C$ has a unique singular point with $(\deg(\xx),\mult_\xx(C))=(1,3)$.
\end{itemize}
\end{proof}

\begin{figure}
    \centering
    \begin{tikzpicture}
    \draw[label=above] (-3.3,0) -- (-0.7,0) node[midway,below] {$L_2$} node[midway,above] {};
    \draw[] (-1,-0.3)--(-2.3,1.3) node[midway, right] {$L_3$} node[near end,left] {};
    \draw[] (-3,-0.3)--(-1.7,1.3) node[midway,left]
    {$L_1$} node[midway,right] {};

    \draw[] (2,-0.3)--(2,1.3) node[very near start,below] {$L_2$} node[near end,right] {};
    \draw[] (1,-0.3)--(3,1.3) node[very near end, below] {$L_3$} node[near end,left] {};
    \draw[] (3,-0.3)--(1,1.3) node[very near end,below]
    {$L_1$} node[midway,below] {};

    \draw[] (6,-0.3)--(6,1.3) node[very near start, below] {$3L$} node[near end,right] {};
    \draw[] (5.8,-0.3)--(6.2,1.3) node[very near end, below] {} node[near end,left] {};
    \draw[] (6.2,-0.3)--(5.8,1.3) node[very near end,below]
    {} node[midway,below] {};
    
    \end{tikzpicture}
    \caption{$C_{\bar\k}$ for $(\deg(\xx),\mult_\xx(C))\in\{ (1,3),(3,2)\}$}
    \label{fig:my_label}
\end{figure}

\begin{theorem} \label{thm: pts_Sarkisov_general_position}
The surface $X$ obtained as the blow-up of $\mathbb{P}^2_{\k}$ at the closed points $\xx_1, \xx_2, \dots, \xx_r$ is a del Pezzo surface if and only if $\sum_{i \geq 1} d_i \leq 8$ and the following conditions are satisfied:
\begin{enumerate}
    \item \label{i:line} there is no closed subscheme $Z \subset T \coloneqq \bigcup_{i=1}^r \xx_i$ of degree at least 3 contained in a line $L$;
    \item \label{i: conic1}there is no closed subscheme $Z \subset T$ containing both the rational point in the singular locus of an integral conic $C$ and a closed subscheme $Z' \subset Z$ of degree at least 4 in the regular locus of $C$;
    \item \label{i: conic2} there is no closed subscheme $Z \subset T$ of degree at least 6 contained in the regular locus of an integral conic $C$;
    \item \label{i:cubic1} there is no closed subscheme $Z \subset T$ containing both a rational point with multiplicity 3 of an integral cubic curve $Q$ and a closed subscheme $Z' \subset Z$ of degree at least 6 contained in the regular locus of $Q$;
    \item \label{i:cubic2} there is no closed subscheme $Z \subset T$ containing both a closed point of degree 3 with multiplicity 2 of an integral cubic curve $Q$ and a closed subscheme $Z' \subset Z$ of degree at least 3 contained in the regular locus of $Q$
    \item \label{i:cubic3}  there is no closed subscheme $Z \subset T$ containing both a rational point with multiplicity 2 of an integral cubic curve $Q$ and a closed subscheme $Z' \subset Z$ of degree at least 7 contained in the regular locus of $Q$.
\end{enumerate}
\end{theorem}

\begin{proof}
We can assume $\k$ is infinite, as otherwise the result can be deduced from \cite[Théorème 1]{demazure80} and the possible actions of the absolute Galois group.
Note that we have  
$$\Pic(X)= \mathbb{Z}[E_0] \oplus \bigoplus_{i=1}^r \mathbb{Z}[E_i],$$ where $E_0$ is the strict transform of a $\k$-line not passing through any of the $\xx_i$ (which exists as $\k$ is infinite), and, for each $i \geq 1$, $E_i$ is the exceptional divisor over $\xx_i$.
Note that $E_0^2=1, E_i^2=-d_i$ and $E_i \cdot E_j=0$ if $i \neq j$.

By \autoref{lem: easy_criterion}, $X$ is del Pezzo if and only if there are no curves $C$ such that $K_X \cdot C \geq 0$. Assume such a curve $C$ exists: we show that one of the conditions (1)--(6) is violated.
We write the linear equivalence $$C \sim a_0E_0-\sum_{i \geq 1} a_i E_i$$ for some coefficients $a_i \in \mathbb{Z}$.
Note that $K_X \cdot C \geq 0$ is equivalent to 
\begin{equation} \label{eq:courbes_anomales}
    3a_0 \leq \sum_{i = 1}^{r} a_id_i. 
\end{equation}

As $\xx_i$ are distinct, $C$ is the strict transform of an integral curve $\Gamma \subset \mathbb{P}^2_\k$ of degree 
$a_0= C \cdot E_0= \Gamma \cdot L = \deg (\Gamma).$
Moreover, by \cite[Proposition 9.2.23]{Liu02}, we have for $i \geq 1$

$$a_id_i=C \cdot E_i= \text{mult}_{\xx_i}(\Gamma) \cdot d_i;$$
which implies that $a_i=\text{mult}_{\xx_i}(\Gamma)$.
Without loss of generality, we can reorder the coefficients such that $a_1 \geq a_2 \geq \dots \geq a_d.$

Suppose $a_0=1$. In this case, $\Gamma$ is a line and thus $a_i=\mult_{\xx_i}\Gamma \leq 1$. Thus a line satisfying \autoref{eq:courbes_anomales} is equivalent to a line violating condition \autoref{i:line}.

Suppose $a_0=2$. In this case, $a_1 =\mult_{\xx_{1}} \Gamma \leq 2$.
Suppose $a_1=2$. As $\Gamma$ is an irreducible conic with a singular point, then the singular point is a rational point and it is unique by \autoref{lem: sing_curves}.
Thus $d_1=1$ and $a_i \leq 1$ for $i \geq 2$. Equation \autoref{eq:courbes_anomales} becomes $4 \leq \sum_{i \geq 2} d_i$, which is equivalent for the curve $C$ to violate \autoref{i: conic1}.
Suppose $a_1=1$. This shows that $a_i=1$ for all $i \geq 1$ and it implies $6 \leq \sum_i d_i$, violating \autoref{i: conic2}.

Suppose $a_0=3$, so that in this case we have $a_1 \leq 3$.
Suppose first that $a_1=3$. As $C$ is an irreducible cubic with a point $\xx_{1}$ of multiplicity 3, we deduce $d_1 = 1$ and $a_i \leq 1$ for $i \geq 2$ by \autoref{lem: sing_curves}. Thus the equation \autoref{eq:courbes_anomales} becomes $6 \leq \sum_{i \geq 2} d_i $, thus showing $C$ violates \autoref{i:cubic1}. 
Suppose $a_1=2$. In this case,
$C$ is an integral cubic with a unique singular point $P$ of multiplicity $2$ by \autoref{lem: sing_curves}, so that $a_i \leq 1$ for $i \geq 2$. We need to distinguish two more subcases.
If $P$ has degree 3, the equation \autoref{eq:courbes_anomales} becomes $3 \leq \sum_{i \geq 2} d_i$, showing $C$ violates condition \autoref{i:cubic2}. 
If $P$ has degree 1, the equation \autoref{eq:courbes_anomales} becomes $7 \leq \sum_{i \geq 2} d_i$, showing $C$ violates condition \autoref{i:cubic3}.
Finally, suppose $a_1=1$. In this case, $a_i=1$ for all $i\geq 1$ and we have $9 \leq \sum_{i} d_i$, which contradicts the hypothesis $8 \leq \sum_i d_i$.

Suppose now that $a_0 \geq 4$. To conclude, we have to show there does not exist a curve $C$ with $K_X \cdot C \geq 0$.
Since $\sum_i d_i \leq 8$, by \autoref{lem:cubic_curve} there exists a cubic curve $Q$ passing through all the closed points $\xx_i$ and at least through another additional closed point $R$ of the irreducible curve $\Gamma.$ 
Note that, as $a_0 \geq 4$, the cubic $Q$ has no common components with the integral curve $\Gamma$, and thus we can apply Bézout to deduce that
$$ \sum_{i=1}^r a_i d_i +[\k(R):\k] \leq Q \cdot \Gamma =3 a_0 \leq \sum_{i=1}^r a_i d_i, $$
reaching a contradiction.
\end{proof}

\begin{example}
We collect some examples of closed points (mostly with purely inseparable residue field) that are not in Sarkisov position.
    \begin{enumerate}
        \item Let $k=\mathbb{F}_3(t)$, and consider the point $(x^3-tz^3,y)$ of degree 3. It lies on the line $(y)$, therefore it violates the condition \autoref{i:line}.
        \item Fix $k=\mathbb{F}_2(t,s)$, and consider the closed point $(x^2+tz^2,y^2+sz^2)$ of degree 4  together with the rational singular point $[0:1:0]$ of the integral non-regular conic $(x^2+tz^2)$. This collection of two points violate the condition \autoref{i: conic1}.
        \item Fix $k=\mathbb{F}_2(s,t)$, and consider the point $(x^2+tz^2,y^2+sz^2)$ of degree 4  and the point $(x^2+tz^2,y)$ of degree 2 . Both are contained in the integral (non-regular) conic $(x^2+tz^2)$ and thus this case violates condition \autoref{i: conic2}.
        \item
        Fix $k=\mathbb{F}_2(s_1, s_2, t_1, t_2)$, and consider the point $(s_1x^2+s_2y^2+z^2, t_1x^2+t_2y^2+z^2)$ of degree 4 together with the $2$-point $(s_1x^2+s_2y^2+z^2,y)$. They both lie on the regular conic $(s_1x^2+s_2y^2+z^2)$ and thus this case violates condition \autoref{i: conic2}.
        \item Fix $\k=\Q$ and consider the point $(x^3-2,x^2-y)$ of degree 3. Let $Q$ be the cubic with $Q_{\bar\k}=L_1\cup L_2\cup L_3$, where $L_1,L_2,L_3$ are the three $\bar\k$-lines passing through two of the three components of $\xx_{\bar\k}$, and let $\yy$ be some other point lying on $Q$. Then $\deg\yy\geq 3$, and the union of $\xx$ and $\yy$ violates condition \autoref{i:cubic2}.
    \end{enumerate}
    
\end{example}

%===========================================
\subsection{Existence and examples of Sarkisov links} \label{ss: existece_Sarkisov_links}
In this section, we provide examples of each Sarkisov link in \autoref{thm:classification links p>2} and describe geometrically the contracted curve.
Our strategy is the following: we use  \autoref{thm: pts_Sarkisov_general_position} to deduce that a point $\xx\in\p^2_\k$ is in Sarkisov general position. Hence, its blow-up $X$ is a del Pezzo surface of $\rho(X)=2$ giving rise to a Sarkisov link $\chi_{\xx,\yy}\colon \p^2_\k\link ab Y$. The information $\deg(\yy), a,b, K_Y^2$ is uniquely determined by $\deg(\xx)$ (and $K_{\p^2_\k}^2$), and the numerological information of the curve contracted onto $\yy$ comes from \autoref{prop: link_I} and \autoref{prop: link II dP}.

	\subsubsection{Link $2:1$ from $\p^2_{\k}$ to $X_8$}

\begin{lemma}\label{lem:general position deg 2}
	Let $\k$ be a field, and let $\xx\in\p^2_\k$ be a point of degree $2$. Then $\xx$ is in Sarkisov general position. Moreover, the associated Sarkisov link $\chi_{\xx,\yy}\colon\p^2_\k\link21 X_8$ contracts the line containing $\xx$.
\end{lemma}

\begin{proof}
%	By \autoref{lem:point in A^2}, we can view $\xx$ as point in $\A^2_{\k}$.
	By \autoref{lem:deg-2-conic-or-line}, up to an automorphism of $\p^2_\k$, we 
	%have $\mathfrak{m}_{\xx}=(F,y)$ for some $F\in\k[x]$ of degree $2$. 
	%Hence, the line $L\subset\p^2_\k$ given by $(y=0)$ contains $\xx$.
	can assume that $\xx$ is contained in the line $L=(y=0) \subset \p^2_\k$.
	As $(\xx_{\bar\k})_{\rm red}$ consists of one or two $\bar\k$-points. As $(\xx_{\bar\k})$ has degree 2, \autoref{thm: pts_Sarkisov_general_position} immediately yields that the blow up $\pi\colon X\to \p^2_\k$ at $\xx$ is a del Pezzo surface, meaning that $\xx$ is in Sarkisov general position.

		Moreover, note that the strict transform $\tilde L$ of $L$ under the blow-up satisfies $\tilde L\equiv H-E_\xx$, where $H$ is the pull-back of a general line, and so it is the curve contracted by the Sarkisov link $\chi_{\xx,\yy}$ by \autoref{prop: link II dP}.
\end{proof}

\begin{example}\label{ex:link deg2}
	Let $\xx\in\p^2_\k$ be a point of degree $2$. By \autoref{lem:general position deg 2}, it is in Sarkisov general position on $\p^2_{\k}$, and hence, there exists a Sarkisov link $\chi_{\xx,\yy}\colon\p^2_{\k}\link21 X_8$ with base point $\xx$ contracting the line through $\xx$ by \autoref{prop: link II dP}.
	Moreover, $X_8$ is a regular quadric in $\p^3_\k$ as in \autoref{rem: points on quadric}.
	We describe now $\chi_{\xx,\yy}$ explicitely, and we will observe that $X_8$ is not smooth exactly if $\k(\xx)/\k$ is purely inseparable.

	We can assume after a linear change of coordinates that $\xx\in\A^2_\k=\p^2_\k\setminus(z=0)$, and that $\mathfrak{m}_\xx=(g(x),y)$, where $g\in\k[x]$ is irreducible of degree $2$ by \autoref{lem:deg-2-conic-or-line}.
	Denoting the homogenisation of $g$ in $\k[x,z]$ again by $g$, set $Q: = \left(g(x,z)=yw\right) \subset \p^3_{\k}$, with coordinates $[w:x:y:z]$. This is a regular quadric with $\rho(Q)=1$.
	We have a birational map
		\[
		\chi_{\xx,\yy}\colon \p^2_\k\dashedrightarrow Q,\quad [x:y:z]\mapsto[g(x,z): xy: y^2: yz],
		\]
	which factors through the blow-up of $\xx$ and contracts the line $(y=0)$. Hence, this is a Sarkisov link $\chi_{\xx,\yy}\colon\p^2_\k\link21 Q$ with $\yy=[1:0:0:0]$.

	If $\k(\xx)$ is separable over $\k$, then $Q$ is smooth. Otherwise, the characteristic of $\k$ is $2$, and $g=x^2-t$ for some $t\in\k\setminus\k^2$, and $Q_{\bar{\k}}$ has a singularity of type $A_1$, since it is a cone over a smooth conic curve.
\end{example}

\subsubsection{Link $3:3$ on $\p^2_{\k}$}

\begin{proposition}\label{prop:link_P2_33}
	Let $\k$ be a field, and let $\xx\in\p^2_\k$ be a point of degree $3$ not contained in a line. %such that $\mathfrak{m}_\xx$ does not contain a linear form.
	Then $\xx$ is in Sarkisov general position, and the associated Sarkisov link $\chi_{\xx,\yy}\colon \p^2_{\k}\link33\p^2_{\k}$ with base point $\xx$ satisfies $k(\yy)\simeq k(\xx)$. Moreover, $\chi_{\xx,\yy}$ contracts the curve $T\subset\p^2_\k$, where $T$ is a curve of degree $3$ as follows:
	\begin{enumerate}
		\item If $k(\xx)/\k$ is separable, then $T_{\bar\k}$ is the union of the three lines going through two of the three points in $(\xx_{\bar\k})$.
		\item If $k(\xx)/\k$ is purely inseparable, let $C\subset\p^2_\k$ be a conic containing $\xx$, and let $(l=0)\subset\p^2_{\bar\k}$ denote the tangent line to $C_{\bar\k}$ at $(\xx_{\bar\k})_{\rm red}$.
		Then $T_{\overline{k}}=(l^3=0)\subset\p^2_{\overline{k}}$.
	\end{enumerate}
\end{proposition}

\begin{proof}
	%By \autoref{lem:point in A^2}, we can assume that $\xx\in\A^2_{\k}=\p^2_{\k}\setminus(z=0)$.
	The conditions of \autoref{thm: pts_Sarkisov_general_position} are satisfied, and hence $\xx$ is in Sarkisov general position.
	Let $C =(g=0)$ be a geometrically integral conic in $\p^2_\k$ passing through $\xx$ given by \autoref{lem:deg-3-conic-or-line}.
	
	If $\k(\xx)/\k$ is separable, then the result is classical.
	So we assume that $k(\xx)/\k$ is purely inseparable and that $\xx=(x^3-t, b(x)y-f(x))$ by \autoref{lem:general_form_maximal_ideal}.
	The equation for the tangent of $C_{\bar\k}$ at the $\bar\k$-rational point $\yy = (\xx_{\bar\k})_{\rm red} $ is given by the zero set of
	\[
		\ell= \partial_x g|_{\yy}\cdot x + \partial_y g|_{\yy}\cdot y +\partial_z g|_{\yy}\cdot z \in k(t^{1/3})[x,y,z],
	\]
	hence $T=(\ell^3=0)$ defines an integral curve on $\p^2_\k$ (that is geometrically a triple line).
	Note that $\xx$ is a singular point of $T$ with multiplicity $2$. Hence, $T$ is the curve contracted by $\chi_{\xx,\yy}$ from \autoref{prop: link II dP}. Moreover, denoting by $\tilde T$ the strict transform of $T$ under the blow-up of $\p^2_\k$ at $\xx$, we see that $k(\yy)\simeq\Gamma(\tilde T,\O_{\tilde T})=\k(t^{1/3})=k(\xx)$.
\end{proof}

\begin{example}\label{ex:3-3-link}
	Assume that $\k$ is an imperfect field of characteristic $3$ and that $t\in\k\setminus\k^3$. The birational involution $\chi\colon\p^2_{\k}\rat\p^2_{\k}$ given by
  		\[[x:y:z]\mapsto [tz^2-xy:y^2-txz:x^2-yz]\]
  	is not defined at the point of degree $3$ corresponding to the point $\xx$ with $\mathfrak{m}_\xx=(x^3-t,x^2-y)$, and contracts the cubic curve $C=\left(tx^3+y^3+t^2z^3=0\right)$.
	Note that the line $(C_{\bar{\k}})_{\text{red}}=\left( t^{1/3}x+y+t^{2/3}z=0 \right)$ is tangent to $\left(x^2-yz=0 \right)$ at $(\xx_{\bar\k})_{\rm red}=[t^{1/3}:t^{2/3}:1]$ in $\p^2_{\bar{\k}}$.
\end{example}

\subsubsection{Link $5:1$ from $\p^2_\k$ to $X_5$}

\begin{proposition}\label{prop:link_P2_51--15}
	Let $\k$ be a field, and let $\xx\in\p^2_\k$ be a closed point of degree $5$.
	Then $\xx$ is in Sarkisov general position if and only if $\xx$ is not contained in a line. %$\mathfrak{m}_\xx$ does not contain a linear form.

	Assume now that $\xx$ is in Sarkisov general position, and let $\chi_{\xx,\yy}\colon \p^2_{\k}\link51 X_5$ be the associated Sarkisov link with base point $\xx$.
	Then $\chi_{\xx,\yy}$ contracts the conic $C\subset\p^2_{\k}$ containing $\xx$.
	Moreover, let $\yy'\in X_5$ be any rational point.
	Then the following hold.
	\begin{enumerate}
		\item\label{it:link_P2_51--15}
		The point $\yy'\in X_5$ is in Sarkisov general position on $X_5$.
		\item\label{it:link_P2_51--15 v2} The Sarkisov link $\chi_{\yy',\xx'}\colon X_5\link15\p^2_\k$ associated to $\yy'\in X_5$ contracts the curve $(\chi_{\xx,\yy})_*(T)$, where  $T\subset\p^2_{\k}$ is a curve of degree $10$ as follows:
		\begin{enumerate}
			\item If $k(\xx)/\k$ is separable, then $T$ is the union of the five $\bar\k$-conics that go through $\chi_{\xx,\yy}^{-1}(\yy')$ and four of the five points of $\xx_{\bar\k}$.
			\item If $k(\xx)/\k$ is purely inseparable, then $D=(T_{\bar\k})_{\rm red}\subset\p^2_{\bar\k}$ is the conic that contains $\chi_{\xx,\yy}^{-1}(\yy_{\bar\k})\in\p^2_\k$ and meets $C_{\bar\k}$ only at $(\xx_{\bar\k})_{\rm red}$.
		\end{enumerate}
		\item $\k(\xx)\simeq\k(\xx')$.
	\end{enumerate}
\end{proposition}

\begin{proof}
	%By \autoref{lem:deg-5-conic-or-line}, either $\xx$ is contained in a line
	%$\mathfrak{m}_\xx$ contains a linear form (and then $\xx$ is not in Sarkisov general position), or there exists a geometrically integral conic $C\subset\p^2_\k$ containing $\xx$. In the latter case,
	As $\xx$ is irreducible and it is not contained in a line, the assumptions of \autoref{thm: pts_Sarkisov_general_position} are satisfied and hence $\xx$ is in Sarkisov general position.
	By \autoref{lem:deg-5-conic-or-line} there exists a geometrically integral conic $C\subset\p^2_\k$ containing $\xx$
	By \autoref{prop: link II dP}, the Sarkisov link $\chi_{\xx,\yy}\colon \p^2_{\k}\link51 X_5$ associated to $\xx$ contracts the conic $C$ onto $\yy$.

	Now let $\yy'\in X_5$ be a rational point. If $\yy'=\yy$, then $\yy'$ is in Sarkisov general position by definition of a Sarkisov link. Suppose that $\yy'\neq\yy$ and let $Y\to X_5$ be the blow-up at $\yy'$, and $Z\to Y$ the blow-up in $\yy$. Note that $Z$ can also be seen as the blow-up $Z\to\p^2_\k$ at $\xx$ and $\yy'$.

	We will construct a curve $T\subset\p^2_\k$ such that its strict transform $\tilde T$ on $Z$ satisfies $\tilde T\equiv 10L-4E_\xx-5E_{\yy'}$. As a consequence of the numerical data, we see that $\tilde T\subset Z$, as well as its strict transform on $Y$ after contracting $E_\yy$, is an exceptional curve of the first kind.
	As $\rho(Y)=2$, this is the second extremal ray on $Y$, and hence $Y$ is a del Pezzo surface. Therefore, $\yy'\in X_5$ is in Sarkisov general position, and the associated Sarkisov link $\chi_{\yy',\xx'}\colon X_5\link15\p^2_\k$ from \autoref{prop: link II dP} contracts $(\chi_{\xx,\yy})_*(T)$ onto a point $\xx'\in\p^2_\k$ of degree $5$.

	We construct now $T$ as in \autoref{it:link_P2_51--15 v2}. Write $K=k(\xx)$. Recall that $C$ is a geometrically integral conic containing $\xx$.
	As $(\xx_K)_{\rm red}$ is a $K$-point of $C_K$, there is a pencil $\mathcal L$ of conics $D$ in $\p^2_K$ that meet $C_K$ only in the point $(\xx_K)_{\rm red}$; that is, the local intersection of $D$ and $C_K$ at $(\xx_K)_{\rm red}$ is $4$.
	Since $\yy$ does not lie on $C$, also $\yy'_K$ does not lie on $C_K$, and so there is a unique conic $(P=0)$ in the pencil $\mathcal L$ that contains $\yy'_K$, with $P\in K[x,y,z]$.
	Then $T:=(P^5=0)$ is defined over $\k$.
	By construction of $T$ it contains $\xx$ with multiplicity $4$, and $\yy'$ with multiplicity $5$.
	Moreover, $\Gamma(T,\O_T)=\k(t^{1/5})$.
	In particular, $k(\xx)\simeq k(T)\simeq k(\xx')$.
\end{proof}

\subsubsection{Link 4:4 on $X_8$}

\begin{example}\label{ex: X8_44}
	Let $\k$ be an imperfect field with $p=2$ with $t\in\k\setminus\k^2$. Let $\yy$ in $\p^2_\k$ be a point of degree $2$ with $\mathfrak{m}_\yy=(x^2-t,y)$ and $\xx$ in $\p^2_\k$ be a point of degree $4$ with $\mathfrak{m}_\xx=(x^4-t,y-x^2)$. Let $\chi_{\yy,\yy'}\colon \p^2_\k\link21 X_8$  be the Sarkisov link based at $\yy$ (it exists by \autoref{lem:general position deg 2}).
	We will show that $\chi_{\yy,\yy'}(\xx)$ is a point of degree $4$ on $X_8$ that is in Sarkisov general position.
	We work over $K=\k(\xx)\supset\k(\yy)$. There exists a unique conic $(g=0)$ in $\p^2_K$ that has local intersection $2$ with the line $(y=0)$ at $(\yy_K)_{\rm red}$, and local intersection $3$ with the geometrically integral conic $(y-x^2=0)$ at $(\xx_K)_{\rm red}$. Then $D=(g^4=0)$ defines a curve in $\p^2_\k$, and its strict transform on $X_8$ is a curve such that on the blow-up $Y$ of $\xx$ it is an exceptional curve of the first kind.
	In particular, it is the second extremal ray of $Y$, so $Y$ is del Pezzo.
\end{example}

\subsubsection{Links of type \I from $\p^2_\k$ to a conic bundle $X_5/\p^1_\k$}

\begin{example}
Let $\k$ be an imperfect field of characteristic $2$ with $t\in\k\setminus\k^2$ and $\xx\in\p^2_\k$ a point with $\mathfrak m_{\xx}=(x^4-t,y-x^2)$. %[t^{1/4}:t^{1/2}:1]
Blowing up $\xx$ yields a regular surface $X$ that carries a conic fibration $X\to\p^1_\k$ given by the pencil $\lambda (yz-x^2)+\mu(y^2+tz^2)=0$, so it is a Mori conic fibration. \autoref{prop:  Mori-conic-bdl-dP} implies that $X$ is a del Pezzo surface of degree $K_X^2=5$.
\end{example}

\begin{example} \label{ex: Mori-conic-double-lines}
	Let $\k$ be an imperfect field of characteristic $p=2$ with $\pdeg(\k)\geq2$, and let $s,t\in\k\setminus\k^2$ be two $p$-independent elements (see \autoref{def: p-basis}), that is, $1,s,t$ are $\k^2$-linearly independent.
	We construct a rational Mori conic bundle $\pi\colon X \to \p^1_\k$ such that $X$ is not geometrically normal.
	Consider the closed point $\xx\in  \mathbb{P}^2_\k$ with $\mathfrak{m}_\xx=(x^2-s, y^2-t)$ and residue field $\k(\xx)=\k(s^{\frac{1}{2}}, t^{\frac{1}{2}})$, and let $X \to \mathbb{P}^2_\k$ be the blow-up of $\xx$.
	The pencil of conics
	\[
	C_{[\mu:\nu]}:=\left(\mu(x^2-sz^2)+\nu(y^2-tz^2)=0 \right)
	\]
through $\xx$ induces a Mori conic bundle
	$\pi \colon X \to \mathbb{P}^1_\k$, and $X$ is not geometrically normal along the exceptional divisor $E$.
	Note that all closed fibres are not smooth as they are geometrically double lines.
	We show now that there is no point $\xx$ on $X/\p^1_\k$ that is in Sarkisov general position. Indeed, otherwise by \autoref{lem:k(x)=k(b)} we have that $\k(\xx)\simeq\k(\pi(\xx))$  and that $\xx$ lies in the regular locus of the fibre $X_{\k(\pi(\xx))}$. Therefore, the fibre is geometrically reduced, which is a contradiction.

	Moreover, all fibres are integral.
	One can show that the fibre $X_{[\mu:\nu]}$ over $[\mu:\nu]$ is regular for infinitely many values of $[\mu: \nu]$. In fact, one can also show, using derivations, that it is non-regular for infinitely many values.

	Notice that \autoref{prop:  Mori-conic-bdl-dP} implies that $X$ is a del Pezzo surface of degree $K_X^2=5$.
\end{example}

\subsubsection{Links of type \I from $X_8$ to a conic bundle $X_6/\p^1_\k$}

\begin{lemma}\label{lem:point deg2 on quadric general position}
	Let $\k$ be any field, and let $Q\subset\p^3_{\k}$ be a rational regular quadric surface with $\rho(Q)=1$. Then any point $\xx\in Q$ of degree $2$ is in Sarkisov general position on $Q$.
\end{lemma}
\begin{proof}
	Let $Y\to Q$ be the blow-up at $\xx$.
	Recall that there is a line $l\subset\p^3_\k$ going through $\xx$, as it is of degree $2$, hence, there is a pencil of planes $\{H_{t}\}_{t\in\p^1_\k}$ in $\p^3_\k$ containing $\xx$. Blowing up $\p^3_\k$ in $l$ frees the pencil  $\{H_t\}_{t\in\p^1_\k}$.
	Since $Q\cdot l=2=\deg(\xx)$, the intersection of $Q$ and $l$ is transversal. Therefore, the restriction of the blow-up of $\p^3_\k$ in $l$ to $Q$ is the blow-up $Y\to Q$ in $\xx$.
	Therefore, the strict transforms $\tilde C_{t}$ on $Y$ of the curves $C_{t}=H_{t}\cap Q$ form a free pencil of conics on $Y$, giving a Mori conic bundle $\pi\colon Y\to\p^1_{\k}$.
\end{proof}

The following example describes the Mori conic bundle that one obtains when blowing up a non-smooth quadric $Q$ at the non-smooth point; this is a special case of \autoref{lem:point deg2 on quadric general position}.

\begin{example}\label{ex:geometrically-non-normal-pdeg1}
	We give an example of a rational Mori conic bundle that is not geometrically normal, and such that every fibre is singular.

	Let $\k$ be a field of characteristic $2$, and $t\in\k\setminus \k^2$.
	Consider the del Pezzo surface $Q= (zw=tx^2+y^2)\subset\p^3_\k$, with non-smooth point $\xx =(tx^2+y^2,z,w)$. The surface $Q$ is regular because of the condition  $t\in\k\setminus \k^2$.
	The blow-up $\eta\colon Y \to Q$ at $\xx$ is a del Pezzo surface of Picard rank $2$ and it admits a conic bundle structure $\pi\colon Y \to \mathbb{P}^1_\k$ (see \autoref{lem:point deg2 on quadric general position}).
	Explicitely, $Y$ is given by $(uz=vw)\subset Q\times\p^1_\k$, where $[u:v]$ are the coordinates of $\p^1_\k$, and $\eta$ is given by
	\begin{align*}
		\eta\colon Y&\to Q,\\
		([x:y:z:w],[u:v]) &\mapsto [x:y:z:w]
	\end{align*}
	and $\pi\colon Y\to \p^1_\k$ is the projection onto $\p^1_\k$. The surface $Y$ is non-smooth along the exceptional divisor, and hence $Y$ is geometrically non-normal.
	Moreover, the birational map \begin{align*}
		Q & \dashrightarrow \p^2_{\k},\\
		[x:y:z:w]&\mapsto [x:y:z]
	\end{align*}
	induces the fibration $\p^2\dashrightarrow\p^1_\k$, $[x:y:z:w]\mapsto [z^2:tx^2+y^2]$, corresponding to the pencil of (geometrically non-reduced) conics through the point $(tx^2+y^2,z)$ in $\p^2_\k$.
	Note that each such conic in $\p^2_\k$ is non-regular at its intersection with the line given by $(x=0)$. The generic fibre, however, is the regular (non-smooth) conic $(tx^2+y^2+sz^2=0)$ in $\p^2_{\k(s)}$, with $k(s)\simeq k(\p^1_{\k})$.

	By \autoref{lem:k(x)=k(b)}\autoref{it:k(x)-1}, there is no Sarkisov link of type \II starting at $Y$.
\end{example}

\subsubsection{Links of type \II between Mori conic bundles}

\begin{example}\label{ex: elementary transformation}(Point in general position on $\p^1_\k\times\p^1_\k/\p^1_\k$)
	Let $\k$ be any field, and let $Q(x)\in\k[x]$ be an irreducible polynomial of degree $d$.
	Then the point $\xx\in\p^1_\k\times\p^1_\k$ given in affine coordinates by $\mathfrak{m}_\xx=(Q(x_0),y_0)\in \k[x_0,y_0]$ satisfies $k(\xx)\simeq k(\pi_1(\xx))$, where $\pi_1\colon\p^1_\k\times\p^1_\k\to\p^1_\k$ denotes the first projection.
	Therefore, $\xx$ is in Sarkisov general position on $\p^1_\k\times\p^1_\k$ with respect to $\pi_1$ by \autoref{lem:k(x)=k(b)}, and hence there exists a link $\chi\colon \p^1_\k\times\p^1_\k\link dd\F_d$ of type \II with base-point $\xx$ (noting that by blowing up $\xx$ the strict tranform of the section $(y_0=0)$ has self-intersection $-d$).

	We give now an explicit birational map $\varphi\colon\p^1_\k\times\p^1_\k\dashrightarrow \p^1_\k\times\p^1_\k$ that has a decomposition into Sarkisov links $\varphi=\chi_1\circ\cdots\circ\chi_d\circ\chi$, where $\chi_i\colon\F_i\link11\F_{i-1}$ is a link of type \II with rational base point for $i=1,\ldots,d$.
	In affine coordinates, consider the map \[(x,y)\mapsto \left(x,\frac{y}{Q(x)}\right),\] which is birational with inverse given by $(x,y)\mapsto (x,yQ(x))$.
	It induces a birational map on $\p^1_{\k }\times\p^1_{\k}$ given by
	\[\varphi\colon ([x_0:x_1],[y_0:y_1]) \mapsto ([x_0:x_1],[y_0x_1^{d}:y_1Q(x_0,x_1)]),\]
	where we denote again by $Q$ its homogenisation of degree $d$, which factors into Sarkisov links as claimed.

	In particular, if $\k$ is an imperfect field of characteristic $p\geq 2$, one can take $Q(x)=x^{p^e}-t$ for $t\in\k\setminus \k^p$ and $e\geq1$ an integer.
\end{example}

\subsubsection{Links $7:7$ and $8:8$ on $\p^2_{\k}$ (Geiser and Bertini)}

We study the links corresponding to Geiser and Bertini involutions on the projective plane over an arbitrary field.

\begin{lemma}\label{lem:Geiser}
	Let $\k$ be any field and let $X=X_2$ be a geometrically integral del Pezzo surface of degree $2$.
	\begin{enumerate}
		\item\label{Geiser:2new}
		Suppose $X$ is obtained from the blow up $\eta\colon X\to\p^2_\k$ of $r$ points $\xx_1,\ldots,\xx_r$, with exceptional divisors $E_1,\ldots,E_r$.
		Then, the Geiser involution $\gamma\in\Aut_\k(X)$ from \autoref{lem: Geiser_Bertini_involutions} exists. 
		Moreover, its action on $\Pic(X)$ is described by $\gamma^*(E_i)=-d_{E_i}K_{X_2}-E_i$, for $i=1,\ldots,r$.

		\item\label{Geiser:3}
		Let $\xx\in\p^2_\k$ be a point of degree $7$ in Sarkisov general position.
		Then there exists a birational involution $\chi\colon\p^2_{\k}\rat\p^2_{\k}$ that is a link of type II
		\[
			\begin{tikzcd}
			&X\ar[dl,"\eta_{\xx}",swap]\ar[dr,"\eta_{\yy}"]&\\
			\p^2_{\k} \ar[rr,dashed,"\chi"]&&\p^2_{\k},
			\end{tikzcd}
		\]
		with base-point $\xx$ and the base-point of $\chi^{-1}$ is a point $\yy$ with $\k(\yy)\simeq\k(\xx)$, and $\eta_{\xx},\eta_{\yy}$ are the blow-ups of $\xx$ and $\yy$, respectively. 
		If we denote by $L$ the pullback by $\eta_{\xx}$ of a general line on $\p^2_{\k}$ and $E$ the exceptional divisor over $\xx$, then $\eta_{\yy}$ contracts a curve $F\equiv 21L-8E$.

		\item\label{Geiser:2} Let $\chi\colon \p^2_\k \rat \p^2_\k$ be a link of type II with a base-point of degree 7. Then there exist automorphisms $\alpha,\beta$ of $\p^2_\k$ such that $\beta \circ \chi \circ \alpha$ is a birational involution.
	\end{enumerate}
\end{lemma}

\begin{proof}
We first note that the $X_2$ appearing in the statements satisfy $H^1(X, \mathcal{O}_X)=0$ and $\rho(X_2) \geq 2$, and thus the Geiser involution exists by \autoref{lem: Geiser_Bertini_involutions}.

We claim that $X$ is geometrically normal. Suppose by contradiction it is not. As $\rho(X) \geq 2$, by \autoref{thm: dP-base-change-classification} the isomorphism class $Y=X_{\overline{\k}}^N$ is $\mathbb{P}^1_{\overline{\k}} \times \mathbb{P}^1_{\overline{\k}}$. In particular, both extremal contractions of $X$ are conic bundles, which is a contradiction. Thus in particular $X$ is geometrically normal.

\autoref{Geiser:2new} Let $E_{j}$ be an exceptional divisor, and consider $E_{j'}=\gamma^*(E_{j})$ as a Cartier divisor. As $\gamma^{*}(E_j+E_{j'})$ is fixed by $\gamma$, we have $E_j+E_{j'} \sim -a K_X$ by \autoref{lem: action-Bertini-Geiser} for some $a>0$.
We show that $a=d_{E_j}$.
By a Riemann--Roch computation, one sees that $h^0(X,\mathcal{O}_X(-d_{E_j} K_X-E_j))>0$ and thus there exists a curve $C$ such that $-d_{E_j}K_X \sim E_j+C$. Thus we have that
$d_{E_j}K_X+C \sim aK_X+E_{j'}$. Intersecting both members with $K_X$, we deduce that 
\[ d_{E_j}=d_{E_j}K_X^2+C \cdot K_X=(aK_X+E_{j'})\cdot K_X=2a-d_{E_j},\] which implies that $a=d_{E_j}$ and thus $C \sim E_{j'}$.
A simple computation now shows that 
$\gamma^{*}E_j \sim -d_{E_j} K_X-\gamma^{*}(E_{j}')=-d_{E_j}K_X-E_j$.

	\autoref{Geiser:3} Let $\eta_{\xx}\colon X\to\p^2_\k$ be the blow-up of $\xx$. Then $X$ is a del Pezzo surface of degree $2$ with its Geiser involution $\gamma$. 
By \autoref{lem: action-Bertini-Geiser}, $\gamma$ exchanges the $K_X$-negative extremal rays.
	In particular, if $E$ is the exceptional divisor of $\xx$, then its image $F=\gamma(E) \simeq E$ is an extremal ray and there is a birational morphism $\eta_{\yy}\colon X\to Z$ contracting $F$ onto a regular point $\yy\in Z$. As $F \simeq E$, we have $H^0(E,\mathcal O_E)\simeq H^0(E,\mathcal O_F)$ as $\k$-algebras, thus $\k(\yy)\simeq\k(\xx)$ and $Z=\p^2_{\k}$.

    We have seen that $\eta_{\yy}$ contracts divisors $F$ and we can see that
    $$F \sim \eta_{\xx}^*(-d_EK_{\mathbb{P}^2_\k})-(d_E+1)E$$ which implies $F \equiv -7\eta_{\xx}^*(K_{\mathbb{P}^2_\k})-8E$.

	\autoref{Geiser:2} Let
		\[
			\begin{tikzcd}
			&X\ar[dl,"\eta_{\xx}",swap]\ar[dr, "\eta_{\yy}"]&\\
			\p^2_\k \ar[rr,dashed,"\chi"]&&\p^2_\k
			\end{tikzcd}
		\]
	be the link of type II and
	let $\gamma$ be the Geiser involution on $X$. It is an involution and $\eta_{\xx} \circ\gamma\circ\eta^{-1}_{\xx}$ is a birational involution of $\p^2_\k$ with the same base-point as $\chi$. Therefore, $\chi^{-1}\circ \eta\circ\gamma\circ\eta^{-1}$ is an automorphism of $\p^2_{\k}$.
\end{proof}

\begin{example}\label{ex:PointOfDeg7}
	Let $\k$ be an imperfect field of characteristic $p=7$ and let $t \in \k\setminus \k^7$.
	Let $\pi \colon Y \to \mathbb{P}^2_\k$ be the blow-up of the closed point $\xx=(x^7-t, x^3-xy-1)$ with purely inseparable residue field. Let $C=(x^3-xyz-z^3=0)\subset\p^2_\k$.
	We show that $\xx$ is in Sarkisov general position by applying \autoref{thm: pts_Sarkisov_general_position}.
	%Write $\yy=(\xx_{\bar\k})_{\rm red}=[t^{1/7}:t^{2/7}-t^{-1/7}:1]\in\p^2_{\bar\k}$.

	The following geometric proof was suggested by one of the referees. Without loss of generality, we can suppose $\k$ is separably closed.
	The curve $C$ is a nodal cubic with a smooth rational point, and thus the smooth locus is isomorphic to $\mathbb{G}_{m,\k}$.
	The subscheme of the smooth locus of $C$ consisting of closed points for which there is a curve of degree $d$ with intersection multiplicity $3d$ at such points is a torsor under the $3d$-torsion $\mu_{3d, \k}$ of $\Pic^0(C) \simeq \mathbb{G}_{m,k}$. By \autoref{thm: pts_Sarkisov_general_position}, if $\xx$ were not in Sarkisov general position, then the degree of the residue field of $\xx$ would divide $3d$, reaching a contradiction.
	%since the degree of $\xx$ is 7. does not divide $3d$, then the residue field is necessarily separable.

	%Let $D$ be a line given by $(ax+by+cz=0)$ in $\p^2_{\bar\k}$. To show that $I_\yy(D,C_{\bar\k})\leq 2$, we show that $D|_{C_{\bar\k}}$ is not given by $(x-t^{1/7})^3$:
	%Note that $D|_{y=\frac{x^3-1}{x}}$ is given by
	%\begin{align*}
	%	0& =ax+b(x^2-x^{-1})+c \\
	%	\Leftrightarrow 0& = bx^3+ax^2+cx-b,
	%\end{align*}
	%which is not equal to any multiple of $(x-t^{1/7})^3=x^3-3x^2t^{1/7}+3xt^{2/7}-t^{3/7}$.

	%Let $D$ be a conic given by $ax^2+bxy+cxz+dy^2+eyz+fz^2$ in $\p^2_{\bar\k}$. Note that $D|_{y=\frac{x^3-1}{x}}$ is similarly given by \begin{align*}
	%	0&= dx^6+bx^5+(a+e)x^4+(c-2d)x^3-bx^2-ex+d+f,
	%\end{align*}
	%and comparing its coefficients of $x^5$ and $x^2$ with the ones of $(x-t^{1/7})^6$, one obtains $b=t^{1/7}$ and $-b=t^{4/7}$, which is impossible.
\end{example}

\begin{lemma}\label{lem:Bertini}
	Let $\k$ be any field and let $X_1$ be a geometrically integral del Pezzo surface of degree $1$ with $H^1(X, \mathcal{O}_X)=0$.
	\begin{enumerate}
		\item\label{Bertini:1}
		Suppose $X=X_1$ is a del Pezzo surface of degree $1$ obtained from the blow up $\eta\colon X\to\p^2_\k$ of $r$ points $\xx_1,\ldots,\xx_r$, with exceptional divisors $E_1,\ldots,E_r$.
		Then, the Bertini involution $\beta\in\Aut_\k(X)$ exists and it satisfies $\beta^*(E_i)=-2d_{E_i}K_{X}-E_i$, for $i=1,\ldots,r$.

		\item\label{Bertini:2}
		Let $\xx\in\p^2$ be a point of degree $8$ in Sarkisov general position.
		Then there exists an involution $\chi\colon\p^2_\k\rat\p^2_\k$ that is a link of type II
		\[
			\begin{tikzcd}
			&X\ar[dl,"\eta_{\xx}",swap]\ar[dr,"\eta_{\yy}"]&\\
			\p^2_{\k} \ar[rr,dashed,"\chi"]&&\p^2_{\k},
			\end{tikzcd}
		\]
		with base-point $\xx$ and the base-point of $\chi^{-1}$ is a point $\yy$ with $\k(\yy)\simeq\k(\xx)$.
		If $\eta_{\xx},\eta_{\yy}$ are the blow-ups of $\xx$ and $\yy$, respectively, let $L$ be the pullback by $\eta_{\xx}$ of a general line on $\p^2$ and $E$ the exceptional divisor of $\xx$. Then $\eta_{\yy}$ contracts a curve $F\equiv 48L-17E$.

		\item\label{Bertini:3} Let $\chi\colon \p^2_\k \rat \p^2_\k$ be a link of type II with a base-point of degree $8$. Then there exist automorphisms $\alpha,\gamma$ of $\p^2_\k$ such that $\gamma\circ \chi\circ\alpha$ is a birational involution.
	\end{enumerate}
\end{lemma}

\begin{proof}
We first note that the $X_1$ appearing in the statements satisfy $H^1(X, \mathcal{O}_X)=0$ and $\rho(X_1) \geq 2$, and thus the Bertini involution exists by \autoref{lem: Geiser_Bertini_involutions}.
We claim that $X$ is geometrically normal. As $\rho(X) \geq 2$ and $K_X^2=1$, this follows from \autoref{thm: dP-base-change-classification} %the isomorphism class $Y=X_{\overline{\k}}^N$ is $\mathbb{P}^1_{\overline{\k}} \times \mathbb{P}^1_{\overline{\k}}$. In particular, both extremal contractions of $X$ should be conic bundles, which is a contradiction. Thus in particular $X$ is geometrically normal.

As for \autoref{Bertini:2} and \autoref{Bertini:3}, the proof is analogous to the one of \autoref{lem:Geiser}. 
\end{proof}

\begin{example}\label{ex:PointOfDeg8}
	Let $\k$ be an imperfect field of characteristic $2$ and let $t\in \k\setminus\k^8$. The point $\xx\in\p^2_\k$ with $\mathfrak{m}_\xx=(x^8-t, xy-x^3+1)$ is in Sarkisov general position:
	Write $\yy=(\xx_{\bar\k})_{\rm red}=[t^{1/8}:t^{2/8}-t^{-1/8}:1]$ and $C=(xyz-x^3+z^3)$.
	Arguing as in \autoref{ex:PointOfDeg7}, we can show that the point $\xx$ is in Sarkisov general position.
%	\FB{erase from now on}
%	Similar to (ref to Example of point of degree 7), one can check that the conditions of \autoref{prop:GeneralPositionCriterion} concerning lines and conics are satisfied. For the condition on cubics, one can check that the cubics $D$ such that $D|_{y=\frac{x^3-1}{x}}$ is divisible by $(x-t^{1/8})^8$ form a pencil generated by
%	\[(xyz+x^3+z^3=0) \text{ and }(t^2(x^2y+xz^2+y^2z) + t(x^3+y^3+z^3) + x^2z+xy^2+yz^2=0).\]
%	Then one can check that no cubic of this pencil is singular at $\yy$.
\end{example}

Finally, we refine the classification \autoref{thm:classification links p>2}.

\begin{theorem}\label{thm: existence of Sarkisov links}
Let $\k$ be a field of characteristic $p>0$.
Let $\chi_{\xx,\yy}\colon \p^2_\k\rat \p^2_\k$ be a link of type II whose base point $\xx$ satisfies that $k(\xx)/\k$ is purely inseparable. Then the following hold:
\begin{enumerate}
\item\label{existence bertini} If $p=2$ and $[\k(\xx):\k]=8$, then $k(\yy)\simeq k(\xx)$.
\item\label{existence quadric} If $p=3$ and $[\k(\xx):\k]=3$, then $k(\yy)\simeq k(\xx)$.
\item\label{existence geiser} If $p=7$ and $[\k(\xx):\k]=7$, then $k(\yy)\simeq k(\xx)$.
\end{enumerate}
Moreover, if $\k$ is imperfect, every kind of
\begin{enumerate}[label=(\alph*)]
	\item\label{Mfs exist} Mori fibre space,
	\item\label{links I exist} link of type I and III and
	\item\label{links II exist} link of type II between del Pezzo surfaces of Picard rank $1$, with base-point of degree $\leq7$.
\end{enumerate}
listed in \autoref{thm:classification links p>2} exists.
\end{theorem}
\begin{proof}
\autoref{existence bertini} follows from \autoref{lem:Bertini}.
\autoref{existence quadric} is \autoref{prop:link_P2_33}.
\autoref{existence geiser} follows from \autoref{lem:Geiser}.

Let us show \ref{Mfs exist}--\ref{links II exist}.
If $p\in\{3,7\}$, links $\chi_{\xx,\yy}\colon \p^2\rat\p^2$ of type II with $[k(\xx):\k]=p$ exist by \autoref{ex:3-3-link} and \autoref{ex:PointOfDeg7}.

If $p=5$, links $\chi_{\xx,\yy}\colon\p^2\rat X_5$ with $[k(\xx):\k]=5$ and $k(\yy)=\k$ exist by taking $\mathfrak m_\xx=(x^5-t,y-x^2)$ with $t\in\k\setminus\k^5$ in \autoref{lem:deg-5-conic-or-line} and applying \autoref{prop:link_P2_51--15}. In particular, rational del Pezzo surfaces of degree $5$ and Picard rank $1$ exist.

For $p=2$, we need to show that every object depicted in \autoref{fig:links_char2} exists.
Links $\p^2_{\k}\rat\p^2_{\k}$ with base point of degree $8$ exist by \autoref{ex:PointOfDeg8}.
A link $X_5\to\p^2_{\k}$ of type I exists by \autoref{ex: Mori-conic-double-lines}. In particular, Mori fibre spaces $X_5/\p^1$ exist.
Links $\p^2_{\k}\rat Q$ to a quadric surface $Q\subset\p^3_{\k}$ exist by \autoref{lem:general position deg 2}, and links $X_6\to Q$ of type I exist by \autoref{lem:point deg2 on quadric general position}. In particular, rational Mori fibre spaces $X_6/\p^1_{\k}$ exist.
Moreover, links $\chi_{\xx,\yy}\colon Q\rat Q$ of type II with $[k(\xx):\k]=4$ exist by \autoref{ex: X8_44}.
\end{proof}

\section{Applications to the planar Cremona groups} \label{s-applications}

In this section, we present three applications of the Sarkisov program and the classification of Sarkisov links: the existence of explicit quotients of the planar Cremona group, the classification of its connected smooth algebraic subgroups and the generation by involutions.

%=================================
%%%%%%%%%%%%%%
\subsection{Quotients of the Cremona group} \label{ss-quotients}

In this section, we construct explicit normal subgroups of the plane Cremona group over any field following the strategy of \cite{LZ20, SchneiderRelations}.

We say that two Mori fibre spaces $X/\p^1_{\k}$, $X'/\p^1_{\k}$ are \emph{equivalent} if they are birational over $\P^1_{\k}$, and write $[X/\p^1_{\k}]$ for their equivalence class.
For $d\in\{5,6,8\}$, denote by $I_d$ the set of such equivalence classes $[X]$ with $K_X^2=d$.
For each class $M=[X/\p^1_{\k}]$ in $I_d$, denote by $\mathcal{N}_M$ the set of integers $e\geq d$ such that there exists a point of degree $e$ on $X$ in Sarkisov general position with respect to $\mathbb{P}^1$.

We denote by $\mathcal{B}$ the set of equivalence classes of points of degree $8$ in Sarkisov general position in $\p^2_\k$ up to automorphisms of $\p^2_\k$.

The following homomorphism exists also if $[\bar\k:\k]\leq2$ but it is trivial. If $[\bar\k:\k]\geq3$, then $\bar\k$ is an infinite extension of $\k$ by \cite{AS1927}.
In particular, this implies that  $\mathcal{N}_{[\F_n/\mathbb{P}^1_{\k}]}$ is countably infinite by \autoref{ex: elementary transformation}.

\begin{theorem}\label{thm: quotients_p2}
	Let $\k$ be any field such that $[\bar\k:\k]\geq3$.
	There exists a non-trivial surjective group homomorphism
	\[
	\phi\colon \Bir(\p^2_{\k})\to (\bigoplus_{\mathcal{N}_{[\F_0/\p^1]}} \Z/2\Z) \ast\left(\Conv_{M\in I_5}(\bigoplus_{\mathcal{N}_{M}} \Z/2\Z)\right) \ast\left(\Conv_{M\in I_6}(\bigoplus_{\mathcal{N}_{M}} \Z/2\Z)\right) \ast\left(\Conv_{\mathcal{B}}\Z/2\right),
	\]
	that is the restriction of a groupoid homomorphism from the groupoid
	\[
	\BirMori_k(\p^2)\to (\bigoplus_{\mathcal{N}_{[\F_0/\p^1]}} \Z/2\Z) \ast\left(\Conv_{M\in I_5}(\bigoplus_{\mathcal{N}_{M}} \Z/2\Z)\right) \ast\left(\Conv_{M\in I_6}(\bigoplus_{\mathcal{N}_{M}} \Z/2\Z)\right) \ast\left(\Conv_\mathcal{B}\Z/2\right),
	\]
	which sends
	\begin{itemize}
		\item any Sarkisov link of type \II between Hirzebruch surfaces with base-point of degree $e\in\mathcal{N}_{[\F_0/\p^1]}$ onto the generator $1_e\in\bigoplus_{\mathcal{N}_{[\F_0/\p^1]}} \Z/2\Z$ indexed by $e\in\mathcal{N}_{[\F_0/\p^1]}$ of the first free factor;
		\item any Sarkisov link of type \II between conic fibrations $X\rat X'$ in $I_d$, where $d\in\{5,6\}$, with base-point of degree $e\in\mathcal{N}_{[X]}$ onto the generator $1_e\in\bigoplus_{\mathcal{N}_{[X]}} \Z/2\Z$ of the free factor indexed by $[X]=[X']\in I_d$.
		\item any Bertini involution with base-point of degree $8$ onto $1_b\in \Conv_{\mathcal{B}}\Z/2$ indexed by its class $b\in \mathcal{B}$;
		\item and all other Sarkisov links and isomorphisms of Mori fibre spaces onto $0$.
	\end{itemize}
	Moreover, $\mathcal{N}_{[\F_0/\p^1]}$ is countably infinite.
\end{theorem}

\begin{proof}
	The groupoid $\BirMori(\p^2_{\k})$ is generated by isomorphisms and Sarkisov links by \autoref{t-sarkisov-program} and by \autoref{thm:classification links p>2} and \autoref{thm:elementary relations}.
	The generating relations for this set of generators are the trivial relations and the ones listed in \autoref{thm:elementary relations}.
	We check that they are all sent onto zero by the map
	\[
	\BirMori(\p^2_{\k})\to (\bigoplus_{\mathcal{N}_{[\F_0/\p^1]}} \Z/2\Z) \ast\left(\ast_{M\in I_5}(\bigoplus_{\mathcal{N}_{M}} \Z/2\Z)\right) \ast\left(\ast_{M\in I_6}(\bigoplus_{\mathcal{N}_{M}} \Z/2\Z)\right) \ast\left(\ast_B\Z/2\right),
	\]
	defined in the statement: The only relevant elementary relation is the one from Figure~\ref{fig:XB_ab} because all others involve only maps sent onto zero (here we use that $e\geq K_X^2$ for $e\in \mathcal{N}_{[X/\p^1_\k]}$ for each conic bundle $X/\p^1_\k$, and that Bertini involutions with a base-point of degree $8$ do not appear in any non-trivial relation).
	
	Thus this map is a homomorphism of groupoids and restricts to a group homomorphism
	\[
	\Bir(\p^2_{\k})\to (\bigoplus_{\mathcal{N}_{[\F_0/\p^1]}} \Z/2\Z) \ast\left(\ast_{M\in I_5}(\bigoplus_{\mathcal{N}_{M}} \Z/2\Z)\right) \ast\left(\ast_{M\in I_6}(\bigoplus_{\mathcal{N}_{M}} \Z/2\Z)\right) \ast\left(\ast_B\Z/2\right).
	\]
	As we only take points in Sarkisov general position in the definition of $\mathcal{N}_{[X]}$, it is immediate that the \textit{groupoid} homomorphism is surjective.
	Let us show that also the restriction to the \textit{group} homomorphism is surjective.

	First of all, observe that by \autoref{lem:Bertini} for any $b\in \mathcal{B}$ there is a Bertini involution $\p^2_\k\rat\p^2_\k$ with base-point represented by $b$, surjecting onto $1_b$.

	For any $e\in\mathcal{N}_{[\F_0/\p^1]}$, there is an elementary transformation $\chi_e\colon\F_0\rat\F_e$ whose base-point is of degree $e$. We consider the birational map
	\[
	\varphi\colon \p^2_{\k}\longleftarrow\F_1\rat\F_0\stackrel{\chi_e}\rat\F_e\rat\F_{e-1}\rat\cdots\rat\F_1\to\p^2_{\k}
	\]
	where the first link $\chi_e$ is the blow-up of a point of degree $e$ in Sarkisov general position over $\p^1_{\k}$ and the remaining links are elementary transformations with a rational base-point. Then $\varphi$ is mapped to $1_e$.

	Let $M\in I_5$ (resp. $M\in I_6$). For any $e\in\mathcal N_{M}$ there is an elementary transformation $X_5\rat X_5'$ (resp. $X_6\rat X_6'$) whose base-point is of degree $e$.
	%, see \autoref{lem:links type II conic fibrations}.
	By \autoref{thm: classification_rational_surfaces}, there exist birational morphisms $X_5\to \p^2_{\k}$ and $X_5'\to\p^2_{\k}$, each contracting a curve onto a point of degree $4$. The resulting birational map of $\p^2_{\k}$ is mapped onto $1_e$.
	By \autoref{thm: classification_rational_surfaces}, there exist birational morphisms $X_6\to X_8$ and $X_6'\to X_8'$, each contracting a point of degree $2$. There exist Sarkisov links $X_8\rat\p^2_{\k}$ and $X_8'\rat \p^2_{\k}$. Indeed, since $X_8$ is rational, it has a rational point ${\bf z}$ by the Lang-Nishimura theorem (see for instance \cite[Theorem 3.6.1]{Poo17}) and its blow-up induces a Sarkisov link $\chi_{{\bf z},\yy}\colon X_8\rat \p^2_{\k}$ of type II with $\deg(\yy)=2$ by \autoref{prop: link II dP} and Châtelet's theorem. The resulting birational map of $\p^2_{\k}$ is sent onto $1_e$.
\end{proof}

\begin{remark}\label{rem: quotient-emptyEquivalenceClass}
Suppose that $\rm{char}(\k)=p=2$ and let $X\to\p^2_{\k}$ be the blow-up of $(x^2-s, y^2-t)$ for some $p$-linearly independent $t,s\in\k\setminus\k^2$.
Then $[X/\p^1_{\k}]\in I_5$ and $\mathcal N_{[X/\p^1_{\k}]}=\varnothing$, as explained in \autoref{ex: Mori-conic-double-lines}.
\end{remark}

%===========================================================

\subsection{Smooth connected algebraic subgroups of $\Bir(\p^2_{\k})$}

As a second application, we classify smooth connected algebraic subgroups of $\Bir(\mathbb{P}^2_\k)$.
Over an algebraically closed field the classification was achived by Enriques in \cite{Enr93} and the fourth author with Terperau extended it over perfect fields in \cite[Proposition 1.3]{RZ24}.

To study algebraic subgroups, we need Weil's regularisation theorem for surfaces, whose statement and proof, due to Zariski and Lipman, we now recall.

\begin{theorem} \label{thm: regularisation_thm}
	Let $\k$ be a field and let $X$ be a geometrically integral quasi-projective surface over $\k$.
	Let $G$ be a connected smooth algebraic group over $\k$ acting rationally on $X$.
	Then $X$ is $G$-birationally isomorphic to a regular projective $G$-surface $Y$.
\end{theorem}

\begin{proof}
The existence of a normal projective variety $Y'$ on which $G$ acts regularly is shown in \cite[Corollary 1.3]{Bri17}.
A resolution of singularities $Y \to Y'$ can be obtained by successively blowing-up closed points and normalising (\cite[Remark B, page 155]{Lip78}), so it is enough to show that every step of this process is $G$-equivariant.
As $G$ is smooth, it always lifts to the normalisation. So it is sufficient to show that a singular closed point of $X$ (with its reduced structure) is invariant under the action of $G$.
It is sufficient to check it over a separably closed field: as $G$ is smooth, $G(\k)$ is dense in $G$ and, as $G$ is connected, the closed singular point is invariant for $G$. This concludes the proof.
\end{proof}

We denote by $\Aut_{X/\k}$ the automorphism group scheme of a variety $X$ defined over $\k$ and by $\Aut^\circ_{X/{\k}}$ its connected component containing the identity.

\begin{theorem} \label{thm: max_alg_subgroups}
	Let $\k$ be a field and let $G$ be a smooth connected algebraic group acting rationally on $\mathbb{P}^2_\k$.
	Then there exists a $G$-birational map $\varphi \colon \mathbb{P}^2_\k \dashrightarrow X$ such that $X$ is a regular projective $G$-surface (equivalently, $\varphi G \varphi^{-1} \subset \Aut^\circ_{X/{\k}}$) in the following list:
	\begin{enumerate}
		\item\label{mas:1} $X \simeq \p^2_\k$;
		\item\label{mas:2} $X \simeq \mathbb{F}_n$ for $n \in \mathbb{N} \setminus \left\{1\right\}$;
		\item\label{mas:3} $X \simeq Q \subset \mathbb{P}^3_\k$ is a quadric (in particular, a del Pezzo of degree 8) of Picard rank 1;
		\item\label{mas:4} $X$ is a del Pezzo surface of degree 6 and Picard rank 1.
		\item $X$ is a del Pezzo surface of degree 5 and Picard rank 1 such that $X_{\bar\k}$ is not smooth. Such $X$ exist only if $p \leq 5$ and $\k$ is not perfect.
	\end{enumerate}
Moreover, the conjugacy classes of the group scheme $\Aut^\circ_{X/\k}$ of cases (1)--(5) are all pairwise disjoint, and \eqref{mas:1}, \eqref{mas:2} are maximal among smooth connected subgroups and \eqref{mas:3} (resp.\ \eqref{mas:4}) is maximal among smooth connected subgroups if and only if $Q_{\overline\k}\simeq\p^1_{\overline\k}\times\p^1_{\overline\k}$ (resp.\ $X$ is smooth).
\end{theorem}

\begin{proof}
	By \autoref{thm: regularisation_thm}, there exists a $G$-birational map $\psi \colon \mathbb{P}^2_\k \dashrightarrow Y$, where $Y$ is a regular projective $G$-surface.
	We run a $K_Y$-MMP $f \colon Y \to X$, which terminates with a Mori fibre space $X \to B$.
	As $G$ is connected, Blanchard's lemma \cite[Theorem 7.2.1]{Bri17b} shows that there is a unique (regular) action of $G$ on $X$ such that $f$ is $G$-equivariant.
	We divide the proof according to the dimension of $B$.

	\begin{enumerate}
	\item[(i)] \emph{$\dim(B)=0$}. In this case, $X$ is a rational del Pezzo surface with $\rho(X)=1$  and it is geometrically normal by \autoref{prop: birational-rigidity-non-norm-dP}.
	By \autoref{cor: rat_deg_5}, we know that $K_X^2 \geq 5$.
	Moreover, as $\rho(X)=1$, we can assume that $K_X^2 \neq 7$ by \autoref{lem:no-minimal-dp7}. If $K_X^2=9$, then $X \simeq \mathbb{P}^2_\k$ by \autoref{lem: SB} because $X$ is rational.
	If $K_X^2=8$, then $X$ is a quadric  by \autoref{prop: divisibility-dp8} and \autoref{lem: dp8-vs-quadric}.
	In the case $K_X^2=6$, there is nothing to add.
	If $K_X^2=5$
	and $X_{\bar{\k}}$ is smooth (which is automatic if $\k$ is perfect, or $p \geq 7$ by \autoref{rmk:weak-dp}), then $\Aut^\circ_{X/\k}$ is trivial as $\Aut^\circ_{X_{\bar{\k}}/\bar{\k}}$ is finite and smooth
	by \cite[Theorem 8.5.8]{Dol12} and \cite{MSRDPdP}.

	\item[(ii)] \emph{$\dim(B)=1$.} In this case, $X \to \p^1_{\k}$ is a Mori conic bundle.
	By the classification of Sarkisov links (see \autoref{thm: classification_rational_surfaces}), we deduce $K_X^2 \in \left\{5, 6, 8 \right\}$.
	If $K_X^2=8$, then $X=\F_n$ is a Hirzebruch surface because these are the only possible Mori conic bundles obtained from $\mathbb{P}^2_\k$ via the Sarkisov program: the only link of type I which has a degree 8 Mori conic bundle is $\mathbb{F}_1$ and a link of type II takes Hirzebruch surfaces to Hirzebruch surfaces.
%	If $K_X^2=8$, then $X \to \mathbb{P}^1_{\k}$ is geometrically a Hirzebruch surface $\mathbb{F}_n$ by \autoref{lem: Mori-conic-p>2}. \textcolor{red}{As the other extremal ray is fixed as well, we can deduce it is a twisted form of Hirzebruch. BE CAREFUL ABOUT CHAR 2.}
	If $n=1$, we have that $\Aut^\circ_{\mathbb{F}_1/k} \subset \Aut^\circ_{\mathbb{P}^2_{\k}/\k}$.
	In the two remaining cases, $X$ is a del Pezzo surface and by \autoref{prop: Mori-conic-bdl-dP} there exists a birational contraction to $\mathbb{P}^2_\k$ and thus $\Aut^\circ_{X/{\k}} \subset \Aut^\circ_{\mathbb{P}^2_{\k}/\k}$.
\end{enumerate}
The final assertion follows as $\Aut^\circ_{X/\k}$ of the surfaces in categories (1)-(5) are all non-isomorphic on the algebraic closure.

Let us show that \eqref{mas:1} and \eqref{mas:2} are maximal smooth subgroups and that \eqref{mas:3} is maximal smooth if and only if $Q_{\overline\k}\simeq\p^1_{\overline\k}\times\p^1_{\overline\k}$.
Let $X$ be one of the surfaces in the above list. Then $G=\Aut^\circ_{X/{\k}}$ is maximal if and only if any $G$-equivariant birational map $\varphi\colon X\dashrightarrow Y$ to a regular projective surface $Y$ is either an isomorphism or satisfies the property $\varphi G\varphi^{-1}=\Aut^\circ_{X/{\k}}$.
%So, let us take  a $G$-equivariant birational map $\varphi\colon X\dashrightarrow Y$ to a regular projective surface $Y$.
%We replace $Y$ by an $\Autz(Y)$-equivariant desingularisation REF. We then run the MMP (see \autoref{t-mmp-excellent}), which is $\Autz(Y)$-equivariant by Blanchard's lemma REF. So, we can assume that $Y$ is a Mori fibre space
%By \autoref{t-sarkisov-program}, any $G$-birational map $\varphi$ is the composition of $G$-equivariant Sarkisov links.

Suppose that $X=\p^2_\k$ or $X=\F_0$. Then $G(\overline\k)$ acts transitively on $X_{\overline\k}$ and hence the action of $\Aut^\circ_{X/k}$ has no fixed points in $X$. In particular, any $G$-equivariant birational map $X\dashrightarrow Y$ is an isomorphism.

Suppose that $X=\F_n$, $n\neq0,1$. Then $G(\overline\k)=\Aut^\circ_{\F_n/{\overline\k}}(\overline \k)$ acts with two orbits on $\F_n$, namely the special section and its complement.  In particular, any $G$-equivariant birational map $X\dashrightarrow Y$ is an isomorphism.

Suppose that $X=Q\subset\p^3_\k$ is a quadric surface.
If $Q$ is regular, not smooth, then the group scheme $G$ is not smooth by \cite[Corollary 2.2]{MSRDPdP}.
In the case $Q$ is smooth, then the group
 $G(\overline\k)$ acts transitively on $(Q_{\overline\k})(\overline \k)$. So, $G$ has no fixed points on $Q$ and again, any $G$-equivariant birational map $X\dashrightarrow Y$ is an isomorphism.

 Suppose $X$ is a regular del Pezzo of degree 6. If $\Aut^\circ_{X/k}$ is smooth, then $X$ is smooth by \cite[Corollary 2.2]{MSRDPdP} and $G(\overline{\k})$ acts transitively on a smooth del Pezzo surface of degree 6 outside of its $(-1)$-curves, showing that any $G$-equivariant birational $X\dashrightarrow Y$ is an isomorphism.
\end{proof}

\begin{remark}
\begin{enumerate}
\item Over a field of positive  characteristic, $\Aut^\circ_{X/{\k}}$ might not be a smooth group scheme over $\k$.
In the case where $X$ is a minimal regular rational surface, then $\Aut^\circ_{X/{\k}}$ is not smooth if and only if $X$ is not geometrically regular.
Indeed, if $X$ is not geometrically regular, it follows from \cite[Corollary 2.2]{MSweakdP} that $\Aut^\circ_{X/{\k}}$ is not smooth.
On the other hand, if $X$ is geometrically regular, then $\Aut^\circ_{X_{\bar{\k}}/{\bar{\k}}}$ is a smooth group as $K_X^2 \geq 5$ by \cite{MSRDPdP}.
\item One might be tempted to replace $\Aut^\circ_{X/{\k}}$ with its reduced subscheme and study the properties of $(\Aut^\circ_{X/{\k}})_{\text{red}}$.
We warn however that, over imperfect fields, the reduced subscheme of a group scheme is not always a sub-group scheme of $\Aut^\circ_{X/k}$. As an example, consider the imperfect field $k=\mathbb{F}_p(t)$ and let $h \colon \mathbb{G}_a \times \mathbb{G}_a \to \mathbb{G}_a$ defined by $(x, y) \to x^p+ty^p$. In this case $G= \ker(h)$ is a group scheme whose reduced part is not a subgroup scheme (cf. \cite[Proposition 1.6]{FS20b}).
\end{enumerate}
\end{remark}

We conclude by applying \autoref{thm: max_alg_subgroups} to the case of separably closed fields. Intriguingly, the (only possible) difference with the case of algebraically closed fields appears only for characteristic $p=2$ and $p=5$.

\begin{corollary} \label{cor: max_separably}
Let $\k$ be a separably closed field of characteristic $p$ and let $G$ be a smooth connected algebraic group acting rationally on $\mathbb{P}^2_\k$.
Then there exists a $G$-birational map $\varphi \colon \mathbb{P}^2_\k \dashrightarrow X$ such that $X$ is a regular projective $G$-surface in the following list
\begin{enumerate}
\item $X \simeq \p^2_\k$;
\item $X \simeq \mathbb{F}_n$ for $n \in \mathbb{N} \setminus \left\{1\right\}$;
\item $X \simeq Q \subset \mathbb{P}^3_{\k}$ is a regular, not geometrically regular, quadric in $\p^3_\k$ with $\rho(X)=1$. Such surfaces exist only if $p=2$, $\k$ is not perfect and $X_{\bar\k}$ has an $A_1$-singularity.
\item $X$ is a regular, not geometrically regular, del Pezzo surface of degree $5$ with $\rho(X)=1$. Such surfaces exist only if $p=5$, $\k$ is not perfect and $X_{\bar{\k}}$ has an $A_4$-singularity.
\end{enumerate}
\end{corollary}

\begin{proof}
	This follows directly from \autoref{thm: max_alg_subgroups} and \autoref{cor: moredP-sepclosed}.
\end{proof}

\subsection{The plane Cremona group is generated by involutions}
\label{ss:involutions}

If $\k$ is algebraically closed, then $\Bir(\p^2_\k)$ is generated by $\Aut(\p^2_\k)\simeq\rm{PGL}_3(\k)=\rm{PSL}_3(\k)$ and the standard quadratic involutions $[x:y:z]\rat[yz:xz:xy]$ \cite{Cas1901}.
Since $\rm{PSL}_3(\k)$ is generated by involutions when $\k$ is algebraically closed, the plane Cremona group is generated by involutions.
In \cite{LS21}, Lamy and the third author extend this property to $\Bir(\p^2_\k)$ when $\k$ is a perfect field.
In this section we explain that their argument adapts rather easily to the case where $\k$ is a separably closed field of characteristic $p \geq 3$. If $p=2$, then there are rational conic fibrations whose geometry is very involved and which we do not want to study at this point (to be compared with \cite[Section~3.5]{LS21}).

\begin{theorem}\label{thm:involution}
	Let $\k$ be a separably closed field of characteristic $p\geq3$. Then $\Bir_\k(\p^2)$ is generated by involutions.
\end{theorem}

Recall that over any field $\k$, any birational map of $\Bir(\p^2_\k)$ is a composition of Sarkisov links by \autoref{t-sarkisov-program}.

\begin{definition}\label{def: Sarkisov length}
Let $\k$ be any field and $f\in\Bir(\p^2_{\k})$ an element. We denote by $\rm{sl}(f)$ the minimal number of Sarkisov links necessary to decompose $f$. It is also called {\em Sarkisov length}.
Since every birational map can be decomposed into Sarkisov links, this number exists and is finite.
We say that $f$ is {\em reducible} if we can write $f=f_2\circ f_1$ with $\rm{sl}(f_i)<\rm{sl}(f)$ for $i=1,2$. Otherwise we say that $f$ is {\em irreducible}.
\end{definition}

\begin{remark}\label{rmk:generation by irred}
Let $\k$ be any field.
A direct consequence of the Sarkisov program (see \autoref{t-sarkisov-program}) is that $\Bir(\p^2_{\k})$ is generated by its irreducible elements.
\end{remark}

\begin{definition}
Let $\k$ be any field.
An irreducible element of $\Bir(\p^2_{\k})$ is called {\em of del Pezzo type} if it admits a minimal decomposition into Sarkisov links of type~\II over $\Spec \k$ between del Pezzo surfaces.

An irreducible element of $\Bir(\p^2_{\k})$ is called {\em of fibering type} if it has a minimal decomposition into Sarkisov links containing a link of type~\I (and hence also of type~\III).
\end{definition}

\begin{remark}\label{rmk: irreducible dP}
Let $\k$ be a separably closed field of characteristic $p\geq3$.
\autoref{thm:classification links p>2} implies that any Sarkisov link of type II between rational del Pezzo surface of Picard rank $1$ is of the following or the inverse of one of the following, where $i=[k(\xx_i):\k]$ is the degree of $\xx_i$ and $X_e$ is a del Pezzo surface of degree $e=K_{X_e}^2$.
\begin{enumerate}
\item\label{link:1} $\chi_{\xx_i,\yy_i}\colon \p^2_\k\rat \p^2_\k$, $i\in\{3,7\}$;
\item\label{link:4} $\chi_{\xx_1,\xx_5}\colon X_5\rat \p^2_\k$;
\end{enumerate}
In particular, the Sarkisov length of an irreducible element $f\in\Bir(\p^2_{\k})$ of del Pezzo type is $\rm{sl}(f)\leq2$.
\end{remark}

By $\mathcal{J}$ we denote the group of birational maps of $\p^2_{\k}$ preserving the pencil of lines through $[1:0:0]$. It is conjugate to the group of birational maps of $\F_1$ preserving the fibration to $\p^1_{\k}$ and is isomorphic to $\mathcal{J}\simeq\rm{PGL}_2(\k(t))\rtimes\rm{PGL}_2(\k)$. Classically, it is called group of Jonquières maps.

\begin{lemma}\label{lem:generation for involutions}
	Let $\k$ be a separably closed field of characteristic $p\geq3$. If $f\in\Bir_\k(\p^2_{\k})$ is of fibering type, then there exist $\alpha,\beta\in\Aut_\k(\p^2_{\k})$ such that $\beta\circ f\circ\alpha \in\mathcal{J}$.

	In particular, $\Bir_\k(\p^2_{\k})$ is generated by $\Aut_\k(\p^2_{\k})$, $\mathcal{J}$ and irreducible elements $f$ of del Pezzo type.
\end{lemma}

\begin{proof}
	Let $f\in\Bir(\p^2_{\k})$ be of fibering type and let $f=\chi_m\circ\cdots\circ\chi_1$ be a minimal factorisation into Sarkisov links $\chi_i$ with $m\geq 2$.
	By \autoref{thm:classification links p>2}, the only minimal rational Mori conic bundles are Hirzebruch surfaces $\F_n/\p^1$, $n\geq0$. Therefore, $\chi_1\colon\p^2_\k\dashrightarrow\F_1$ is a link of type~\I and $\chi_m\colon\F_1\to\p^2_{\k}$ is a link of type~\III.

	If $m=2$, then $f\in\Aut_\k(\p^2_{\k})$.
	Assume now that $m\geq 3$. Since $f$ is irreducible the intermediate Sarkisov links $\chi_i\colon\F_{n_i}\dashrightarrow \F_{n_{i+1}}$ are between Hirzebruch surfaces, and they are of type~II since a minimal factorisation never contains the link of type~\IV given by the exchange of the ruling $\F_0\simeq\F_0$ (see~\cite[Proposition~2.5(1)]{LS21}).
	Therefore, $f$ sends the pencil of lines through the base point of $\chi_1$ onto the pencil of lines through the base point of $\chi_m^{-1}$. Therefore, $\alpha,\beta\in\Aut_\k(\p^2_{\k})$ can be chosen such that $\beta\circ f\circ\alpha \in\mathcal{J}$.

	The second claim follows since $\Bir_\k(\p^2_{\k})$ is generated by its irreducible elements, which are of del Pezzo type and/or of fibering type by \autoref{rmk:generation by irred}.
\end{proof}

\begin{lemma}[{\cite[Lemma 3.6]{LS21}}]\label{lem:aut-J-involutions}
	Let $\k$ be a any field. Then $\Aut_\k(\p^2_{\k})$ and $\mathcal{J}$ are contained in the subgroup of $\Bir_\k(\p^2_{\k})$ generated by involutions.
\end{lemma}

We are ready for the proof of \autoref{thm:involution}.

\begin{proof}[Proof of \autoref{thm:involution}]
	By \autoref{lem:generation for involutions}, $\Bir_\k(\p^2_{\k})$ is generated by $\Aut_\k(\p^2_{\k})$, $\mathcal{J}$ and irreducible elements $f$ of del Pezzo type.
	By \autoref{lem:aut-J-involutions}, $\Aut_\k(\p^2_{\k})$ and $\mathcal{J}$ are contained in the subgroup of $\Bir_\k(\p^2_{\k})$ generated by involutions.
	It remains to show that irreducible elements $f$ of del Pezzo type are in the subgroup of $\Bir_\k(\p^2_{\k})$ generated by involutions. \autoref{rmk: irreducible dP} lists all Sarkisov links over separably closed fields in characteristic $p\geq3$.

	$\bullet$ Consider the link $\chi_{\xx_3,\yy_3}\colon\p^2_{\k}\rat\p^2_{\k}$. By \autoref{lem:deg-3-conic-or-line}, there exists a rational conic $C$ through $\xx_3$. Let $\mathbf{z}\in C$ be a $\k$-rational point. Then $\xx_3$ and $\zz$ are together in Sarkisov general position by \autoref{thm: pts_Sarkisov_general_position}. Hence, the elementary relation $\piece{\p^2}13$ in \autoref{P2_13} exists and yields a decomposition $\chi_{\xx_3,\yy_3}=\chi_5\circ\cdots\circ\chi_1$, where $\chi_1^{-1},\chi_5\colon \F_1\to\p^2$ are the blow-up of $\mathbf{z}$, and $\chi_2,\chi_4$ are links of type II between Hirzebruch surfaces (preserving a fibration), and $\chi_3$ is a link of type IV.
	We conclude that $\chi_{\xx_3,\yy_3}$ is contained in the subgroup of $\Bir(\p^2_\k)$ generated by involutions.

	$\bullet$ The link $\chi_{\xx_7,\yy_7}\colon\p^2_{\k}\rat\p^2_{\k}$ is, up to an automorphism of $\p^2_\k$, an involution \autoref{lem:Geiser}\autoref{Geiser:2}.

	$\bullet$ By \autoref{rmk: irreducible dP}, the last type of irreducible element is $f =\chi_{\xx_5,\yy_1}\circ \chi_{\yy_1',\xx_5'}$.
	If $\yy_1=\yy_1'$, then $f\in\Aut(\p^2_\k)$. Suppose that $\yy_1\neq\yy_1'$. Then the point $\mathbf{z}_1:=\chi_{\yy_1',\xx_5'}^{-1}(\yy_1)\in\p^2$ is not contained in the conic through $\xx_5'$.
	\autoref{thm: pts_Sarkisov_general_position} implies that $\xx_5',\mathbf{z}_1$ are in general position and hence the elementary relation $\piece{\p^2}15$ in \autoref{P2_15} exists.
	It yields a decomposition $f=\chi_5\circ\cdots\circ \chi_1$, where the $\chi_1,\chi_5^{-1}$ are links of type I, $\chi_2,\chi_4$ are links of type II between Hirzebruch surfaces and $\chi_3$ is a link of type IV.
	We conclude that $f$ is contained in the subgroup of $\Bir(\p^2_\k)$ generated by involutions.
\end{proof}

%===========================================================
%\section{}
%\label{}

%===========================================================

%===========================================================

\bibliographystyle{amsalpha}
\bibliography{biblio}

\end{document}